\documentclass
[10pt
reqno,
centertags,
twoside,
openright]{amsbook}

\usepackage{amsmath, amsfonts, amsthm, amssymb, bbm, enumerate, stmaryrd}
\usepackage[retainorgcmds]{IEEEtrantools}
\usepackage[mathscr]{euscript}
\usepackage{graphicx}
\usepackage[all]{xy}
\xyoption{curve}
\SelectTips{cm}{}

\usepackage{wasysym}
\usepackage[english]{babel}
\usepackage[utf8]{inputenc}
\usepackage{paralist}
\usepackage{textcomp}
\usepackage[colorinlistoftodos]{todonotes}
\usepackage{caption}

\usepackage{fancyhdr}
\usepackage[bottom]{footmisc}
\usepackage[hyphens]{url}
\usepackage[
colorlinks=true, linkcolor=red, citecolor=green,
filecolor=black, urlcolor=cyan
]{hyperref}

\makeatletter
\def\@tocline#1#2#3#4#5#6#7{\relax
  \ifnum #1>\c@tocdepth
  \else
    \par \addpenalty\@secpenalty\addvspace{#2}
    \begingroup \hyphenpenalty\@M
    \@ifempty{#4}{%
      \@tempdima\csname r@tocindent\number#1\endcsname\relax
    }{%
      \@tempdima#4\relax
    }%
    \parindent\z@ \leftskip#3\relax \advance\leftskip\@tempdima\relax
    \rightskip\@pnumwidth plus4em \parfillskip-\@pnumwidth
    #5\leavevmode\hskip-\@tempdima #6\nobreak\relax
    \ifnum#1<0\hfill\else\dotfill\fi\hbox to\@pnumwidth{\@tocpagenum{#7}}\par
    \nobreak
    \endgroup
  \fi}
\makeatother

\calclayout
\makeatletter
\g@addto@macro{\thm@space@setup}{\thm@headfont{\bf}}

\newtheorem{lem}{Lemma}[section]
\newtheorem{prop}[lem]{Proposition}
\newtheorem{cor}[lem]{Corollary}
\newtheorem{thm}[lem]{Theorem}

\newtheorem*{Thm1}{Theorem~1}
\newtheorem*{Thm2}{Theorem~2}
\newtheorem*{Thm3}{Theorem~3}
\newtheorem*{Thm4}{Theorem~4}

\newtheorem*{Cor}{Corollary}

\newtheorem*{Prop1}{Proposition~1}
\newtheorem*{Prop2}{Proposition~2}
\newtheorem*{Conj}{Conjecture}

\theoremstyle{remark}
\newtheorem{rem}[lem]{Remark}

\theoremstyle{definition}
\newtheorem{exa}[lem]{Example}
\newtheorem{exas}[lem]{Examples}
\newtheorem{defn}[lem]{Definition}
\newtheorem{quest}[lem]{Question}
\newtheorem{nn}[lem]{}
\newtheorem{nota}[lem]{Notation}

\numberwithin{equation}{section}

\newtheorem{nameless}[lem]{}

\renewcommand{\thesection}{\arabic{chapter}.\arabic{section}}

\newcommand{\smatrix}[1]{\left[\begin{smallmatrix}#1\end{smallmatrix}\right]}

\renewcommand{\mod}{\operatorname{\mathsf{mod}}\nolimits}
\newcommand{\inj}{\operatorname{\mathsf{inj}}\nolimits}

\newcommand{\Ann}{\operatorname{Ann}\nolimits}

\newcommand{\proj}{\operatorname{\mathsf{proj}}\nolimits}
\newcommand{\Proj}{\operatorname{\mathsf{Proj}}\nolimits}

\newcommand{\Der}{\operatorname{Der}\nolimits}

\newcommand{\Add}{\operatorname{\mathsf{Add}}\nolimits}
\newcommand{\Res}{\operatorname{Res}\nolimits}

\newcommand{\Inj}{\operatorname{\mathsf{Inj}}\nolimits}

\newcommand{\id}{\operatorname{id}\nolimits}
\newcommand{\Id}{\operatorname{Id}\nolimits}

\newcommand{\Mod}{\operatorname{\mathsf{Mod}}\nolimits}

\newcommand{\Aut}{\operatorname{Aut}\nolimits}

\newcommand{\sfEnd}{\operatorname{\mathsf{End}}\nolimits}
\newcommand{\End}{\operatorname{End}\nolimits}
\newcommand{\Hom}{\operatorname{Hom}\nolimits}

\renewcommand{\Im}{\operatorname{Im}\nolimits}
\newcommand{\Ker}{\operatorname{Ker}\nolimits}

\newcommand{\Coker}{\operatorname{Coker}\nolimits}
\newcommand{\coker}{\operatorname{coker}\nolimits}

\renewcommand{\dim}{\operatorname{dim}\nolimits}
\newcommand{\gldim}{\operatorname{gldim}\nolimits}

\newcommand{\Ext}{\operatorname{Ext}\nolimits}
\newcommand{\Map}{\operatorname{Map}\nolimits}

\newcommand{\Tor}{\operatorname{Tor}\nolimits}

\newcommand{\ev}{\mathrm{ev}}

\newcommand{\Ch}{\operatorname{\mathbf{C}}\nolimits}
\newcommand{\HH}{\operatorname{HH}\nolimits}
\newcommand{\OH}{\operatorname{H}\nolimits}
\newcommand{\Mor}{\operatorname{Mor}\nolimits}
\newcommand{\Ob}{\operatorname{Ob}\nolimits}
\newcommand{\Fun}{\operatorname{\mathsf{Fun}}\nolimits}

\newcommand{\abs}[1]{\ensuremath{\left\vert#1\right\vert}}

\newenvironment{bibcomment}{%
  \item[]%
  \begingroup%
  \par%
  \parshape0%
}{%
  \par%
  \endgroup%
}

\def\li{\varprojlim}

\def\A{{\mathsf A}}
\def\B{{\mathsf B}}
\def\cB{{\mathcal B}}
\def\sfB{{\mathsf B}}
\def\BB{{\mathbb B}}
\def\C{{\mathsf C}}

\def\D{{\mathsf D}}

\def\H{{\mathcal H}}

\def\L{{\mathcal L}}
\def\scrL{{\mathscr L}}

\def\P{{\mathsf P}}

\def\PP{{\mathbb P}}

\def\U{{\mathsf U}}

\def\Z{{\mathbb Z}}

\def\scrL{{\mathscr L}}
\def\scrR{{\mathscr R}}
\def\scrX{{\mathscr X}}
\def\scrY{{\mathscr Y}}
\def\scrZ{{\mathscr Z}}

\def\op{\mathrm{op}}

\newdir{> }{{}*!/+5pt/@{>}}
\newdir{ <}{{}*!/-5pt/@{<}}
\newdir{>>> }{{}*!/+10pt/@{>}}

\entrymodifiers={+!!<0pt,\fontdimen22\textfont2>}

\begin{document}
%
%
\setcounter{footnote}{1}
\begin{titlepage}\centering
\vspace*{48pt}
\textbf{
{
{
\Large{MONOIDAL CATEGORIES AND THE GERSTENHABER BRACKET IN HOCHSCHILD COHOMOLOGY}
}
}
}\\
\vspace*{48pt} \large{Reiner Hermann\footnote{\textsc{
Reiner Hermann, Institutt for matematiske fag, NTNU, 7491 Trondheim, Norway
}

\hspace*{12pt}\textit{E-mail address}: \url{
%rhermann@math.uni-bielefeld.de
reiner.hermann@math.ntnu.no
}}}
\end{titlepage}
\pagenumbering{gobble}
\hspace*{2pt}\newpage

\pagenumbering{gobble}
\cleardoublepage
\thispagestyle{empty}
\addtocontents{toc}{\protect\setcounter{tocdepth}{0}}
\section*{Abstract}
In this monograph, we extend S.\,Schwede's exact sequence interpretation of the Gerstenhaber bracket in Hochschild cohomology to certain exact and monoidal categories. Therefore we establish an explicit description of an isomorphism by A.\,Neeman and V.\,Retakh, which links $\Ext$-groups with fundamental groups of categories of extensions and relies on expressing the fundamental group of a (small) category by means of the associated Quillen groupoid.

As a main result, we show that our construction behaves well with respect to structure preserving functors between exact monoidal categories. We use our main result to conclude, that the graded Lie bracket in Hochschild cohomology is an invariant under Morita equivalence. For quasi-triangular bialgebras, we further determine a significant part of the Lie bracket's kernel, and thereby prove a conjecture by L.~Menichi. Along the way, we introduce $n$-extension closed and entirely extension closed subcategories of abelian categories, and study some of their properties.\\
\vfill\noindent
\textbf{MSC 2010 Subject Classification:} Primary 16E40; Secondary 14F35, 16T05, 18D10, 18E10, 18G15.\\
\textbf{Keywords:} Exact categories; Gerstenhaber algebras; Hochschild cohomology; Homological algebra; Hopf algebras; Monoidal categories.

\frontmatter

\tableofcontents

\setcounter{page}{0}
\chapter*{Introduction}
\addtocontents{toc}{\protect\setcounter{tocdepth}{0}}
\section*{Background}
\addtocontents{toc}{\protect\setcounter{tocdepth}{1}}
In 1945, G.\,Hochschild published the first out of a consecutive sequence of three articles (\cite{Ho45}, \cite{Ho46}, \cite{Ho47}), where he introduced the fundamental cohomology theory one now refers to as the theory of \textit{Hochschild cohomology}. It became a major tool in, for instance, the study of associative algebras and their representations, and was extended to various other algebraic fields. In the past decades, Hochschild's cohomology theory has been studied deeply by a large number of authors approaching the topic with very different backgrounds and hence leading to multifaceted developments; just to mention a few, P.\,A.\,Bergh (\cite{Ber07}),\linebreak R.-O.\,Buchweitz (\cite{Bu03}), H.\,Cartan and S.\,Eilenberg (\cite{CaEi56}), M.\,Gerstenhaber (\cite{Ge63}), D.\,Happel (\cite{Ha89}), T.\,Holm (\cite{Holm00}), B.\,Keller (\cite{Ke04},), J.-L.\,Lo\-day (\cite{Lo98}), S.\,MacLane (\cite{MaL58}), D.\,Quillen (\cite{Qu68}), S.\,Schwede (\cite{Sch98}), N.\,Snashall and \O.\,Solberg (\cite{SnSo04}), R.\,Taillefer (\cite{Ta04}), M.\,van den Bergh (\cite{VdB98}), S.\,Witherspoon (\cite{Wi04}) and A.\,Zimmermann (\cite{Zi07}) are amongst those. Since it seems that G.\,Hochschild did not notice all higher structures appearing in his theory in the early stages, he restricted himself to the study of the behavior of the resulting abelian groups. The present monograph focusses on a specific type of algebraic structure that naturally arises within the theory of Hochschild cohomology, namely its Gerstenhaber algebra structure (cf. \cite{Ge63}, \cite{GeSch86}).

Let $A$ be an associative and unital algebra over a commutative ring $k$, and let $A^\ev = A \otimes_k A^\op$ be its \textit{enveloping algebra}. It is easy to see that $A$-$A$-bimodules with central $k$-action are equivalent to left $A^\ev$-modules. The \textit{bar resolution}
$$
\xymatrix@C=20pt{
\cdots \ar[r]^-{b_2} & A \otimes_k A \otimes_k A \ar[r]^-{b_1} & A \otimes_k A \ar[r]^-{b_0} & A \ar[r] & 0 \ ,
}
$$
denoted by $\mathbb B_A \rightarrow A \rightarrow 0$ for short, resolves $A$ as an $A^\ev$-module. Let us assume that $A$ is projective over $k$. Then $\mathbb B_A$ is even a resolution of $A$ by projective $A^\ev$-modules. For an $A^\ev$-module $M$, the $n$-th cohomology of the complex $\Hom_{A^\ev}(\mathbb B_A, M)$, denoted by $\HH^n(A,M)$, computes the module $\Ext^n_{A^\ev}(A,M)$. It is the $n$-th \textit{Hochschild cohomology group of $A$} (\textit{with coefficients in $M$}). One easily verifies, that $\HH^0(A,A)$ is the center of $A$. If $\HH^1(A,M) = 0$, then any \textit{derivation} $A \rightarrow M$ has to be an inner one. Without relying on the $k$-projectivity of $A$, G.\,Hochschild further proved that $A$ is separable over $k$ if, and only if, $\HH^n(A,M) = 0$ for all $n > 0$ and all bimodules $M$. This is the same as saying that $A$ is a projective $A^\ev$-module.

The \textit{Hochschild $($co$)$complex} $\Hom_{A^\ev}(\mathbb B_A, A)$ carries a rich structure. To begin with, the assignment 
$$
(f \cup g)(a_1 \otimes \cdots \otimes a_{m+n+2}) = f(a_1 \otimes \cdots \otimes a_{m} \otimes 1_A)g(1_A \otimes a_{m+1} \otimes \cdots \otimes a_{m+n})
$$
for $A^\ev$-linear homomorphisms $f: A^{\otimes_k m} \rightarrow A$, $g: A^{\otimes_k n} \rightarrow A$, turns $\Hom_{A^\ev}(\mathbb B_A, A)$ into a differential graded $k$-algebra. Thus, it defines a product on the graded module $\HH^\bullet(A,A)$, the so called \textit{cup product}. It seems that, for the most part, the cup product in Hochschild cohomology is quite well understood. Although it is still very hard to calculate in concrete examples, there are several equivalent abstract descriptions. Classically, the cup product on $\HH^\bullet(A,A)$ is (as explained above) defined in terms of the bar resolution $\mathbb B_A$. However, given any other projective resolution $\mathbb P_A \rightarrow A \rightarrow 0$, one may express the cup product by means of a Hochschild diagonal approximation $\mathbb P_A \rightarrow \mathbb P_A \otimes_A \mathbb P_A$. The conceptually right way though is to see the cup product as the Yoneda product on $\Ext^\bullet_{A^\ev}(A,A)$, or, equivalently (but more intuitively), as the composition of morphisms inside the graded endomorphism ring $\Hom_{\mathbf D(A^\ev)}(A,A[\bullet])$ of $A$ in the derived category $\mathbf D(A^\ev)$ of left $A^\ev$-modules; note that $\Hom_{\mathbf D(A^\ev)}(A,A[\bullet])$ is the graded endomorphism ring of the unit object in the tensor triangulated category $(\mathbf D(A^\ev), \otimes_A^{\mathbf L}, A)$. From an abstract point of view, the cup product's latter interpretation is a very satisfying one, since it is an intrinsic formulation that does not rely on any choices. No matter which way of approaching it, the cup product has an important property: it is \textit{graded commutative} in the sense that $x \cup y = (-1)^{\abs{x}\abs{y}}y \cup x$ for all homogeneous $x,y \in \HH^\bullet(A,A)$, a result, which is attributed to M.\,Gerstenhaber (see \cite{Ge63}, and also \cite{Su02}).

It seems that, for reasons that are still very unclear, the actual mystery of Hochschild's theory lies in the additional piece of structure living on the Hochschild complex $\Hom_{A^\ev}(\mathbb B_A,A)$. A part of this extra structure has been discovered by M.\,Gersten\-haber in 1963 which he presented in his ground-breaking article \cite{Ge63}. In this article, he defined the so called \textit{circle product} on $\Hom_{A^\ev}(\mathbb B_A,A)$, which we (in divergence to Gerstenhaber's classical terminology) will denote by $\bullet$. This product is, in general, highly non-associative and does not commute with the differentials, but, interestingly enough, gives rise to a well-defined bilinear map $\{-,-\}_A$ of degree $-1$ on $\HH^\bullet(A,A)$. Representativewise, it is given by the graded commutator
$$
\{f,g\}_A = f \bullet g - (-1)^{(m-1)(n-1)} g \bullet f
$$
with respect to the circle product, where $f: A^{\otimes_k(m+2)} \rightarrow A$, $g: A^{\otimes_k(n+2)} \rightarrow A$ are $A^\ev$-linear maps. Thanks to this \textit{bracket}, the shifted module $\HH^{\bullet+1}(A,A)$ becomes a graded Lie algebra over $k$. And even more than that, the Lie bracket is linked to the cup product on $\HH^\bullet(A,A)$ through the \textit{graded Poisson identity} (i.e., for any homogeneous element $x \in \HH^\bullet(A,A)$, the $k$-linear map $\{x,-\}_A$ is a graded derivation on $\HH^\bullet(A,A)$ of degree $\abs{x}-1$). Thus, the triple $(\HH^\bullet(A,A), \cup, \{-,-\}_A)$ became the prototype of what we call a \textit{Gerstenhaber algebra} today. As it turned out, it unfortunately might also be one of the most difficult to understand incarnations of its kind. In what follows, the map $\{-,-\}_A$ is referred to as the \textit{Gerstenhaber bracket on $\HH^\bullet(A,A)$}.
\vspace*{6pt}\\
\indent Various authors bemoaned the inaccessibility of Gerstenhaber's bracket in Hoch\-schild cohomology. For instance, A.\,A.\,Voronov offers the following com\-plaint in his lecture notes for a graduate course on \textit{Topics in mathematical physics} (cf. \cite{Vo01}):
\begin{quote}
\it{The formula for this bracket, defined by M.\,Gerstenhaber, is not really inspirational to me, in spite of all those years I have spent staring at it.}
\end{quote}
J.\,Stasheff tries to demystify the unapproachability of the bracket by wondering if the issue might be that the Hochschild complex itself is not a differential graded Lie algebra on its own (cf. \cite{St93}):
\begin{quote}
\it{In his pioneering work on deformation theory of associative algebras, Gerstenhaber created a bracket on the Hochschild cohomology $\HH^\bullet(A,A)$, but this bracket seemed to be rather a tour de force since it was not induced from a differential graded Lie algebra structure on the underlying complex.}
\end{quote}
In his notes on differential graded categories, B.\,Keller introduces Hochschild cohomology for dg categories $\A$ as the cohomology of the differential graded endomorphisms $\Id_\A \rightarrow \Id_\A$ of the identity functor of $\A$. He makes the following crucial observation (cf. \cite{Ke06}):
\begin{quote}
\it{The Hochschild cohomology is naturally interpreted as the homology of the complex $\mathcal Hom(\Id_\A, \Id_\A)$ computed in the dg category $\mathbf R \mathcal Hom(\A,\A)$, where $\Id_\A$ denotes the identity functor of $\A$ $[\ldots]$. Then the cup product corresponds to the composition $($whereas the Gerstenhaber bracket has no obvious interpretation$)$.}
\end{quote}
A related remark may be found in an article by A.\,V.\,Shepler and S.\,Witherspoon, wherein the authors state the following when analyzing the Gerstenhaber bracket in Hochschild cohomology for skew group algebras (cf. \cite{ShWi12}):
\begin{quote}
\it{The cup product has another description as Yoneda composition of extensions of modules, which can be transported to any other projective resolution. However, the Gerstenhaber
bracket has resisted such a general description.}
\end{quote}
Indeed, a big part of the mysteriousness of the Gerstenhaber bracket lies in the fact, that it could only be understood in terms of the bar resolution and the resulting Hochschild complex for a long time. Results like the theorem of Hochschild-Kostant-Rosenberg (see \cite{HKR62}) led to first attempts to unearth the bracket in the special case of smooth commutative algebras. It was then Stasheff (see \cite{St93}), who described the Gerstenhaber bracket in Hochschild cohomology in terms of coderivations and the canonical bracket given in this situation. More explicitly, he interpreted the Hochschild complex by means of a differential graded tensor coalgebra $T=T_k(A)$; the differential $D(\delta) = b_{\bullet - 2} \circ \delta - (-1)^{\abs{\delta}} \delta \circ b_{\bullet - 2}$ turns the graded module $\mathrm{Coder}_k^\bullet(T) \subseteq \End_k^\bullet(T)$ of $k$-linear graded coderivations into a differential graded one. In fact, thanks to the graded commutator, $(\mathrm{Coder}_k^\bullet(T), D)$ is a differential graded Lie algebra. A more enhanced interpretation was introduced by S.\,Schwede in \cite{Sch98} describing the Gerstenhaber bracket in terms of bimodule extensions, that is, in terms of sequences of morphisms in the abelian monoidal category $(\Mod(A^\ev),\otimes_A,A)$. Finally, B.\,Keller offers in \cite{Ke04} an interpretation of the bracket in terms of the Lie algebra associated to the derived Picard group $\mathbf{DPic}^\bullet_A$ being the first and single attempt known to the author trying to describe the Gerstenhaber bracket in the world of triangulated categories. Keller's result implies that the Gerstenhaber bracket is a derived invariant, which (in finite characteristic) has been extended to the whole \textit{restricted Lie algebra structure} on $\HH^{\bullet+1}(A,A)$ (see \cite{Zi07}). However, it is still widely open, how one may \textit{intrinsically} see the Gerstenhaber bracket in the derived category $\mathbf D(A^\ev)$, that is, how to see the morphism $\{f,g\}_A : A \rightarrow A[m+n-1]$ for given morphisms $f: A \rightarrow A[m]$, $g: A \rightarrow A[n]$. This monograph does not deal with this question, but will provide a generalization of Schwede's construction for exact and monoidal categories, and state some of its applications -- both of which we will expose in the following section.
\addtocontents{toc}{\protect\setcounter{tocdepth}{0}}
\section*{Main results}
\addtocontents{toc}{\protect\setcounter{tocdepth}{1}}
Let $A$ be an associative and unital algebra over a commutative ring $k$ which is projective as a $k$-module. In 1998, Stefan Schwede gave an interpretation of the Gerstenhaber bracket in Hochschild cohomology by means of extensions in the abelian category of left modules over $A^\ev$ (cf. \cite{Sch98}). By regarding $\Mod(A^\ev)$ as a monoidal category with monoidal product functor $\otimes_A$ and unit object $A$, he introduced a map on self-extensions of the $A^\ev$-module $A$ (i.e., on self-extensions of the unit object),
$$
\boxtimes_A: \Ext^m_{A^\ev}(A,A) \times \Ext^n_{A^\ev}(A,A) \longrightarrow \Ext^{m+n}_{A^\ev}(A,A),
$$
leading to the following diagram of $(m+n)$-self-extensions of $A$
$$
\xymatrix@!C=25pt{
& \xi \boxtimes_A  \zeta \ar[dr] \ar[dl] &\\
\quad \xi \circ \zeta \quad & & (-1)^{mn} \zeta  \circ \xi \ ,
}
$$
where $\xi \circ \zeta$ denotes the Yoneda product of the $m$-self-extension $\xi$ and the $n$-self-extension $\zeta$. By mirroring the diagram, Schwede obtained a loop
$$
\xymatrix@!C=25pt{
& \xi \boxtimes_A  \zeta \ar[dr] \ar[dl] &\\
\quad \xi \circ \zeta \quad & & (-1)^{mn} \zeta \circ \xi\\
& (-1)^{mn }\zeta \boxtimes_A  \xi \ar[ur] \ar[ul] &
}
$$
which he denoted by $\Omega(\xi, \zeta)$. Observe that if we let $\mathcal Ext^n_{A^\ev}(A,A)$ be the category of $n$-extensions (the morphisms are morphisms of complexes $(f_n)_{n \in \mathbb Z}$ whose bordering morphisms $f_{-1}$ and $f_n$ equal the identity of $A$), we therefore have obtained an element $\Omega(\xi, \zeta)$ inside $\pi_1(\mathcal Ext^{m+n}_{A^\ev}(A,A), \xi \circ \zeta)$ (this requires to reinterpret the fundamental group of a category in terms of the \textit{Quillen groupoid}; see \cite[\S 1]{Qu72}). To turn this loop into a $(m+n-1)$-self-extension of $A$, Schwede used the fundamental isomorphism
$$
\xymatrix@C=20pt{
\mu : \Ext^{n-1}_{A^\ev}(A,A) \ar[r]^-\sim & \pi_1 \mathcal Ext^{n}_{A^\ev}(A,A)
} \quad \text{(for $n \geq 1$)}
$$
established by V.\,Retakh. In fact, Retakh stated it in the following significantly stronger version (cf. \cite{Re86}): Let $\A$ be an abelian category. Then for $X, Y \in \Ob \A$, 
$$
\xymatrix@C=20pt{
\Ext^{n-i}_{\A}(X,Y) \ar[r]^-\sim & \pi_i \mathcal Ext^{n}_{\A}(X,Y)
} \quad \text{(for $n \geq 1$ and $0 \leq i \leq n$)}.
$$
S.\,Schwede explicitly described the isomorphism $\mu$ by using the fact that $\Mod(A^\ev)$ has enough projective objects. Finally, the composition $[-,-]_A = \mu^{-1}\Omega(-,-)$ defines a map
$$
[-,-]_A : \Ext^{m}_{A^\ev}(A,A) \times \Ext^{n}_{A^\ev}(A,A) \longrightarrow \Ext^{m+n-1}_{A^\ev}(A,A),
$$
called the \textit{loop bracket}. Schwede proved that, under the canonical isomorphism $\HH^\bullet(A,A) \cong \Ext^\bullet_{A^\ev}(A,A)$, his loop bracket on $\Ext^\bullet_{A^\ev}(A,A)$ agrees with the Gerstenhaber bracket on $\HH^\bullet(A,A)$ (up to a sign; cf. \cite[Thm.\,3.1]{Sch98}). This monograph is aiming for widening Schwede's construction to monoidal categories (remember, that the category of $A^\ev$-modules is such a category) that do not necessarily have enough projective objects. 

To do so, we first find that a generalization of Retakh's isomorphism to \textit{Waldhausen categories} was presented by A.\,Neeman and V.\,Retakh in \cite{NeRe96}. Both of the results \cite{Re86} and \cite{NeRe96} are proven from a homotopy theoretical point of view and rely on maps that are given fairly abstractly. Inspired by some of their ideas, and by generalizing several of S.\,Schwede's observations, we will prove the following theorem.

\begin{Thm1}[= Lemmas \ref{lem:u_injective}, \ref{lem:u_surjective}, \ref{lem:isopi1} and Theorem \ref{thm:iso_pi0pi1}] 
Fix an integer $n \geq 1$. Let $\C$ be a factorizing exact category, let $X,Y$ be objects in $\C$, and let $\xi$ be an admissible $n$-extension of $X$ by $Y$. Then there there are explicit and mutually inverse isomorphisms
$$
\mathrm{NR}_\C^+: \Ext^{n-1}_\C(X,Y) \longrightarrow \pi_1(\mathcal Ext^{n}_\C(X,Y), \xi)
$$
and
$$
\mathrm{NR}_\C^-: \pi_1(\mathcal Ext^{n}_\C(X,Y), \xi) \longrightarrow \Ext^{n-1}_\C(X,Y).
$$
\end{Thm1}

See Definition \ref{def:factorizing} for an explanation of the term factorizing exact category. For now, we only remark that every abelian category is factorizing. Now assume that $\C$ is not only exact, but also \textit{monoidal}, i.e., that it comes equipped with \textit{monoidal product functor} $\otimes_\C: \C \times \C \rightarrow \C$ and a \textit{unit object} $\mathbbm 1_\C$ (with respect to $\otimes_\C$). Furthermore, we have to assume that $\C$ is (what we call) a \textit{strong exact monoidal category} (see Definition \ref{def:exactmonocat}). By basically copying S.\,Schwede's construction (however, one has to put some thought into the exact details), we construct a loop $\Omega_\C(\xi, \zeta)$ in $\mathcal Ext^{m+n}_\C(\mathbbm 1_\C, \mathbbm 1_\C)$ based at the Yoneda composition $\xi \circ \zeta$ of $\xi$ and $\zeta$, as well as a loop $\square_\C(\xi)$ based at $\xi \circ \xi$ if $n$ is even and $m = n$.
Thus, we obtain maps
\begin{align*}
[-,-]_\C &: \Ext^{m}_\C(\mathbbm 1_\C, \mathbbm 1_\C) \times \Ext^{n}_\C(\mathbbm 1_\C, \mathbbm 1_\C) \longrightarrow \Ext^{m+n-1}_\C(\mathbbm 1_\C, \mathbbm 1_\C)\\
sq_\C &:  \Ext^{2n}_{\C}(\mathbbm 1_\C, \mathbbm 1_\C) \longrightarrow \Ext^{4n-1}_{\C}(\mathbbm 1_\C, \mathbbm 1_\C)
\end{align*}
by assigning
$$
[\xi,\zeta]_\C = \mathrm{NR}_\C^- (\Omega_\C(\xi,\zeta)) \quad \text{and} \quad sq_\C(\xi) = \mathrm{NR}_\C^-(\square_\C(\xi)),
$$
where equivalence classes are implicitly taken when needed. As expected, our construction recovers Schwede's original one (see Lemma \ref{lem:consistancy}). The following observation is inspired by \cite{Ta04}.

\begin{Prop1}[$\subseteq$ Theorem \ref{thm:trivial_bracket}]
Assume that the strong exact monoidal category $\C$ is lax braided. Then the map $[-,-]_\C$ is constantly zero.
\end{Prop1}

For another strong exact monoidal category $\D$, we show the following main result. The proof heavily relies on the explicit form of the morphisms $\mathrm{NR}_\C^+$ and $\mathrm{NR}_\C^-$ exposed in Theorem 1.

\begin{Thm2}[= Theorem \ref{thm:bracketcomm} and Remark \ref{rem:thmcomrem}]
Let $\scrL: \C \rightarrow \D$ be an exact and almost strong $($or costrong$)$ monoidal functor. Then the induced graded algebra homomorphism $\Ext^\bullet_\C(\mathbbm 1_\C, \mathbbm 1_\C) \rightarrow \Ext^\bullet_\D(\mathbbm 1_\D, \mathbbm 1_\D)$ makes the diagrams
$$
\xymatrix@C=30pt{
\Ext^m_\C(\mathbbm 1_\C, \mathbbm 1_\C) \times \Ext^n_\C(\mathbbm 1_\C, \mathbbm 1_\C) \ar[r]^-{[-,-]_\C} \ar[d] & \Ext^{m+n-1}_\C(\mathbbm 1_\C, \mathbbm 1_\C) \ar[d]\\
\Ext^m_{\D}(\mathbbm 1_\D, \mathbbm 1_\D) \times \Ext^n_{\D}(\mathbbm 1_\D, \mathbbm 1_\D) \ar[r]^-{[-,-]_\D} & \Ext^{m+n-1}_{\D}(\mathbbm 1_\D,\mathbbm 1_\D)
}
$$
and
$$
\xymatrix@C=30pt{
\Ext^{2n}_\C(\mathbbm 1_\C, \mathbbm 1_\C) \ar[r]^-{{sq}_\C} \ar[d] & \Ext^{4n-1}_\C(\mathbbm 1_\C, \mathbbm 1_\C) \ar[d]\\
\Ext^{2n}_{\D}(\mathbbm 1_\D, \mathbbm 1_\D) \ar[r]^-{sq_{\D}} &
\Ext^{4n-1}_{\D}(\mathbbm 1_\D,\mathbbm 1_\D)
}
$$
commutative for all integers $m, n \geq 1$.
\end{Thm2}

The proof of Theorem 2 is involved and requires a lot of tedious calculations. That the effort pays off is illustrated by the upcoming Theorem 3, stating that the Gerstenhaber bracket is invariant under Morita equivalence. Recall that the Gerstenhaber bracket is already known to even be a derived invariant thanks to B.\,Keller. However, the methods Keller uses to manifest his result are completely different from ours.

\begin{Thm3}[= Theorem \ref{thm:moritaHH}]
Assume that $A$ and $B$ are Morita equivalent $k$-algebras of which one, and hence both, is supposed to be $k$-projective. Then there is an exact and almost strong monoidal equivalence inducing an isomorphism $\HH^\bullet(A,A) \rightarrow \HH^\bullet(B,B)$ of graded $k$-algebras such that the diagrams 
$$
\xymatrix@C=35pt{
\HH^m(A,A) \times \HH^n(A,A) \ar[r]^-{\{-,-\}_A} \ar[d]_-\cong & \HH^{m+n-1}(A,A) \ar[d]^-\cong \\
\HH^{m}(B,B) \times \HH^n(B,B) \ar[r]^-{\{-,-\}_B} & \HH^{m+n-1}(B,B)
}
$$and$$
\xymatrix@C=35pt{
\HH^{2n}(A,A) \ar[r]^{sq_A} \ar[d]_\cong & \HH^{4n-1}(A,A) \ar[d]^\cong\\
\HH^{2n}(B,B) \ar[r]^{sq_B} & \HH^{4n-1}(B,B)
}
$$
commute for any choice of integers $m,n \geq 1$.
\end{Thm3}

In order to show the result, we use the fact, that Morita equivalent algebras are characterized in terms of finitely generated projective generators, i.e., progenerators. Given such a progenerator for a pair of Morita equivalent algebras, we construct a progenerator for the corresponding enveloping algebras, which also is a comonoid in the category of bimodules it is living in. Thus, it will give rise to the desired monoidal functor.

In view of Proposition 1 and our Main Theorem 4, our attention is drawn towards the kernel of the Gerstenhaber bracket in Hochschild cohomology. Additional motivation is given by the following question raised by R.-O.\,Buchweitz.
\begin{quote}
\it{Can we determine the kernel of the adjoint representation $($map\-ping an element $x \in \HH^\bullet(A,A)$ to $\mathrm{ad}(x) = \{x,-\}_A)$, or, less restrictively, elements $x$ such that $$\{-,-\}_A : \HH^m(A,A) \otimes_k \HH^{n}(A,A) \longrightarrow \HH^{m+n-1}(A,A)$$ vanishes on $x \otimes x$?}
\end{quote}
Indeed, this is an intriguing problem. On the one hand, the graded Lie algebra $\mathfrak g^\bullet = (\HH^{\bullet + 1}(A,A), \{-,-\}_A)$ is of course a graded Lie representation over itself, and hence, one is interested in describing the annihilators $\Ann_{\mathfrak g^\bullet}(x) = \Ker(\mathrm{ad}(x))$ (for $x \in \HH^\bullet(A,A))$ and $\Ann_{\mathfrak g^\bullet}(\mathfrak g^\bullet) = \Ker(\mathrm{ad}(-))$ $($the latter is sometimes also called the \textit{center} $Z(\mathfrak g^\bullet)$ of the Lie algebra $\mathfrak g^\bullet)$. On the other hand, one knows that a deformation of the algebra $A$ arises from a so called \textit{$($non-commutative$)$ Poisson structure}, that is, an element $x \in \HH^2(A,A)$ with $\{x,x\}_A = 0$ (see \cite{Ber12}, \cite{Cr11}, \cite{HaTa10} and, above all, \cite{BlGe91}, \cite{Kon03}, \cite{Xu94}). Thus, a first step in understanding an algebra’s deformation theory lies in understanding the Gerstenhaber bracket's kernel. We will do so for a particular class of bialgebroids over $k$. For expository purposes, we state the result in the following weakened version.

\begin{Thm4}[$\subseteq$ Theorem \ref{thm:commutativeHopfalgebroid} and Corollary \ref{cor:vanishHopfalgebroid}]
Let $\mathcal B = (B, \nabla, \eta, \Delta, \varepsilon)$ be a quasi-triangular $($e.g.,~a cocommutative$)$ bialgebra over $k$ such that $B$ is $k$-projective. Then $\OH^\bullet(B,k) = \Ext^\bullet_B(k,k)$ identifies with a subalgebra of $\HH^\bullet(B,B)$ and the Gerstenhaber bracket $\{-,-\}_B$ on $\HH^\bullet(B,B)$ vanishes on $\OH^\bullet(B,k)$, i.e.,
$$
\{\alpha, \beta\}_B = 0 \quad \text{$($for all $\alpha, \beta \in \OH^\bullet(B,k))\,.$}
$$
\end{Thm4}
The theorem in particular implies the following conjecture raised by L.~Menichi in \cite[Conj.~23]{Me11}.
\begin{Conj}[$\subseteq$ Corollary \ref{cor:gerstenhabervanishhopfalgebra}]
Let $K$ be a field, and let $\mathcal B = (B, \nabla, \eta, \Delta, \varepsilon)$ be a quasi-triangular and finite dimensional bialgebra over $K$. The Lie bracket on the sub-Gerstenhaber algebra $\Ext^\bullet_B(K,K)$ of $\HH^\bullet(B,B)$ is trivial.
\end{Conj}
In \cite{Me11} L.~Menichi provides a proof of the conjecture in case $\mathcal B$ is a cocommutative Hopf algebra. It heavily relies on the fact, that, since being cocommutative, the Hopf algebra's antipode $S$ is involutive, i.e., $S \circ S = \id_B$. As we approach the question from a completely different direction, this assumption will turn out to not be needed.

 M.\,Linckelmann proved that if $\mathcal H$ is a commutative Hopf algebra such that it is a finitely generated projective $k$-module, then $\HH^\bullet(H,H) \cong H \otimes_k \OH^\bullet(H,k)$. We combine Linckelmann's observation with Theorem 4 to obtain the following intriguing insight. It is inspired by supportive examples which easily can be found in related literature (see for instance \cite{LeZh13}; we will, however, elaborate on these examples in the respective section of this monograph).

\begin{Cor}[= Corollary \ref{cor:centerdetermins}]
Let $\mathcal H = (H, \nabla, \eta, \Delta, \varepsilon, S)$ be a commutative and quasi-triangular Hopf algebra over $k$, such that $H$ is finitely generated projective over $k$. The Gerstenhaber bracket $\{-,-\}_H$ on $\HH^\bullet(H,H)$ is completely determined by the induced action of the center $Z(H) = H$ on the cohomology ring $\OH^\bullet(H,k)$.
%
\end{Cor}

If $\A$ is an abelian $k$-category, the category $\mathsf{End}_k(\A)$ of all $k$-linear endofunctors on $\A$ is a monoidal category in a natural way. The unit object, given by $\Id_\A$, gives rise to the $\Ext$-algebra $\Ext^\bullet_{\sfEnd_k(\A)}(\Id_\A, \Id_\A)$ which one may (and we will do so) call the \textit{Hochschild cohomology ring of $\A$}. If $\A$ is the category of left modules over an algebra, we prove the proposition below, which also shows, that the terminology is a sensible one. 

\begin{Prop2}[= Proposition \ref{abel:monoidal} and Corollary \ref{cor:abel:monoidal}]
Let $A$ be a $k$-algebra and let $\A$ be the category $\Mod(A)$. If $A$ is $k$-projective, then the left adjoint $\ev_A^\lambda$ of the evaluation functor $\ev_A: \mathsf{End}_k(\A) \rightarrow \Mod(A^\ev)$ defined by $\ev_A(\scrX) = \scrX(A)$ induces an isomorphism
$$
\xymatrix@C=20pt{
\HH^\bullet(A) \ar[r]^-{\sim} & \Ext^\bullet_{\mathsf{End}_k(\A)}(\Id_\A,\Id_\A)
}
$$
of graded $k$-algebras taking $\{-,-\}_A$ to $[-,-]_{\sfEnd_k(\A)}$.
\end{Prop2}
\addtocontents{toc}{\protect\setcounter{tocdepth}{0}}
\section*{Outline}
\addtocontents{toc}{\protect\setcounter{tocdepth}{1}}
In Chapter \ref{cha:prerequ}, we are going to recall the basic notions of Quillen exact categories and monoidal categories, including various relaxations of the standard definitions. Further, we will state some elementary observations on homological properties of exact categories, such as the existence of certain pullbacks and pushouts, and how functors between exact categories interact with them. We will conclude the chapter by a couple of examples being relevant for future investigations. Amongst those, the category of bimodules over a ring, and certain full subcategories of it will occur.

Chapter \ref{cha:extcats} focusses on the study of categories of admissible extensions over a fixed exact category $\C$. Following \cite{Sch98}, we will discuss some homotopical properties of those extension categories. The fundamental group of a (small) category admits a description as morphisms in the Quillen groupoid associated to it. We will recall the necessary setup to understand these constructions, and apply them to the categories of admissible extensions over $\C$. Before closing the chapter by defining $n$-extension closed subcategories, we will recall the so called Baer sum which turns $\Ext_\C^n(-,-)$ into a bifunctor with values in the category of abelian groups.

In \cite{Re86}, V.\,Retakh described isomorphisms, relating the $(n-1)$-st extension group with the fundamental group of the category of $n$-extensions (for all $n \geq 1$). In Chapter \ref{cha:retakh}, we are going to reestablish those isomorphisms in a very explicit manner, provided that the underlying exact category $\C$ is ``nice" enough. We will further discuss how these isomorphisms interact with functors between exact categories. If $\C$ is not only exact, but also monoidal, the corresponding monoidal product functor yields a candidate for an external operation on the category of $n$-extensions. We will discuss the interplay between this candidate and Retakh's isomorphisms.

Chapter \ref{cha:hochschild} addresses the classical theory of Hochschild cohomology. We will briefly define the cup product and the Gerstenhaber bracket at the level of the Hochschild complex, and explain why the Hochschild cohomology ring of an algebra is not only graded commutative, but in fact a so called strict Gerstenhaber algebra. We conclude the chapter by an example, which is widely known as the theorem of Hochschild-Kostant-Rosenberg.

Chapter \ref{ch:bracket} introduces a bracket and a squaring map on the $\Ext$-algebra of the tensor unit of a given exact and monoidal category. We will show that, firstly, by an argument presented in \cite{Ta04}, the bracket is constantly zero if the given monoidal category admits a braiding, and that, secondly, by the results obtained in Chapter \ref{cha:extcats}, the bracket and the squaring map behave nicely with respect to exact and monoidal functors. Afterwards, we will compare our construction with the original one by S.\,Schwede. As a result, we will conclude that if we apply our construction to the special case of bimodules over an algebra, we indeed recover Schwede's map (up to minor differences). This insight will enable us to prove that the Gerstenhaber bracket is an invariant under Morita equivalence. We are going to end the chapter by a short survey on the existence of braidings on the monoidal category of bimodules over an algebra.

In Chapter \ref{cha:app1}, we are going to describe parts of the kernel of the Gerstenhaber bracket in Hochschild cohomology for quasi-triangular bialgebras and, in particular, quasi-triangular Hopf algebras. The key ingredients leading to the result were developed in Chapter \ref{ch:bracket}, and its preceding chapters, and will be combined here. If $H$ is a finite dimensional commutative Hopf algebra over a field $k$, M.\,Linckelmann described its Hochschild cohomology ring in terms of its cohomology ring: $\HH^\bullet(H, H) \cong H \otimes_k \OH^\bullet(H,k)$. By previous considerations on the kernel of the Gerstenhaber bracket on $\HH^\bullet(H,H)$, we will deduce that it is determined by the images of the adjoint representation $\mathrm{ad}(h) = \{h,-\}_H$, $h \in H$, if $H$ is quasi-triangular.

Finally, in Chapter \ref{chap:identityfunc}, we will concern ourselves with Hochschild cohomology for abelian categories $\A$, provide a bracket in this setting (by using Chapter \ref{ch:bracket}), and conclude that it is the usual Gerstenhaber bracket in Hochschild cohomology if $\A$ is the category of left modules over an algebra being projective over the base ring.
\newpage
%
\begin{figure}[h!v!]
\centering
\small{\xymatrix@C=30pt@R=22pt{
&  &\\
& *+[F-:<3pt>]{\txt{Exact categories\\ (Section \ref{sec:exactcats})}} \ar@{-> }[d] \ar@{-> }[r] & *+[F-:<3pt>]{\txt{Exact and\\ monoidal categories\\ (Sections \ref{sec:monoidalcats}, \ref{sec:extcats_monoidal})}} \ar@/^2.3pc/@{->>> }[ddddl]\\
& *+[F-:<3pt>]{\txt{Homotopical properties \\ of extension categories\\ (Chapter \ref{cha:extcats})}} \ar@{-> }[d] &\\
*+[F--:<3pt>]{\txt{Graded Lie bracket\\ $\{-,-\}_A$ on $\HH^\bullet(A,A)$\\ (Gerstenhaber, 1963)\\ (Overview: Chapter \ref{cha:hochschild})}} \ar@{ <-> }[dd] & *+[F-:<3pt>]{\txt{Explicit description of \\ the Neeman-Retakh\\ isomorphism\\ (Chapter \ref{cha:retakh})}} \ar@{-> }[dd] & \\
& &\\
*+[F--:<3pt>]{\txt{Exact sequence inter-\\ pretation of $\{-,-\}_A$\\ (Schwede, 1998)}} \ar@{~> }[r] \ar@{-> }[d] & *+[F-:<3pt>]{\txt{Construction of\\ the bracket $[-,-]_\C$\\ (Chapter \ref{ch:bracket})}} \ar@{-> }[d] \ar@/_3pc/@{->>> }[l]_-{\C = \P(A)} \ar@{-> }[r] & *+[F-:<3pt>]{\txt{Vanishing of $[-,-]_\C$\\ for lax braided\\ categories\\ (Theorem \ref{thm:trivial_bracket})}} \ar@{-> }[d]\\
*+[F-:<3pt>]{\txt{Morita invariance\\ of $\{-,-\}_A$\\ (Theorem \ref{thm:moritaHH})}} & *+[F-:<3pt>]{\txt{Abstract structure\\ theorem for $[-,-]_\C$\\ (Theorem \ref{thm:bracketcomm})}} \ar@{-> }[dd] \ar@{-> }[r] \ar@/^3pc/@{->>> }[l]^-{\textmd{+ Morita theory}} & *+[F-:<3pt>]{\txt{Behaviour of $\{-,-\}_B$\\ for (pre-triangular)\\ bialgebroids $B$\\ (Theorem \ref{thm:commutativeHopfalgebroid})}} \ar@{-> }[d]\\
 & & *+[F-:<3pt>]{\txt{Vanishing of $\{-,-\}_B$ on\\ $\OH^\bullet(B,k)$ for pre-triangular\\ bialgebras $B$ (as conjec-\\ tured by Menichi, 2011)\\ (Corollary \ref{cor:gerstenhabervanishhopfalgebra})}} \ar@{-> }[d]\\
 & *+[F-:<3pt>]{\txt{A bracket for\\ endofunctors on\\ an abelian category\\ (Chapter \ref{chap:identityfunc})}} & *+[F-:<3pt>]{\txt{Structure of $\{-,-\}_A$ for\\ commutative quasi-trian- \\ gular Hopf algebras $A$\\ (Section \ref{sec:linckelmann})}} \\
}}\vspace*{5pt}
\caption*{\textbf{Figure 1.} Map of the main constructions, techniques and results.}
\end{figure}
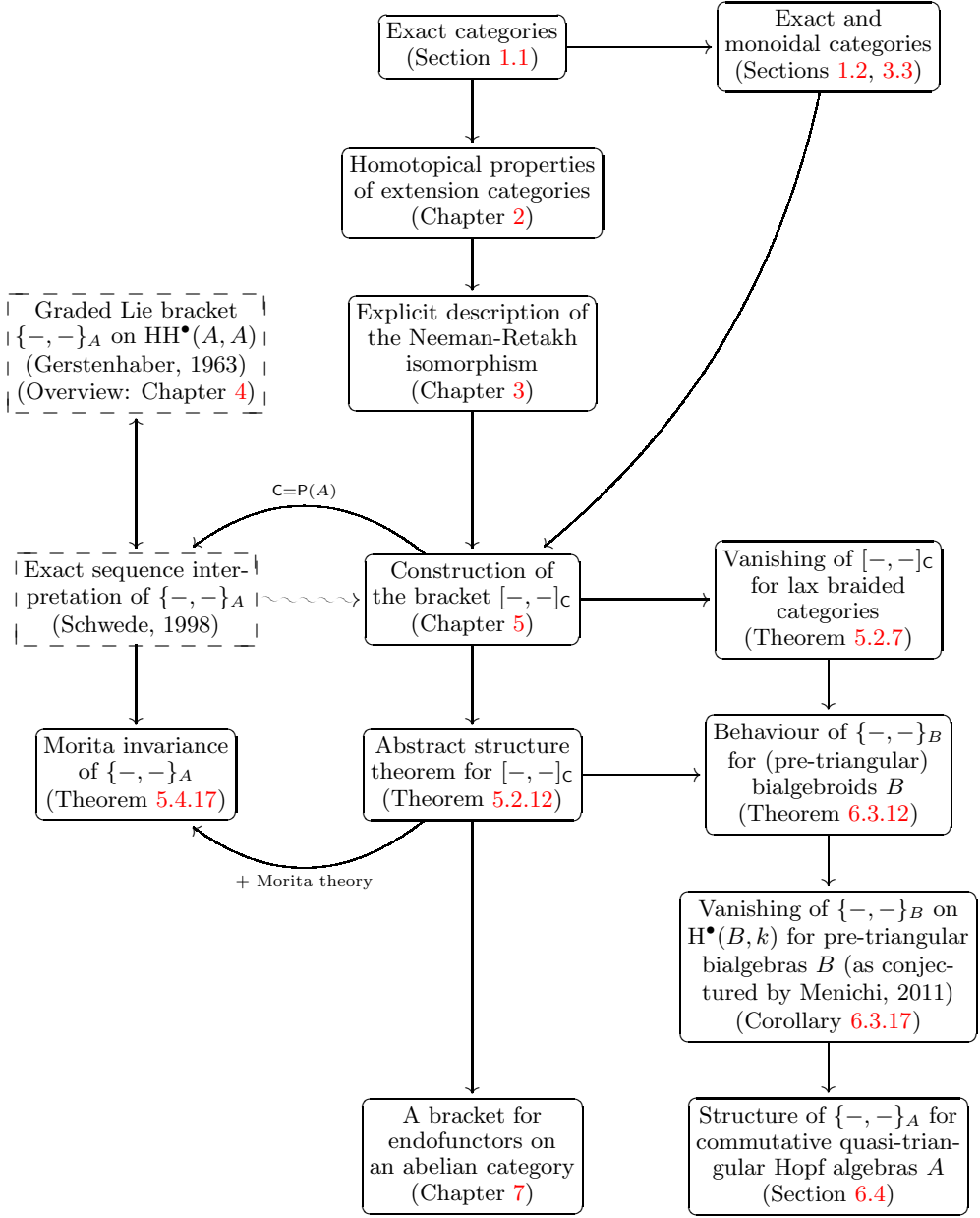
\vfill
\addtocontents{toc}{\protect\setcounter{tocdepth}{0}}
\section*{Conventions}
\addtocontents{toc}{\protect\setcounter{tocdepth}{1}}
We will not pay attention to set theoretical issues that can appear when applying constructions that rely on a base category (for instance, forming homotopy groups of a category, or the derived category of an abelian category). To feel more comfortable, the reader might want to assume that all occurring categories are (skeletally) small. For the entire monograph at hand, we further agree on the following notations.

\begin{enumerate}[\rm$\bullet$]
\item The term ring (algebra) shall always stand for \textit{associative and unital} ring (algebra). If $R$ is a ring, we denote its (additive and multiplicative) units by $0_R$ and $1_R$ or, for short (if there is no ambiguity), by $0$ and $1$. The \textit{opposite ring of $R$}, denoted by $R^\op$, has the same underlying abelian group as $R$, but its multiplicative structure is given by $r \star s := sr$ (for $r,s \in R$).
\item If not stated otherwise, $k$ will always denote a commutative ring.
\item The term module will always mean \textit{left} module.
\item For two integers $n_1 \leq n_2$ we put $[n_1,n_2] := \{n_1, n_1 + 1,\dots,n_2\}$. We let $$\mathfrak{S}_n := \{ \sigma: [1,n] \longrightarrow [1,n] \mid \sigma \text{ is bijective} \}$$ be the symmetric group on $n$ letters.
\item Let $V_1, \dots, V_n$ be $k$-modules. Every element $\sigma$ of the symmetric group $\mathfrak{S}_n$ gives rise to a $k$-linear map $V_1 \otimes_k \cdots \otimes_k V_n \rightarrow V_{\sigma(1)} \otimes_k \cdots \otimes_k V_{\sigma(n)}$ by
$$
\sigma(v_1 \otimes \cdots \otimes v_n) = v_{\sigma(1)} \otimes \cdots v_{\sigma(n)},
$$
where $v_i \in V_i$ for $i = 1, \dots, n$. In the particular case of $n = 2$, the symmetric group has just one non-trivial element. We denote the corresponding \textit{twist} map by $\tau = \tau_{V_1, V_2}$:
$$
\tau: V_1 \otimes_k V_2 \longrightarrow V_2 \otimes_k V_1, \ \tau(v_1 \otimes v_2) = v_2 \otimes v_1.
$$
\item If $R$ is a ring, let $\Mod(R)$ be the category of all $R$-modules, $\mod(R)$ be the full subcategory of all finitely generated $R$-modules and let $\mod_0(R)$ be the full subcategory of all $R$-modules of finite length. Accordingly, we use the notation $\Proj(R)$, $\proj(R)$, $\proj_0(R)$ and $\Inj(R)$, $\inj(R)$, $\inj_0(R)$ for the corresponding full subcategories of all projective and injective $R$-modules. Likewise, $\mathsf{Flat}(R)$, $\mathsf{flat}(R)$ and $\mathsf{flat}_0(R)$ denote the full subcategories of all flat $R$-modules (being finitely generated/of finite length over $R$).
\item If $\C$ is a category, we let $\Ob \C$ be the class of all objects in $\C$ and $\Mor \C$ the class of all morphisms in $\C$. For objects $X$ and $Y$ in $\C$, let $\Hom_\C(X,Y)$ ($\mathrm{Iso}_\C(X,Y)$) be the set of all morphisms (isomorphisms) between $X$ and $Y$ in $\C$. Further, we put $\End_\C(X) := \Hom_\C(X,X)$ and $\Aut_\C(X) := \mathrm{Iso}_\C(X,X)$. If $\C$ is a category of modules over a ring $R$, we write $\Hom_R$, $\End_R$, $\Aut_R$ and $\mathrm{Iso}_R$. We denote by $\id_X \in \End_\C(X)$ the identity morphism of $X$; alternative notions will be $1_{\End_\C(X)}$ or, for short, $1_X$. The \textit{opposite category of $\C$}, denoted by $\C^\op$, is defined by $\Ob \C^\op := \Ob \C$ and $\Hom_{\C^\op}(X,Y) := \Hom_\C(Y,X)$ (for objects $X, Y \in \C$). 
\item Let $\C$ be a category with a zero object (that is, an initial and terminal object) and $f \in \Hom_\C(X,Y)$. If $f$ has a kernel in $\C$, we denote it by $(\Ker(f), \ker(f))$, where $\Ker(f) \in \Ob \C$ and $\ker(f)$ is the defining morphism $\ker(f): \Ker(f) \rightarrow X$. Similarly, we denote a cokernel of $f$ (if existent) by $(\Coker(f), \coker(f))$. If $X' \rightarrow X \leftarrow X''$ and $X' \leftarrow X \rightarrow X''$ are diagrams in $\C$ that admit a pullback/pushout in $\C$, we denote by $X' \times_X X''$ the pullback of the first diagram, and by $X' \oplus_X X''$ the pushout of the second. 
\item The term functor shall always stand for \textit{covariant} functor. If $\scrX: \C \rightarrow \D$ is a functor, we denote by $\Im(\scrX) = \scrX\C$ its \textit{essential image}, i.e., the full subcategory
$$
\{D \in \D \mid D \cong \scrX(C) \text{ for some object $C \in \C$}\} \subseteq \D.
$$
\item If $\C$ is a category with a zero object, we use blackboard bold letters for complexes over $\C$. (Recall that a complex is a sequence $\cdots \rightarrow X_{i+1} \rightarrow X_i \rightarrow X_{i-1} \rightarrow \cdots$ of morphisms in $\C$, indexed over $\mathbb Z$, such that the composite of two successive morphisms is zero.) If $\mathbb X$ is a complex, we let $X_i$ be the object in degree $i$, and let $x_i: X_i \rightarrow X_{i-1}$ be the corresponding morphism in $\mathbb X$ ($i \in \mathbb Z$). We abbreviate this by $\mathbb X = (X_i, x_i)_i$. We let $\Ch(\C)$ be the category of complexes over $\C$. For \textit{cocomplexes} over $\C$ (i.e., complexes in $\mathbf C(\C^\op)$), we use upper indices: $\mathbb X = (X^i, x^i)_i$.
\item If $\C$ and $\D$ are categories, we let $\C \times \D$ be the corresponding \textit{product category}. Its objects are pairs $(X,Y)$, where $X$ runs through the objects of $\C$ and $Y$ runs through the objects of $\D$; its morphism sets are given by
$$
\Hom_{\C \times \D}((X,Y), (X',Y')) := \Hom_\C(X,X') \times \Hom_\D(Y,Y').
$$
If $\mathscr X: \C \times \D \rightarrow \mathsf T$ is a functor to some target category $\mathsf T$, then we write
$$
\mathscr X(f, Y) := \mathscr X(f, \id_Y) \quad \text{and} \quad \mathscr X(X, g) := \mathscr X(\id_X,g),
$$
where $X \in \Ob \C$, $Y \in \Ob \D$ and $f \in \Mor \C$, $g \in \Mor \D$. For 
$$
\scrX = \Hom_\C(-,-): \C^\op \times \C \longrightarrow \mathsf{Sets}
$$
we also introduce the notation
$$
f^\ast = f^\ast_Y := \Hom_\C(f,Y) \quad \text{and} \quad g_\ast = g_\ast^X := \Hom_\C(X,g).
$$
In terms of this convention, we get $f_\ast \circ g^\ast = g^\ast \circ f_\ast : \Hom_\C(Y',X) \rightarrow \Hom_\C(Y,X')$ if $f:X \rightarrow X'$ and $g: Y \rightarrow Y'$.
\item Let $\C$ be a category with finite biproducts (denoted by $\oplus$). If $A$ and $B$ are objects in $\C$, we let 
\begin{align*}
j^A &: A \longrightarrow A \oplus B, & q^A & : A \oplus B \longrightarrow A,\\
j_B &: B \longrightarrow A \oplus B, &  q_B & : A \oplus B \longrightarrow B
\end{align*}
be the canonical morphisms. Remember that they satisfy the following equations:
\begin{align*}
q^A \circ j^A &= \id_A, & j^A \circ q^A + j_B \circ q_B = \id_{A \oplus B} \\
q_B \circ j_B &= \id_B.
\end{align*}
Further, we let $d_A: A \rightarrow A \oplus A$ and $d^A: A \oplus A \rightarrow A$ denote the unique morphisms with
$$
q^A \circ d_A = \id_A = q_A \circ d_A \quad \text{and} \quad d^A \circ j^A = \id_A = d^A \circ j_A .
$$
More generally, if $A_1, \dots, A_r$ and $B_1, \dots, B_s$ are objects in $\C$, $j_\ell: A_\ell \rightarrow \bigoplus_i A_i$, and $q_k: \bigoplus_i B_i \rightarrow B_k$ are the canonical morphisms, and if $f_{k,\ell}: A_\ell \rightarrow B_k$ (for $\ell = 1, \dots, r$, $k = 1, \dots, s$) are morphisms in $\C$, then the matrix
$$
M: = \left[\begin{matrix}
f_{1,1} & f_{1,2} & \cdots & f_{1,r} \\
  f_{2,1} & f_{2,2} & \cdots & f_{2,r} \\
  \vdots  & \vdots  & \ddots & \vdots  \\
  f_{s,1} & f_{s,2} & \cdots & f_{s,r}
\end{matrix}\right] : \bigoplus_{\ell=1}^r A_\ell \longrightarrow \bigoplus_{k=1}^s B_k
$$
is defined to be the unique morphism in $\C$ with
$$
q_k \circ M \circ j_\ell = f_{k,\ell} \quad (\text{for $\ell = 1, \dots, r$, $k = 1, \dots, s$}).
$$
If $r = 1$, we will also use the notation
$$
\left[\begin{matrix}
f_{1,1} \\
f_{2,1} \\
\vdots \\
f_{s,1}
\end{matrix}\right] = f_{1,1} \oplus f_{2,1} \oplus \cdots \oplus f_{s,1}.
$$
Occasionally, a slightly more ambiguous notation will be employed in the case $s=1$:  $$\left[\begin{matrix}
f_{1,1} & f_{1,2} & \cdots & f_{1,r}
\end{matrix}\right] = f_{1,1} + \dots + f_{1,r}.$$
\end{enumerate}


\mainmatter


\chapter{Prerequisites}\label{cha:prerequ}
Within this chapter we will recall definitions and results being crucial for the observations we will make in the consecutive chapters. Thereunder, basics from the theory of (Quillen) exact categories and monoidal categories. Bialgebras and, more specifically, Hopf algebras will give rise to examples of (braided and symmetric) monoidal categories. Further examples will be stated in the chapter's final section.
\section{Exact categories}\label{sec:exactcats}
\begin{nn}
Exact categories were introduced by D.\,Quillen in \cite{Qu72} in oder to study the K-theory of additive categories which are not necessarily abelian. He defines them as full subcategories of abelian categories such that their image in the abelian category defines an extension closed subcategory. Equivalently, exact categories may be defined (without explicitly naming an ambient abelian category) as additive categories equipped with a distinguished class of sequences which are subject to certain axioms. In \cite{Ke90} B.\,Keller discusses these axioms, discovers some redundancy, and hence is able to restate the classical definition in a very concise way. In what follows, we will consider exact categories as (extension closed) subcategories of abelian categories. Let $k$ be a commutative ring.
\end{nn}
\begin{defn}\label{def:extensioncl}
Let $\C$ be a full subcategory of an abelian category $\A$. The subcategory $\C$ is called \textit{extension closed $($in $\A)$}, if whenever
$$
0 \longrightarrow C' \longrightarrow A \longrightarrow C'' \longrightarrow 0
$$
is a short exact sequence in $\A$ with $C'$ and $C''$ in $\C$, then $A$ belongs to $\C$.
\end{defn}
\begin{defn}\label{def:exactcat}
Let $\C$ be an additive $k$-category and let $i_\C: \C \rightarrow \A_\C$ be a full and faithful $k$-linear functor into an abelian $k$-category $\A_\C$.
\begin{enumerate}[\rm(1)]
\item The pair $(\C,i_\C)$ is an \textit{exact $k$-category} (or a \textit{Quillen exact $k$-category}) if the essential image $\Im(i_\C) = i_\C \C$ of $i_\C$ is an extension closed subcategory of $\A_\C$. 
\item A sequence $0 \rightarrow E' \rightarrow E \rightarrow E'' \rightarrow 0$ in $\mathsf C$ is called an \textit{admissible short exact sequence} (or a \textit{conflation}) provided its image under $i_\C$ is a short exact sequence in $\A_\C$. We denote the class of all admissible short exact sequences in $\C$ by $\sigma({\mathsf C})$.
\item A morphism $E' \rightarrow E$ in $\C$ is an \textit{admissible monomorphism} (or an \textit{inflation}) if it fits inside an admissible short exact sequence $0 \rightarrow E' \rightarrow E \rightarrow E'' \rightarrow 0$.
\item A morphism $E \rightarrow E''$ in $\C$ is an \textit{admissible epimorphism} (or a \textit{deflation}) if it fits inside an admissible short exact sequence $0 \rightarrow E' \rightarrow E \rightarrow E'' \rightarrow 0$.
\end{enumerate}
Let $\A$ be an abelian $k$-category, and let $\D$ be a full subcategory of $\A$. We call $\D$ an \textit{exact subcategory of $\A$}, if $\D$ together with the inclusion functor $\D \rightarrow \A$ defines an exact $k$-category.
\end{defn}
\begin{rem}\label{rem:exactsubcat}
\begin{enumerate}[\rm(1)]
\item When denoting an exact $k$-category, we will frequently suppress the structure giving inclusion and write $\C$ instead of $(\C, i_\C)$.
\item Evidently, extension closed subcategories are closed under taking finite biproducts. Note that if $\A$ is an abelian category and $\C$ is an extension closed subcategory of $\A$ containing a zero object and being closed under isomorphisms, then the pair $(\C, i_\C)$, where $i_\C = \mathrm{inc}: \C \rightarrow \A$ denotes the inclusion functor, is an exact category. In particular any abelian category $\A$ is Quillen exact, in such a way that $\sigma(\A)$ coincides with the class of all short exact sequences in $\A$ (which is obvious from the definition).
\end{enumerate}
\end{rem}
\begin{prop}[{\cite{Qu72}}]\label{prop:exactcat}
Let $\mathsf C$ be an exact $k$-category and $i = i_\C: \C \rightarrow \A_\C$ its defining embedding into the abelian $k$-category $\A = \A_\C$.
\begin{enumerate}[\rm (1)]
\item\label{prop:exactcat:0} For every object $E$ in $\C$, the identity morphism $\id_E$ is an admissible monomorphism and an admissible epimorphism.
\item\label{prop:exactcat:1} The class $\sigma(\C)$ of all admissible short exact sequences is
closed under isomorphisms and finite direct sums $($taken in the category of chain complexes over $\C$$)$. 
\item\label{prop:exactcat:1a} The composite of two admissible monomorphisms
\emph{(}admissible epimorphisms\emph{)} is an admissible monomorphism \emph{(}admissible epimorphism\emph{)}.
\item\label{prop:exactcat:2} In every admissible short exact sequence $0 \rightarrow E' \rightarrow E \rightarrow E''
\rightarrow 0$ in $\C$, $E' \rightarrow E$ is a kernel of $E \rightarrow E''$, and $E \rightarrow
E''$ is a cokernel of $E' \rightarrow E$.
\item\label{prop:exactcat:3} Let $0 \rightarrow E' \xrightarrow{f} E \xrightarrow{g} E''
\rightarrow 0$ be an admissible short exact sequence in $\C$ and let $p \in \Hom_\C(E', X)$, $q \in \Hom_\C(X,E'')$ be morphisms in $\C$. The pair $(f,p)$ has a pushout in $\C$, while the pair $(g,q)$ has a pullback in $\C$.
\end{enumerate}
\end{prop}
\begin{proof}
The assertions (\ref{prop:exactcat:0}) and (\ref{prop:exactcat:1}) are valid for trivial reasons, and basically just use the fact that $i$ is an additive functor. In order to prove (\ref{prop:exactcat:1a}), let $f: E' \rightarrow E$ and $g: E \rightarrow E''$ be admissible monomorphisms. We obtain a commutative diagram
$$
\xymatrix@C=20pt{
0 \ar[r] & iE' \ar@{=}[r] \ar[d]_{i(f)} & iE' \ar[r] \ar[d]^{i(g \circ f)} & 0 \ar[r] \ar[d] & 0\\
0 \ar[r] & iE \ar[r]^{i(g)} \ar[d] & iE'' \ar[r] \ar[d] & \Coker(i(g)) \ar[r] \ar@{=}[d] & 0 \\
 & \Coker(i(f)) \ar[r]|-{\bbslash} & \Coker(i(g \circ f)) \ar[r]|-{\bbslash} & \Coker(i(g))
}
$$
where the highlighted arrows come from the universal property that cokernels carry. An application of the Snake Lemma yields the exactness of
$$
\xymatrix@C=20pt{
0 \ar[r] & \Coker(i(f)) \ar[r] & \Coker(i(g \circ f)) \ar[r] & \Coker(i(g)) \ar[r] & 0
}
$$
and hence $g \circ f$ is admissible since $i\C$ is extension closed. Analogously, $g \circ f$ is an
admissible epimorphism, if $g$ and $f$ are. The second statement is obvious. To see
(\ref{prop:exactcat:3}), we remark that in $\A$, we are allowed to complete the following diagram given by the labeled arrows via the unlabelled ones such that the lower row is also a short exact sequence in $\A$:
$$
\xymatrix{
0 \ar[r] & iE' \ar[r]^{i(f)} \ar[d]_-{i(p)} & iE \ar[r]^-{i(g)} \ar[d] & iE'' \ar[r] \ar@{=}[d] & 0 \ \ \\
0 \ar[r] & iX \ar[r]        & P \ar[r]            & iE'' \ar[r]                 & 0 \ .
}
$$
Since $\C$ is exact, the pushout $P = iX \oplus_{iE'} iE$ belongs to the essential image of $i$. Its fibre clearly is a pushout of $(f,p)$. In a similar way, the pullback along $g$ and $q$ exists in $\C$, too.
\end{proof}
\begin{defn}\label{def:exfunctor}
Let $\C$ and $\D$ be exact $k$-categories. A (not necessarily $k$-linear) functor $\scrX: \C \rightarrow \D$ is called \textit{exact} if it carries admissible short exact sequences in $\C$ to admissible short exact sequences in $\D$, that is, if $\scrX$ induces a map $\sigma(\C) \rightarrow \sigma(\D)$. Assume that $\D = \A$ is abelian. The functor $\scrX: \C \rightarrow \A$ is called 
\begin{enumerate}[\rm(1)]
\item \textit{left exact} if any admissible short exact sequence $0 \rightarrow E' \rightarrow E \rightarrow E'' \rightarrow 0$ in $\C$ is carried to an exact sequence 
$$
0 \longrightarrow \scrX(E') \longrightarrow \scrX(E) \longrightarrow \scrX(E'')
$$
in $\A$.
\item \textit{right exact} if any admissible short exact sequence $0 \rightarrow E' \rightarrow E \rightarrow E'' \rightarrow 0$ in $\C$ is carried to an exact sequence 
$$
\scrX(E') \longrightarrow \scrX(E) \longrightarrow \scrX(E'') \longrightarrow 0
$$
in $\A$.
\end{enumerate}
The $k$-linear left (right) exact functors $\C \rightarrow \A$ form a full subcategory $\underline{\Fun}_k^\lambda(\C,\A)$ ($\underline{\Fun}_k^\varrho(\C,\A)$) of the category of $k$-linear functors $\Fun_k(\C,\A)$. Their intersection,
$$
\underline{\Fun}_k(\C,\A) := \underline{\Fun}_k^\lambda(\C,\A) \cap \underline{\Fun}_k^\varrho(\C,\A),
$$
is the full subcategory of $\Fun_k(\C,\A)$ containing all exact $k$-linear functors $\C \rightarrow \A$.
\end{defn}
\begin{rem}\label{rem:properties_exfunc}
Let $\C$ be an additive $k$-category equipped with a class $\sigma(\C)$ of sequences $0 \rightarrow E' \xrightarrow{f} E \xrightarrow{g} E'' \rightarrow 0$ of morphisms in $\C$ (where $f$ is an ``admissible monomorphism" and $g$ is an ``admissible epimorphism") satisfying the properties (\ref{prop:exactcat:0})--(\ref{prop:exactcat:3}) of Proposition \ref{prop:exactcat} and, moreover, fulfilling the following two conditions.
\begin{enumerate}[\rm(i)]
\item\label{rem:properties_exfunc:1} If $E \rightarrow E''$ is a morphism in $\C$ that possesses a kernel in $\C$, and if $D \rightarrow E$ is a morphism in $\C$ such that $D \rightarrow E \rightarrow E''$ is an admissible epimorphism, then $E \rightarrow E''$ is an admissible epimorphism.
\item\label{rem:properties_exfunc:2} If $E' \rightarrow E$ is a morphism in $\C$ that possesses a cokernel in $\C$, and if $E \rightarrow F$ is a morphism in $\C$ such that $E' \rightarrow E \rightarrow F$ is an admissible monomorphism, then $E' \rightarrow E$ is an admissible epimorphism.
\end{enumerate}
Consider the category $\A_\C := \underline{\Fun}_k^\lambda(\C^\op, \Mod(k))$ of left exact $k$-linear functors $\scrX \colon \C^\op \rightarrow \Mod(k)$ (left exact in the sense that a distinguished sequence $0 \rightarrow E' \rightarrow E \rightarrow E'' \rightarrow 0$ in $\sigma(\C)$ is mapped to an exact sequence $0 \rightarrow \scrX(E'') \rightarrow \scrX(E) \rightarrow \scrX(E')$ in $\Mod(k)$). As exhibited in \cite{Qu72}, one can show, that $\C$ embeds into $\A_\C$ via the Yoneda functor and that $\A_\C$ is $k$-linear abelian. Furthermore, the essential image of $\C$ in $\A_\C$ is an extension closed subcategory and $\sigma(\C)$ is precisely the class of sequences in $\C$ which are carried over to short exact sequences in $\A_\C$. Thus $\C$ is exact in the sense of Definition \ref{def:exactcat}. B.\,Keller observed that, in fact, it suffices to assume that $\sigma(\C)$ satisfies \ref{prop:exactcat}(\ref{prop:exactcat:0})--(\ref{prop:exactcat:1a}) and \ref{prop:exactcat}(\ref{prop:exactcat:3}).
\end{rem}
\begin{prop}[{\cite{Ke90}}]
Let $\C$ be an additive $k$-category equipped with a class $\sigma(\C)$ of sequences $0 \rightarrow E' \rightarrow E \rightarrow E'' \rightarrow 0$ of morphisms in $\C$ fulfilling the properties $(\ref{prop:exactcat:0})$, $(\ref{prop:exactcat:1})$, $(\ref{prop:exactcat:1a})$ and $(\ref{prop:exactcat:3})$ of Proposition $\ref{prop:exactcat}$. Then also $\ref{prop:exactcat}(\ref{prop:exactcat:2})$ and $\ref{rem:properties_exfunc}$\emph{(\ref{rem:properties_exfunc:1})}--\emph{(\ref{rem:properties_exfunc:2})} hold true.\qed
\end{prop}
\begin{lem}\label{lem:exact_func_push_pull}
Let $\mathsf C$ and $\mathsf D$ be exact $k$-categories. Let $\mathscr X: \mathsf C \rightarrow \mathsf D$ be an exact $k$-linear functor. Then $\mathscr X$ preserves pushouts along admissible monomorphisms and pullbacks along admissible epimorphisms.
\end{lem}
\begin{proof}
If $f: X \rightarrow Y$ is an admissible monomorphism in $\mathsf C$, then it fits inside some admissible
exact sequence
$$
0 \longrightarrow X \longrightarrow Y \longrightarrow \Coker{(f)} \longrightarrow 0
$$
in $\mathsf C$. Pushing out along any other morphism $f': X \rightarrow Y'$ in $\mathsf C$ yields a sequence
$$
(0 \longrightarrow Y' \longrightarrow Y \oplus_{X} Y' \longrightarrow \Coker{(f)} \longrightarrow 0) \quad \in \quad \mathsf \sigma({\mathsf C}).
$$
Furthermore, the sequence
$$
\xymatrix@C=26pt
{
0 \ar[r] & X \ar[r]^-{\left[\begin{smallmatrix}
f\\
-f'
\end{smallmatrix}\right]} & Y \oplus Y' \ar[r] & Y \oplus_X Y' \ar[r] & 0
}
$$
is an admissible exact one in $\mathsf C$ and hence will stay so in $\mathsf{D}$ after applying $\mathscr X$. By dualizing this simple argument, we also obtain that $\mathscr X$ preserves pullbacks along admissible epimorphisms.
\end{proof}
\section{Monoidal categories}\label{sec:monoidalcats}
In this section we will recall the notions of monoidal categories and structure preserving functors (monoidal functors) between them. We mainly follow the textbooks \cite{AgMa10} and \cite{MaL98}.
\begin{defn}\label{def:monoidalcat}
Let $\mathsf \C$ be a category, $\otimes: \C \times \C \rightarrow \C$ be a functor, $\mathbbm 1$ be an object in $\C$,
\begin{gather*}
\alpha: - \otimes (- \otimes-) \longrightarrow (- \otimes -) \otimes - \ , \\
\lambda: \mathbbm 1 \otimes - \longrightarrow \Id_{\C} \ , \\
\varrho: - \otimes \mathbbm 1 \longrightarrow \Id_{\C}
\end{gather*}
be isomorphisms of functors and let $\gamma_{X,Y}: X \otimes Y \rightarrow Y \otimes X$ (for  $X, Y \in \Ob \C$) be natural morphisms.

\begin{enumerate}[\rm(1)]
\item The 6-tuple $(\mathsf C, \otimes, \mathbbm 1, \alpha, \lambda, \varrho)$ is a \textit{monoidal category} provided that
the following diagrams commute:
$$
\xymatrix@C=40pt{
X \otimes (\mathbbm 1 \otimes Y) \ar[r]^{\alpha_{X,\mathbbm 1,Y}} \ar[d]_{X \otimes \lambda_Y} & (X \otimes \mathbbm 1) \otimes Y \ar[d]^{\varrho_{X} \otimes Y}\\
X \otimes Y \ar@{=}[r] & X \otimes Y
}
$$
and
$$
\xymatrix@C=40pt{
W \otimes (X \otimes (Y \otimes Z)) \ar[r]^-{\alpha_{W,X,Y \otimes Z}} \ar[d]_-{W \otimes \alpha_{X,Y,Z}} & (W\otimes X) \otimes (Y \otimes Z) \ar[r]^{\alpha_{W \otimes X, Y, Z}} & ((W \otimes X) \otimes Y) \otimes Z \ \ \\
W \otimes ((X \otimes Y) \otimes Z) \ar[rr]^-{\alpha_{W,X \otimes Y,Z}} & & (W \otimes (X \otimes Y)) \otimes Z \ar[u]_{\alpha_{W,X,Y} \otimes Z} \ ,
}
$$
where $W, X,Y,Z \in \Ob \mathsf C$. In this situation, $\otimes$ is a \textit{monoidal} (or \textit{tensor}) \textit{product functor for $\C$} and $\mathbbm 1$ is the \textit{$($tensor$)$ unit} of $\otimes$.
\item The 7-tuple $(\C, \otimes, \mathbbm 1, \alpha, \lambda, \varrho, \gamma)$ is a \textit{pre-braided monoidal category}, provided that $(\mathsf C, \otimes, \mathbbm 1, \alpha, \lambda, \varrho)$ is a monoidal category and $\gamma$ is an isomorphism such that $\varrho_{\mathbbm 1} \circ \gamma_{\mathbbm 1, \mathbbm 1} = \lambda_{\mathbbm 1}$. In this case, $\gamma$ is a \textit{pre-braiding} on the monoidal category $\C$.
\item The 7-tuple $(\C, \otimes, \mathbbm 1, \alpha, \lambda, \varrho, \gamma)$ is a \textit{lax braided monoidal category} provided that $(\mathsf C, \otimes, \mathbbm 1, \alpha, \lambda, \varrho)$ is a monoidal category such that $\varrho_{\mathbbm 1} \circ \gamma_{\mathbbm 1, \mathbbm 1} = \lambda_{\mathbbm 1}$ and such that the diagrams
$$
\xymatrix@C=40pt{
X \otimes (Y \otimes Z) \ar[r]^-{\alpha_{X,Y,Z}} \ar[d]_-{X \otimes \gamma_{Y,Z}} & (X\otimes Y) \otimes Z \ar[r]^{\gamma_{X \otimes Y, Z}} & Z \otimes (X \otimes Y) \ar@<-3pt>[d]^{\alpha_{Z,X,Y}} \ \ \\
X \otimes (Z \otimes Y) \ar[r]^-{\alpha_{X,Z,Y}} & (X \otimes Z) \otimes Y \ar[r]^-{\gamma_{X,Z} \otimes Y} & (Z \otimes X) \otimes Y \ ,
}
$$
and
$$
\xymatrix@C=40pt{
(X \otimes Y) \otimes Z \ar[r]^-{\alpha_{X,Y,Z}^{-1}} \ar[d]_-{\gamma_{X,Y} \otimes Z} & X \otimes (Y \otimes Z) \ar[r]^-{\gamma_{X, Y \otimes Z}} & (Y \otimes Z) \otimes X \ar[d]^-{\alpha_{Y,Z,X}^{-1}} \\
(Y \otimes X) \otimes Z \ar[r]^-{\alpha_{Y,X,Z}^{-1}} & Y \otimes (X \otimes Z) \ar[r]^-{Y \otimes \gamma_{X, Z}} & Y \otimes (Z \otimes X) 
}
$$
commute for all $X,Y,Z \in \Ob \C$. These diagrams define the so called \textit{triangle equations}. In this case, $\gamma$ is a \textit{lax braiding} on the monoidal category $\C$.
\item The 7-tuple $(\C, \otimes, \mathbbm 1, \alpha, \lambda, \varrho, \gamma)$ is a \textit{braided monoidal category} provided that $(\mathsf C, \otimes, \mathbbm 1, \alpha, \lambda, \varrho)$ is a monoidal category and $\gamma$ is an isomorphism satisfying the triangle equations. In this case, $\gamma$ is a \textit{braiding} on the monoidal category $\C$.
\item A braided monoidal category is called a \textit{symmetric monoidal
category}, if\linebreak$(\gamma_{X,Y})^{-1} = \gamma_{Y,X}$ for all $X, Y \in \Ob \C$. In this case, $\gamma$ is a \textit{symmetry} on the monoidal category $\C$.
\end{enumerate}
\end{defn}
\begin{rem}\label{rem:opposite} Let $(\C, \otimes, \mathbbm 1, \alpha, \lambda, \varrho)$ be a monoidal category.
\begin{enumerate}[\rm(1)]
\item We will often suppress a huge part of the structure morphisms and simply write $(\C, \otimes, \mathbbm 1)$ instead of $(\C, \otimes, \mathbbm 1, \alpha, \lambda, \varrho)$; if they are needed without explicitly being mentioned, we will refer to them as $\alpha_\C$, $\lambda_\C$ and $\varrho_\C$.
\item Note that in the definition of a lax braided monoidal category we do \textit{not} require that $\gamma$ is an isomorphism.
\item It follows from the axioms (cf. \cite[Prop.\,1.1]{JoSt93}) that the following diagrams commute for all $X, Y, Z \in \Ob \C$.
$$
\xymatrix{
\mathbbm 1 \otimes \mathbbm 1 \ar[r]^-{\lambda_{\mathbbm 1}} \ar@{=}[d] & \mathbbm 1 \ar@{=}@<-3.25pt>[d] \ \ \\
\mathbbm 1 \otimes \mathbbm 1 \ar[r]^-{\varrho_{\mathbbm 1}} & \mathbbm 1 \ ,
}\quad
\xymatrix{
X \otimes (Y \otimes \mathbbm 1) \ar[r]^-{X \otimes \varrho} \ar[d]_-{\alpha} & X \otimes Y \ar@{=}@<-3pt>[d] \ \ \\
(X \otimes Y) \otimes \mathbbm 1 \ar[r]^-{\varrho} & X \otimes Y \ ,
}\quad
\xymatrix{
\mathbbm 1 \otimes (X \otimes Y) \ar[r]^-{\lambda} \ar[d]_-{\alpha} & X \otimes Y \ar@{=}@<-3pt>[d] \ \ \\
(\mathbbm 1 \otimes X) \otimes Y \ar[r]^-{\lambda \otimes Y} & X \otimes Y \ .
}
$$
\item Let $\gamma$ be a braiding on $(\C, \otimes, \mathbbm 1, \alpha, \lambda, \varrho)$. The following diagrams commute (cf. \cite[Prop.\,2.1]{JoSt93}).
$$
\xymatrix{
\mathbbm 1 \otimes X \ar[r]^-{\lambda_X} \ar[d]_-{\gamma_{\mathbbm 1,X}} & X \ar@{=}@<-3pt>[d] \ \ \\
X \otimes \mathbbm 1 \ar[r]^-{\varrho_X} & X \ ,
}
\quad \quad \quad
\xymatrix{
X \otimes \mathbbm 1 \ar[r]^-{\varrho_X} \ar[d]_-{\gamma_{X,\mathbbm 1}} & X \ar@{=}@<-3pt>[d] \ \ \\
\mathbbm 1 \otimes X \ar[r]^-{\lambda_X} & X \ .
}
$$
In particular, $\gamma$ is a pre-braiding and a lax braiding on the monoidal category $\C$.
\item Note that if $(\mathsf C, \otimes, \mathbbm 1, \alpha, \lambda, \varrho)$ is a monoidal category (with braiding $\gamma$), then so is $\mathsf C^\op$ together with the structure morphisms $\alpha^{-1}$, $\lambda^{-1}$ and $\varrho^{-1}$ (with braiding $\gamma^{-1}$).
\end{enumerate}
\end{rem}
\begin{defn}
A monoidal category $(\C, \otimes, \mathbbm 1)$ is called \textit{tensor $k$-category}, if $\C$ is $k$-linear and the tensor product functor $\otimes: \C \times \C \rightarrow \C$ is $k$-bilinear on morphisms, that is, it factors through the \textit{tensor product category} $\C \otimes_k \C$ which is defined as follows:
\begin{align*}
\Ob(\C \otimes_k \C) &:= \Ob(\C \times \C),\\
\Hom_{\C \otimes_k \C}(\underline{X}, \underline{Y}) &:= \Hom_\C(X_1,Y_1) \otimes_k \Hom_\C(X_2, Y_2)
\end{align*}
for objects $\underline{X} = (X_1, X_2)$ and $\underline{Y} = (Y_1, Y_2)$ in $\C \times \C$. A tensor $k$-category is \textit{pre-braided} (\textit{lax braided}, \textit{braided}, \textit{symmetric}) if its underlying monoidal category is pre-braided (lax braided, braided, symmetric).
\end{defn}
\begin{defn}\label{def:monoidalfunc}
Let $(\mathsf C, \otimes_{\mathsf C}, \mathbbm 1_{\mathsf C})$ and $(\mathsf D, \otimes_{\mathsf D}, \mathbbm 1_{\mathsf D})$ be monoidal categories. Let $\mathscr L: \mathsf C \rightarrow \mathsf D$ be a functor,
\begin{align*}
& \phi_{X,Y}: \mathscr LX \otimes_{\mathsf D} \mathscr LY \longrightarrow \mathscr L(X \otimes_{\mathsf C} Y),\\
& \psi_{X,Y}: \scrL (X \otimes_{\mathsf C} Y) \longrightarrow \scrL X \otimes_{\mathsf D} \scrL Y,
\end{align*}
be natural morphisms in $\D$ ($X,Y \in \Ob\mathsf C$), and let
$$
\phi_0: \mathbbm 1_{\mathsf D} \longrightarrow \scrL \mathbbm 1_{\mathsf C}, \quad \psi_0: \scrL \mathbbm 1_{\mathsf C} \longrightarrow \mathbbm 1_{\mathsf D}
$$
be morphisms in $\mathsf D$.

\begin{enumerate}[\rm(1)]
\item The triple $(\scrL, \phi, \phi_0)$ is called \textit{lax monoidal functor} if the following diagrams commute for all $X,Y,Z \in \Ob \mathsf C$:

\begin{equation*}
\xymatrix@C=40pt{
\scrL X \otimes_\D (\scrL Y \otimes_\D \scrL Z) \ar[r]^{\scrL X \otimes_\D \phi_{Y,Z}} \ar[d]_{\alpha_\D\scrL} & \scrL X
\otimes_\D \scrL(Y \otimes_\C Z) \ar[r]^{\phi_{X,Y \otimes_\C Z}} & \scrL(X \otimes_\C (Y \otimes_\C Z)) \ar[d]^{\scrL \alpha_\C}\\
(\scrL X \otimes_\D \scrL Y) \otimes_\D \scrL Z \ar[r]^-{\phi_{X,Y} \otimes_\D \scrL Z} & \scrL (X \otimes_\C Y) \otimes_\D \scrL Z
\ar[r]^-{\phi_{X \otimes_\C Y,Z}} & \scrL ((X \otimes_\C Y) \otimes_\C Z)
}
\end{equation*}
\begin{equation*}
\xymatrix{
\mathbbm 1_\D \otimes_\C \scrL X  \ar[d]_{\phi_0 \otimes_\D \scrL X} \ar[r]^-{\lambda_{\D, \scrL X}} & \scrL X  \\
\scrL \mathbbm 1_\C \otimes_\D \scrL X \ar[r]^-{\phi_{\mathbbm 1_\C, X}} & \scrL(\mathbbm 1_\C \otimes_\C X) \ar[u]_{\scrL (\lambda_{\C,X})}
}
\quad
\xymatrix{
\scrL X \otimes_\D  \mathbbm 1_\D \ar[d]_{\scrL X \otimes_\D \phi_0} \ar[r]^-{\varrho_{\D,\scrL X}} & \scrL X  \\
\scrL X \otimes_\D \scrL \mathbbm 1_\C \ar[r]^-{\phi_{X, \mathbbm 1_\C}} & \scrL(X \otimes_\C \mathbbm 1_\C)
\ar[u]_{\scrL (\varrho_{\C,X})}
}
\end{equation*}
\item The triple $(\scrL, \psi, \psi_0)$ is called \textit{colax monoidal functor} if the functor $\scrL^\op: \C^\op \rightarrow \D^\op$ gives rise to a lax monoidal functor $(\scrL^\op, \psi, \psi_0)$.
\item The 5-tuple $(\scrL, \phi, \phi_0, \psi, \psi_0)$ is called \textit{bilax monoidal functor} if $(\scrL, \phi,\phi_0)$ is a lax monoidal functor and $(\scrL, \psi, \psi_0)$ is a colax monoidal functor such that $\psi_0 \circ \phi_0 = \id_{\mathbbm 1_\D}$ and the following diagrams commute for all objects $X, Y$ in $\C$.
\begin{align*}
\xymatrix@C=35pt{
\mathbbm 1_\D \ar[r]^-{\phi_0} \ar[d]_-{\lambda_{\D, \mathbbm 1_\D}} & \scrL \mathbbm 1_\C \ar[r]^-{\scrL(\lambda_{\C, \mathbbm 1_\C})} & \scrL(\mathbbm 1_\C \otimes_\C \mathbbm 1_\C) \ar[d]^-{\psi_{\mathbbm 1_\C, \mathbbm 1_\C}} \\
\mathbbm 1_\D \otimes_\D \mathbbm 1_\D \ar[rr]^-{\phi_0 \otimes_\D \phi_0} && \scrL \mathbbm 1_\C \otimes_\D \scrL\mathbbm 1_\C
}\\
\xymatrix@C=35pt{
\mathbbm 1_\D & \scrL \mathbbm 1_\C \ar[l]_-{\psi_0} & \scrL(\mathbbm 1_\C \otimes_\C \mathbbm 1_\C)  \ar[l]_-{\scrL(\lambda^{-1}_{\C, \mathbbm 1_\C})}  \\
\mathbbm 1_\D \otimes_\D \mathbbm 1_\D \ar[u]^-{\lambda^{-1}_{\D, \mathbbm 1_\D}} && \scrL \mathbbm 1_\C \otimes_\D \scrL \mathbbm 1_\C \ar[u]_-{\phi_{\mathbbm 1_\C, \mathbbm 1_\C}} \ar[ll]_-{\psi_0 \otimes_\D \psi_0} 
}
\end{align*}
$$
\xymatrix@C=65pt{
\scrL(X \otimes_\C \mathbbm 1_\C) \otimes_\D \scrL(\mathbbm 1_\C \otimes_\C Y) \ar[r]^-{\psi_{X,\mathbbm 1_\C} \otimes_\D \psi_{\mathbbm 1_\C, Y}}  \ar[d]_-{\phi_{X \otimes_\C \mathbbm 1_\C, \mathbbm 1_\C \otimes_\C Y}} & (\scrL X \otimes_\D \scrL \mathbbm 1_\C) \otimes_\D (\scrL \mathbbm 1_\C \otimes_\D \scrL Y) \ar[d]^-{\phi_{X,\mathbbm 1_\C} \otimes_\D \phi_{\mathbbm 1_\C, Y}} \\
\scrL(X \otimes_\C \mathbbm 1_\C \otimes_\C \mathbbm 1_\C \otimes_\C Y) \ar[r]^-{\psi_{X \otimes_\C \mathbbm 1_\C, \mathbbm 1_\C \otimes_\C Y}} & \scrL(X \otimes_\C \mathbbm 1_\C) \otimes_\D \scrL(\mathbbm 1_\C \otimes_\C Y)
}
$$
\item The triple $(\scrL, \phi, \phi_0)$ is called an \textit{almost strong monoidal functor} if it is a lax monoidal functor and $\phi_0$ is invertible. The triple $(\scrL, \psi, \psi_0)$ is called an \textit{almost costrong monoidal functor} if $(\scrL^\op, \psi, \psi_0)$ is an almost strong monoidal functor (that is, $(\scrL, \psi, \psi_0)$ is colax and $\psi_0$ is invertible).
\item The triple $(\scrL, \phi, \phi_0)$ is called a \textit{strong monoidal functor} if it is an almost strong monoidal functor and $\phi$ is invertible. The triple $(\scrL, \psi, \psi_0)$ is called a \textit{costrong monoidal functor} if it is an almost costrong monoidal functor and $\psi$ is invertible.
\item The 5-tuple $(\scrL, \phi, \phi_0, \psi, \psi_0)$ is called a \textit{bistrong monoidal functor} if it is a bilax monoidal functor such that $\phi$, $\phi_0$, $\psi$ and $\psi_0$ are invertible.
\end{enumerate}
\end{defn}
\begin{rem}
\begin{enumerate}[\rm(1)]
\item In \cite{AgMa10} the authors define bilax monoidal functors for functors between braided monoidal categories. Since the applications we are interested in frequently address the interplay of (ordinary) monoidal categories and monoidal categories that admit a braiding, we will stick to the weaker definition made above.
\item If $(\scrL, \phi, \phi_0)$ is a strong monoidal functor, then $(\scrL, \phi^{-1}, \phi_0^{-1})$ is costrong. In fact, this assignment defines a bijection between strong monoidal and costrong monoidal functors. Moreover, it is easily checked that $(\scrL, \phi, \phi_0, \phi^{-1}, \phi_0^{-1})$ is bistrong. One can show that if $(\scrL, \phi, \phi_0, \psi, \psi_0)$ is a bistrong monoidal functor, then $\psi = \phi^{-1}$ and $\psi_0 = \phi_0^{-1}$ is necessarily implied (compare the proof of \cite[Prop.\,3.41]{AgMa10} and the statement \cite[Prop.\,3.45]{AgMa10}). That is, a strong (or, equivalently, costrong, bistrong) monoidal structure on $\mathscr L$ is uniquely determined by $\mathscr L$.
\end{enumerate}
\end{rem}
\begin{defn}\label{def:flatcoflat}
Let $(\C, \otimes, \mathbbm 1)$ be a monoidal category.
\begin{enumerate}[\rm(1)]
\item Assume that $\C$ is an exact $k$-category. An object $X \in \Ob \C$ is called \textit{flat} (respectively \textit{coflat}), if $X \otimes - $ (respectively $- \otimes X$) is an exact functor. Denote by $\mathsf{Flat}(\C)$ ($\mathsf{Coflat}(\C)$) the full subcategory of $\C$ consisting of all flat (coflat) objects.
\item The monoidal category $(\C, \otimes, \mathbbm 1)$ is a \textit{weak exact monoidal $k$-category} (\textit{coweak exact monoidal $k$-category}) if $\C$ is an exact $k$-category and every object in $\C$ is flat (coflat). 
\end{enumerate}
\end{defn}
\begin{lem}\label{prop:flatcoflat23}
Let $(\C, \otimes, \mathbbm 1)$ be a monoidal category, such that $\C$ is an exact $k$-category. Then $\mathsf{Flat}(\C)$ and $\mathsf{Coflat}(\C)$ are monoidal categories whose monoidal structure is inherited from $\C$. 
\end{lem}
\begin{proof}
By definition, the functors $\mathbbm 1 \otimes - \cong \Id_\C \cong - \otimes \mathbbm 1$ are exact. Hence $\mathbbm 1$ is flat and coflat. If $X$, $Y$ are flat, and $X'$, $Y'$ are coflat, then
$$
(X \otimes Y) \otimes - \cong X \otimes (Y \otimes -) \cong (X \otimes -) \circ (Y \otimes -)
$$
and
$$
- \otimes (X' \otimes Y') \cong (- \otimes X') \otimes Y' \cong (- \otimes Y') \circ (- \otimes X')
$$
imply that $X \otimes Y$ is flat and $X' \otimes Y'$ is coflat. 
\end{proof}
\section{Examples: Exact and monoidal categories}\label{sec:examples}
\begin{exa}[Chain complexes]\label{exa:functors}
Let $\C$ be an exact $k$-category with defining embedding $i=i_\C: \C \rightarrow \A_\C$ into the abelian $k$-category $\A = \A_\C$. Denote by $\Ch(\C)$ the category of chain complexes in $\C$ and by $\Ch(\A)$ the category of chain complexes in $\A$. Both $\Ch(\C)$ and $\Ch(\A)$ are $k$-linear, additive categories and the functor $i$ induces
a $k$-linear, full and faithful functor $\Ch(i): \Ch(\C) \rightarrow \Ch(\A)$. By taking kernels and cokernels
degreewise, $\Ch(\A)$ is abelian, and the image of $\Ch(i)$ defines an extension closed subcategory of $\Ch(\A)$. For this reason, $\Ch(\C)$ is an exact category. In fact, a map in $\Ch(\C)$ is an admissible monomorphism (admissible epimorphism) if and only if it is an admissible monomorphism (admissible epimorphism) in $\C$ in every degree.
\end{exa}
\begin{exa}[Functors]
Let $\mathsf P$ be a (small) preadditive category and let $(\C,i_\C)$ be an exact $k$-category. The additive $k$-category of additive functors $\mathsf P \rightarrow \C$,
$$
\C^{\mathsf P} = \Fun_{\mathbb Z}(\mathsf P, \C) = \Add(\mathsf P, \C),
$$
embeds into the abelian category $\A_\C^{\mathsf P} = \Add(\mathsf P, \A_\C)$ via $i^{\mathsf P}_\C$. $\C^{\mathsf P}$ is extension closed because $\C$ is. Thus $\C^{\mathsf P}$ is a $k$-linear exact category. This construction enables us to recover the exact category of chain complexes over $\C$.

Let $\overrightarrow{\Delta}$ be the infinite quiver
$$
\xymatrix{
\cdots \ar[r]^-{a_{-1}} & \underset{-1}{\bullet} \ar[r]^-{a_{0}} & \underset{0}{\bullet} \ar[r]^-{a_1} & \underset{1}{\bullet} \ar[r]^-{a_2} & \cdots \ .
}
$$
We let $\underline{\mathsf W}_{\overrightarrow{\Delta}}$ be the category with objects $\Delta_0$ (the vertices of $\overrightarrow{\Delta}$) and, for $i, j \in \Delta_0$, with morphisms
$$
\Hom_{\underline{\mathsf W}_{\overrightarrow{\Delta}}}(i, j) = k{\Delta}(i,j) / k{\Delta}_2(i,j),
$$
where ${\Delta}_r(i,j)$ is the set of all paths $w$ form $i$ to $j$ of length $\mathrm{length}(w) \geq r$, and ${\Delta}(i,j) := {\Delta}_0(i,j)$. In particular, $\Hom_{\underline{\mathsf W}_{\overrightarrow{\Delta}}}(i, j) \neq 0$ if, and only if, $j = i$ or $j = i+1$. It is obvious that $\Ch(\C) \cong \Add(\underline{\mathsf W}_{\overrightarrow{\Delta}}^\op,\C) = \C^{\underline{\mathsf W}_{\overrightarrow{\Delta}}^\op}.$
\end{exa}
\begin{exa}[Bimodules]\label{exa:bimodules}
Let $A$ be a $k$-algebra. Then the categories $\Mod(A)$, $\mod(A)$, $\Proj(A)$, $\proj(A)$ are
$k$-linear and exact in the evident way. More generally, every abelian category $\A$ is exact, and so are its full subcategories $\Proj(\A)$ and $\Inj(\A)$. Let $A^\ev = A \otimes_k A^\op$ be the \textit{enveloping algebra of $A$}. We regard $A$ as an $A^\ev$-module through
$$
(a_1 \otimes a_2)a := a_1 a a_2 \quad \text{(for $a,a_1,a_2 \in A$)}.
$$
Every $A^\ev$-module $M$ may be turned into a left and right $A$-module (with central $k$-action) via
$$
a m := (a \otimes 1_A)m, \quad m a := (1_A \otimes a)m \quad (\text{for $a \in A, \ m \in M$}).
$$
Therefore we have three forgetful functors:
\begin{align*}
\scrL_A: \Mod(A^\ev)&\longrightarrow \Mod(A), \quad \mathscr R_A: \Mod(A^\ev) \longrightarrow \Mod(A^\op),\\ &\mathscr X_{A^\ev}: \Mod(A^\ev) \longrightarrow \Mod(k).
\end{align*}
Observe that $\mathscr X_{A^\ev}$ factors through both $\scrL_A$ and $\mathscr R_A$. Denote by $\P(A)$ and $\mathsf{p}(A)$ the full subcategories of $\Mod(A^\ev)$ respectively $\mod(A^\ev)$ consisting of all $A$-modules which are $A$-projective on both sides, that is,
$$
\P(A) = \big\{ M \in \Mod(A^\ev) \mid \scrL_A(M) \in \Proj(A), \ \mathscr R_A(M) \in \Proj(A^\op) \big\}
$$
and
$$
\mathsf{p}(A) = \big\{ M \in \mod(A^\ev) \mid \scrL_A(M) \in \Proj(A), \ \mathscr R_A(M) \in \Proj(A^\op) \big\}.
$$
The adjunction isomorphism $\Hom_{A}(A \otimes_k A, -) \cong \Hom_k(A, \Hom_A(A,-))$ implies $\Proj(A^\ev) \subseteq \P(A)$ and $\proj(A^\ev) \subseteq \mathsf{p}(A)$ in case $A$ is projective over $k$. But since, for example, $A$ does not need to be projective as an $A^\ev$-module, the other inclusion will not hold in general, even if $A$ is $k$-projective. In fact, G.\,Hochschild showed in \cite{Ho45} that, if $k$ is a field, $A$ being projective over $A^\ev$ is equivalent to $A$ being separable over $k$.

The tensor product $\otimes_A$ over $A$ defines a functor $\Mod(A^\ev) \times \Mod(A^\ev) \rightarrow \Mod(A^\ev)$. For every $A^\ev$-module $M$, there are natural $A^\ev$-module homomorphisms
$$
M \otimes_A A \cong M \cong A \otimes_A M
$$
and hence it is easy to believe, that $(\Mod(A^\ev), \otimes_A, A)$ is a monoidal category. The functor $\otimes_A$ restricts to $\P(A) \times \P(A) \rightarrow \P(A)$ (since, for instance, $\Hom_A(P \otimes_A Q, - ) \cong \Hom_A(P, \Hom_A(Q,-))$ for all $A^\ev$-modules $P,Q$). It follows that the triple $(\P(A), \otimes_A, A)$ is a monoidal category. Since projective $A$-modules are flat over $A$, every object in the monoidal category $(\P(A), \otimes_A, A)$ is flat and coflat. The category $\P(A)$ is going to reappear in various stages of this monograph.
\end{exa}
\begin{exa}[Endofunctors]\label{exa:endofunc}
 Let $\A$ be a $k$-linear abelian category. The category $\mathsf{End}_k(\A) = \mathsf{Fun}_k(\A, \A)$ of all $k$-linear endofunctors of $\A$ is $k$-linear and abelian, hence exact. The category $\mathsf{End}_k(\mathsf{A})$ has the structure of a monoidal category; with respect to this monoidal structure, every endofunctor of $\A$ is coflat, and every \textit{exact} endofunctor of $\A$ is flat and coflat (see Chapter \ref{chap:identityfunc}). The monoidal product functor is given by the composition of functors:
$$
\circ: \End_k(\mathsf{A}) \times \End_k(\mathsf{A}) \rightarrow \End_k(\mathsf{A}), \ (\mathscr X, \mathscr Y) \mapsto \mathscr X \circ \mathscr Y.
$$
The associativity isomorphism is simply given by the associativity of the composition of functors, whereas the left and right unit morphisms are given by the identity morphisms (e.g., $\Id_\mathsf{A} \circ \mathscr X = \mathscr X = \mathscr X \circ \Id_\mathsf{A}$ clearly does hold for evey $k$-linear functor $\mathscr X: \mathsf{A} \rightarrow \mathsf{A}$). Hence $(\mathsf{End}_k(\A), \circ, \Id_\A)$ is a monoidal category whose monoidal structure is \textit{strict} (that is, the structure morphisms are the \textit{identity} morphisms). In Chapter \ref{chap:identityfunc} we are going to investigate the category of $k$-linear endofunctors on an abelian $k$-linear category in further detail.
\end{exa}
\begin{exa}[Bialgebras]\label{exa:bialgebras}
Let $\mathcal B = (B, \nabla, \eta, \Delta, \varepsilon)$ be a bialgebra over $k$ (see Section \ref{sec:algcoalhopf} of the appendix of this monograph for examples and further details on bialgebras). Thanks to the comultiplication map $\Delta$, the $k$-module $M \otimes_k N$ is a $B$-module for every pair of $B$-modules $M$ and $N$. More concretely, the $B$-action on $M \otimes_k N$ is given by
$$
b(m \otimes n) := \Delta(b)\cdot(m \otimes n) = \sum_{(b)}b_{(1)}m \otimes b_{(2)}n \quad \text{(for $b \in B$, $m \in M$ and $n \in N$)}.
$$
The $B$-module that arises in this way will be denoted by $M \boxtimes_k N$. It is not difficult to see that ($\Mod(B), \boxtimes_k, k)$ carries the structure of a monoidal category. In the special case when $k$ is a field, $- \boxtimes_k M$ and $M \boxtimes_k -$ are exact functors, i.e., every object in $(\Mod(B), \boxtimes_k, k)$ is flat and coflat. In general we have to pass to smaller categories to achieve flatness (respectively coflatness) for objects. For instance, let $\Sigma \subseteq \mathsf{Flat}(k)$ be a subclass of objects, such that $\{0, k\}, \Sigma \otimes_k \Sigma \subseteq \Sigma$ and such that the full subcategory of $\Mod(k)$ defined by $\Sigma$ is closed under isomorphisms and an extension closed subcategory of $\Mod(k)$. Then the full subcategory 
$$
\C_\Sigma(\mathcal B) := \big\{ M \in \Mod(B) \mid \mathscr X_B(M) \in \Sigma \big\} \subseteq \Mod(B)
$$
is closed under isomorphisms and an extension closed subcategory of $\Mod(B)$ which contains $0$, $k$ and also $\C_\Sigma(\mathcal B) \boxtimes_k \C_\Sigma(\mathcal B)$, and whose objects are flat and coflat. Here $\mathscr X_B: \Mod(B) \rightarrow \Mod(k)$ denotes the forgetful functor. By specializing to $\Sigma = \Proj(k)$, we get the exact and monoidal $k$-category $(\C_{\Proj(k)}(\mathcal B), \boxtimes_k, k)$. It is closed under taking direct summands, and will be of particular interest in Section \ref{sec:kernel}.
\end{exa}



\chapter{Extension categories}\label{cha:extcats}
For a given exact category $\C$, and for fixed objects $X$ and $Y$, one may consider the class of admissible short exact sequences $0 \rightarrow Y \rightarrow E \rightarrow X \rightarrow 0$ in $\C$. It gives rise to a subcategory of $\Ch(\C)$ by putting $X$ in degree $-1$, and $Y$ in degree $1$. The morphisms are obtained by just allowing chain morphisms, that are the identity morphisms in degrees $-1$ and $1$. In an analogous way, the category of \textit{admissible $n$-extension of $X$ by $Y$} arises. Within this chapter, we will study some homotopical properties of those categories, leading to the insight, that their (lower) homotopy groups can be identified (see Chapter \ref{cha:retakh}). The content of the following two chapters is strongly related to \cite{Bu13}, \cite{NeRe96}, \cite{Re86} and \cite{Sch98}.
\section{Definition and properties}\label{sec:defprop}
\begin{nn}
In what follows, we let $\C$ be an exact $k$-category with defining embedding $i=i_\C: \C \rightarrow \A_\C$ into the abelian $k$-category $\A = \A_\C$. Let $X$ and $Y$ be two objects in $\C$, and let $n \geq 1$ be an integer. Within this section, we introduce and investigate the category of admissible $n$-extensions in $\C$. It should be pointed out that most of the observations presented here are direct generalizations of the ones in \cite{Sch98}.
\end{nn}
\begin{defn}\label{def:adexseq}
A sequence 
$$
\xymatrix@C=15pt{
\xi & \equiv & 0 \ar[r] & Y \ar[r] & E_{n-1} \ar[r] & \cdots \ar[r] & E_1 \ar[r] & E_0 \ar[r] & X \ar[r] & 0
}
$$
of morphisms in $\C$ is called an \textit{admissible $n$-extension of $X$ by $Y$} (or an \textit{admissible exact sequence of
length $n$}) if there exist factorizations
$$
\xymatrix{
&&& K_{n-2} \ar@{=}[r] & \cdots \ar@{=}[r] & K_2 \ar@[red][d] &&&\\
0 \ar[r] & Y \ar[r] & E_{n-1} \ar@[red][d] \ar[r] & E_{n-2} \ar@[red][u] \ar[r] & \cdots \ar[r] & E_1 \ar[r] \ar@[red][d] & E_0 \ar[r] & X \ar[r] & 0\\
&& K_{n-1} \ar@{=}[r] & K_{n-1} \ar@[red][u] && K_1 \ar@{=}[r] & K_1 \ar@[red][u] &&
}
$$
in $\mathsf C$ such that the sequences
$$
\xymatrix@C=18pt@R=10pt{
0 \ar[r] & E_i \ar@[red][r] & K_i \ar@[red][r] & E_{i-1} \ar[r] & 0\, \ \\
0 \ar[r] &Y \ar[r] & E_{n-1} \ar@[red][r] & K_{n-1} \ar[r] & 0\, , \\
0 \ar[r] & K_1 \ar@[red][r] & E_0 \ar[r] & X \ar[r] & 0\, ,
}
\quad \quad (\text{for $1 \leq i \leq n-1$}),
$$
are admissible short exact sequences in $\mathsf C$.
\end{defn}

\begin{defn}\label{def:catofext}
Let $\mathcal Ext^n_\C(X,Y)$ denote the category whose objects are the admissible $n$-extensions $\xi$ of $X$ by $Y$ in $\mathsf C$:
$$
\xymatrix@C=18pt{
\xi & \equiv & 0 \ar[r] & Y \ar[r]^-{e_n} & E_{n-1} \ar[r]^-{e_{n-1}} & \cdots \ar[r]^-{e_2} & E_1 \ar[r]^-{e_1} & E_0 \ar[r]^-{e_0} & X \ar[r] & 0 \ .
}
$$
A morphism in $\mathcal Ext^n_\C(X,Y)$ is given by a commutative diagram in $\C$ of the following shape:
$$
\xymatrix@C=15pt{
\xi \ar[d]_-\alpha & \equiv & 0 \ar[r] & Y \ar[r] \ar@{=}[d] & E_{n-1} \ar[r] \ar[d] & E_{n-2} \ar[r] \ar[d]
&\cdots \ar[r] & E_0
\ar[r] \ar[d] & X \ar@{=}[d] \ar[r] & 0 \ \ \\
\zeta & \equiv & 0 \ar[r] & Y \ar[r]                 & F_{n-1} \ar[r]          & F_{n-2} \ar[r]          &
\cdots \ar[r] & F_0
\ar[r]             & X \ar[r] & 0 \ .
}
$$
Let $\mathcal Ext^0_\C(X,Y)$ be the category with object set $\Hom_\C(X,Y)$ and trivial morphisms.
\end{defn}
\begin{nota}
We will abbreviate the middle segment of an admissible $n$-extension $\xi$ of $X$ by $Y$ by $\mathbb E$ and use the shorthand notation
$$
\xymatrix@!C=12pt{
\xi & \equiv & 0 \ar[r] & Y \ar[r] & \mathbb E \ar[r] & X \ar[r] & 0 \ .
}
$$
$\mathbb E$ may be regarded as a complex concentrated in degrees $0$ up to $n-1$. Clearly, every morphism $\alpha: \xi \rightarrow \mathsf \zeta$ in $\mathcal Ext^n_\C(X,Y)$ gives rise to a morphism $\alpha: \mathbb E \rightarrow \mathbb F$ of complexes.
\end{nota}
\begin{lem}\label{lem:prop_ext1} Let $\alpha: \xi \rightarrow \zeta$ be a morphism in $\mathcal Ext^n_{\C}(X,Y)$.
\begin{enumerate}[\rm(1)]
\item\label{lem:prop_ext1:1} If $\alpha$ is an admissible monomorphism in $\Ch(\C)$, the pushout of any diagram of the form
$$
\xymatrix{\zeta' & \mathsf \xi \ar[r]^-\alpha \ar[l] & \mathsf \zeta}
$$
exists in $\mathcal Ext^n_{\C}(X,Y)$.
\item\label{lem:prop_ext1:2} If $\alpha$ is an admissible epimorphism in $\Ch(\C)$, the pullback of any diagram of the form
$$
\xymatrix{\xi' \ar[r] & \zeta & \mathsf \xi \ar[l]_-\alpha}
$$
exists in $\mathcal Ext^n_{\C}(X,Y)$.
\end{enumerate}
\end{lem}
\begin{proof}
We only prove (\ref{lem:prop_ext1:1}) since (\ref{lem:prop_ext1:2}) follows similarly. The pushout $\mathbb P$ of the described diagram surely exists in $\Ch(\C)$ for $\Ch(\C)$ is exact. Since an admissible monomorphism in $\Ch(\C)$ is given by a chain morphism being an admissible monomorphism in every degree and, furthermore, a pushout is nothing other than a cokernel of a certain morphism, it is taken degreewise. But if $A$ is an object in $\C$, the pushout of
$$
\xymatrix{A & A \ar[r]^-{\id_A} \ar[l]_-{\id_A}  & A}
$$
clearly is $A$ together with the identity morphisms. So the diagram
$$
\xymatrix{\xi \ar[r]^-\alpha \ar[d] & \zeta \ar[d]\\
\zeta' \ar[r] & \mathbb P}
$$
in $\Ch(\C)$ actually belongs to $\mathcal Ext^n_{\C}(X,Y)$.
\end{proof}
\begin{defn}\label{def:closedker}
The exact category $\C$ is
\begin{enumerate}[\rm(1)]
\item \textit{closed under kernels of epimorphisms} if for a morphism $f$ in $\C$ the following condition is satisfied:
$$
\text{$f$ is an admissible epimorphism in $\C$ \ $\Longleftrightarrow$ \ $i(f)$ is an epimorphism in $\A$}.
$$
That is, $i$ \textit{detects admissible epimorphisms}.
\item \textit{closed under cokernels of mono\-morphisms} if for a morphism $f$ in $\C$ the following condition is satisfied:
$$
\text{$f$ is an admissible monomorphism in $\C$ \ $\Longleftrightarrow$ \ $i(f)$ is a monomorphism in $\A$}.
$$
That is, $i$ \textit{detects admissible monomorphisms}.\end{enumerate}
\end{defn}
\begin{exas}\label{exa:condition}
\begin{enumerate}[\rm(1)]
\item Clearly, if $\C$ is closed under kernels of epimorphisms (cokernels of monomorphisms), then $\C^\op$ is closed under cokernels of monomorphisms (kernels of epimorphisms).
\item Every abelian category is closed under kernels of epimorphisms and under cokernels of monomorphisms.
\item Let $A$ be a $k$-algebra. The exact category $\P(A)$ (cf. Example \ref{exa:bimodules}) is closed under kernels of epimorphisms: Let $f: P \rightarrow Q$ be a surjective map between objects in $\P(A)$. Then $\Ker(f)$ is, as a left $A$-module, a direct summand of $P$, and therefore $A$-projective on the left. Similarly, $\Ker(f)$ is also $A$-projective on the right and hence belongs to $\P(A)$.
\item Let $\mathcal B$ be a bialgebra over $k$. The exact category $\C_{\Proj(k)}(\mathcal B)$ (see Example \ref{exa:bialgebras}) is closed under kernels of epimorphisms.
\end{enumerate}
\end{exas}
\begin{nn}Let $\beta: \xi \rightarrow \zeta$ be a morphism in $\mathcal Ext^n_\C(X,Y)$, and assume that $\xi$ and $\zeta$ are given as follows.
$$
\xymatrix@C=18pt@R=10pt{
\xi & \equiv & 0 \ar[r] & Y \ar[r]^-{e_n} & E_{n-1} \ar[r]^-{e_{n-1}} & E_{n-2} \ar[r]^-{e_{n-2}}
&\cdots \ar[r]^-{e_1} & E_0 \ar[r]^-{e_0} & X \ar[r] & 0 \\
\zeta & \equiv & 0 \ar[r] & Y \ar[r]^-{f_{n}} & F_{n-1} \ar[r]^{f_{n-1}} & F_{n-2} \ar[r]^-{f_{n-2}} &
\cdots \ar[r]^-{f_1} & F_0 \ar[r]^-{f_0} & X \ar[r] & 0
}
$$
The objects
$$
X_p =
\begin{cases}
  E_{n-2},  & \text{if $p = n-1$},\\
  E_p \oplus E_{p-1}, & \text{if $1 \leq p \leq n-2$},\\
  E_0,  & \text{if $p = 0$},
\end{cases}
$$
in $\mathsf C$ form a complex $\mathbb X$ whose differentials $x_p: X_p \rightarrow X_{p-1}$ are given by $x_p = 0$ for $p \notin [1,n-1]$ and
$$
\xymatrix@C=18pt{
x_p: E_p \oplus E_{p-1} \ar[r]^-{\mathrm{can}} & E_{p-1} \ar[r]^-{\mathrm{can}} & E_{p-1} \oplus E_{p-2}
} \quad \text{(for $2 \leq p \leq n-2$)},
$$
whereas $x_{n-1}: E_{n-2} \rightarrow E_{n-2} \oplus E_{n-3}$ and $x_1: E_1 \oplus E_0 \rightarrow E_0$ are given by the canonical morphisms. The complex $\mathbb X \in \Ch(\C)$ is such that $\zeta \oplus \mathbb X$ belongs to $\mathcal Ext^n_\C(X,Y)$. Let us define a morohism $\widehat{\beta}: \xi \rightarrow \zeta \oplus \mathbb X$ in $\mathcal Ext^n_\C(X,Y)$; degreewise, it is given by $\widehat{\beta}_n = \id_Y$, $\widehat{\beta}_{-1} = \id_X$,
\begin{align*}
\widehat{\beta}_{n-1}:& \ E_{n-1} \longrightarrow F_{n-1} \oplus E_{n-2}, \ &&\widehat{\beta}_{n-1} = \beta_{n-1} \oplus e_{n-1},\\
\widehat{\beta}_p:& \ E_p \longrightarrow F_p \oplus E_p \oplus E_{p-1}, &&\widehat{\beta}_p = \beta_p \oplus \id_{E_p} \oplus
e_p \quad \quad \quad (\text{for $1 \leq p \leq n-2$}),\\
\intertext{and}
\widehat{\beta}_{0}:& \ E_0 \longrightarrow F_0 \oplus E_0, &&\widehat{\beta}_n = \beta_0 \oplus \id_{E_0}.
\end{align*}
One immediately verifies that $\widehat{\beta}$ is a morphism of chain complexes (see \cite{Sch98} for the case of modules over a ring).
\end{nn}
\begin{lem}\label{lem:monomorphism}
The morphisms $\widehat{\beta}$ and $\gamma := i(\widehat{\beta})$ are degreewise monomorphisms \emph{(}in particular, they are respective monomorphisms in $\Ch(\C)$ and $\Ch(\A)$\emph{)}.
\end{lem}
\begin{proof}
For every $-1 \leq p \leq n$, $p \neq n-1$, $\gamma_p$ is a monomorphism by definition. This also holds true for $p = n-1$: We have to show that if $t$ is a morphism in $\A$ with $\gamma_{n-1} \circ t = 0$, then $t = 0$. In the abelian category $\A$, $\gamma_{n-1}$ has a kernel $\Ker(\gamma_{n-1})$; let $\ker(\gamma_{n-1})$ be the corresponding universal morphism. In particular, we get that 
$$
i(\beta_{n-1}) \circ \ker(\gamma_{n-1}) = 0 = i(e_{n-1}) \circ \ker(\gamma_{n-1}).
$$
The latter equality yields a unique morphism $h: \Ker(\gamma_{n-1}) \rightarrow iY$ with $i(e_n) \circ h= \ker(\gamma_{n-1})$ (remember, that $\xi$ is admissible exact, and thus $i(e_n)$ is the kernel of $i(e_{n-1})$). Therefore
$$
i(f_n) \circ h = i(\beta_{n-1} \circ e_n) \circ h = i(\beta_{n-1}) \circ i(e_n) \circ h = 0,
$$
and thus $h = 0$ (since $i(f_n)$ is a monomorphism). Hence the kernel of $\gamma_{n-1} = i(\widehat{\beta}_{n-1})$ is trivial. Since $i$ is injective on morphisms, also $\widehat{\beta}$ is a monomorphism in every degree.
\end{proof}
\begin{lem}\label{lem:admono}
Consider the following statements.
\begin{enumerate}[\rm(1)]
\item\label{lem:admono:1} $\C$ is closed under kernels of epimorphisms.
\item\label{lem:admono:2} $\C$ is closed under cokernels of monomorphisms.
\item\label{lem:admono:3} The morphism $\widehat{\beta}$ is an admissible monomorphism in $\Ch(\C)$.
\end{enumerate}
Then the implications $(\ref{lem:admono:1}) \Longrightarrow (\ref{lem:admono:3})$ and $(\ref{lem:admono:2}) \Longrightarrow (\ref{lem:admono:3})$ hold true.
\end{lem}
\begin{proof}
Assume that $\C$ is closed under kernels of epimorphisms. First of all, note that, since $i(\widehat{\beta})$ is a monomorphism between acyclic complexes in $\Ch(\A)$, its cokernel in $\Ch(\A)$ will also be acyclic. Let $1 \leq p \leq n-2$. We have the following commutative diagram in $\A$, having exact rows.
$$
\xymatrix@R=40pt{
0 \ar[r] & iE_p \ar[r]^-{i(\widehat{\beta}_p)} \ar@{=}[d] & iF_p \oplus iE_p \oplus iE_{p-1} \ar[r] \ar[d]^{
\left[\begin{smallmatrix}
\id & -i(\beta_{p}) & 0 \\
0 & \id & 0 \\
0 & -i(e_{p}) & \id
\end{smallmatrix}\right]
} & \Coker(i(\widehat{\beta}_p)) \ar[r]  \ar@{-->}[d]& 0 \\
0 \ar[r] & iE_p \ar[r] & iF_p \oplus iE_p \oplus iE_{p-1} \ar[r] & iF_p \oplus iE_{p-1} \ar[r] & 0
}
$$
Here the dashed arrow is induced by the universal property of $\Coker(i(\widehat{\beta}_p))$. By the 5-Lemma, this morphism is already an isomorphism and therefore $\Coker(i(\widehat{\beta}_p)) \in i\C$. Similarly, $\Coker(i(\widehat{\beta}_0)) \in i\C$. It follows, that
$$
\Ker\left(\Coker(i(\widehat{\beta}_1)) \rightarrow \Coker(i(\widehat{\beta}_0))\right) \in i\C
$$
since $\C$ is closed under kernels of epimorphisms. By induction, we obtain:
$$
\Ker\left(\Coker(i(\widehat{\beta}_p)) \rightarrow \Coker(i(\widehat{\beta}_{p-1}))\right) \in i\C \quad (\text{for all $1 \leq p \leq n-2$}).
$$
But, notably, this means that
$$
\Coker(i(\widehat{\beta}_{n-1})) = \Ker\left(\Coker(i(\widehat{\beta}_{n-2})) \rightarrow
\Coker(i(\widehat{f}_{n-3}))\right) \in i\C
$$
and we are done. Since we as well could have investigated the cokernels instead of the kernels of the
morphisms occuring in $\Coker(i(\widehat{\beta}))$, the same proof works in case $\C$ is closed under cokernels of monomorphisms.
\end{proof}
\begin{defn}\label{def:factorizing}
The exact $k$-category $\C$ is called \textit{factorizing} if for any choice of $n \geq 1$, $X,Y \in \Ob \C$ and any morphism $\beta$ in $\mathcal Ext^n_\C(X,Y)$, the morphism $\widehat{\beta}$ is an admissible monomorphism in $\Ch(\C)$.
\end{defn}
By Lemma \ref{lem:admono} the exact category $\C$ is factorizing if it is closed under kernels of epimorphisms \textit{or} under cokernels of monomorphisms. The terminology is motivated by the following lemma.
\begin{lem}[{\cite[Lem.\,4.4]{Sch98}}]\label{lem:prop_ext2} If $\C$ is factorizing, then there is an
admissible monomorphism $\iota$ and a split epimorphism $\pi$ in $\Ch(\C)$, both belonging to $\mathcal Ext_\C^n(X,Y)$, such that the morphism $\beta: \xi \rightarrow \zeta$ factors as $\beta =
\pi \circ \iota$.
\end{lem}
\begin{proof}
Let $\pi: \zeta \oplus \mathbb X \rightarrow \zeta$ be the canonical projection and put $\iota = \widehat{\beta}$. Clearly, $\pi$ is a split epimorphism and $\pi \circ
\iota = \beta$. Moreover, $\widehat{\beta}$ is an admissible monomorphism since $\C$ is factorizing.
\end{proof}
\section{Homotopy groups}\label{sec:homo_groups}
\begin{nn}
Let $\C$ be a small category. Recall that the \textit{classifying space $\mathscr B(\C)$ of $\C$} is given by the geometrical realization of the nerve $\mathscr N \C$ of $\C$. The \textit{$i$-th homotopy group $\pi_i\C$ of $\C$} then arises as the $i$-th homotopy group of the classifying space of $\C$, that is, $\pi_i\C := \pi_i(\mathscr B(\C))$ for $i \geq 0$. The $1$st homotopy group of $\C$ (i.e., $\pi_1 \C$) is also called the \textit{fundamental group of $\C$} (or, to be more precise, the fundamental group of $\C$ at some (suppressed) base point $X \in \Ob \C$). We will explain how to understand $\pi_0 \C$ and $\pi_1 \C$ in terms of morphisms in $\C$.
\end{nn}
\begin{nn}
Let $X, Y$ be objects in $\C$. A \textit{path $w = w(X,Y)$ from $X$ to $Y$ in $\C$} is a sequence of
objects $X_0, X_1, \dots, X_n \in \C$ and morphisms $a_0, a_1, \dots, a_{n-1}$ in $\C$ with
$$
a_i \in \Hom_\C(X_i, X_{i+1}) \quad \text{or} \quad a_i \in \Hom_\C(X_{i+1}, X_{i}) \quad \text{(for $0 \leq i \leq n-1$)},
$$
and $X_0 = X$ and $X_n = Y$. We will denote such a path $w$ by
$$
\xymatrix{
X = X_0 \ar@{<->}[r]^-{a_0} & X_1 \ar@{<->}[r]^-{a_1} & \cdots \ar@{<->}[r]^-{a_{n-2}} & X_{n-1}
\ar@{<->}[r]^-{a_{n-1}} & X_n = Y \ .
}
$$
The \textit{length of the path $w$} is $\mathrm{length}(w) :=n$. A path
from $X$ to $X$ in $\C$ is called a \textit{loop based at $X$}. Let $\mathrm{Path}_\C(X,Y)$ be the set of all paths
form $X$ to $Y$ in $\C$. Clearly, $\mathrm{Path}_\C(X,Y) \neq \emptyset$ if $\Hom_\C(X,Y) \cup \Hom_\C(Y,X) \neq \emptyset$. Moreover, note that the assignments
$$
X_i \mapsto Y_{n-i} = X_i, \ a_i \mapsto b_{n-i} = a_i \quad \text{(for $0 \leq i \leq n$)}
$$
induce a bijection $\mathrm{Path}_\C(X,Y) \rightarrow \mathrm{Path}_\C(Y,X)$. If $w$ is a path from $X$ to $Y$ in $\C$, we will denote its image under this map by (the formal symbol) $w^{-1}$. Henceforth,
$$
X \leftrightsquigarrow Y \ \Longleftrightarrow \ \mathrm{Path}_\C(X,Y) \neq \emptyset
$$
defines an equivalence relation on the set of objects of $\C$. Let $\pi_0 \C$ be the set of equivalence classes with respect to
$\leftrightsquigarrow$. If $\D$ is another small category and if $\scrX: \C \rightarrow \D$ is a functor, then $\scrX$ will respect the
relation $\leftrightsquigarrow$, i.e., $X \leftrightsquigarrow Y$ in $\C$ will imply $\scrX X \leftrightsquigarrow \scrX Y$ in $\D$. This shows that $\scrX$ gives rise to a well-defined map $\scrX^\sharp = \pi_0 \scrX : \pi_0 \C \rightarrow \pi_0 \D$. The category $\C$ is called \textit{connected} if $\mathrm{Path}_\C(X,Y) \neq \emptyset$ for every choice of objects $X$ and $Y$.
The \textit{path category of $\C$} is the category $\Pi(\C)$ given by the data
\begin{align*}
\mathrm{Ob}(\Pi(\C)) & = \mathrm{Ob}(\C)\\
\Hom_{\Pi(\C)}(X,Y) & = \mathrm{Path}_\C(X,Y) \quad (\text{for $X,Y \in \C$}).
\end{align*}
The composition of morphisms in $\Pi(\C)$ is given by splicing paths together (in the evident way). It is apparent that $\C$ faithfully embeds into $\Pi(\C)$. Let $w$ and $w'$ be paths from $X$ to $Y$ such that $\mathrm{length}(w) = \mathrm{length}(w') - 1$.
We say that $w$ and $w'$ are \textit{elementary homotopic} if there is a commutative diagram
\begin{equation*}
\Delta \quad \equiv \quad
\begin{aligned}
\xymatrix@!C=15pt{
& X_1 & \\
X_2 \ar[ur] \ar[rr] & & X_3 \ar[ul]
}
\end{aligned}
\end{equation*}
in $\C$ such that one arrow occurs in $w'$ and $w$ arises from $w'$ by replacing this arrow by the other two. Let $\sim$ be the smallest equivalence relation on $\mathrm{Path}_\C(X,Y)$ such that two paths of length difference $1$ are equivalent if, and only if, they are elementary homotopic. If $w \sim w'$ for some paths $w$ and $w'$, we say that $w$ and $w'$ are \textit{homotopically equivalent}. Put $$\underline{\mathrm{Path}}_\C(X,Y) := \frac{\mathrm{Path}_\C(X,Y)}{\sim}.$$ Since composition of paths respects the elementary homotopic relations, it induces a well defined map modulo $\sim$, i.e., we obtain a category $\mathsf G(\C)$:
\begin{align*}
\mathrm{Ob}(\mathsf G(\C)) & = \mathrm{Ob}(\C)\\
\Hom_{\mathsf G(\C)}(X,Y) & = \underline{\mathrm{Path}}_\C(X,Y) \quad (\text{for $X,Y \in \C$}).
\end{align*}
\end{nn}
\begin{defn}\label{def:quigroup}
Let $\mathsf G$ be a category. We call $\mathsf G$ a \textit{groupoid} if every morphism in $\mathsf G$ is invertible.
\end{defn}
\begin{nn}
The category $\mathsf G(\C)$ clearly is a groupoid and we call it the \textit{Quillen} (or \textit{fundamental}) \textit{groupoid associated to $\C$}. By mapping a morphism $f: X \rightarrow Y$ in $\C$ to the equivalence class of $f$ in $\Hom_{\mathsf G(\C)}(X,Y)$, we obtain a dense functor $g_\C : \C \rightarrow \mathsf G(\C)$. This functor is in general neither full nor faithful. For instance, if $\C$ has an initial or a terminal object, then $\C$ is contractible by \cite[\S 1, Cor.\,2]{Qu72}, and hence $\End_{\mathsf G(\C)}(X)$ is trivial for every object $X$ (since it is isomorphic to $\pi_1(\C,X)$ by the following proposition), whereas $\End_\C(X)$ does not have to be. However, the pair $(\mathsf G(\C), g_\C)$ has the following universal property: If $\scrX: \C \rightarrow \D$ is a functor such that $\scrX(f)$ is invertible in $\D$ for every morphism $f$ in $\C$, then $\scrX$ uniquely factors through $g_\C$.
\end{nn}
\begin{prop}\label{prop:iso_pi1}
Let $\C$ be an essentially small category. Then, for every object $X$ in $\C$, $\pi_1(\C,X) \cong \End_{\mathsf{G}(\C)}(X)$ as groups. Therefore every element of $\pi_1(\C,X)$ may be considered as a loop based at $X$ \emph{(}up to homotopy\emph{)}.
\end{prop}
\begin{proof}
See \cite[Prop.\,1]{Qu72}.
\end{proof}
\begin{nn}
Let $\scrX: \C \rightarrow \D$ be a functor into a small category $\D$. By the universal property of the Quillen groupoid, we obtain a commutative diagram
$$
\xymatrix{
\C \ar[r]^-{\scrX} \ar[d]_-{g_\C} & \D \ar[d]^-{g_\D}\\
\mathsf G(\C) \ar@{-->}[r]^-{\scrX^\flat} & \mathsf G(\D)
}
$$
telling us that $\scrX$ induces a group homomorphism
$$
\pi_1(\scrX, X): \pi_1(\C,X) \cong \End_{\mathsf G(\C)}(X) \longrightarrow \End_{\mathsf G(\D)}(\scrX X) \cong \pi_1(\D,\scrX X)
$$
for every object $X$ in $\C$.
\end{nn}
\begin{nn}
Let $X'$ and $Y'$ be objects in $\C$ and $w_1$ and $w_2$ be paths from $X'$ to $X$ and $Y'$ to $Y$ respectively. The map
$$
\mathrm{Path}_\C(X,Y) \longrightarrow \mathrm{Path}_\C(X',Y'), \ w \mapsto w_1 w w_2^{-1}
$$
induces an isomorphism of groups:
$$
c_{w_1, w_2}: \underline{\mathrm{Path}}_\C(X,Y) \longrightarrow \underline{\mathrm{Path}}_\C(X',Y').
$$
For $w_1 = w_2 = w$ we will denote $c_{w,w}$ by $c_w$; it is simply conjugating with the group element defined by ${w}$. In
particular, we obtain an isomorphism
$$
c_w: \pi_1(\C,X) \cong \End_{\mathsf G(\C)}(X) \longrightarrow \End_{\mathsf G(\C)}(Y) \cong \pi_1(\C,Y),
$$
if $w$ is a path from $Y$ to $X$ in $\C$.
\end{nn}
\begin{lem}\label{lem:fund_conj_commutative}
Let $\scrX: \C \rightarrow \D$ be a functor and let $X$ and $Y$ be objects in $\C$. Let $v,w \in \mathrm{Path}_\C(Y,X)$. Then the diagram
\begin{equation}\label{eq:pi1_diagram}
\begin{aligned}
\xymatrix@C=35pt{
\pi_1(\C,X) \ar[r]^-{\pi_1(\scrX, X)} \ar[d]_{c_{v,w}} & \pi_1(\D,\scrX X) \ar[d]^-{c_{\scrX v, \scrX w}}\\
\pi_1(\C,Y) \ar[r]^-{\pi_1(\scrX, Y)} & \pi_1(\D,\scrX Y)
}
\end{aligned}
\end{equation}
commutes.
\end{lem}
\begin{proof}
This is almost obvious, since the following extended version of the diagram (\ref{eq:pi1_diagram}) is commutative.
$$
\xymatrix@C=35pt{
\Hom_{\mathsf G(\C)}(X,X) \ar[r]^-{\scrX^\flat_{X,X}} \ar[d]_{\Hom_{\mathsf G(\C)}(X,w^{-1})} & \Hom_{\mathsf G(\D)}(\scrX X,\scrX X)
\ar[d]^-{\Hom_{\mathsf G(\D)}(\scrX X,\scrX w^{-1})}\\
\Hom_{\mathsf G(\C)}(X,Y) \ar[r]^-{\scrX^\flat_{X,Y}} \ar[d]_{\Hom_{\mathsf G(\C)}(v,Y)} & \Hom_{\mathsf G(\D)}(\scrX X,\scrX Y)
\ar[d]^-{\Hom_{\mathsf G(\D)}(\scrX v,\scrX Y)} \\
\Hom_{\mathsf G(\C)}(Y,Y) \ar[r]^-{\scrX^\flat_{Y,Y}} & \Hom_{\mathsf G(\D)}(\scrX Y,\scrX Y)
}
$$
Observe that the compositions of the vertical arrows correspond to $c_{v,w}$ and $c_{\scrX v, \scrX w}$ respectively.
\end{proof}
\section{Lower homotopy groups of extension categories}\label{sec:lowhomotopy}
\begin{nn}
Throughout this section, let $\C$ be an exact $k$-category and $n \geq 1$ be an integer. We fix objects $A,B, X, Y$ in $\C$. The $k$-modules $\Hom_\C(A,X)$ and $\Hom_\C(Y,B)$ (externally) act on the category $\mathcal Ext^n_\C(X,Y)$. In order to be more specific, let
$$
\xi \quad \equiv \quad 0 \longrightarrow Y \longrightarrow \mathbb E \longrightarrow X \longrightarrow 0
$$
be an admissible $n$-extension in $\mathcal Ext^n_\C(X,Y)$ and let $f \in \Hom_\C(A,X)$ and $g \in \Hom_\C(Y,B)$. We obtain admissible $n$-extensions $\xi \dashv g$ in $\mathcal Ext^n_\C(A,Y)$ and $f \vdash \xi$ in $\mathcal Ext^n_\C(X,B)$ through the following commutative diagram (arising by respectively taking the pullback and the pushout of the obvious morphisms).
$$
\xymatrix@C=15pt@R=18pt{
\xi \dashv f & \equiv & 0 \ar[r] & Y \ar@{=}[d] \ar[r] & E_{n-1} \ar@{=}[d] \ar[r] & E_{n-2} \ar@{=}[d] \ar[r] & \cdots \ar[r] & E_1 \ar@{=}[d] \ar[r] & P_f \ar[d] \ar[r] & A \ar[r] \ar[d]^-{f} & 0\\
\xi & \equiv & 0 \ar[r] & Y \ar[d]_-g \ar[r]^-{e_n} & E_{n-1} \ar[d] \ar[r] & E_{n-2} \ar@{=}[d] \ar[r] & \cdots \ar[r] & E_1\ar@{=}[d] \ar[r] & E_0 \ar[r]^-{e_0} \ar@{=}[d] & X \ar@{=}[d] \ar[r] & 0\\
g \vdash \xi & \equiv & 0 \ar[r] & B \ar[r] & Q_g \ar[r] & E_{n-2} \ar[r] & \cdots \ar[r] & E_1 \ar[r] & E_0 \ar[r] & X \ar[r] & 0
}
$$
In fact, by the universality of pullbacks and pushouts, we get functors
\begin{align*}
(-) \dashv f& : \mathcal Ext^n_\C(X,Y) \longrightarrow \mathcal Ext^n_\C(A,Y), \\
g \vdash(-) &: \mathcal Ext^n_\C(X,Y) \longrightarrow \mathcal Ext^n_\C(X,B).
\end{align*}
Assume that $f$ and $g$ are isomorphisms. Then $\xi \dashv f$ and $g \vdash \xi$ are of the following form:
\begin{equation}\label{eq:iso_mult}
\begin{aligned}
\xymatrix@C=30pt@R=12pt{
0 \ar[r] & Y \ar[r]^-{e_n} & E_{n-1} \ar[r]^-{e_{n-1}} & \cdots \ar[r]^-{e_1} & E_0 \ar[r]^-{f^{-1} \circ e_0} & A \ar[r] & 0 \ ,\\
0 \ar[r] & B \ar[r]^-{g^{-1} \circ e_n} & E_{n-1} \ar[r]^-{e_{n-1}} & \cdots \ar[r]^-{e_1} & E_0 \ar[r]^-{e_0} & X \ar[r] & 0 \ .
}
\end{aligned}
\end{equation}
Now assume that $A = X$ and $B = Y$. Since $\End_\C(X)$ and $\End_\C(Y)$ are $k$-algebras, we obtain two $k$-actions on $\mathcal Ext^n_\C(X,Y)$ which we also denote by $\dashv$ and $\vdash$. They are equivalent in the following sense:
$$
\forall a \in k : \xi \dashv (a \id_X) = (a \id_Y) \vdash \xi \quad \text{in $\pi_0\mathcal Ext^n_\C(X,Y)$.}
$$
If $a \in k$ is an invertible element, the $n$-extensions $\xi \dashv (a \id_X)$ and $(a \id_Y) \vdash \xi$ appear as the rows of the following commutative diagram (compare with (\ref{eq:iso_mult})):
$$
\xymatrix@C=25pt{
0 \ar[r] & Y \ar[r]^-{e_n} \ar@{=}[d] & E_{n-1} \ar[r]^-{e_{n-1}} \ar[d]_-{a^{-1}} & \cdots \ar[r]^-{e_1} & E_0 \ar[r]^-{a^{-1}e_0} \ar[d]^-{a^{-1}} & X \ar[r] \ar@{=}[d] & 0 \ ,\\
0 \ar[r] & Y \ar[r]^-{a^{-1}e_n} & E_{n-1} \ar[r]^-{e_{n-1}} & \cdots \ar[r]^-{e_1} & E_0 \ar[r]^-{e_0} & X \ar[r] & 0 \ .
}
$$
Let us remark that if $\C$ is abelian and $\xi$ is a short exact sequence $0 \rightarrow Y \rightarrow E \rightarrow X \rightarrow 0$ in $\C$, then $f \mapsto \xi \dashv f$ and $g \mapsto (-g) \vdash \xi$ are precisely the connecting homomorphisms
$$
\Hom_\C(A,X) \longrightarrow \Ext^1_\C(A,Y), \quad \Hom_\C(Y,B) \longrightarrow \Ext^1_\C(X,B)
$$
in the long exact sequences corresponding to $\xi$ (cf. \cite[\S 7, Prop.\,6.5]{Bou07} and \cite[XIV.1]{CaEi56}).
\end{nn}
\begin{nn} We are going to endow $\pi_0 \mathcal Ext^n_\C(X,Y)$ with the structure of an abelian group. Let $\xi$ and $\zeta$ be objects in $\mathcal Ext^n_\C(X,Y)$: 
\begin{align*}
\xi &\quad \equiv \quad 0 \longrightarrow Y \longrightarrow \mathbb E \longrightarrow X \longrightarrow 0\ ,\\
\zeta &\quad \equiv \quad 0 \longrightarrow Y \longrightarrow \mathbb F \longrightarrow X \longrightarrow 0 \ .
\end{align*}
For $n = 1$, we let $\xi \boxplus \zeta$ be the lower row of the following diagram:
$$
\xymatrix{
0 \ar[r] & Y \oplus Y \ar[r]  \ar@{=}[d] & E_0 \oplus F_0 \ar[r] & X \oplus X \ar[r] & 0\\
0 \ar[r] & Y \oplus Y \ar[d]_{
\left[\begin{smallmatrix}
1 & 1 \end{smallmatrix}\right]} \ar[r] & P \ar[d] \ar[r] \ar[u] & X \ar[u]_{
\left[\begin{smallmatrix}
1\\
1
\end{smallmatrix}\right]
} \ar[r] & 0\\
0 \ar[r] & Y \ar[r] & Q \ar[r] & X \ar@{=}[u] \ar[r] & 0
}
$$
If $n \geq 2$ we define $\xi \boxplus \zeta$ to be the equivalence class of the lower row of the commutative diagram 
$$
\xymatrix@C=18pt{
0 \ar[r] & {Y \oplus Y} \ar[r] & {E_{n-1} \oplus F_{n-1}} \ar[r] & \cdots \ar[r] & E_0 \oplus F_0 \ar[r]
& {X \oplus X} \ar[r] & 0 \ \ \\
0 \ar[r] & Y \oplus Y \ar[r] \ar@{=}[u] \ar[d]_{
\left[\begin{smallmatrix}
1 & 1 \end{smallmatrix}\right]
} & E_{n-1} \oplus F_{n-1} \ar[r] \ar@{=}[u] \ar[d] & \cdots \ar[r] & P \ar[r] \ar[u] \ar@{=}[d]  & X  \ar[u]_{
\left[\begin{smallmatrix}
1\\
1
\end{smallmatrix}\right]
} \ar@{=}[d] \ar[r] & 0 \ \ \\
0 \ar[r] & Y \ar[r] & Q \ar[r] & \cdots \ar[r] & P \ar[r] & X \ar[r] & 0 \ .
}
$$
Here, in both situations, $P$ denotes the pullback and $Q$ the pushout of the obvious diagrams. Clearly, the above construction is functorial. Therefore we obtain a bifunctor
$$
\boxplus: \mathcal Ext^n_\C(X,Y) \times \mathcal Ext^n_\C(X,Y) \rightarrow \mathcal Ext^n_\C(X,Y)
$$
and hence a well defined map
$$
+: \pi_0 \mathcal Ext^n_\C(X,Y) \times \pi_0 \mathcal Ext^n_\C(X,Y) \longrightarrow \pi_0 \mathcal Ext^n_\C(X,Y), \ ([\xi], [\zeta]) \mapsto [\xi \boxplus \zeta].
$$
There is a distinguished admissible $n$-extension $\sigma_n(X,Y)$ in $\mathcal Ext^n_\C(X,Y)$. Namely, if $n = 1$, let $\sigma_n(X,Y)$ be the
split admissible short exact sequence
$$
\xymatrix@C=20pt{
0 \ar[r] & Y \ar[r] & Y \oplus X \ar[r] & X \ar[r] & 0 \ .
}
$$
In all other cases, $\sigma_n(X,Y)$ is defined to be the admissible $n$-extension
$$
\xymatrix@C=20pt{
0 \ar[r] & Y \ar@{=}[r] & Y \ar[r] & 0 \ar[r] & \cdots \ar[r] & 0 \ar[r] & X \ar@{=}[r] & X \ar[r] & 0 \ .
}
$$
It is apparent that $[\xi \boxplus \sigma_n(X,Y)] = [\xi] = [\sigma_n(X,Y) \boxplus \xi]$ for every admissible $n$-extension $\xi$ of $X$ by $Y$.
\end{nn}
\begin{lem}[{\cite[III.5]{MaL95}}]\label{lem:extmodule}
The functors $\boxplus$ and $\vdash$ turn $\pi_0 \mathcal Ext^n_\C(X,Y)$ into a $k$-module. The identity element with respect to $+$ is given by the equivalence class of the sequence $\sigma_n(X,Y)$, whereas the inverse element of the equivalence class of some admissible $n$-extension $\xi$ is given by the equivalence class of $-\xi = (- \id_Y) \vdash \xi$.
\end{lem}
\begin{nn}
The operation $+$ on $\pi_0 \mathcal Ext^n_\C(X,Y)$ is commonly known as the \textit{Baer sum}. We will denote the $k$-module obtained from $\pi_0 \mathcal Ext^n_\C(X,Y)$ by $\Ext^n_\C(X,Y)$ and call it the \textit{$k$-module of admissible $n$-extensions of $X$ by $Y$}. Recall that if $\C$ is an abelian category with enough projectives, $\Ext^n_\C(X,Y)$ coincides with the value of the $n$-th right derived funtor of $\Hom_\C(-,Y)$ at $X$: $$\Ext^n_\C(X,Y) \cong R^n \Hom_\C(-,Y)(X) = H^n(\Hom_\C(\mathbb P_X,Y)),$$ where $\mathbbm P_X \rightarrow X \rightarrow 0$ is some projective resolution of $X$ in $\C$. 
\end{nn}
\begin{nn}\label{nn:pathximinxi}
For later usage, we will convince ourselves that, for given $\xi \in \mathcal
Ext^n_\C(X,Y)$,
$$
\xi \boxplus (-\xi) \ \text{is equivalent to} \ \sigma_n(X,Y)
$$
which precisely means that $[\xi] + [-\xi] = [\sigma_n(X,Y)]$. In Section \ref{sec:retakh2} it will be helpful to have a concrete path between $\xi \boxplus (-\xi)$ and $\sigma_n(X,Y)$ at hand.

To begin with, we consider the case $n = 1$. We have the following extended commutative diagram defining $\xi \boxplus (-\xi)$
$$
\xymatrix@C=27.5pt{
\xi & \equiv & 0 \ar[r] & Y \ar[r] & E_0 \ar[r]^{e_0} & X \ar[r] & 0 \ \ \\
\xi \oplus (-\xi) & \equiv & 0 \ar[r] & Y \oplus Y \ar[u]^{
\left[\begin{smallmatrix}
1 & 1 \end{smallmatrix}\right]} \ar[r]|{ \ \ }  \ar@{=}[d] & E_0 \oplus E_0 \ar[u]^{
\left[\begin{smallmatrix}
1 & 1
\end{smallmatrix}\right]
}
\ar[r]^{
\left[\begin{smallmatrix}
e_0 & 0 \\
0 & -e_0
\end{smallmatrix}\right]
} & X \oplus X \ar[u]_
{
\left[\begin{smallmatrix}
1 & -1
\end{smallmatrix}\right]
} \ar[r] & 0 \ \ \\
& & 0 \ar[r] & Y \oplus Y \ar[d]_{
\left[\begin{smallmatrix}
1 & 1 \end{smallmatrix}\right]} \ar[r] & P \ar[uul]_(0.7)u \ar[d] \ar[r] \ar[u] & X \ar[u]_{
\left[\begin{smallmatrix}
1\\
1
\end{smallmatrix}\right]} \ar[r] & 0 \ \
\\
\xi \boxplus (-\xi) & \equiv & 0 \ar[r] & Y \ar[r]^f & Q \ar[r]^g & X \ar@{=}[u] \ar[r] & 0 \ ,
}
$$
where the arrow $u : P \rightarrow Y$ is induced by the universal property of the kernel of $e_0$
(the kernel of $e_0$ is given by the pair $(Y,e_1)$). But now one easily checks that left diagram below commutes, which yields a unique morphism $s: Q \rightarrow Y$ completing it to the commutative diagram on the right hand side.
$$
\xymatrix{
Y \oplus Y \ar[d]_{
\left[\begin{smallmatrix}
1 & 1 \end{smallmatrix}\right]} \ar[r] & P \ar@/^1pc/[ddr]^u \ar[d] &\\
Y \ar@/^-1pc/[drr]_{\id_Y} \ar[r]^f & Q &\\
&& Y
}
\quad \quad \quad
\xymatrix{
Y \oplus Y \ar[d]_{
\left[\begin{smallmatrix}
1 & 1 \end{smallmatrix}\right]} \ar[r] & P \ar@/^1pc/[ddr]^u \ar[d] &\\
Y \ar@/^-1pc/[drr]_{\id_Y} \ar[r]^f & Q \ar[dr]^s &\\
&& Y
}
$$
Therefore, $f$ is a split monomorphism and the desired isomorphism $Q \cong Y \oplus X$ is given by
$$
\left[\begin{smallmatrix}
f \circ s\\
\id_Q - f \circ s
\end{smallmatrix}\right] :Q \rightarrow \Ker(g) \oplus \Ker(s) \cong Y \oplus X.
$$
Suppose that $n \geq 2$. The above considerations tell us, that the lower sequence in the diagram
$$
\xymatrix@C=35pt{
0 \ar[r] & \Ker(e_0) \oplus \Ker(e_0) \ar[r]  \ar@{=}[d] & E_0 \oplus E_0
\ar[r]^{\left[\begin{smallmatrix}
e_0 & 0 \\
0 & -e_0
\end{smallmatrix}\right]
} & X \oplus X \ar[r] & 0 \\
0 \ar[r] & \Ker(e_0) \oplus \Ker(e_0) \ar[d]_{
\left[\begin{smallmatrix}
1 & 1 \end{smallmatrix}\right]} \ar[r] & P \ar[d]^-v \ar[r] \ar[u] & X \ar[u]_{
\left[\begin{smallmatrix}
1\\
1
\end{smallmatrix}\right]
} \ar[r] & 0 \\
0 \ar[r] & \Ker(e_0) \ar[r] & \widetilde{Q} \ar[r]^{\widetilde{g}} & X \ar@{=}[u] \ar[r] & 0
}
$$
is split. Let $r: X \rightarrow \widetilde{Q}$ be such that $\widetilde{g} \circ r = \id_X$. We obtain the following commutative
diagram
$$
\xymatrix{
0 \ar[r] & {Y \oplus Y} \ar[r] & {E_{n-1} \oplus E_{n-1}} \ar[r] & \cdots \ar[r] & E_0 \oplus E_0 \ar[r] & {X \oplus X} \ar[r] & 0\\
0 \ar[r] & Y \oplus Y \ar[r] \ar@{=}[u] \ar[d]_{
\left[\begin{smallmatrix}
1 & 1 \end{smallmatrix}\right]
} & E_{n-1} \oplus E_{n-1} \ar@/^2pc/[dd]^(0.3){
\left[\begin{smallmatrix}
1 & 1 \end{smallmatrix}\right]
}
 \ar[r] \ar@{=}[u] \ar[d] & \cdots \ar[r] & P \ar[r] \ar[u] \ar@{=}[d]  & X  \ar[u]_{
\left[\begin{smallmatrix}
1\\
1
\end{smallmatrix}\right]
} \ar@{=}[d] \ar[r] & 0\\
0 \ar[r] & Y \ar@{=}[d] \ar[r] & Q \ar[d]_-t \ar[r]|(0.37){\ \ } & \cdots \ar[r] & P \ar[d]^-v \ar[r] & X \ar@{=}[d] \ar[r] & 0\\
0 \ar[r] & Y \ar[r] & E_{n-1} \ar[r] & \cdots \ar[r] & \widetilde{Q} \ar[r] & X \ar[r] & 0
}
$$
wherein the arrow $t$ is induced by the universal property of $Q$. (Observe that the squares in between the last two rows
really do commute.) We end up with the path
$$
\xymatrix@C=22pt{
\xi \boxplus (-\xi) \ar[d]_-{} & \equiv & 0 \ar[r] & Y \ar@{=}[d] \ar[r] & Q \ar[d]_t \ar[r] & \cdots \ar[r] & P
\ar[d]^-v \ar[r] & X \ar@{=}[d] \ar[r] &
0 \\
\widetilde{\xi} & \equiv & 0 \ar[r] & Y \ar@{=}[d] \ar[r] & E_{n-1} \ar[r] & \cdots \ar[r] & \widetilde{Q} \ar[r] & X \ar@{=}[d]
\ar[r] &
0 \\
\sigma_n(X,Y) \ar[u]^-{} & \equiv & 0 \ar[r] & Y \ar@{=}[r] & Y \ar[u] \ar[r] & \cdots \ar[r] & X \ar[u]_r \ar@{=}[r] & X
\ar[r] & 0
}
$$
and hence we are done.
\end{nn}
For extension categories, the $0$-th and the $i$-th homotopy groups are linked in the following manner.
\begin{prop}[{\cite[Thm.\,1]{Re86}}]\label{prop:retakhiso}
Let $\A$ be an abelian category. Then, for every $n \geq 1$ and $i = 0, \dots, n$, 
$$
\Ext^{n-i}_\A(B,A) 
\cong \pi_i \mathcal Ext^n_\A(B,A) \quad \text{$($for $A, B \in \Ob\A)$},
$$
as groups. In particular, $\pi_i \mathcal Ext^n_\A(B,A)$ is abelian for all $A, B \in \Ob\A$.
\end{prop}
We will (explicitly) reestablish this isomorphism for $i=1$ in the case of exact $k$-categories which are factorizing (in the sense of Definition \ref{def:factorizing}). This recovers a generalization of Proposition \ref{prop:retakhiso} stated in \cite{NeRe96}. A key ingredient will be the following result.
\begin{lem}[{\cite[Proof of Lem.\,4.5]{Sch98}}]\label{lem:looplength2}
Assume that $\C$ is factorizing. Let $\xi, \xi' \in \mathcal Ext^n_{\C}(X,Y)$ be admissible $n$-extensions and let $w$ be a path in $\mathcal Ext^n_\C(X,Y)$ from $\xi$ to $\xi'$. Then $w$ is homotopically equivalent to a path of the form
$$
\xymatrix{\xi \ar[r] & \zeta & \xi' \ar[l] \ .}
$$
\end{lem}
\begin{proof}
If two adjacent arrows in $w$ point in the same direction, we may compose them to obtain a homotopically equivalent loop $w'$. So, without loss of generality, it can be assumed that two adjacent arrows in $w$ point in opposite directions. In this case, $w$ locally looks as follows:
$$
\xymatrix{
\cdots & \chi_{i-1} \ar[l]_{\alpha_{i-1}} \ar[r]^{\alpha_{i}} & \chi_i & \chi_{i+1} \ar[l]_{\alpha_{i+1}} \ar[r]^{\alpha_{i+2}}& \chi_{i+2} & \cdots \ . \ar[l]_{\alpha_{i+3}}
}
$$
We factorise $\alpha_{i+1} = \pi \circ \iota$, where $\iota$ is an admissible monomorphism in $\Ch(\C)$ and $\pi$ is a split epimorpism (as described in Lemma \ref{lem:prop_ext2}). Since $\pi$ splits, there is a morphism $q$ in $\mathcal Ext^n_\C(X,Y)$ such that $\pi \circ q = \id_{\chi_i}$. So $w$ is homotopically equivalent to
$$
\xymatrix{
\cdots & \chi_{i-1} \ar[l]_{\alpha_{i-1}} \ar[r]^-{q \circ\alpha_{i}} & \chi & \chi_{i+1} \ar[l]_-{\iota} \ar[r]^{\alpha_{i+2}}& \chi_{i+2} & \cdots \ . \ar[l]_{\alpha_{i+3}}
}
$$
But, since $\iota$ is an admissible monomorphism, the pushout of $\iota$ and $\alpha_{i+2}$ exists, and lies in $\mathcal Ext^n_\C(X,Y)$ (by Lemma \ref{lem:prop_ext1}). Therefore, we have a commutative diagram
$$
\xymatrix
{
\chi_{i+1} \ar[d]_\iota \ar[r]^{\alpha_{i+2}}& \chi_{i+2} \ar[d]^{p_2}\\
\chi  \ar[r]_{p_1} & \mathbb P
}
$$
yielding the loop
$$
\xymatrix@C=32pt{
\cdots & \chi_{i-1} \ar[l]_{\alpha_{i-1}} \ar[r]^-{p_1 \circ q \circ\alpha_{i}} & \mathbb P & \chi_{i+3} \ar[l]_-{p_2 \circ \alpha_{i+3}} \ar[r]^{\alpha_{i+4}} & \cdots 
}
$$
which is homotopically equivalent to $\omega$ and has length length$(w) - 2$. By proceeding in this manner, we either end up with one of the following types of loops.
$$
\xymatrix{\xi \ar[r]  & \chi & \xi' \ar[l]}
$$
$$
\xymatrix{\xi \ar[r]  & \chi_1 & \chi_2 \ar[l] \ar[r] & \xi'} 
$$
$$
\xymatrix{\xi & \chi_1 \ar[l] \ar[r] & \chi_2 & \xi' \ar[l]}
$$
$$
\xymatrix{\xi & \chi_1 \ar[l] \ar[r] & \chi_2 & \chi_3 \ar[l] \ar[r] & \xi'}
$$
If the resulting loop is the first one, we are done. If it is one of the others, we just have to insert $\id_{\xi}: \xi \longleftrightarrow \xi$ pointing in the appropriate direction, and move on with the algorithm described above.
\end{proof}
\section{$n$-Extension closed subcategories}
\begin{nn}
Let $k$ be a commutative ring. Throughout this section, fix an abelian $k$-category $\A$ and a full additive subcategory $\C \subseteq \A$. Let $j: \C \rightarrow \A$ be the corresponding inclusion functor. For objects $C$ and $D$ in $\C$ and an integer $n \geq 1$, we let $\mathcal Ext^n_\C(D,C) \subseteq \mathcal Ext^n_\A(D,C)$ be the full subcategory of $n$-extensions $0 \rightarrow C \rightarrow \mathbb E \rightarrow D \rightarrow 0$ in $\A$ whose middle term $\mathbb E$ belongs to $\Ch(\C)$. We let $\Ext^n_\C(D,C)$ denote $\pi_0 \mathcal Ext^n_\C(D,C)$ which is consistent with the definition for exact subcategories $\C$. 
\end{nn}
\begin{nn}\label{nn:nextcloseintro}
Let $0 \rightarrow C' \rightarrow A \rightarrow C'' \rightarrow 0$ be a short exact sequence in $\A$, whose bordering terms $C'$ and $C''$ are objects in $\C$. Then $A$ belongs to the essential image of $j$ in $\A$ if, and only if, there is a short exact sequence $0 \rightarrow C' \rightarrow C \rightarrow C'' \rightarrow 0$ in $\A$ with $C$ in $\C$ such that it is isomorphic to the one we have started with. It follows that the essential image of $j$ is an extension closed subcategory of $\A$ if, and only if, $j$ induces a bijection $\Ext^1_\C(C'',C') \rightarrow \Ext^1_\A(C'',C')$ for all $C',C'' \in \Ob \C$. In this section, we will study subcategories $\C$ which give rise to a bijection $\Ext^n_\C(C'',C') \cong \Ext^n_\A(C'',C')$ for $n \geq 2$.
\end{nn}
\begin{defn}\label{def:nextclosed}
Let $n \geq 1$ be an integer. 
\begin{enumerate}[\rm(1)]
\item The full subcategory $\C$ of $\A$ is \textit{$n$-extension closed} if the induced maps
$$
j^\sharp_m: \Ext^m_{\C}(X,Y) \longrightarrow \Ext^m_{\A}(X,Y) \quad \text{(for $m \leq n$)}
$$
are bijective for any pair of objects $X,Y$ in $\C$. 
\item The category $\C$ is \textit{entirely extension closed} if it is $n$-extension closed for all $n$.
\end{enumerate}
\end{defn}
\begin{rem}
\begin{enumerate}[\rm(1)]
\item From \ref{nn:nextcloseintro} the following observation is immediate: The subcategory $\C$ is $1$-extension closed if, and only if, the essential image of $j$ is extension closed, and thus, $(\C, j:\C \rightarrow \A)$ is a Quillen exact $k$-category if, and only if, $\C$ is $1$-extension closed.
\item Assume that $\C$ is $1$-extension closed (i.e., that it is an exact subcategory of $\A$). The maps $j^\sharp_m: \Ext^m_{\C}(X,Y) \rightarrow \Ext^m_{\A}(X,Y)$ are $k$-linear for all $m \geq 0$, and all $X$, $Y$ in $\C$. This follows from the fact that $j$ preserves pushouts (along admissible monomorphisms) and pullbacks (along admissible epimorphisms). We get the following obvious consequence.
\end{enumerate}
\end{rem}
\begin{lem}\label{lem:exh_triviality}
Assume that $\C$ is entirely extension closed. For every $X, Y \in \Ob \C$, the inclusion $j: \mathsf C \rightarrow \A$ defines an isomorphism
$$
j^\sharp_\bullet: \Ext^\bullet_{\mathsf C}(X,Y) \longrightarrow \Ext^\bullet_{\A}(X,Y)
$$
of graded $k$-modules. \qed
\end{lem}
\begin{prop}\label{lem:exhproj}
Let $X$ and $Y$ be in $\C$. Assume that $X$ admits a projective resolution $\cdots \rightarrow P_2 \rightarrow P_1\rightarrow P_0 \rightarrow X \rightarrow 0$ in $\A$. Further, assume that
\begin{enumerate}[\rm(1)]
\item\label{lem:exhproj:0} $\C$ is $1$-extension closed,
\item\label{lem:exhproj:1} $P_i$ belongs to $\C$ for every $i \geq 0$, and
\item\label{lem:exhproj:2} for every object $A \in \Ob \C$, every $i \geq 0$ and for every epimorphism $f: P_i \rightarrow A$ in $\A$, the kernel $\Ker(f)$ of $f$ lies in $\C$ $($up to isomorphism$)$.
\end{enumerate}
\emph{(}For instance, one may assume that $\C$ is closed under kernels of epimorphisms to ensure that \emph{(\ref{lem:exhproj:2})} holds.\emph{)} Then the inclusion functor $j: \C \rightarrow \A$ induces an isomorphism $\Ext^\bullet_\C(X,Y) \rightarrow \Ext^\bullet_\A(X,Y)$.
\end{prop}
\begin{proof}
Let $X, Y \in \Ob \C$ and let $n \geq 1$ be an integer. We show that $j^\sharp_n$ is surjective. To this end, let $\xi\ : \ 0 \rightarrow Y \rightarrow E_{n-1} \rightarrow \cdots \rightarrow E_0 \rightarrow X \rightarrow 0$ be an exact sequence in $\A$. If $n=1$, the equivalence class of $\xi$ already belongs to $\Ext^1_\C(X,Y)$. Hence let us assume that $n \geq 2$. There is a morphism of complexes in $\Ch(\A)$ lifting the indetity of $X$:
$$
\xymatrix@C=20pt{
\cdots \ar[r] & P_{n+1} \ar[r] \ar[d] & P_n \ar[r] \ar[d]_-{\varphi_n} & P_{n-1} \ar[r] \ar[d]_-{\varphi_{n-1}} & P_{n-2} \ar[r] \ar[d]^-{\varphi_{n-2}} & \cdots \ar[r]  & P_0 \ar[r] \ar[d]^-{\varphi_0} & X \ar[r] \ar@{=}[d] & 0 \ \ \\
\cdots \ar[r] & 0 \ar[r] & Y \ar[r] & E_{n-1} \ar[r] & E_{n-2} \ar[r] & \cdots \ar[r] & E_0
\ar[r] & X \ar[r] & 0 \ .
}
$$
Now, consider the pushout diagram
$$
\xymatrix@C=18pt{
\cdots \ar[r] & P_{n+1} \ar[r] \ar[d] & P_n \ar[r] \ar[d]_-{\varphi_n} & P_{n-1} \ar[r] \ar[d] & P_{n-2} \ar[r] \ar@{=}[d] &
\cdots \ar[r]  & P_0 \ar[r] \ar@{=}[d] & X \ar[r] \ar@{=}[d] & 0 \ \ \\
\cdots \ar[r] & 0 \ar[r] & Y \ar[r] & Y \oplus_{\varphi_n} P_{n-1} \ar[r] & P_{n-2} \ar[r] & \cdots \ar[r] &
P_0 \ar[r] & X \ar[r] & 0 \ ,
}
$$
wherein the lower row defines an $n$-extension $\xi'$ of $X$ by $Y$ which is equivalent to the $n$-extension $\xi$. By (\ref{lem:exhproj:1}), $P_0, \dots, P_{n-2} \in \Ob \C$. Moreover, by (\ref{lem:exhproj:2}) and induction,
$$
\Ker(P_{i} \longrightarrow P_{i-1}) \in \Ob \C \quad \text{for $i = 0, \dots, n-2$},
$$
where $P_{-1} = X$. Hence $Y \oplus_{\varphi_n} P_{n-1} \in \Ob \C$ since it occurs as a $1$-extension of objects in $\C$:
$$
0 \longrightarrow Y \longrightarrow Y \oplus_{\varphi_n} P_{n-1} \longrightarrow \Ker(P_{n-2} \longrightarrow P_{n-3}) \longrightarrow 0 \ .
$$
Therefore $\xi'$ is an admissible $n$-extension of $X$ by $Y$ in $\C$, whose equivalence class in $\Ext^n_\C(X,Y)$ is mapped, via $j^\sharp_n$, to the equivalence class of $\xi$ in $\Ext^n_\A(X,Y)$. To deduce the injectivity of $j^\sharp_n$, we claim that the above construction is functorial in the following weak sense.
\begin{enumerate}
\item[\bf Claim:] Let $\xi \rightarrow \zeta$ be a morphism in $\mathcal Ext^n_\A(X,Y)$. Then there is a morphism $\xi' \rightarrow \zeta'$ in $\mathcal Ext^n_\C(X,Y)$.
\end{enumerate}
For the moment, let us assume that the claim is valid. Take an admissible $n$-extension $\xi$ in $\mathcal Ext^n_{\C}(X,Y)$ which is equivalent to the trivial $n$-extension in $\mathcal Ext^n_{\A}(X,Y)$; say $\xi$ is connected to $\sigma_n$ via a sequence $(\alpha_0, \alpha_1, \dots, \alpha_{r-1})$ of morphisms in $\mathcal Ext^n_{\A}(X, Y)$ with
$$
\alpha_i: \xi_i \longrightarrow \xi_{i+1} \quad \ \text{or} \quad \ \alpha_i: \xi_{i+1} \longrightarrow \xi_i \quad \text{(for $i = 0, \dots, r-1$)},
$$
where $\xi_{\mathrm{add}} = \xi$ and $\xi_r = \sigma_n$. What we want to show is that $\xi$ is linked to $\sigma_n$ by a sequence of morphisms in $\mathcal Ext^n_{\C}(X,Y)$. By the claimed weak functoriality, morphisms
$$
\alpha_i ': \xi_i' \longrightarrow \xi_{i+1}' \quad \ \text{or} \quad \ \alpha_i': \xi_{i+1}' \longrightarrow \xi_i' \quad \text{(for $i = 0, \dots, r-1$)},
$$
will be given. The morphisms $\beta_0 := \alpha_\xi: \xi' \rightarrow \xi$, $\beta_{r+1}:=\alpha_{\sigma_n}: \sigma_n' \rightarrow \sigma_n$ and $\beta_i = \alpha_{i-1}'$ (for $i = 1, \dots, r$) then define the desired connecting sequence $(\beta_0, \beta_1, \dots, \beta_{r+1})$ in $\mathcal Ext^n_{\C}(X,Y)$ and we are done.

Let us prove the claim. Let $\alpha: \xi \rightarrow \zeta$ be a morphism in $\mathcal Ext^n_\A(X,Y)$. As before, there is a morphism of complexes,
$$
\xymatrix@C=17pt{
\mathbb P \ar[d]_-{\varphi} & \cdots \ar[r] & P_{n+1} \ar[r] \ar[d] & P_n \ar[r] \ar[d]_-{\varphi_n} & P_{n-1} \ar[r] \ar[d]_-{\varphi_{n-1}} & P_{n-2} \ar[r] \ar[d]^-{\varphi_{n-2}} & \cdots \ar[r]  & P_0 \ar[r] \ar[d]^-{\varphi_0} & X \ar[r] \ar@{=}[d] & 0 \ \ \\
\xi & \cdots \ar[r] & 0 \ar[r] & Y \ar[r] & E_{n-1} \ar[r] & E_{n-2} \ar[r] & \cdots \ar[r] & E_0
\ar[r] & X \ar[r] & 0 \ ,
}
$$
which, when composed with $\alpha$, yields a morphism $\psi = \alpha \circ \varphi$ of complexes between the projective resolution of $X$ and $\zeta$ which lifts the identity of $X$. We have the commutative diagram
$$
\footnotesize
\xymatrix@!C=0pt@!R=0pt{
& 0 \ar[rr] & & Y \ar[rr]
   & & F_{n-1} \ar[rr] && F_{n-2} \ar[rr] && \cdots \ar[rr] && F_0 \ar[rr] && X \ar[rr] \ar@{=}[dd]|!{[d];[d]}\hole && 0
\\
0 \ar[rr] & & Y \ar@{=}[ur]\ar[rr]
 & & E_{n-1} \ar[rr] \ar[ur]^(0.4){\alpha_{n-1}} && E_{n-2} \ar[ur] \ar[rr] && \cdots \ar[rr] && E_0 \ar[ur] \ar[rr] && X \ar@{=}[dd] \ar@{=}[ur] \ar[rr] && 0
\\
& \cdots \ar'[r][rr] & & P_n \ar[dd]|!{[d];[d]}\hole \ar'[r][rr] \ar[uu]^-(.2){\psi_n}|!{[u];[u]}\hole
& & P_{n-1} \ar[uu]|!{[u];[u]}\hole \ar[dd]|!{[d];[d]}\hole  \ar'[r][rr] && P_{n-2} \ar[uu]|!{[u];[u]}\hole \ar[rr] && \cdots \ar'[r][rr] && P_0 \ar[uu]|!{[u];[u]}\hole \ar'[r][rr] && X \ar[rr] && 0
\\
\cdots \ar[rr] & & P_n \ar[dd] \ar[rr]\ar@{=}[ur] \ar[uu]^(.3){\varphi_n}
 & & P_{n-1} \ar@{=}[ur] \ar[uu] \ar[rr] \ar[dd]  && P_{n-2} \ar@{=}[ur] \ar[uu] \ar[rr] && \cdots \ar[rr] && P_0 \ar[uu] \ar@{=}[ur] \ar[rr] && X \ar@{=}[ur] \ar[rr] && 0
\\
& 0 \ar'[r][rr] & & Y \ar'[r][rr]
   & & Q \ar'[r][rr] && P_{n-2} \ar@{=}[uu]|!{[u];[u]}\hole \ar[rr] && \cdots \ar'[r][rr] && P_0 \ar@{=}[uu]|!{[u];[u]}\hole \ar'[r][rr] && X \ar[rr] \ar@{=}[uu]|!{[u];[u]}\hole && 0
\\
0 \ar[rr] & & Y \ar@{=}[ur]\ar[rr]
 & & P  \ar[rr] \ar[ur]|{\, f \,} && P_{n-2} \ar@{=}[uu] \ar@{=}[ur] \ar[rr] && \cdots \ar[rr] && P_0 \ar@{=}[uu] \ar@{=}[ur] \ar[rr] && X \ar@{=}[uu] \ar@{=}[ur] \ar[rr] && 0
}
$$
wherein $P = Y \oplus_{\varphi_n} P_{n-1}$, $Q = Y \oplus_{\psi_n} P_{n-1}$ and the arrow $f : P \rightarrow Q$ is induced by the universal property of $Y \oplus_{\varphi_n} P_{n-1}$. It completes the diagram to a commutative one and hence gives rise to a morphism $\alpha': \xi' \rightarrow \zeta'$ which differs from the identity morphism precisely in degree $n-1$, where it is given by $\alpha'_{n-1} = f$.
\end{proof}
\begin{rem}
The statement of Proposition \ref{lem:exhproj} remains valid if one replaces condition (\ref{lem:exhproj:2}) by the following one: For every $A \in \Ob \C$, every $i \geq 0$ and every monomorphism $f: A \rightarrow P_i$ in $\A$, the cokernel $\Coker(f)$ of $f$ lies in $\C$ (up to isomorphism). This is guaranteed, if $\C$ is closed under cokernels of monomorphisms.
\end{rem}
\begin{cor}\label{cor:exhproj}
Assume that $\A$ has enough projective objects. Further, assume that
\begin{enumerate}[\rm(1)]
\item $\C$ is $1$-extension closed,
\item\label{cor:exhproj:1} $\Proj(\A) \subseteq \C$, and
\item\label{cor:exhproj:2} for every object $A \in \Ob \C$ and for every epimorphism $f: P \rightarrow A$ in $\A$ with $P \in \Ob \Proj(\A)$, the kernel $\Ker(f)$ of $f$ lies in $\C$ $($up to isomorphism$)$.
\end{enumerate}
Then $\C$ is entirely extension closed. \qed
\end{cor}
\begin{lem}\label{lem:kprojproj}
Let $A$ be a $k$-algebra and consider the full subcategory $\P(A)$ of $\Mod(A^\ev)$ defined in Example $\ref{exa:bimodules}$. Then $\mathsf P(A)$ contains $\Proj(A^\ev)$ if, and only if, $A$ is projective as a $k$-module.
\end{lem}
\begin{proof}
Every projective $A^\ev$-module belongs to $\mathsf P(A)$ if, and only if, the module $A \otimes_k A$ belongs to $\mathsf P(A)$ which holds true if, and only if,
$$
\Hom_A(A \otimes_k A, -) \cong \Hom_k(A, \Hom_A(A,-)) \cong \Hom_k(A,-)
$$
and
$$
\Hom_{A^\op}(A \otimes_k A, -) \cong \Hom_k(A, \Hom_{A^\op}(A,-)) \cong \Hom_k(A,-)
$$
are exact. But the latter (by definition) is valid if, and only if, $A$ is a projective $k$-module.
\end{proof}
\begin{cor}\label{cor:isohochschildproj}
Let $A$ be a $k$-algebra . If $A$ is $k$-projective, then $\P(A)$ is entirely extension closed, i.e., the inclusion functor $j: \P(A) \rightarrow \Mod(A^\ev)$ induces an isomorphism of graded $k$-modules $$j^\sharp_\bullet: \Ext^\bullet_{\P(A)}(M,N) \longrightarrow \Ext^\bullet_{A^\ev}(M,N)$$ for all objects $M$ and $N$ in $\P(A)$. $($It is an isomorphism of graded algebras in case $M = N$.$)$
\end{cor}
\begin{proof}
If $A$ is $k$-projective, then $\Proj(A^\ev) \subseteq \P(A)$ by Lemma \ref{lem:kprojproj}. The assertion now follows from Corollary \ref{cor:exhproj}, because $\P(A)$ is closed under kernels of epimorphisms. $($Since $j$ is an exact functor, it will preserve the \textit{Yoneda product} on $\Ext^\bullet_{\P(A)}(M,M)$ which we will introduce later; see Section \ref{sec:yoneda}.$)$
\end{proof}



\chapter{The Retakh isomorphism}\label{cha:retakh}
In \cite{Re86} V.\,Retakh established a connection between the extension groups and the homotopy groups of extension categories over a fixed abelian category $\A$. More precisely, he showed that for every pair of objects $A, B$ in $\A$, and integers $n \geq 1$ and $i = 0, \dots, n$, there is an isomorphism $\Ext^{n-i}_\A(B,A) \cong \pi_i\mathcal Ext^n_\A(B,A)$ of (abelian) groups. In what follows, we will focus on the case $i = 1$. While S.\,Schwede gave an explicit description of these isomorphisms for the category of left modules over a ring (cf. \cite{Sch98}), we will do so for any exact category which is factorizing.
\section{An explicit description}\label{sec:retakh}
\begin{nn}
Let $k$ be a commutative ring. Throughout this section we fix an integer $n \geq 0$ as well as an exact $k$-category $\C$ with defining embedding $i=i_\C: \C \rightarrow \mathsf A_\C$ into the abelian $k$-category $\A = \A_\C$. Moreover, we fix objects $X,Y \in \Ob \C$ and an automorphism $a \in \Aut_\C(X)$. We assume that $\C$ is factorizing. Recall our convention that for objects $A$ and $B$ in $\C$,
\begin{align*}
j^A &: A \longrightarrow A \oplus B, & q^A & : A \oplus B \longrightarrow A,\\
j_B &: B \longrightarrow A \oplus B, &  q_B & : A \oplus B \longrightarrow B
\end{align*}
denote the canonical morphisms.
\end{nn}
\begin{nn}
The main result of the following two sections, namely, that $\Ext^{n}_\C(X,Y)$ and $\pi_1(\mathcal Ext^{n+1}_\C(X,Y), \xi)$ are isomorphic groups for any base point $\xi$, was (as mentioned above) proven by V.\,Retakh in \cite{Re86} for abelian categories and A.\,Neeman and V.\,Retakh in \cite{NeRe96} in the more general context of so called Waldhausen categories. However, their proofs involve a good amount of abstract homotopy theory and therefore are not quite concrete. In contrast to their approach, we will construct the desired isomorphism explicitly.

The first step in doing so, is to acquire a group homomorphism
$$
u^{a,+}_\C:  \Ext^{n}_\C(X,Y) \longrightarrow \pi_1 (\mathcal Ext^{n+1}_\C(X,Y), \mathsf \sigma_{n+1}(X,Y))
$$
which will be vital for our further investigations. In fact, $u^{a,+}_\C$ will turn out to be bijective. In the following, we denote the equivalence class of a loop $w$ in $\mathcal Ext^n_\C(X,Y)$ based at $\sigma_n(X,Y)$ by $[w] \in \pi_1 (\mathcal Ext^n_\C(X,Y), \sigma_n(X,Y))$.
\end{nn}
\begin{nn}\label{nn:defnuplus}
Let us define for each $\xi$ in $\mathcal Ext^n_\C(X,Y)$ a loop $w(\xi)$ in $\mathcal Ext^{n+1}_\C(X,Y)$ based at $\sigma_{n+1}(X,Y)$. To begin with, consider the case $n = 0$, and recall that $\Ext^0_\C(X,Y) = \Hom_\C(X,Y)$ by definition. For a given morphism $g \in \Hom_\C(X,Y)$, the matrix 
$$L_a(g) = 
\left[\begin{matrix}
\id_Y & g\circ a\\
0 & \id_X
\end{matrix}\right] : Y \oplus X \longrightarrow Y \oplus X
$$
gives rise to a morphism $\alpha: \sigma_1(X,Y) \rightarrow \sigma_1(X,Y)$, i.e., a loop $w(g)$ of length 1 based at $\sigma_1(X,Y)$. Suppose that $n \geq 1$.
Let
$$
\xymatrix{\xi & \equiv & 0 \ar[r] & Y \ar[r]^-{e_{n}} & E_{n-1} \ar[r]^-{e_{n-1}} & \cdots \ar[r]^-{e_1} & E_0 \ar[r]^-{e_0} & X
\ar[r] & 0}
$$
be an admissible $n$-extension of $X$ by $Y$. For a morphism $f: A \rightarrow X$ in $\C$ let $M_a(f)$ be the matrix
$$
M_a(f) = \left[\begin{matrix}
a \circ f\\
- a \circ f
\end{matrix}\right] : A \longrightarrow X \oplus X .
$$
It gives rise to an admissible $(n+1)$-extension $0 \rightarrow Y \rightarrow \mathbb E_+ \rightarrow X \rightarrow 0$ in $\C$:
$$
\xi_+ \quad \equiv \quad
\xymatrix@C=23pt{
0 \ar[r] & Y \ar[r]^-{e_{n}} & E_{n-1} \ar[r]^-{e_{n-1}} & \cdots \ar[r]^-{e_1} & E_0 \ar[r]^-{M_a(e_0)} & X \oplus X \ar[r]^-{
\left[\begin{smallmatrix}
1 & 1
\end{smallmatrix}\right]
} & X \ar[r] & 0 \ . \\
}
$$
The extension $\xi_+$ fits into the commutative diagram
$$
\xymatrix@C=22pt@R=22pt{
0 \ar[r] & Y \ar@{=}[d] \ar@{=}[r] & Y \ar[d]_-{e_{n}} \ar[r] & 0 \ar[r] \ar[d]& \cdots \ar[r] & 0 \ar[r] \ar[d] & X \ar@{=}[r] \ar[d]^-{\left[\begin{smallmatrix}
1\\
0
\end{smallmatrix}\right]} & X \ar@{=}[d] \ar[r] & 0\\
0 \ar[r] & Y \ar[r]^-{e_{n}} & E_{n-1} \ar[r]^-{e_{n-1}} & E_{n-2} \ar[r]^{e_{n-2}} & \cdots \ar[r]^-{e_1} & E_0 \ar[r]^-{M_a(e_0)} & X \oplus X \ar[r]^-{
\left[\begin{smallmatrix}
1 & 1
\end{smallmatrix}\right]
} & X \ar[r] & 0\\
0 \ar[r] & Y \ar@{=}[u] \ar@{=}[r] & Y \ar[u]^-{e_{n}} \ar[r] & 0 \ar[r] \ar[u] & \cdots \ar[r] & 0 \ar[r] \ar[u] & X \ar@{=}[r] \ar[u]_{\left[\begin{smallmatrix}
0\\
1
\end{smallmatrix}\right]} & X \ar@{=}[u] \ar[r] & 0
}
$$
which is a loop $w(\xi)$ in $\mathcal Ext^{n+1}_\C(X,Y)$ based at $\sigma_{n+1}(X,Y)$. Note that any morphism $\xi \rightarrow \zeta$ in $\mathcal Ext^n_\C(X,Y)$ will give rise to a morphism $\xi_+ \rightarrow \zeta_+$. Hence the assignment $u^{a,+}_\C([\xi]) = [w(\xi)]$ yields a well-defined map
$$
u^{a,+}_\C: \Ext^n_\C(X,Y) \longrightarrow \pi_1(\mathcal Ext^{n+1}_\C(X,Y), \sigma_{n+1}(X,Y)).
$$
\end{nn}
\begin{lem}\label{lem:u_grouphom}
The map $u^{a,+}_\C$ is a group homomorphism.
\end{lem}
\begin{proof}
First of all, assume that $n = 0$ and let $g$ and $g'$ be in $\Hom_\C(X,Y) = \Ext^0_\C(X,Y)$. Since
$$
\left[\begin{matrix}
\id_Y & g \circ a\\
0 & \id_X
\end{matrix}\right] \cdot
\left[\begin{matrix}
\id_Y & g' \circ a\\
0 & \id_X
\end{matrix}\right] 
=
\left[\begin{matrix}
\id_Y & (g + g')\circ a\\
0 & \id_X
\end{matrix}\right] 
$$
it follows that $u^{a,+}_\C(g + g') = u_\C^{a,+}(g)u_\C^{a,+}(g')$. For the remainder of the proof, let $n \geq 1$ and $\xi$ and $\zeta$ be admissible $n$-extensions in $\mathcal Ext^{n}_\C(X,Y)$. Let us first investigate the case $n = 1$; we will deduce the general result from it. The sum $[\xi] + [\zeta]$ is given by the equivalence class of the lower sequence in the
diagram
$$
\xymatrix@C=22pt@R=22pt{
0 \ar[r] & Y \oplus Y \ar[r]  \ar@{=}[d] & E_0 \oplus F_0 \ar[r] & X \oplus X \ar[r] & 0\\
0 \ar[r] & Y \oplus Y \ar[d]_{
\left[\begin{smallmatrix}
1 & 1 \end{smallmatrix}\right]} \ar[r] & P \ar[d] \ar[r] \ar[u] & X \ar[u]_{
\left[\begin{smallmatrix}
1\\
1
\end{smallmatrix}\right]
} \ar[r] & 0\\
0 \ar[r] & Y \ar[r]^-{d_1} & Q \ar[r]^-{d_0} & X \ar@{=}[u] \ar[r] & 0
}
$$
being obtained by successively taking a pullback and then a pushout. On the other hand, $u_\C^{a,+}([\xi])u_\C^{a,+}([\zeta])$ is represented by the loop
$$
\xymatrix@C=26pt{
\sigma_2(X,Y) \ar[d] & \equiv & 0 \ar[r] & Y \ar@{=}[d] \ar@{=}[r] & Y \ar[d]_-{e_{1}} \ar[r]^-{0}  \ar[d] & X \ar@{=}[r]
\ar[d]^-{
\left[\begin{smallmatrix}
1\\
0
\end{smallmatrix}\right]
} & X \ar@{=}[d] \ar[r] & 0\\
\xi_+ & \equiv & 0 \ar[r] & Y \ar[r]^-{e_1} & E_0 \ar[r]^-{
M_a(e_0)
} & X \oplus X \ar[r]^-{
\left[\begin{smallmatrix}
1 & 1
\end{smallmatrix}\right]
} & X \ar[r] & 0 \\
\sigma_2(X,Y) \ar[u] \ar[d] & \equiv & 0 \ar[r] & Y \ar@{=}[u] \ar@{=}[r] \ar@{=}[d] & Y \ar[u]^-{e_1} \ar[r]^-{0} \ar[u]
\ar[d]_{f_1} & X \ar[d]^-{\left[\begin{smallmatrix}
1\\
0
\end{smallmatrix}\right]} \ar@{=}[r] \ar[u]_{
\left[\begin{smallmatrix}
0\\
1
\end{smallmatrix}\right]
} & X \ar@{=}[u] \ar@{=}[d] \ar[r] & 0 \\
\zeta_+ & \equiv & 0 \ar[r] & Y \ar[r]^-{f_1} & F_0 \ar[r]^-{
M_a(f_0)
} & X \oplus X \ar[r]^-{
\left[\begin{smallmatrix}
1 & 1
\end{smallmatrix}\right]
} & X \ar[r] & 0 \\
\sigma_2(X,Y) \ar[u] & \equiv & 0 \ar[r] & Y \ar@{=}[u] \ar@{=}[r] & Y \ar[u]^-{f_1} \ar[r]^-{0} \ar[u] & X \ar@{=}[r]
\ar[u]_{
\left[\begin{smallmatrix}
0\\
1
\end{smallmatrix}\right]
} & X \ar@{=}[u] \ar[r] & 0
}
$$
which, after pushing out $\xi_+ \leftarrow \sigma_2(X,Y) \rightarrow \zeta_+$, is homotopically equivalent to a loop of the form
\begin{equation}\label{eq:lem:u_grouphom:result}
\begin{aligned}
\xymatrix@C=30pt{
0 \ar[r] &Y \ar@{=}[d] \ar@{=}[r] & Y \ar[d]_-{p \circ e_1} \ar[r]^-{0}  \ar[d] & X \ar@{=}[r] \ar[d]^-{
\left[\begin{smallmatrix}
0\\
0\\
1
\end{smallmatrix}\right]
} & X \ar@{=}[d] \ar[r] & 0 \ \ \\
0 \ar[r] & Y \ar[r] & \widehat{P} \ar[r] & X \oplus X \oplus X \ar[r]^-{
\left[\begin{smallmatrix}
1 & 1 & 1
\end{smallmatrix}\right]
} & X \ar[r] & 0 \ \ \\
0 \ar[r] & Y \ar@{=}[u] \ar@{=}[r] & Y \ar[u]^-{q \circ f_1} \ar[r]^-{0} \ar[u] & X \ar@{=}[r] \ar[u]_{
\left[\begin{smallmatrix}
0\\
1\\
0
\end{smallmatrix}\right]
} & X \ar@{=}[u] \ar[r] & 0 \ .
}
\end{aligned}
\end{equation}
In order to see this, one has to take the following pushout diagrams in $\C$ into account:
$$
\xymatrix{
Y \ar[r]^-{e_1} \ar[d]_-{f_1} & E_0 \ar@<-3pt>[d]^p \ \ &\\
F_0 \ar[r]^q & \widehat{P} \ , &
}
\xymatrix{
& X \ar[r]^-{\left[\begin{smallmatrix}
1\\
0
\end{smallmatrix}\right]} \ar[d]_-{\left[\begin{smallmatrix}
0\\
1
\end{smallmatrix}\right]} & X \oplus X \ \  \ar@<-3pt>[d]^-{\left[\begin{smallmatrix}
1 & 0\\
0 & 1\\
0 & 0
\end{smallmatrix}\right]} \\
& X \oplus X \ar[r]^-{\left[\begin{smallmatrix}
0 &1\\
0 & 0\\
1 & 0
\end{smallmatrix}\right]} & X \oplus X \oplus X \ .
}
$$
The universal property of $Q$ yields the (uniquely determined) arrow $Q \rightarrow \widehat{P}$ in the commutative diagram
$$
\xymatrix{
0 \ar[r] & Y \oplus Y \ar[d]_-{\left[\begin{smallmatrix}
1\\
1
\end{smallmatrix}\right]} \ar[r] & P \ar[r] \ar[d] & E_0 \oplus F_0 \ar@<-3pt>[d]^-{\left[\begin{smallmatrix}
p & q
\end{smallmatrix}\right]} \ \ \\
0 \ar[r] & Y \ar@{=}[d] \ar[r]^{d_1} & Q \ar[r] & \widehat{P} \ \ \\
0 \ar[r] & Y \ar[rr]^-{\left[\begin{smallmatrix}
e_1\\
f_1
\end{smallmatrix}\right]} & & E_0 \oplus F_0 \ , \ar@<3pt>[u]_-{\left[\begin{smallmatrix}
p & q
\end{smallmatrix}\right]}
}
$$
and this arrow is such that
$$
\xymatrix@C=30pt{
0 \ar[r] & Y \ar@{=}[d] \ar@{=}[r] & Y \ar[d] \ar[r]^-{0}  \ar[d] & X \ar@{=}[r] \ar[d]^-{
\left[\begin{smallmatrix}
1\\
0
\end{smallmatrix}\right]
} & X \ar@{=}[d] \ar[r] & 0\\
0 \ar[r] & Y \ar@{=}[d] \ar[r] & Q \ar[r]^-{
M_a(d_0)
} \ar[d] & X \oplus X \ar[d]^-{\left[\begin{smallmatrix}
0 & 0\\
0 & 1\\
1 & 0
\end{smallmatrix}\right]} \ar[r]^-{
\left[\begin{smallmatrix}
1 & 1
\end{smallmatrix}\right]
} & X \ar@{=}[d] \ar[r] & 0\\
0 \ar[r] & Y \ar[r] & \widehat{P} \ar[r] & X \oplus X \oplus X \ar[r]^-{
\left[\begin{smallmatrix}
1 & 1 & 1
\end{smallmatrix}\right]
} & X \ar[r] & 0\\
0 \ar[r] & Y \ar@{=}[u] \ar[r] & Q \ar[u] \ar[r]^-{
M_a(d_0)
} & X \oplus X \ar[u]_-{\left[\begin{smallmatrix}
0 & 0\\
0 & 1\\
1 & 0
\end{smallmatrix}\right]} \ar[r]^-{
\left[\begin{smallmatrix}
1 & 1
\end{smallmatrix}\right]
} & X \ar@{=}[u] \ar[r] & 0\\
0 \ar[r] & Y \ar@{=}[u] \ar@{=}[r] & Y \ar[u] \ar[r]^-{0} \ar[u] & X \ar@{=}[r] \ar[u]_{
\left[\begin{smallmatrix}
0\\
1
\end{smallmatrix}\right]
} & X \ar@{=}[u] \ar[r] & 0
}
$$
commutes. But this diagram is a loop which is homotopically equivalent to the loop (\ref{eq:lem:u_grouphom:result}) above, and to
$$
\xymatrix@C=34pt{
0 \ar[r] & Y \ar@{=}[d] \ar@{=}[r] & Y \ar[d]_-{d_{1}} \ar[r]^-{0}  \ar[d] & X \ar@{=}[r] \ar[d]^-{
\left[\begin{smallmatrix}
1\\
0
\end{smallmatrix}\right]
} & X \ar@{=}[d] \ar[r] & 0\\
0 \ar[r] & Y \ar[r]^-{d_{1}} & Q \ar[r]^-{
\left[\begin{smallmatrix}
M_a(d_0)
\end{smallmatrix}\right]
} & X \oplus X \ar[r]^-{
\left[\begin{smallmatrix}
1 & 1
\end{smallmatrix}\right]
} & X \ar[r] & 0\\
0 \ar[r] & Y \ar@{=}[u] \ar@{=}[r] & Y \ar[u]^-{d_{1}} \ar[r]^-{0} \ar[u] & X \ar@{=}[r] \ar[u]_{
\left[\begin{smallmatrix}
0\\
1
\end{smallmatrix}\right]
} & X \ar@{=}[u] \ar[r] & 0
}
$$
which defines $u_\C^{a,+}([\xi] + [\zeta])$. We get $u_\C^{a,+}([\xi] + [\zeta]) = u_\C^{a,+}([\xi])u_\C^{a,+}([\zeta])$ as required.

Now, finally, assume that $n \geq 2$. The product $u_\C^{a,+}([\xi])u_\C^{a,+}([\zeta])$ is represented by the loop
$$
\sigma_{n+1}(X,Y) \longrightarrow \xi_+ \longleftarrow \sigma_{n+1}(X,Y) \longrightarrow \zeta_+ \longleftarrow \sigma_{n+1}(X,Y)
$$
which expands as
$$
\xymatrix@C=22pt@R=22pt{
0 \ar[r] & Y \ar@{=}[d] \ar@{=}[r] & Y \ar[d] \ar[r] & 0 \ar[r] \ar[d] & \cdots \ar[r] & 0 \ar[r] \ar[d] & X \ar@{=}[r] \ar[d]^-{\left[\begin{smallmatrix}
1\\
0
\end{smallmatrix}\right]} & X \ar@{=}[d] \ar[r] & 0 \ \ \\
0 \ar[r] & Y \ar[r]^-{e_n} & E_{n-1} \ar[r]^-{e_{n-1}} & E_{n-2} \ar[r]^-{e_{n-2}} & \cdots \ar[r]^-{e_1} & E_0 \ar[r]^-{M_a(e_0)} & X \oplus X \ar[r]^-{\left[\begin{smallmatrix}
1 & 1
\end{smallmatrix}\right]} & X \ar[r] & 0 \ \ \\
0 \ar[r] & Y \ar@{=}[u] \ar@{=}[d] \ar@{=}[r] & Y \ar[u] \ar[r] \ar[d] & 0 \ar[d] \ar[r] \ar[u] & \cdots \ar[r] & 0 \ar[r] \ar[u] \ar[d] & X \ar@{=}[r] \ar[u]_-{\left[\begin{smallmatrix}
0\\
1
\end{smallmatrix}\right]} \ar[d]^-{\left[\begin{smallmatrix}
1\\
0
\end{smallmatrix}\right]} & X \ar@{=}[u] \ar@{=}[d] \ar[r] & 0 \ \ \\
0 \ar[r] & Y \ar[r]^-{f_{n}} & F_{n-1} \ar[r]^-{f_{n-1}} & F_{n-2} \ar[r]^-{f_{n-2}} & \cdots \ar[r]^-{f_1} & F_0 \ar[r]^-{M_a(f_0)} & X \oplus X \ar[r]^-{\left[\begin{smallmatrix}
1 & 1
\end{smallmatrix}\right]} & X \ar[r] & 0 \ \ \\
0 \ar[r] & Y \ar@{=}[u] \ar@{=}[r] & Y \ar[u] \ar[r] & 0 \ar[r] \ar[u] & \cdots \ar[r] & 0 \ar[r] \ar[u] & X \ar@{=}[r] \ar[u]_-{\left[\begin{smallmatrix}
0\\
1
\end{smallmatrix}\right]} & X \ar@{=}[u] \ar[r] & 0 \ .
}
$$
As above, we take the pushout of $\xi_+ \leftarrow \sigma_{n+1}(X,Y) \rightarrow \zeta_+$ to obtain a loop of shorter length. Degreewise, the pushout is given by
$$
\xymatrix{
Y \ar[r]^-{e_{n}} \ar[d]_-{f_n} & E_{n-1} \ar@<-3pt>[d]^r \ \ \\
F_{n-1} \ar[r]^s & Q \ ,
}
\xymatrix{
& 0 \ar[r] \ar[d] & E_i \ \  \ar@<-3pt>[d]^-{\left[\begin{smallmatrix}
1\\
0
\end{smallmatrix}\right]} \\
& F_i \ar[r]^-{\left[\begin{smallmatrix}
0\\
1
\end{smallmatrix}\right]} & E_i \oplus F_i \ ,
}
\xymatrix{
& X \ar[r]^-{\left[\begin{smallmatrix}
1\\
0
\end{smallmatrix}\right]} \ar[d]_-{\left[\begin{smallmatrix}
0\\
1
\end{smallmatrix}\right]} & X \oplus X \ \  \ar@<-3pt>[d]^-{\left[\begin{smallmatrix}
1 & 0\\
0 & 1\\
0 & 0
\end{smallmatrix}\right]} \\
& X \oplus X \ar[r]^-{\left[\begin{smallmatrix}
0 &1\\
0 & 0\\
1 & 0
\end{smallmatrix}\right]} & X \oplus X \oplus X \ ,
}
$$
where $i = 0, \dots, n-2$. The resulting loop, denoted by $w$, looks as follows:
\begin{equation*}\label{eq:lem:u_grouphom:result2}
\begin{aligned}
\footnotesize{
\xymatrix@C=18pt@R=20pt{
0 \ar[r] & Y \ar@{=}[r] \ar@{=}[d] & Y \ar[r] \ar[d] & 0 \ar[d] \ar[r] & \cdots \ar[r] & 0 \ar[r] \ar[d] & X \ar@{=}[r]  \ar[d]^-{\left[\begin{smallmatrix}
0\\
0\\
1
\end{smallmatrix}\right]} & X \ar@{=}[d] \ar[r] & 0 \ \ \\
0 \ar[r] & Y \ar[r] & Q \ar[r] & E_{n-2} \oplus F_{n-2} \ar[r] & \cdots \ar[r] & E_0 \oplus F_0 \ar[r] & X \oplus X \oplus X \ar[r]^-{\left[\begin{smallmatrix}
1 & 1 & 1
\end{smallmatrix}\right]} & X \ar[r] \ar@{=}[d] & 0 \ \ \\
0 \ar[r] & Y \ar@{=}[u] \ar@{=}[r] & Y \ar[u] \ar[r] & 0 \ar[r] \ar[u] & \cdots \ar[r] & 0 \ar[r] \ar[u] & X \ar@{=}[r] \ar[u]_-{\left[\begin{smallmatrix}
0\\
1\\
0
\end{smallmatrix}\right]} & X \ar[r] & 0 \ .
}}
\end{aligned}
\end{equation*}
Now consider the pullback diagram
$$
\xymatrix@C=35pt{
E_0 \oplus F_0 \ar[r]^-{\left[\begin{smallmatrix}
e_0 & 0\\
0 & f_0
\end{smallmatrix}\right]} & X \oplus X \ \ \\
P \ar[u]^-{\left[\begin{smallmatrix}
v_1\\
v_2
\end{smallmatrix}\right]}   \ar[r]^{v} & X \ . \ar@<3pt>[u]_-{\left[\begin{smallmatrix}
1\\
1
\end{smallmatrix}\right]}
}
$$
It gives rise to factorizations of the morphisms in the loop $w$:
$$
\footnotesize{
\xymatrix@C=18pt@R=20pt{
0 \ar[r] & Y \ar@{=}[r] \ar@{=}[d] & Y \ar[r] \ar[d] & 0 \ar[d] \ar[r] & \cdots \ar[r] & 0 \ar[r] \ar[d] & X \ar@{=}[r]  \ar[d]^-{\left[\begin{smallmatrix}
1\\
0
\end{smallmatrix}\right]} & X \ar@{=}[d] \ar[r] & 0 \ \ \\
0 \ar[r] & Y \ar[r] & Q \ar[r] \ar@{=}[d] & E_{n-2} \oplus F_{n-2} \ar[r] \ar@{=}[d] & \cdots \ar[r] & P \ar[r] \ar[d]^-{
\left[\begin{smallmatrix}
v_1\\
v_2
\end{smallmatrix}\right]} & X \oplus X \ar[d]^-{
\left[\begin{smallmatrix}
0 & 0\\
0 & 1\\
1 & 0
\end{smallmatrix}\right]
} \ar[r]^-{\left[\begin{smallmatrix}
1 & 1
\end{smallmatrix}\right]} & X \ar[r] & 0 \ \ \\
0 \ar[r] & Y \ar@{=}[u] \ar@{=}[d] \ar[r] & Q \ar[r] & E_{n-2} \oplus F_{n-2} \ar[r] & \cdots \ar[r] & E_0 \oplus F_0 \ar[r] & X \oplus X \oplus X \ar[r]^-{\left[\begin{smallmatrix}
1 & 1 & 1
\end{smallmatrix}\right]} & X \ar@{=}[u] \ar@{=}[d] \ar[r] & 0 \ \ \\
0 \ar[r] & Y \ar[r] & Q \ar[r] \ar@{=}[u] & E_{n-2} \oplus F_{n-2} \ar[r] \ar@{=}[u] & \cdots \ar[r] & P \ar[r] \ar[u]_-{
\left[\begin{smallmatrix}
v_1\\
v_2
\end{smallmatrix}\right]}  & X \oplus X \ar[u]_-{
\left[\begin{smallmatrix}
0 & 0\\
0 & 1\\
1 & 0
\end{smallmatrix}\right]
} \ar[r]^-{\left[\begin{smallmatrix}
1 & 1
\end{smallmatrix}\right]} & X \ar[r] & 0 \ \ \\
0 \ar[r] & Y \ar@{=}[u] \ar@{=}[r] & Y \ar[u] \ar[r] & 0 \ar[r] \ar[u] & \cdots \ar[r] & 0 \ar[r] \ar[u] & X \ar@{=}[r] \ar[u]_-{\left[\begin{smallmatrix}
0\\
1
\end{smallmatrix}\right]} & X \ar@{=}[u] \ar[r] & 0 \ .
}}
$$
But this loop is homotopically equivalent to $u_\C^{a,+}([\xi] + [\zeta])$ since the diagram
$$
\xymatrix@C=55pt{
Y \oplus Y \ar[r]^-{\left[\begin{smallmatrix}
e_{n} & 0\\
0 & f_{n}
\end{smallmatrix}\right]} \ar[d]_-{\left[\begin{smallmatrix}
1 & 1\\
\end{smallmatrix}\right]} & E_{n-1} \oplus F_{n-1} \ar[d]^-{\left[\begin{smallmatrix}
r & s
\end{smallmatrix}\right]}  \\
Y \ar[r]^{r \circ e_n} & Q
}
$$
defines a pushout. This completes the proof.
\end{proof}
\begin{nn}
In order to see that $u_\C^{a,+}$ is a bijection, we will construct its inverse map
$$
u_\C^{a,-}: \pi_1 (\mathcal Ext^{n+1}_\C(X,Y), \sigma_{n+1}(X,Y)) \longrightarrow \Ext^{n}_\C(X,Y).
$$
Assume that $n = 0$ and let $w$ be a loop in $\mathcal Ext^1_\C(X,Y)$ based at $\sigma_{1}(X,Y)$. Since every morphism in $\mathcal Ext^1_\C(X,Y)$ automatically is an isomorphism (due to the 5-Lemma), $w$ is homotopically equivalent to a morphism $\alpha: \sigma_{1}(X,Y) \rightarrow \sigma_{1}(X,Y)$ in $\mathcal Ext^1_\C(X,Y)$, which, for its part, is again an isomorphism. By assumption, we have that $q_X \circ (\alpha_0 - \id_{Y \oplus X}) = 0$, so there is a unique morphism $f(w): Y \oplus X \rightarrow Y$ in $\C$ such that $\alpha_0 - \id_{Y \oplus X} = j^Y \circ f(w)$. Hence the morphism 
$$
u_\C^{a,-}([w]) := f(w) \circ j_X \circ a^{-1}
$$
belongs to $\Hom_\C(X,Y) = \mathcal Ext^0_\C(X,Y)$. Now let us suppose that $n \geq 1$ and that $w$ is a loop in $\mathcal Ext^{n+1}_\C(X,Y)$ based at $\sigma_{n+1}(X,Y)$. By Lemma \ref{lem:looplength2} we know that there is a loop $w' = (\sigma_{n+1}(X,Y) \xrightarrow{\alpha} \xi \xleftarrow{\beta} \sigma_{n+1}(X,Y))$,
$$
\xymatrix@C=20pt@R=20pt{
0 \ar[r] & Y \ar@{=}[d] \ar@{=}[r] & Y \ar[d]_-{\alpha_{n}} \ar[r] & 0 \ar[r] \ar[d]_-{\alpha_{n-1}} & \cdots \ar[r] & 0 \ar[r] \ar[d]^-{\alpha_1} & X \ar@{=}[r] \ar[d]^-{\alpha_0} & X \ar@{=}[d] \ar[r] & 0 \ \ \\
0 \ar[r] & Y \ar[r]^-{e_{n+1}} & E_{n} \ar[r]^-{e_{n}} & E_{n-1} \ar[r]^{e_{n-1}} & \cdots \ar[r]^-{e_2} & E_1 \ar[r]^{e_1} & E_0 \ar[r]^{e_0} & X \ar[r] & 0 \ \ \\
0 \ar[r] & Y \ar@{=}[u] \ar@{=}[r] & Y \ar[u]^-{\beta_{n}} \ar[r] & 0 \ar[r] \ar[u]^-{\beta_{n-1}} & \cdots \ar[r] & 0 \ar[r] \ar[u]_{\beta_1} & X \ar@{=}[r] \ar[u]_{\beta_0} & X \ar@{=}[u] \ar[r] & 0 \ ,
}
$$
such that $[w] = [w']$. Without loss of generality, we may assume that $w = w'$. By forming the difference $\alpha - \beta$, we obtain a commutative diagram
$$
\xymatrix{
X \ar@{=}[r] \ar[d]_{\alpha_0 - \beta_0} & X \ar[d]^0 \\
E_0 \ar[r]^{e_0} & X
}
$$
telling us, that $e_0 \circ (\alpha_0 - \beta_0) = 0$. Therefore there is a unique morphism $w_-: X \rightarrow \Ker(e_0)$ such that $\ker(e_0) \circ w_- = \alpha_0 - \beta_0$. We get the following pullback diagram.
\begin{equation}\label{eq:inverseloop}
\begin{aligned}
\xymatrix@C=22pt{
0 \ar[r] & Y \ar@{=}[d] \ar[r]^-{e_{n+1}} & E_{n} \ar@{=}[d] \ar[r]^-{e_{n}} & \cdots \ar[r]^-{e_3} & E_2 \ar@{=}[d] \ar[r] & P \ar[d] \ar[r] & X \ar[d]^{w_- \circ a} \ar[r] & 0 \\
0 \ar[r] & Y \ar[r]^-{e_{n+1}} & E_{n} \ar[r]^-{e_{n}} & \cdots \ar[r]^-{e_3} & E_2 \ar[r]^-{e_{2}} & E_1 \ar[r]^-{e_{1}} & \Ker(d_0) \ar[r] & 0
}
\end{aligned}
\end{equation}
The upper sequence, denoted by $\xi(w)$, is an admissible $n$-extension and hence belongs to $\mathcal Ext^{n}_\C(X,Y)$. Thus we put $$u_\C^{a,-}([w]) = [\xi(w)].$$ We will divide the proof of $u_\C^{a,+}$ and $u_\C^{a,-}$ being mutually inverse into two parts.
\end{nn}
\begin{lem}\label{lem:u_injective}
The composition $u_\C^{a,-} \circ u_\C^{a,+}$ is the identity. Hence $u_\C^{a,+}$ is injective.
\end{lem}
\begin{proof}
We start by looking at the case $n = 0$. Let $g \in \Hom_\C(X,Y)$. One easily checks that $f(w(g))$ coincides with $g \circ a \circ q_X$ and thus 
\begin{align*}
(u_\C^{a,-} \circ u_\C^{a,+})(g) &= f(w(g)) \circ j_X \circ a^{-1}\\
&= g \circ a \circ q_X \circ j_X \circ a^{-1} = g.
\end{align*}
Now let $n \geq 1$ and choose an admissible $n$-extension $\xi$ in $\mathcal Ext^{n}_\C(X,Y)$. By definition, $u_\C^{a,+}([\xi])$ is represented by the loop $w(\xi)$ given by
$$
\xymatrix@C=22pt{
0 \ar[r] & Y \ar@{=}[d] \ar@{=}[r] & Y \ar[d]_-{e_{n-1}} \ar[r] & 0 \ar[r] \ar[d]& \cdots \ar[r] & 0 \ar[r] \ar[d] & X \ar@{=}[r] \ar[d]^-{
\left[\begin{smallmatrix}
1\\
0
\end{smallmatrix}\right]
} & X \ar@{=}[d] \ar[r] & 0 \\
0 \ar[r] & Y \ar[r]^-{e_{n}} & E_{n-1} \ar[r]^-{e_{n-1}} & E_{n-2} \ar[r]^{e_{n-2}} & \cdots \ar[r]^-{e_1} & E_0 \ar[r]^-{
M_a(e_0)
} & X \oplus X \ar[r]^-{
\left[\begin{smallmatrix}
1 & 1
\end{smallmatrix}\right]
} & X \ar[r] & 0 \\
0 \ar[r] & Y \ar@{=}[u] \ar@{=}[r] & Y \ar[u]^-{e_{n-1}} \ar[r] & 0 \ar[r] \ar[u] & \cdots \ar[r] & 0 \ar[r] \ar[u] & X \ar@{=}[r] \ar[u]_{
\left[\begin{smallmatrix}
0\\
1
\end{smallmatrix}\right]
} & X \ar@{=}[u] \ar[r] & 0
}
$$
and the corresponding morphism $w(\xi)_-$ is simply the automorphism $a^{-1}$. So in this situation, the upper row in (\ref{eq:inverseloop}) displays itself as
$$
\xymatrix@C=20pt{
0 \ar[r] & Y \ar[r]^-{e_{n}} & E_{n-1} \ar[r]^-{e_{n-1}} & E_{n-1} \ar[r]^{e_{n-2}} & \cdots \ar[r]^-{e_1} & E_0 \ar[r]^{e_0} & X \ar[r] & 0
}
$$
which verifies that $(u_\C^{a,-} \circ u_\C^{a,+})([\xi]) = [\xi]$.
\end{proof}
\begin{lem}\label{lem:u_surjective}
The composition $u_\C^{a,+} \circ u_\C^{a,-}$ is the identity. Hence $u_\C^{a,+}$ is surjective.
\end{lem}
\begin{proof}
First of all, we turn our attention towards the case $n=0$. Let $w$ be a loop in $\mathcal Ext^1_\C(X,Y)$ based at $\sigma_1(X,Y)$ which is homotopically equivalent to an isomorphism $\alpha: \sigma_1(X,Y) \rightarrow \sigma_1(X,Y)$. By definition, $u_\C^{a,-}([w]) = f(w) \circ j_X \circ a^{-1}$ for a map $f(w)$ such that $\alpha_0 - \id_{Y \oplus X}=j^Y \circ f(w)$. From the latter formula we deduce the following equalities:
\begin{align*}
q_X \circ \alpha_0 \circ j_X = \id_X, & \quad q^Y \circ \alpha_0 \circ j_X = f(w) \circ j_X,\\
q^Y \circ \alpha_0 \circ j^Y = \id_Y, & \quad q_X \circ \alpha_0 \circ j^Y = 0 \ .
\end{align*}
Therefore the morphism $Y \oplus X \rightarrow Y \oplus X$ given by the matrix 
$$
L_a(f(w)\circ j_X \circ a^{-1}) = 
\left[\begin{matrix}
\id_Y & f(w) \circ j_X \\
0 & \id_X
\end{matrix}\right]
$$
coincides with 
$$
\alpha_0 = 
\left[\begin{matrix}
q_X \circ \alpha_0 \circ j_X & q^Y \circ \alpha_0 \circ j_X \\
q_X \circ \alpha_0 \circ j^Y & q^Y \circ \alpha_0 \circ j^Y
\end{matrix}\right]
$$
and hence $(u_\C^{a,+} \circ u_\C^{a,-})([w])=u_\C^{a,+}(f(w) \circ j_X \circ a^{-1}) = [w]$.

Now let us consider the case $n \geq 1$. Let $w$ be a loop in $\mathcal Ext^{n+1}_\C(X,Y)$ based at $\sigma_{n+1}(X,Y)$.  We claim that $(u_\C^{a,+} \circ u_\C^{a,-})([w])=u_\C^{a,+}([\xi(w)]) = [w]$. Remember that $\xi(w)$ is given by the upper sequence in the following pullback diagram:
$$
\xymatrix@C=23pt{
0 \ar[r] & Y \ar@{=}[d] \ar[r]^-{e_{n+1}} & E_{n} \ar@{=}[d] \ar[r]^-{e_{n}} & \cdots \ar[r]^-{e_3} & E_2 \ar@{=}[d] \ar[r]^{\overline{e}_2} & P \ar[d]^-{h} \ar[r]^{\overline{e}_1} & X \ar[d]^{w_- \circ a} \ar[r] & 0 \ \ \\
0 \ar[r] & Y \ar[r]^-{e_{n+1}} & E_{n} \ar[r]^-{e_{n}} & \cdots \ar[r]^-{e_3} & E_2 \ar[r]^-{d_{2}} & E_1 \ar[r]^-{e_{1}} & \Ker(e_0) \ar[r] & 0 \ ,
}
$$
where $w_-: X \rightarrow \Ker(e_0)$ is such that $\ker(e_0) \circ w_- = \alpha_0 - \beta_0$. There is a loop $w'$ of the form $\sigma_{n+1}(X,Y) \xrightarrow{\alpha} \xi \xleftarrow{\beta} \sigma_{n+1}(X,Y) $ which is homotopically equivalent to $w$:
$$
\xymatrix@C=23pt@R=23pt{
0 \ar[r] & Y \ar@{=}[d] \ar@{=}[r] & Y \ar[d]_-{\alpha_{n}} \ar[r] & 0 \ar[r] \ar[d]_-{\alpha_{n-1}} & \cdots \ar[r] & 0 \ar[r] \ar[d]^-{\alpha_1} & X \ar@{=}[r] \ar[d]^-{\alpha_0} & X \ar@{=}[d] \ar[r] & 0 \ \ \\
0 \ar[r] & Y \ar[r]^-{e_{n+1}} & E_{n} \ar[r]^-{e_{n}} & E_{n-1} \ar[r]^{e_{n-2}} & \cdots \ar[r]^-{e_2} & E_1 \ar[r]^{e_1} & E_0 \ar[r]^{e_0} & X \ar[r] & 0 \ \ \\
0 \ar[r] & Y \ar@{=}[u] \ar@{=}[r] & Y \ar[u]^{\beta_{n}} \ar[r] & 0 \ar[r] \ar[u]^{\beta_{n-1}} & \cdots \ar[r] & 0 \ar[r] \ar[u]_{\beta_1} & X \ar@{=}[r] \ar[u]_{\beta_0} & X \ar@{=}[u] \ar[r] & 0 \ .
}
$$
Since $e_0 \circ \alpha_0 = \id_X = e_0 \circ \beta_0$ and 
$$
\left[\begin{matrix}
\alpha_0 & \beta_0
\end{matrix}\right] \circ \left[\begin{matrix}
a \circ \overline{e}_1\\
-a \circ\overline{e}_1
\end{matrix}\right] = (\alpha_0-\beta_0) \circ a \circ \overline{e}_1 = \ker(e_0) \circ w_- \circ a \circ \overline{e}_1 = e_1 \circ h
$$
the diagram
$$
\xymatrix@C=23pt@R=23pt{
0 \ar[r] & Y \ar@{=}[d] \ar@{=}[r] & Y \ar[d]_-{e_{n+1}} \ar[r] & 0 \ar[r] \ar[d]& \cdots \ar[r] & 0 \ar[r] \ar[d] & X \ar@{=}[r] \ar[d]^-{f_0} & X \ar@{=}[d] \ar[r] & 0\\
0 \ar[r] & Y \ar@{=}[d] \ar[r]^-{e_{n+1}} & E_{n} \ar[r] \ar@{=}[d] & E_{n-1} \ar[r] \ar@{=}[d] & \cdots \ar[r] & E_1  \ar[r]^-{e_1} & E_0 \ar[r] & X \ar@{=}[d] \ar[r] & 0\\
0 \ar[r] & Y \ar[r]^-{e_{n+1}} \ar@{=}[d] & E_{n} \ar@{=}[d] \ar[r]^-{e_{n}} & E_{n-1} \ar@{=}[d] \ar[r]^{e_{n-1}} & \cdots \ar[r]^-{\overline{e}_2} & P \ar[d]_h \ar[u]^h \ar[r]^-{
M_a(\overline{e}_1)
} & X \oplus X \ar[r]^-{
\left[\begin{smallmatrix}
1 & 1
\end{smallmatrix}\right]
} \ar[u]_{
\left[\begin{smallmatrix}
\alpha_0 & \beta_0
\end{smallmatrix}\right]
} \ar[d]^{
\left[\begin{smallmatrix}
\alpha_0 & \beta_0
\end{smallmatrix}\right]
}& X \ar@{=}[d] \ar[r] & 0 \\
0 \ar[r] & Y  \ar[r]^-{e_{n+1}} & E_{n} \ar[r] & E_{n-1} \ar[r] & \cdots \ar[r] & E_1 \ar[r]^-{e_1} & E_0 \ar[r] & X \ar[r] & 0\\
0 \ar[r] & Y \ar@{=}[u] \ar@{=}[r] & Y \ar[u]^-{e_{n+1}} \ar[r] & 0 \ar[r] \ar[u] & \cdots \ar[r] & 0 \ar[r] \ar[u] & X \ar@{=}[r] \ar[u]_{\beta_0} & X \ar@{=}[u] \ar[r] & 0
}
$$
is commutative and hence defines a loop $w''$ based at $\sigma_{n+1}(X,Y)$ which is homotopically equivalent to $w'$. Since the equalities
$$
\left[\begin{matrix}
\alpha_0 & \beta_0
\end{matrix}\right] \circ \left[\begin{matrix}
\id_X\\
0
\end{matrix}\right] = \alpha_0, \quad \left[\begin{matrix}
\alpha_0 & \beta_0
\end{matrix}\right] \circ \left[\begin{matrix}
0\\
\id_X
\end{matrix}\right] = \beta_0
$$
hold for trivial reasons, $w''$ is homotopically equivalent to
$$
\xymatrix@C=23pt{
0 \ar[r] & Y \ar@{=}[d] \ar@{=}[r] & Y \ar[d]_-{e_{n+1}} \ar[r] & 0 \ar[r] \ar[d]& \cdots \ar[r] & 0 \ar[r] \ar[d] & X \ar@{=}[r] \ar[d]^-{
\left[\begin{smallmatrix}
1\\
0
\end{smallmatrix}\right]
} & X \ar@{=}[d] \ar[r] & 0\\
0 \ar[r] & Y \ar[r]^-{e_{n+1}} & E_{n} \ar[r]^-{e_{n}} & E_{n-1} \ar[r]^{e_{n-1}} & \cdots \ar[r]^-{\overline{e}_2} & P \ar[r]^-{
M_a(\overline{e}_1)
} & X \oplus X \ar[r]^-{
\left[\begin{smallmatrix}
1 & 1
\end{smallmatrix}\right]
} & X \ar[r] & 0 \\
0 \ar[r] & Y \ar@{=}[u] \ar@{=}[r] & Y \ar[u]^-{e_{n+1}} \ar[r] & 0 \ar[r] \ar[u] & \cdots \ar[r] & 0 \ar[r] \ar[u] & X \ar@{=}[r] \ar[u]_{
\left[\begin{smallmatrix}
0\\
1
\end{smallmatrix}\right]
} & X \ar@{=}[u] \ar[r] & 0
}
$$
which represents the equivalence class $u_\C^{a,+}([\xi(w)])$. Therefore we are done.
\end{proof}
We subsume our observations in the following theorem.
\begin{thm}\label{thm:iso_pi0pi1}
For every integer $n \geq 0$, the map $$u_\C^{a,+}:  \Ext^{n}_\C(X,Y) \rightarrow \pi_1 (\mathcal Ext^{n+1}_\C(X,Y), \sigma_{n+1}(X,Y))$$ is an isomorphism of groups whose inverse map is given by $u_\C^{a,-}$. In particular, $\pi_1 (\mathcal Ext^{n+1}_\C(X,Y), \sigma_{n+1}(X,Y))$ is abelian. \qed
\end{thm}
\begin{nn}
Let $\tau$ and $\xi$ be addmissile $n$-extension of $X$ by $Y$. If $w$ is a loop in $\mathcal Ext^n_\C(X,Y)$ based at $\tau$, we may regard it as a diagram in $\mathcal Ext^n_\C(X,Y)$ to which one may apply the functor $- \boxplus \xi$. Thus we obtain a loop $w \boxplus \xi$ based at $\tau \boxplus \xi$. Recall that, for fixed $\xi$, we have a canonical path $\tau \boxplus \xi \boxplus (-\xi) \rightarrow \tau \boxplus \widetilde{\xi} \leftarrow \tau \boxplus \sigma_n(X,Y) \xleftarrow{\cong} \tau$ in $\mathcal Ext^n_\C(X,Y)$ (cf. the exposition in subparagraph \ref{nn:pathximinxi}).
\end{nn}
\begin{lem}\label{lem:isopi1}
Let $n \geq 1$ be an integer and $\tau, \xi \in \mathcal Ext^n_\C(X,Y)$ be two admissible $n$-extensions. The map
$$
v_\C(\tau, \xi): \pi_1 (\mathcal Ext^n_\C(X,Y), \tau) \longrightarrow \pi_1 (\mathcal Ext^n_\C(X,Y), \tau \boxplus \xi), \ [w]
\mapsto [w \boxplus \xi]
$$
is an isomorphism of groups.
\end{lem}
\begin{proof}
We show that 
$$
v_\C(\tau \boxplus \xi, \tau \boxplus \xi \boxplus (-\xi)) \circ v_\C(\tau, \xi) = c_{v},
$$
where $c_v$ is conjugation by the path $v$ given by $\tau \boxplus \xi \boxplus (-\xi) \rightarrow \tau \boxplus \widetilde{\xi} \leftarrow \tau$. Put $\xi_{\mathrm{add}} : = \xi \boxplus (-\xi)$ for short, and let $w$ be a loop in $\mathcal Ext^n_\C(X,Y)$ based at $\tau$. Assume that $w$ is given as follows:
$$
\tau = \chi_0 \longrightarrow \chi_1 \longleftarrow \chi_2 \longrightarrow \cdots \longleftarrow \chi_{r-1} \longrightarrow \chi_r = \tau.
$$
We show, by induction on $i$, that the loop
\begin{equation}\label{eq:looptau}
\begin{aligned}
\tau \boxplus \xi_{\mathrm{add}} \longrightarrow \tau &\boxplus \widetilde{\xi} \longleftarrow \tau = \chi_0 \longrightarrow \chi_1 \longleftarrow \chi_2 \longrightarrow \cdots\\ &\longleftarrow \chi_{r-1} \longrightarrow \chi_r = \tau \longrightarrow  \tau \boxplus \widetilde{\xi} \longleftarrow \tau \boxplus \xi_{\mathrm{add}}
\end{aligned}
\end{equation}
is homotopically equivalent to the loop
\begin{align*}
\tau \boxplus \xi_{\mathrm{add}} \longrightarrow \chi_1&\boxplus \xi_{\mathrm{add}} \longleftarrow \chi_2 \boxplus \xi_{\mathrm{add}} \longrightarrow  \cdots \\ &\longleftarrow \chi_{i} \boxplus \xi_{\mathrm{add}} \longrightarrow \chi_i \boxplus \widetilde{\xi} \longleftarrow \chi_i \longrightarrow \chi_{i+1} \longleftarrow \cdots\\ &\longrightarrow \chi_r = \tau \longrightarrow  \tau \boxplus \widetilde{\xi} \longleftarrow \tau \boxplus \xi_{\mathrm{add}}
\end{align*}
for any $i \geq 0$. It then follows that the loop $w \boxplus \xi \boxplus (-\xi)$ is homotopically equivalent to the conjugate of $w$ by $v$. Let us just deal with the case $i = 0$. The induction step then is an easy exercise. Since $- \boxplus -$ is a bifunctor, we get the commutative diagram
\begin{equation*}
\begin{aligned}
\xymatrix{
\tau \boxplus \xi_{\mathrm{add}} \ar[r] & \tau \boxplus \widetilde{\xi} & \tau \ar[r] \ar[l] \ar[d] & \chi_1 \ar[d] & \chi_2 \ar[l] \\
&& \tau \boxplus \widetilde{\xi} \ar[r] & \chi_1 \boxplus \widetilde{\xi} &\\
&& \tau \boxplus \xi_{\mathrm{add}} \ar[u] \ar[r] & \chi_1 \boxplus \xi_{\mathrm{add}} \ar[u] &
}
\end{aligned}
\end{equation*}
Thus it follows that the loop (\ref{eq:looptau}) is homotopically equivalent to
\begin{align*}
\tau \boxplus \xi_{\mathrm{add}} \longrightarrow \chi_1 &\boxplus \xi_{\mathrm{add}} \longrightarrow \chi_1 \boxplus \widetilde{\xi} \longleftarrow \chi_1 \longleftarrow \chi_2 \longrightarrow \cdots\\ &\longleftarrow \chi_{r-1} \longrightarrow \chi_r = \tau \longrightarrow  \tau \boxplus \widetilde{\xi} \longleftarrow \tau \boxplus \xi_{\mathrm{add}} \ .
\end{align*}
\end{proof}
By combining Theorem \ref{thm:iso_pi0pi1} with Lemma \ref{lem:isopi1} we instantaneously get.
\begin{thm}\label{prop:pi0_pi1_iso}
For any integer $n \geq 1$ and any admissible $n$-extension $\xi$ of $X$ by $Y$, the group $\pi_1(\mathcal Ext^n_\C(X,Y),\xi)$ is
isomorphic to $\Ext^{n-1}_\C(X,Y)$. Therefore, $\pi_1(\mathcal Ext^n_\C(X,Y),\xi)$ is abelian.\qed
\end{thm}
\section{Compatibility results}\label{sec:retakh2}
Within this section, we discuss reformulations of the isomorphisms introduced priorly, and study their behavior with respect to exact functors between exact categories. We keep the notations and conventions of the previous section.
\begin{nn}
Let $b \in \Aut_\C(Y)$ be an automorphism of $Y$. We obtain additional maps
$$
\underline{u}_\C^{b,+}: \Ext^n_\C(X,Y) \rightarrow \pi_1(\mathcal Ext^{n+1}_{\C}(X,Y), \sigma_{n+1}(X,Y)) \quad \text{(for $n \in \mathbb Z_{\geq
0}$)},
$$
by dualizing the loop construction made for $u_\C^{a,+}$: In the case $n = 0$, the loop $\sigma_{1}(X,Y) \rightarrow \sigma_{1}(X,Y)$ a morphism $g \in \Hom_\C(X,Y)$ will be mapped to by $\underline{u}_\C^{b,+}$ is given by the matrix
$$
\underline{L}_b(g) =
\left[\begin{matrix}
\id_Y & b \circ g \\
0 & \id_X
\end{matrix}\right] : Y \oplus X \longrightarrow Y \oplus X.
$$
For $n \geq 1$ consider the matrix
$$
\underline{M}(f) = \underline{M}_b(f)=
\left[\begin{matrix}
f \circ b\\
-f \circ b
\end{matrix}\right]^T : Y \oplus Y \longrightarrow B
$$
defined for any morphism $f: Y \rightarrow B$ in $\C$; given some admissible $n$-extension $\xi \in \Ob \mathcal Ext^n_\C(X,Y)$, we obtain a loop in the following manner:
$$
\xymatrix@C=23pt{
0 \ar[r] & Y \ar@{=}[d] \ar@{=}[r] & Y \ar[r] & 0 \ar[r] & \cdots \ar[r] & 0 \ar[r]  & X \ar@{=}[r]  & X \ar@{=}[d] \ar[r] & 0 \ \ \\
0 \ar[r] & Y \ar[r]^-{
\left[\begin{smallmatrix}
1\\
1
\end{smallmatrix}\right]} & Y \oplus Y
\ar[u]^{
\left[\begin{smallmatrix}
1 & 0
\end{smallmatrix}\right]
}
\ar[d]_{
\left[\begin{smallmatrix}
0 & 1
\end{smallmatrix}\right]
}
\ar[r]^-{
\underline{M}(e_n)
} & {E_{n-1}} \ar[u] \ar[d] \ar[r]^{e_{n-1}} & \cdots \ar[r]^-{e_2} & E_1 \ar[u] \ar[d] \ar[r]^-{e_1} & E_0 \ar[r]^-{e_0} \ar[u]^{e_0} \ar[d]_{e_0} & X \ar[r] & 0 \ \ \\
0 \ar[r] & Y \ar@{=}[u] \ar@{=}[r] & Y \ar[r] & 0 \ar[r] & \cdots \ar[r] & 0 \ar[r] & X \ar@{=}[r]  & X \ar@{=}[u] \ar[r] & 0 \ .
}
$$
Unsurprisingly, $\underline{u}_\C^{b,+}$ is bijective, and its inverse map $\underline{u}_\C^{b,-}$ may be constructed in a very similar fashion as the one for ${u}_\C^{a,+}$. We will elaborate on the correlation between the maps $u_\C^{a,+}$ and $\underline{u}_\C^{b,+}$.
\end{nn}

\begin{defn}\label{def:homotopic}
Let $\alpha, \beta : \xi \rightarrow \zeta$ be morphisms in $\mathcal Ext^n_\C(X,Y)$. We say that
$\alpha$ and $\beta$ are \textit{homotopic} if there exist morphisms $s_i: E_i \rightarrow F_{i+1}$ for $0 \leq i \leq n-2$ such
that
\begin{align*}
f_1 \circ s_0 &= f_0 - \beta_0,\\
f_i \circ s_{i} + s_{i-1} \circ e_i &= \alpha_i - \beta_i & & \text{(for $1 \leq i \leq n-2$)},\\
s_{n-2} \circ e_{n-1} &= \alpha_{n-1} - \beta_{n-1}.
\end{align*}
Such a family $(s_i)_i$ is then called a (\textit{chain}) \textit{homotopy between $\alpha$ and $\beta$}.
\end{defn}
\begin{rem}
Observe that $\alpha$ and $\beta$ being homotopic precisely means that the induced morphisms $\alpha^\natural$ and $\beta^\natural$ between the middle segments $\mathbb E$ and $\mathbb F$ are homotopic in the usual sense (cf. \cite[Sec.\,1.4]{Wei94}).
\end{rem}
\begin{lem}[{\cite[Lem.\,4.2]{Sch98}}]\label{lem:homotopic}
Let $\alpha, \beta : \xi \rightarrow \zeta$ be homotopic morphisms in $\mathcal Ext^n_\mathsf{C}(X,Y)$. Then $\alpha$ and $\beta$ are homotopically equivalent $($when regarded as paths in $\mathcal Ext^n_\mathsf{C}(X,Y))$.
\end{lem}
\begin{proof}
We recall (and adapt) S.\,Schwede's proof for convenience of the reader. Fix a chain homotopy $s_i: E_i \rightarrow F_{i+1}$, $i = 0, \dots, n-2$, between $\alpha$ and $\beta$. Assume that there is an admissible $n$-extension $\mathbb C$ in $\mathcal Ext^n_\C(X,Y)$ admitting morphisms $\iota_0, \iota_1 \colon \xi \rightarrow \mathbb C$, $\pi: \mathbb C \rightarrow \xi$ and $\Sigma : \mathbb C \rightarrow \zeta$ such that $\pi \circ \iota_0 = \id_\xi = \pi \circ \iota_1$ and $\Sigma \circ \iota_0 = \beta$, $\Sigma \circ \iota_1 = \alpha$. (Such a $\mathbb C$ is the admissible extension analogue of a \textit{cylinder object} for complexes; cf. \cite[III.2 and V.2.7]{GeMa03}.) It then follows that the path represented by $\alpha$ is homotopically equivalent to
$$
\xymatrix{
\xi \ar[r]^-{\iota_1} & \mathbb C \ar[r]^-\Sigma & \zeta \ .
}
$$
Because of $\pi \circ \iota_0 = \id_\xi = \pi \circ \iota_1$ this path is homotopically equivalent to
$$
\xymatrix{
\xi \ar[r]^-{\iota_0} & \mathbb C \ar[r]^-\Sigma & \zeta \ ,
}
$$
that is, it is homotopically equivalent to $\beta$. Thus it suffices to construct an extension $\mathbb C$ having the properties above.

Define $\mathbb C$ to be the following complex $(C_\bullet, c_\bullet)$: Let it be concentrated in degrees $-1, \dots, n$, wherein it is given by
$$
C_i := 
\begin{cases}
Y & \text{if $i = n$},\\
\Coker(\delta) & \text{if $i = n-1$},\\
E_i \oplus E_{i-1} \oplus E_i & \text{if $0 < i < n-1$},\\
E_0 \oplus E_0 & \text{if $i = 0$},\\
X & \text{if $i = -1$}.
\end{cases}
$$
Here $\delta$ is the morphism $\id_{E_{n-1}} \oplus e_{n-1} \oplus (-\id_{E_{n-1}}): E_{n-1} \rightarrow E_{n-1} \oplus E_{n-2} \oplus E_{n-1}$ (i.e., $C_{n-1}$ is the pushout of $(\id_{E_{n-1}} \oplus e_{n-1}, \id_{E_{n-1}})$, which is an object of $\C$). The differential $c_0: E_0 \oplus E_0 \rightarrow X$ is given by the matrix $\left[\begin{matrix} e_0 & e_0 \end{matrix}\right]$, whereas $c_n: Y \rightarrow \Coker(\delta)$ is simply induced by the inclusion
$$
\xymatrix{ Y \ar[r]^-{e_n} & E_{n-1} \ar[r]^-{j^{E_{n-1}}} & E_{n-1} \oplus (E_{n-2} \oplus E_{n-1})}
$$
into the first summand. For $i > 1$, the morphisms $\widetilde{c}_i: E_i \oplus E_{i-1} \oplus E_i \rightarrow E_{i-1} \oplus E_{i-2} \oplus E_{i-1}$ defined by
$$
\widetilde{c}_i = 
\left[\begin{matrix}
e_i & (-1)^{n-i} \id_{E_{i-1}} & 0\\
0 & e_{i-1} & 0\\
0 & (-1)^{n-i-1} \id_{E_{n-1}} & e_{i}
\end{matrix}\right]
$$
give rise to the differentials $c_i$ for $i > 2$; namely, set $c_i :=\widetilde{c}_i$ for $i \neq n-1$ and $c_{n-1}$ is the unique morphism $\Coker(\delta) \rightarrow E_{n-2} \oplus E_{n-3} \oplus E_{n-2}$ with $c_{n-1} \circ \coker(\delta) = \widetilde{c}_{n-1}$ (note that $\widetilde{c}_{n-1} \circ \delta = 0$). Finally, $c_1: E_1 \oplus E_0 \oplus E_1 \rightarrow E_0 \oplus E_0$ is the morphism
$$
c_1 = 
\left[\begin{matrix}
e_1 & (-1)^{n-1} \id_{E_{0}} & 0\\
0 & (-1)^n \id_{E_0} & e_{1} 
\end{matrix}\right] .
$$
It is not hard to see that the identity morphism of $i_\C \mathbb C$ in $\Ch(\A)$ is null-homotopic, and thus $i_\C \mathbb C$ is exact. By construction, $\mathbb C$ is hence admissible exact. The maps $\iota_0$ and $\iota_1$ are induced by the inclusions of $E_i$ into the outer summands of $C_i$. Consider the morphisms $\widetilde{\pi}_i: E_i \oplus E_{i-1} \oplus E_i \rightarrow E_i$,
$$
\widetilde{\pi}_i = \left[\begin{matrix}\id_{E_i} & 0 & \id_{E_i}\end{matrix}\right] \quad \text{(for $0 < i < n-2$).}
$$
Put $\pi_i :=  \widetilde{\pi}_i$ for $0 < i < n-1$, and let $\pi_{n-1}$ be the morphism which is induced by $\widetilde{\pi}_{n-1}$. Further, let $\pi_0: E_0 \oplus E_0 \longrightarrow E_0$ be given by the matrix
$$
\pi_0 = \left[\begin{matrix}\id_{E_0} & \id_{E_0}\end{matrix}\right].
$$
The morphism $\Sigma: \mathbb C \rightarrow \zeta$ of admissible $n$-extensions may be realized as the matrices
\begin{align*}
\Sigma_0 &= \left[\begin{matrix}
\beta_0 & \alpha_0
\end{matrix}\right],\\
\Sigma_i &= \left[\begin{matrix}
\beta_i & (-1)^{n-i}s_{i-1} & \alpha_i
\end{matrix}\right] \quad \text{(for $0 < i < n-1$)}
\intertext{and $\Sigma_{n-1}$ is induced by}
\widetilde{\Sigma}_{n-1} &= \left[\begin{matrix}
\beta_{n-1} & - s_{n-2}  & \alpha_{n-1}
\end{matrix}\right] .
\end{align*}
Clearly $\iota_0$, $\iota_1$, $\pi$ and $\Sigma$ have the desired properties.
\end{proof}
\begin{prop}\label{lem:mapsagree}
Let $n \geq 0$ be a fixed integer. Consider the automorphisms $a=(-1)^{n+1} \id_X$ and $b = (-1)^{n+1}\id_Y$. Then $u_\C^{\id_X,+} = \underline{u}_\C^{b,+}$ and $u_\C^{a,+} = \underline{u}_\C^{\id_Y,+}$ $($and therefore, $u_\C^{\id_X,-} = \underline{u}_\C^{b,-}$ and $u_\C^{a,-} = \underline{u}_\C^{\id_Y,-})$.
\end{prop}
\begin{proof}
We merely show the first equation; the second one follows similarly. For any morphism $g: X \rightarrow Y$ in $\C$ we have
$$
L_{\id_X}(g)^{-1} = 
\left[\begin{matrix}
\id_Y & g \\
0 & \id_X
\end{matrix}\right]^{-1} = \left[\begin{matrix}
\id_Y & -g \\
0 & \id_X
\end{matrix}\right] = \underline{L}_{b}(g)
$$
Therefore $u_\C^{\id_X,+} = \underline{u}_\C^{b, +}$ follows immediately if $n = 0$. 

Let $n \geq 1$ and let $\xi$ be an admissible $n$-extension of $X$ by $Y$ in $\C$. By taking the previous lemma into account, we see that it suffices to show that the morphisms in the diagram
$$
\xymatrix{
\sigma_{n+1}(X,Y) \ar[r]^-{\alpha_1} & \xi_+ \\
\xi^+ \ar[u]^-{\beta_1} \ar[r]^-{\beta_2} & \sigma_{n+1}(X,Y) \ar[u]_-{\alpha_2} \ ,
}
$$
resulting out of the loops defining $u_\C^{\id_X,+}$ and $\underline{u}_\C^{b,+}$, are such that $\alpha_1 \circ \beta_1$ and $\alpha_2 \circ \beta_2$ are homotopic. Put $\Delta :=\alpha_1 \circ \beta_1 - \alpha_2 \circ \beta_2$ and $M(e_0) := M_{\id_X}(e_0)$. Clearly, $\Delta$ is given by
$$
\small
\xymatrix@C=23pt{
0 \ar[r] & Y \ar[r] \ar[d]_0 & Y \oplus Y \ar[r]^-{\underline{M}(e_n)} \ar[d]_-{
\left[\begin{smallmatrix}
e_n & - e_n
\end{smallmatrix}\right]} & E_{n-1} \ar[d]_-0 \ar[r] & \cdots \ar[r] & E_1 \ar[d]^-0 \ar[r] & E_0 \ar[d]^-{
\left[\begin{smallmatrix}
e_0 \\
-e_0
\end{smallmatrix}\right]
} \ar[r] & X \ar[d]^-0 \ar[r] & 0 \\
0 \ar[r] & Y \ar[r] & E_{n-1} \ar[r] & E_{n-2} \ar[r] & \cdots \ar[r] & E_0 \ar[r]^-{M_{}(e_0)} & X \oplus X \ar[r] & X \ar[r] & 0
}
$$
and the maps
$$
s_i = (-1)^i \id_{E_i} \quad (\text{for $0 \leq i \leq n-1$}),
$$
satisfy
\begin{align*}
M(e_0) \circ s_0 &= \Delta_0,\\
e_{i}\circ s_{i} + s_{i-1} \circ e_{i} &= \Delta_i \quad \quad (\text{for $1 \leq i \leq n-1$}),\\
s_{n-1} \circ \underline{M}(e_n) &= \Delta_{n}.
\end{align*}
Therefore we are done.
\end{proof}
We close this section with a compatibility result relating the functors $\pi_0$, $\pi_1$ with the maps $u_\C^{a,+}$ and $v_\C(\xi) := v_\C(\sigma_n(X,Y), \xi)$ defined previously.
\begin{lem}\label{lem:grhomcomm}
Let $\scrX : \mathsf C \rightarrow \D$ be a $k$-linear exact functor between exact $k$-categories and let $n \geq 0$ be an integer. Let $X$ and $Y$ be objects in $\C$ and $\xi$ be an object in $\mathcal Ext ^n_{\C}(X,Y)$.
\begin{enumerate}[\rm(1)]
\item $\scrX$ induces a functor $\scrX_n : \mathcal E xt^n_{\mathsf C}(X,Y) \rightarrow \mathcal Ext^n_{\mathsf D}(\scrX X,\scrX Y)$ with $\scrX_n(\sigma_n(X,Y)) = \sigma_n(\scrX X,\scrX Y)$. Hence we obtain group homomorphisms
$$
\scrX^\sharp_n = \pi_0 \scrX_n : \Ext^n_{\mathsf C}(X,Y) \longrightarrow \Ext^n_{\mathsf D}(\scrX X, \scrX Y)
$$
and
$$
\scrX^\flat_n = \pi_1 (\scrX_n, \xi) : \pi_1 (\mathcal E xt^n_{\mathsf C}(X,Y), \xi) \longrightarrow \pi_1(
\mathcal E xt^n_{\D}(\scrX X,\scrX Y), \scrX_n \xi).
$$
\item Assume that $\C$ and $\D$ are factorizing and set $\xi' :=  \scrX_n \xi$. Then the following diagrams commute:
$$
\xymatrix@C=40pt{
\pi_1 (\mathcal E xt^n_{\C}(X,Y), \sigma_n(X,Y)) \ar[r]^-{v_\C(\xi)} \ar[d]_-{\pi_1 (\scrX_n, \sigma_n(X,Y))} & \pi_1 (\mathcal
E xt^n_{\mathsf C}(X,Y), \xi) \ar[d]^-{\pi_1 (\scrX_n, \xi)} \ \ \\
\pi_1( \mathcal E xt^n_{\mathsf D}(\scrX X,\scrX Y), \sigma_n(\scrX X,\scrX Y)) \ar[r]^-{v_\D(\xi')} & \pi_1( \mathcal Ext^n_{\mathsf D}(\scrX X, \scrX Y), \xi')
}
$$
and, for every automorphism $a \in \Aut_\C(X)$,
$$
\xymatrix@C=40pt{
\Ext^n_{\mathsf C}(X,Y) \ar[r]^-{u_\C^{a,+}} \ar[d]_-{\scrX^\sharp_n} & \pi_1 (\mathcal E xt^{n+1}_{\mathsf C}(X,Y), \sigma_n(X,Y)) \ar[d]^-{\scrX^\flat_n} \ \ \\
\Ext^n_{\mathsf D}(\scrX X, \scrX Y) \ar[r]^-{u_\D^{\scrX a,+}} & \pi_1( \mathcal E xt^{n+1}_{\mathsf D}(\scrX X, \scrX Y), \sigma_n(\scrX X,\scrX Y)) \ .
}
$$
\end{enumerate}
\end{lem}
\begin{proof}
By applying it degreewise, the functor $\scrX$ induces a functor $\Ch^b(\C) \rightarrow \Ch^b(\D)$ that
sends admissible extensions to admissible extensions (for $\scrX$ is exact). Hence we obtain the functor $\scrX_n$ by simply restricting to the subcategory $\mathcal Ext^n_{\C}(X,Y)$ of $\Ch^b(\C)$. The maps $\scrX^\sharp_n$ and $\scrX^\flat_n$ come for free as explained in section \ref{sec:homo_groups}, and $\scrX^\flat_n$ is a group homomorphism. Since $\scrX$ preserves pushouts along admissible monomorphisms and pullbacks along admissible epimorphisms by Lemma \ref{lem:exact_func_push_pull}, $\scrX^\sharp_n \colon \Ext^n_{\C}(X,Y) \rightarrow \Ext^n_{\D}(\scrX X,\scrX Y)$ will respect the Baer sum.

Let us prove the claimed commutativity of the diagrams. Since $\scrX$ is exact,
$$
\scrX_n(\xi \boxplus \sigma_n(X,Y)) \cong \scrX_n\xi \boxplus \scrX_n\sigma_n(X,Y) = \scrX_n\xi \boxplus \sigma_n(\scrX X,\scrX Y).
$$
When combining this with Lemma \ref{lem:fund_conj_commutative}, we see that the diagram
$$
\xymatrix@R=30pt@!C=73pt{
& \pi_1 (\mathcal E xt^n_{\C}(X,Y), \sigma_n(X,Y)) \ar[dl]_-{\pi_1 (\xi \boxplus -, \sigma_n(X,Y)) \quad \ } \ar[dr]^-{\quad \ \pi_1 (\scrX_n, \sigma_n(X,Y))} \ar@/^3pc/[ddl]^(0.6){v_\C(\xi)} &\\
\pi_1 (\mathcal E xt^n_{\C}(X,Y), \xi \boxplus \sigma_n(X,Y)) \ar[d]_-{c_{w}} & & \pi_1( \mathcal E xt^n_{\D}(\scrX X, \scrX Y), \sigma_n(\scrX X,\scrX Y)) \ar[d]^-{\pi_1 (\xi' \boxplus -,\sigma_n(\scrX X,\scrX Y))} \ar@/_3pc/[ddl]_(.5){v_\D(\xi')}\\
\pi_1 (\mathcal E xt^n_{\C}(X,Y), \xi) \ar[dr]_-{\pi_1 (\scrX_n, \xi) \ } & & \pi_1( \mathcal E xt^n_{\mathsf D}(\scrX X, \scrX Y), \xi' \boxplus \sigma_n(\scrX X,\scrX Y)) \ar[dl]^-{c_{\scrX w}}\\
& \pi_1( \mathcal Ext^n_{\D}(\scrX X, \scrX Y), \xi') &
}
$$
commutes (insert the morphism $\scrX^\flat_n = \pi_1 (\scrX_n, \xi \boxplus \sigma_n(X,Y))$,
$$
\scrX^\flat_n: \pi_1 (\mathcal E xt^n_{\C}(X,Y), \xi \boxplus \sigma_n(X,Y)) \longrightarrow  \pi_1( \mathcal E xt^n_{\mathsf D}(\scrX X, \scrX Y), \xi' \boxplus \sigma_n(\scrX X,\scrX Y)),
$$
at the right place), where $w$ is the natural isomorphism $\xi \rightarrow \xi \boxplus \sigma_n(X,Y)$. For the second diagram, we just remark that, for any $n$-extension $\xi$ in $\mathcal Ext^n_\C(X,Y)$, the $(n+1)$-extension $\scrX(\xi_+)$ is given by
$$
\xymatrix@C=33pt
{
0 \ar[r] & \scrX Y \ar[r]^-{\scrX e_{n}} & \cdots \ar[r]^-{\scrX e_1} & \scrX E_0 \ar[r]^-{M_{\scrX a}(\scrX e_0)} & \scrX X \oplus \scrX X \ar[r]^-{\left[\begin{smallmatrix}
1 & 1
\end{smallmatrix}\right]} & \scrX X \ar[r] & 0
}
$$
which coincides with $(\scrX \xi)_+$ (see \ref{nn:defnuplus} to bring the definitions of $\xi_+$ and $(\scrX \xi)_+$ back to mind).
\end{proof}
\section{Extension categories for monoidal categories}\label{sec:extcats_monoidal}
\begin{nn}
Throughout the following paragraph, let $\C$ and $\D$ be exact $k$-categories together with their defining embeddings $i_\C: \C \rightarrow \A_\C$, $i_\D: \D \rightarrow \A_\D$ into the abelian $k$-categories $\A_\C$ and $\A_\D$. Furthermore, fix objects $X, X', Y, Y' \in \Ob \C$ and integers $m,n \geq 1$. Let
$$
(\C, \otimes_\C, \mathbbm 1_\C) \quad \text{and} \quad (\D,\otimes_\D, \mathbbm 1_\D)
$$
be fixed monoidal structures on $\mathsf C$ and $\D$ respectively. We abbreviate $\otimes_\C$ and $\mathbbm 1_\C$ by $\otimes$ and $\mathbbm 1$. Recall that a monoidal category is a tensor $k$-category if its monoidal product functor is $k$-bilinear on morphisms.
\end{nn}
\begin{nn}Let $\xi$ be an admissible $m$-extension of $X$ by $Y$ and let $\zeta$ be an admissible $n$-extension of $X'$ by $Y'$:
$$
\xymatrix@R=10pt@C=15pt{
\xi & \equiv & 0 \ar[r] & Y \ar[r] & \mathbb E \ar[r] & X \ar[r] & 0\ , \\
\zeta & \equiv & 0 \ar[r] & Y' \ar[r] & \mathbb F \ar[r] & X' \ar[r] & 0 \ .
}
$$
We define their \textit{monoidal product} $\xi \boxtimes_\C \zeta$ (which will turn out to be an object in $\mathcal Ext^{m+n}_\mathsf{C}(X \otimes X', Y \otimes Y')$ if $\mathsf C$ is adequately chosen) to be the (usual) monoidal product of the truncated complexes $\xi^\natural: 0 \rightarrow Y \rightarrow \mathbb E$ and $\zeta^\natural: 0 \rightarrow Y' \rightarrow \mathbb F$ (as described in, for instance, \cite[Sec.\,2.7]{Be98}) followed by $e_0 \otimes f_0 : E_0 \otimes F_0
\rightarrow X \otimes X'$. More explicitly, $\xi \boxtimes_\C \zeta$ is given by
$$
(\xi \boxtimes_\C \zeta)_i := \bigoplus_{p = 0}^i{E_p \otimes F_{i-p}} \quad \text{(for $0 \leq i \leq m + n$)},
$$
and $(\xi \boxtimes_\C \zeta)_{-1} := X \otimes X'$ along with the differentials $\partial_i: (\xi \boxtimes_\C \zeta)_i \rightarrow (\xi \boxtimes_\C \zeta)_{i-1}$,
\begin{align*}
\partial_i := \sum_{p=1}^i \big( e_p \otimes {F_{i-p}}& + (-1)^{p} {E_p} \otimes f_{i-p} \big) \quad \text{(for $1 \leq i \leq m+n$)}
\end{align*}
and $\partial_{0} := e_0 \otimes f_0 \in \Hom_{\mathsf C}(E_0 \otimes F_0, X \otimes X')$. Be aware of the fact, that if one does not assume that $(\C, \otimes, \mathbbm 1)$ is a tensor $k$-category, then the above construction does not even need to yield a complex, much less an admissible extension.
\end{nn}
\begin{defn}\label{def:exactmonocat}
The triple $(\C, \otimes, \mathbbm 1)$ is
\begin{enumerate}[\rm(1)]
\item an \textit{exact monoidal category} if $\xi \boxtimes_\C \zeta$ is an admissible $(m+n)$-extension for all $m,n \in \mathbb Z_{\geq 1}$ and all $\xi \in \mathcal Ext^m_\C(\mathbbm 1,\mathbbm 1), \ \zeta \in \mathcal Ext^n_\C(\mathbbm 1,\mathbbm 1)$. 
\item a \textit{strong exact monoidal category} if it is an exact monoidal category and the exact $k$-category $\C$ is factorizing. 
\item a \textit{very strong exact monoidal category} if the exact category $\C$ is factorizing, and $\xi \boxtimes_\C \zeta$ is an admissible $(m+n)$-extension for all $m,n \in \mathbb Z_{\geq 1}$, $X, X', Y, Y' \in \Ob \C$ and all $\xi \in \mathcal Ext^m_\C(X,Y), \ \zeta \in \mathcal Ext^n_\C(X',Y')$.
\end{enumerate}
\end{defn}
\begin{lem}\label{lem:exactcomplex}
Let $\A_1, \A_2$ and $\sfB$ be abelian categories. Let $\C_1 \subseteq \A_1$ and $\C_2 \subseteq \A_2$ be exact subcategories. Assume that $\scrX: \C_1 \times \C_2 \rightarrow \sfB$ is an additive functor which is bilinear on morphisms. Then, for given bounded complexes $\mathbb{X} \in \Ch^b(\C_1)$, $\mathbb{Y} \in \Ch^b(\C_2)$ with $X_p = Y_p = 0$ for $p \leq -2$, the bounded complex $\mathbb T_\scrX (\mathbb X, \mathbb Y)$ defined by $\mathbb T_\scrX (\mathbb X, \mathbb Y)_{-1} = \scrX(X_{-1}, Y_{-1})$ and
$$
\mathbb T_\scrX (\mathbb X, \mathbb Y)_i:=\mathrm{Tot}(\scrX(\mathbb X_{\geq 0}, \mathbb Y_{\geq 0}))_i = \bigoplus_{p=0}^i \scrX(X_p, Y_{i-p})
\quad \text{$($for  $i \geq 0)$}, 
$$
with the usual differentials $\partial_i$ for $i \neq 0$ and $\partial_{0} = \scrX(x_0, y_0)$, is acyclic in the following two situations.
\begin{enumerate}[\rm(1)]
\item\label{lem:exactcomplex:1} The complexes $\mathbb{X}$ and $\mathbb Y$ are acyclic, and the functors $\scrX(X, -)$ $($for all $X$ in $\C_1)$ and $\scrX(-,Y_{-1})$ are exact.
\item\label{lem:exactcomplex:2} The complexes $\mathbb{X}$ and $\mathbb Y$ are acyclic, and the functors $\scrX(X_{-1}, -)$ and $\scrX(-,Y)$ $($for all $Y$ in  $\C_2)$ are exact.
\end{enumerate}
\end{lem}
\begin{proof}
Let us only show that (\ref{lem:exactcomplex:1}) implies the acyclicity of $\mathbb T_\scrX (\mathbb X, \mathbb Y)$ (dual arguments apply when (\ref{lem:exactcomplex:2}) is assumed). In the following we will write $\abs{\scrX(\mathbb C, \mathbb D)}$ for
$\mathrm{Tot}(\scrX(\mathbb C, \mathbb D))$.  Without loss of generality, we may assume that both complexes $\mathbb X$ and $\mathbb Y$ are different from the zero complex. If $\C$ is any exact $k$-category, $r \in \mathbb Z$ is an integer, and $\mathbb C$ is a complex in $\Ch(\C)$, then it yields a commutative diagram
$$
\xymatrix{
& \vdots \ar[d] & \vdots \ar[d] & \vdots \ar[d] &\\
0 \ar[r] & 0 \ar[r] \ar[d] & C_{r+2} \ar[r] \ar[d] & C_{r+2} \ar[r] \ar[d] & 0\\
0 \ar[r] & 0 \ar[r] \ar[d] & C_{r+1} \ar[r] \ar[d] & C_{r+1} \ar[r] \ar[d] & 0\\
0 \ar[r] & C_{r} \ar[r] \ar[d] & C_{r} \ar[r] \ar[d] & 0 \ar[r] \ar[d] & 0\\
0 \ar[r] & C_{r-1} \ar[r] \ar[d] & C_{r-1} \ar[r] \ar[d] & 0 \ar[r] \ar[d] & 0\\
& \vdots & \vdots & \vdots &
}
$$
in $\Ch(\C)$ and hence an admissible short exact sequence
$$
\xi_{\mathbb C}^{r} \quad \equiv \quad 0 \longrightarrow \mathbb{C}_{\leq r} \longrightarrow \mathbb C \longrightarrow \mathbb C_{\geq r+1}
\longrightarrow 0 
$$
in $\Ch(\C)$. With this observation at hand, we are going to divide the proof into two intermediate steps.
\begin{enumerate}[\rm(a)]
\item\label{item:proofexacttensor} We show that the induced sequences
\begin{align*}
0 &\longrightarrow \abs{\scrX(\mathbb{X}_{\leq -1}, \mathbb C)} \longrightarrow \abs{\scrX(\mathbb{X},
\mathbb C)} \longrightarrow
\abs{\scrX(\mathbb{X}_{\geq 0}, \mathbb C)} \longrightarrow 0 \ ,
\\ &\\
0 &\longrightarrow \abs{\scrX(\mathbb D, \mathbb{Y}_{\leq -1})} \longrightarrow \abs{\scrX(\mathbb D, \mathbb{Y})} \longrightarrow
\abs{\scrX(\mathbb D, \mathbb{Y}_{\geq 0})} \longrightarrow 0
\end{align*}
are short exact ones in $\Ch^b(\mathsf B)$ for every $\mathbb C \in \Ch^b(\C_2)$ and $\mathbb D \in \Ch^b(\C_1)$.

\item\label{item:proofexacttensor:2} By using (\ref{item:proofexacttensor}) we deduce that $\scrX(\mathbb X_{\geq 0}, \mathbb Y_{\geq 0})$ has trivial homology in positive, non-zero degrees if $\mathbb X$ and $\mathbb Y$ are acyclic.
\end{enumerate}
Observe that for every $i \in \mathbb Z$, 
$$
0 \longrightarrow \abs{\scrX(\mathbb{X}_{\leq -1}, \mathbb C)}_i \longrightarrow \abs{\scrX(\mathbb{X},
\mathbb C)}_i \longrightarrow
\abs{\scrX(\mathbb{X}_{\geq 0}, \mathbb C)}_i \longrightarrow 0
$$
and
$$
0 \longrightarrow \abs{\scrX(\mathbb D, \mathbb{Y}_{\leq -1})}_i \longrightarrow \abs{\scrX(\mathbb D, \mathbb{Y})}_i \longrightarrow
\abs{\scrX(\mathbb D, \mathbb{Y}_{\geq 0})}_i \longrightarrow 0
$$
simply are the split exact sequences
\begin{align*}
0 &\longrightarrow \scrX(X_{-1}, C_{i+1}) \longrightarrow \bigoplus_{p \geq -1} \scrX(X_p, C_{i-p}) \longrightarrow
\bigoplus_{p \geq 0} \scrX(X_p, C_{i-p}) \longrightarrow 0 \ , \\
&\\
0 &\longrightarrow {\scrX(D_{i+1}, {Y}_{-1})} \longrightarrow \bigoplus_{p \geq -1}{\scrX(D_p, Y_{i-p})} \longrightarrow \bigoplus_{p \geq 0}{\scrX(D_p, {Y}_{i-p})} \longrightarrow 0 \ .
\end{align*}
Hence assertion (\ref{item:proofexacttensor}) follows immediatly. Let us prove (\ref{item:proofexacttensor:2}). The complex $\abs{\scrX(\mathbb X_{\geq 0}, \mathbb Y)}$ is acyclic: There is a convergent spectral sequence for the double complex $\scrX(\mathbb X_{\geq 0}, \mathbb Y)$ with $E^2$- and $E^\infty$-terms
$$
E^2_{p,q} = H_p^h H^v_q \scrX(\mathbb X_{\geq 0}, \mathbb Y) \ \Longrightarrow \ H_{p+q}\abs{\scrX(\mathbb X_{\geq 0}, \mathbb Y)} = E^\infty_{p,q}
$$
(cf. \cite[Sec.\,5.6]{Wei94}); here $H_q^v$ denotes the homology with respect to the vertical differentials, and $H_p^h$ denotes the homology with respect to the horizontal differentials. Since $\mathbb Y$ is acyclic and $\scrX(X,-)$ is exact for all objects $X$ in $\C_1$, it follows that $H^v_q \scrX(\mathbb X_{\geq 0}, \mathbb Y) = 0$ for all $q \in \mathbb Z$ and thus
$$
0 = E^2_{p,q} = E^3_{p,q} = \cdots = E^r_{p,q} = \cdots = E^\infty_{p,q} \quad (\text{for all $p,q \in \mathbb Z$ and $r \geq 2$})
$$
by \cite[Prop.\,XV.5.1]{CaEi56}. In particular, $H_n\abs{\scrX(\mathbb X_{\geq 0}, \mathbb Y)} = 0$ for all $n \in \mathbb Z$. Therefore the long exact sequence in homology (see \cite[Thm.\,6.10]{Ro09}),
$$
\xymatrix{
\cdots \ar[r] & H_{n+1} \abs{\scrX(\mathbb X_{\geq 0}, \mathbb Y_{\geq 0})} \ar[r] & H_n \abs{\scrX(\mathbb X_{\geq 0}, \mathbb Y_{\leq -1})} \ar[r] & H_n \abs{\scrX(\mathbb X_{\geq 0}, \mathbb Y)} \ar `r[d] `[l] `[llld] `[dll] [dll]\\
& H_n \abs{\scrX(\mathbb X_{\geq0},\mathbb Y_{\geq 0})} \ar[r] & H_{n-1} \abs{\scrX(\mathbb X_{\geq 0}, \mathbb
Y_{\leq -1})} \ar[r] & \cdots \ ,
}
$$
yields that $H_n\abs{\scrX(\mathbb X_{\geq 0}, \mathbb Y_{\geq 0})} \cong H_{n-1}\abs{\scrX(\mathbb X_{\geq 0}, \mathbb Y_{\leq -1})}$, where the right hand side is zero for $n \neq 0$. Hence the second claim follows.
\begin{center}
\textbf{Conclusion}
\end{center}
We have proven that $H_n \abs{\scrX(\mathbb X_{\geq 0}, \mathbb Y_{\geq 0})} = 0$ for all $n \neq 0$.
Moreover, 
$$
H_0 \abs{\scrX(\mathbb X_{\geq 0}, \mathbb Y_{\geq 0})} \cong H_{-1} \abs{\scrX(\mathbb
X_{\geq 0}, \mathbb Y_{\leq -1})} \cong \scrX(X_{-1}, Y_{-1}).
$$
The universal property of the cokernel provides the commutative diagram
$$
\xymatrix@C=30pt{
\cdots \ar[r] & \abs{\scrX(\mathbb X, \mathbb Y)}_1 \ar[r] \ar@{=}[d] & \abs{\scrX(\mathbb X, \mathbb Y)}_0 \ar@{->>}[r]
\ar@{=}[d] & H_0 \abs{\scrX(\mathbb X_{\geq 0}, \mathbb Y_{\geq 0})} \ar@{-->}[d] \ar[r] & 0\\
\cdots \ar[r] & \abs{\scrX(\mathbb X, \mathbb Y)}_1 \ar[r] & \abs{\scrX(\mathbb X, \mathbb Y)}_0 \ar[r]^-{\scrX(x_0,y_0)}
& \scrX(X_{-1}, Y_{-1}) \ar[r] & 0\\
}
$$
where the induced dotted arrow has to be an isomorphism. Indeed, it is an epimorphism right away, since $\scrX(x_0,y_0) = \scrX(x_0,Y_{-1}) \circ \scrX(X_0,y_0)$ is epimorphic. The following observation will immediately imply that it is also monomorphic, and hence the lemma is proven.
\begin{enumerate}[\rm(c)]
\item Let $\xi \ \equiv \ 0 \rightarrow A \xrightarrow{f} B \xrightarrow{g} C \rightarrow 0$ be an admissible short exact sequence in $\C_1$ and $\zeta \ \equiv \ 0 \rightarrow X \xrightarrow{s} Y \xrightarrow{t} Z \rightarrow 0$ be an admissible short exact sequence in $\C_2$. Assume that $\scrX(-,Z)$ is an exact functor $\C_1 \rightarrow \B$. Then $\abs{\scrX(\xi, \zeta)}$ is a $2$-extension in $\B$.
\end{enumerate}
In degrees different form zero, the differentials in $\abs{\scrX(\xi, \zeta)}$ are given by $\delta_2 = \scrX(f,X) \oplus \scrX(A,s)$ and $\delta_1 = - \scrX(B,s) + \scrX(f,Y)$. Thus, we have the commutative diagram
$$
\xymatrix{
0 \ar[r] & 0 \ar[r] \ar[d] & \scrX(B,X) \ar@{=}[r] \ar[d]_-{\mathrm{can}} & \scrX(B,X) \ar[r] \ar[d]^-{\scrX(B,s)} & 0 \ar[d] \ar[r] & 0 \\
0 \ar[r] & \scrX(A,X) \ar[r]^-{\delta_2} \ar@{=}[d] & \scrX(B,X) \oplus \scrX(A,Y) \ar[r]^-{\delta_1} \ar[d]_-{\mathrm{can}} & \scrX(B,Y) \ar[r]^-{\scrX(g,t)} \ar[d]^-{\scrX(B,t)} & \scrX(C,Z) \ar@{=}[d] \ar[r] & 0\\
0 \ar[r] & \scrX(A,X) \ar[r]^-{\scrX(A,s)} & \scrX(A,Y) \ar[r]^-{\scrX(f,t)} & \scrX(B,Z) \ar[r]^-{\scrX(g,Z)} & \scrX(C,Z) \ar[r] & 0
}
$$
which, when read from top to bottom, defines a short exact sequence in $\Ch(\B)$, and in which the upper row and the lower row are exact in $\B$. But then, also the middle row $\abs{\scrX(\xi, \zeta)}$ has to be exact (for instance, by the long exact sequence in homology).
\end{proof}
\begin{cor}\label{cor:tensorexact}
Consider the following statements on $(\C, \otimes, \mathbbm 1)$.
\begin{enumerate}[\rm(1)]
\item\label{cor:tensorexact:1} The exact $k$-category $\C$ is closed under kernels of epimorphisms.
\item\label{cor:tensorexact:2} The exact $k$-category $\C$ is closed under cokernels of monomorphisms.
\item\label{cor:tensorexact:3} $(\C, \otimes, \mathbbm 1)$ is a tensor $k$-category, and every object is flat.
\item\label{cor:tensorexact:3*} $(\C, \otimes, \mathbbm 1)$ is a tensor $k$-category, and every object is co\-flat.
\item\label{cor:tensorexact:4} $(\C, \otimes, \mathbbm 1)$ is a very strong exact monoidal category.
\item\label{cor:tensorexact:5} $(\C, \otimes, \mathbbm 1)$ is a strong exact monoidal category.
\item\label{cor:tensorexact:6} $(\C, \otimes, \mathbbm 1)$ is an exact monoidal category.
\end{enumerate}
Then the implications
\begin{align*}
(\ref{cor:tensorexact:1}) + (\ref{cor:tensorexact:3}) \Longrightarrow (\ref{cor:tensorexact:5}), \quad (\ref{cor:tensorexact:1}) + (\ref{cor:tensorexact:3*}) \Longrightarrow (\ref{cor:tensorexact:5}), \quad (\ref{cor:tensorexact:2}) + (\ref{cor:tensorexact:3}) \Longrightarrow (\ref{cor:tensorexact:5}), \quad (\ref{cor:tensorexact:2}) + (\ref{cor:tensorexact:3*}) \Longrightarrow (\ref{cor:tensorexact:5})\\
(\ref{cor:tensorexact:1}) + (\ref{cor:tensorexact:3}) + (\ref{cor:tensorexact:3*}) \Longrightarrow (\ref{cor:tensorexact:4}), \quad (\ref{cor:tensorexact:2}) + (\ref{cor:tensorexact:3}) + (\ref{cor:tensorexact:3*}) \Longrightarrow (\ref{cor:tensorexact:4}) \quad \text{and} \quad (\ref{cor:tensorexact:4}) \Longrightarrow (\ref{cor:tensorexact:5}) \Longrightarrow (\ref{cor:tensorexact:6})
\end{align*}
hold true.
\end{cor}
\begin{proof}
The implications $(\ref{cor:tensorexact:4}) \Longrightarrow (\ref{cor:tensorexact:5}) \Longrightarrow (\ref{cor:tensorexact:6})$ are clear. All others are immediate consequences of Lemma \ref{lem:exactcomplex} and Lemma \ref{lem:admono} (since every acyclic complex in $\A$ with terms in $\C$ will be, by the assumptions made, admissible exact in $\C$; see the arguments inside the proof of Lemma \ref{lem:admono}). Notice that the tensor unit $\mathbbm 1$ is, by definition, both flat and coflat in $\C$.
\end{proof}
\begin{defn}[Translation functor]\label{def:transfunc}
Let $f: X \rightarrow X'$ and $g: Y' \rightarrow Y$ be isomorphisms in $\mathsf C$. The \textit{translation functor $\mathrm{Tr} = \mathrm{Tr}_{f,g}$} (\textit{with respect to $f$ and $g$}) $$\mathrm{Tr} : \mathcal Ext^n_\C(X,Y) \longrightarrow \mathcal Ext^n_\C(X',Y')$$ is given by the assignment 
\begin{align*}
\mathrm{Tr}&(
\xymatrix@C=15pt{
0 \ar[r] & Y \ar[r] & E_{n-1} \ar[r] & \cdots \ar[r] & E_0 \ar[r] & X \ar[r] & 0
})\\
&= \xymatrix@C=15pt{
0 \ar[r] & Y' \ar[r]^-g & Y \ar[r] & E_{n-1} \ar[r] & \cdots \ar[r] & E_0 \ar[r] & X \ar[r]^-f &  X' \ar[r] & 0
} .
\end{align*}
It is an isomorphism of categories with inverse functor $\mathrm{Tr}_{f^{-1},g^{-1}}$.
\end{defn}
\begin{nn}Assume that $(\C, \otimes, \mathbbm 1)$ is an exact monoidal category, and let $\mathrm{Tr}$ be $\mathrm{Tr}_{\lambda_{\mathbbm 1}^{-1}, \lambda_{\mathbbm 1}}$. We obtain a family of bifunctors
$$
\boxtimes_\C: \mathcal Ext^m_\C(\mathbbm 1,\mathbbm 1) \times \mathcal Ext^n_\C(\mathbbm 1,\mathbbm 1)
\rightarrow \mathcal Ext^{m+n}_\C(\mathbbm 1,\mathbbm 1) \quad (m, n \geq 1).
$$
In fact, $\boxtimes_\C$ sends $(\xi, \zeta)$ to $\mathrm{Tr}(\xi \boxtimes_\C \zeta)$. By the usual abuse of notation, we are going to use the identification $\xi \boxtimes_\C \zeta = \mathrm{Tr}(\xi \boxtimes_\C \zeta)$ for simplicity.
\end{nn}
\begin{nn} 
Assume that both $(\C, \otimes_\C, \mathbbm 1_\C)$ and $(\D, \otimes_\D, \mathbbm 1_\D)$ are exact monoidal categories. Let $(\scrL, \phi, \phi_0): \C \rightarrow \D$ be an exact and lax monoidal functor such that $\phi_0: \mathbbm 1_\D \rightarrow \scrL\mathbbm 1_\C$ is an isomorphism (i.e., $(\scrL, \phi, \phi_0)$ is exact and almost costrong). Note that since
$$
\xymatrix@C=30pt{
\mathbbm 1_\D \otimes_\D \scrL \mathbbm 1_\C \ar[r]^-{\lambda_{\scrL \mathbbm 1_\C}} \ar[d]_{\phi_0 \otimes \scrL\mathbbm 1_\C} & \scrL \mathbbm 1_\C\\
\scrL \mathbbm 1_\C \otimes_\D \scrL \mathbbm 1_\C \ar[r]^-{\phi_{\mathbbm 1_\C, \mathbbm 1_\C}} & \scrL(\mathbbm 1_\C \otimes_\C \mathbbm
1_\C) \ar[u]_-{\scrL(\lambda_{\mathbbm 1_\C})}
}
$$
commutes, this forces $\phi_{\mathbbm 1_\C, \mathbbm 1_\C}: \scrL\mathbbm 1_\C \otimes_\D \scrL \mathbbm 1_\C \rightarrow \scrL(\mathbbm 1_\C \otimes_\C \mathbbm 1_\C)$ to be an isomorphism, too. The lax monoidal functor $(\scrL, \phi,\phi_0)$ gives rise to the following morphisms in $\D$:
\begin{align*}
\Delta : \mathbbm 1_\D \rightarrow \scrL \mathbbm 1_\C \otimes_\D \scrL \mathbbm 1_\C,& \quad \Delta = (\phi_0
\otimes \scrL(\id_{\mathbbm 1_\C})) \circ \lambda_{\scrL \mathbbm 1_\C}^{-1} \circ \phi_0, \\
\underline{\Delta}: \mathbbm 1_\D \rightarrow \scrL (\mathbbm 1_\C \otimes_\C \mathbbm 1_\C),& \quad \underline{\Delta}
= \scrL(\lambda_{\mathbbm 1_\C}^{-1}) \circ \phi_0
\intertext{as well as}
\nabla : \scrL \mathbbm 1_\C \otimes_\D \scrL \mathbbm 1_\C \rightarrow \mathbbm 1_\D,& \quad \nabla = \phi_0^{-1} \circ
\scrL(\lambda_{\mathbbm 1_\C}) \circ \phi_{\mathbbm 1_\C, \mathbbm 1_\C}, \\
\overline{\nabla}: \scrL(\mathbbm 1_\C \otimes_\C \mathbbm 1_\C) \rightarrow \mathbbm 1_\D, & \quad \overline{\nabla}
= \nabla \circ \phi_{\mathbbm 1_\C, \mathbbm 1_\C}^{-1} = \phi_0^{-1} \circ \scrL(\lambda_{\mathbbm 1_\C}).
\end{align*}
One may ask oneself, if the induced diagram
\begin{equation}\label{eq:extdiag}\begin{aligned}
\xymatrix@R=30pt@!C=50pt{
& \mathcal Ext^m_\C(\mathbbm 1_\C,\mathbbm 1_\C) \times \mathcal Ext^n_\C(\mathbbm 1_\C,\mathbbm 1_\C) \ar[dl]_-{\boxtimes_\C} \ar[dr]^-{\scrL_m \times \scrL_n} &\\
\mathcal Ext^{m+n}_\C(\mathbbm 1_\C,\mathbbm 1_\C) \ar[d]_{\scrL_{m+n}} && \mathcal Ext^m_{\D}(\scrL \mathbbm 1_\C,\scrL \mathbbm 1_\C) \times \mathcal Ext^n_{\D}(\scrL \mathbbm 1_\C,\scrL \mathbbm 1_\C) \ar[d]^-{\boxtimes_\D}\\
\mathcal Ext^{m+n}_\D(\scrL \mathbbm 1_\C,\scrL \mathbbm 1_\C) \ar[dr]_-{\mathrm{Tr}_{\phi_0, \phi_0^{-1}}} &&\mathcal Ext^{m+n}_{\D}(\scrL \mathbbm 1_\C \otimes_\D \scrL \mathbbm 1_\C, \scrL \mathbbm 1_\C \otimes_\D \scrL \mathbbm 1_\C) \ar[dl]^-{\mathrm{Tr}_{\Delta, \nabla}}\\
& \mathcal Ext^{m+n}_\D(\mathbbm 1_\D,\mathbbm 1_\D) &
}
\end{aligned}\end{equation}
is going to commute then. In general, this cannot be expected. Nevertheless, we have the following sufficiently nice statement.
\end{nn}
\begin{prop}\label{prop:laxfunc_comm}
Let $(\scrL, \phi, \phi_0): \C \rightarrow \D$ be an exact and almost strong monoidal functor. Then, with the notations made above, the diagram $(\ref{eq:extdiag})$ commutes on the level of $\pi_0$.
\end{prop}
\begin{proof}
Fix two sequences $\xi \in \mathcal Ext^m_\C(\mathbbm 1_\C, \mathbbm 1_\C)$ and $\zeta \in
\mathcal Ext^n_\C(\mathbbm 1, \mathbbm 1)$,
$$
\xymatrix@C=15pt@R=10pt{
\xi & \equiv & 0 \ar[r] & \mathbbm 1_\C \ar[r] & \mathbb E \ar[r] & \mathbbm 1_\C \ar[r] & 0 \ ,\\
\zeta & \equiv & 0 \ar[r] & \mathbbm 1_\C \ar[r] & \mathbb F \ar[r] & \mathbbm 1_\C \ar[r] & 0 \ .
}
$$
It suffices to name a morphism
$$
\vartheta: \mathrm{Tr}_{\Delta, \nabla}(\scrL_m\xi \boxtimes_\D \scrL_n \zeta) \longrightarrow
\mathrm{Tr}_{\phi_0, \phi_0^{-1}}(\scrL_{m+n}(\xi \boxtimes_\C \zeta))
$$
within the category $\mathcal Ext^{m+n}_{\D}(\mathbbm 1_\D, \mathbbm 1_\D)$. Let the $i$-th component of $\vartheta$ be given by the following diagonal matrix:
$$
\vartheta_i := \left[
\begin{matrix}
\phi_{E_0, F_i} & 0 & \cdots & 0 & 0\\
0 & \phi_{E_1, F_{i-1}} & \cdots & 0 & 0\\
\vdots & \vdots & \ddots & \vdots & \vdots \\
0 & 0 & \cdots  & \phi_{E_{i-1}, F_1} & 0\\
0 & 0 & \cdots & 0 & \phi_{E_i, F_0}
\end{matrix}\right] \quad (\text{for $1 \leq i \leq m+ n-1$}). 
$$
We convince ourselves that this actually gives rise to the desired morphism. The commutativity of
$$
\xymatrix{
\scrL E_0 \otimes_\D \scrL F_0 \ar[r] \ar[d] & \scrL \mathbbm 1_\C \otimes_\D \scrL \mathbbm 1_\C \ar[r] \ar[d] & \scrL(\mathbbm 1_\C \otimes_\C \mathbbm 1_\C) \ar[r] \ar@{=}[d] & \scrL \mathbbm 1_\C \ar@{=}[d] \ar[r] & \mathbbm 1_\D \ar@{=}[d]\\
\scrL (E_0 \otimes_\C F_0) \ar[r] & \scrL (\mathbbm 1_\C \otimes_\C \mathbbm 1_\C) \ar@{=}[r] & \scrL(\mathbbm 1_\C \otimes_\C \mathbbm 1_\C) \ar[r] & \scrL \mathbbm 1_\C \ar[r] & \mathbbm 1_\D
}
$$
is automatic. The same holds for the diagrams
$$
\xymatrix{
\scrL E_p \otimes_\D \scrL F_q \ar[r] \ar[d] & (\scrL E_{p-1} \otimes_\D \scrL F_q) \oplus (\scrL
E_p \otimes_\D \scrL F_{q-1}) \ \ \ar@<-3.75pt>[d]\\
\scrL (E_p \otimes_\C F_q) \ar[r] & \scrL (E_{p-1} \otimes_\C F_q) \oplus \scrL (E_p
\otimes_\C F_{q-1}) \ ,
}
$$
(for $1 \leq p \leq m$ and $1 \leq q \leq n$ with $p + q < n+m$). Finally, set
$$
\delta:=(\phi_0 \otimes \scrL \mathbbm 1_\C) \circ \lambda_{\scrL \mathbbm 1_\C}^{-1}
$$
and consider the diagram
$$
\xymatrix@C=30pt{
\mathbbm 1_\D \ar[r]^{\phi_0} \ar@{=}[d] & \scrL \mathbbm 1_\C \ar[r]^-{\delta} \ar@{=}[d] &
\scrL \mathbbm 1_\C \otimes_\D \scrL \mathbbm 1_\C \ar[r] \ar[d]^-{\phi_{\mathbbm 1_\C, \mathbbm 1_\C}} & (\scrL \mathbbm 1_\C \otimes_\D \scrL F_{n-1}) \oplus (\scrL E_{m-1} \otimes_\D \scrL \mathbbm 1_\C) \ar@<-1.75pt>[d]\\
\mathbbm 1_\D \ar[r]_{\phi_0} & \scrL \mathbbm 1_\D \ar[r]^-{\scrL(\lambda_{\mathbbm
1_\C}^{-1})} & \scrL(\mathbbm 1_\C \otimes_\C \mathbbm 1_\C) \ar[r] & \scrL(\mathbbm 1_\C \otimes_\C F_{n-1}) \oplus \scrL(E_{m-1} \otimes_\C \mathbbm 1_\C)
}
$$
wherein the very leftmost squre evidently is commutative. The very rightmost square commutes because $\phi$ is a morphism of functors. Since $\scrL$ is a lax monoidal functor,
$$
\lambda_{\scrL \mathbbm 1_\C} = \scrL(\lambda_{\mathbbm 1_\C}) \circ \phi_{\mathbbm 1_\C, \mathbbm 1_\C} \circ (\phi_0 \otimes_\D \scrL\mathbbm 1_\C),
$$
so the middle square commutes, too. This proves the claim.
\end{proof}
\begin{rem}\label{rem:laxfunc_comm}
Note that by a similar argument, Proposition \ref{prop:laxfunc_comm} also holds true, if one replaces the exact and almost strong monoidal functor $(\scrL, \phi, \phi_0)$ by an exact and almost costrong monoidal functor $(\scrL, \psi, \psi_0)$.
\end{rem}



\chapter{Hochschild cohomology}\label{cha:hochschild}
\section{Basic definitions}\label{sec:basdef}
\begin{nn}
In this section, we will recall the definition of Hochschild cohomology as it was introduced by G.\,Hochschild in \cite{Ho45}. To this end, we fix a commutative ring $k$ and a $k$-algebra $A$. If $B$ is any other $k$-algebra, the $k$-module $A \otimes_k B$ is an associative $k$-algebra via
$$
(a \otimes b)(a'\otimes b') = aa' \otimes bb' \quad (\text{for $a,a' \in A$, $b,b' \in B$})
$$
with unit given by $1_A \otimes 1_B$. Hence $A^\ev := A \otimes_k A^\op$ is a $k$-algebra, the so called \textit{enveloping algebra of $A$}. Note that every left module $M$ over $A^\ev$ is an $A$-$A$-bimodule (with central $k$-action) thanks to
$$
am := (a \otimes 1_A)m \quad (\text{for $a \in A$, $m \in M$})
$$
and
$$
ma :=m(1_A \otimes a) \quad (\text{for $a \in A$, $m \in M$}).
$$
In fact, the hereby defined assignment ${_{A^\ev}}M \mapsto {_A M _A}$ establishes an equivalence between the category of $A^\ev$-modules and the category of $A$-$A$-bimodules on which $k$ acts centrally:
$$
\xymatrix{\Mod(A^\ev) \ar[r]^-{\sim} & \mathsf{Bimod}_k(A,A)}.
$$
\end{nn}

\begin{nn}
There are two natural $k$-algebra homomorphisms:
$$
e_\lambda: A \longrightarrow A \otimes_k A^\op, \ e_\lambda(a) = a \otimes 1_A
$$
and
$$
e_\varrho: A^\op \longrightarrow A \otimes_k A^\op, \ e_\varrho(a) = 1_A \otimes a.
$$
They give rise to two forgetful functors $e_\lambda^\ast = \mathrm{Res}(e_\lambda): \Mod(A^\ev) \rightarrow \Mod(A)$, $e_\varrho^\ast = \mathrm{Res}(e_\varrho)\colon \Mod(A^\ev) \rightarrow \Mod(A^\op)$. The functors coincide with the ones defined in Example \ref{exa:bimodules}.
\end{nn}

\begin{nameless}[The Hochschild complex]
Let $M$ be an $A^\ev$-module. For any integer $n \geq 0$, let $S^n(A) = A^{\otimes_k n}$ be the $n$-fold tensor product of $A$ with itself over $k$. Furthermore, put $C^n(A,M) = \Hom_k(S^n(A), M)$. Together with the maps $\partial^n \colon C^{n}(A,M) \rightarrow C^{n+1}(A,M)$,
\begin{align*}
(\partial^n f)(a_1 \otimes \cdots \otimes a_{n+1}) = a_1 & f(a_2 \otimes \cdots \otimes a_{n+1})\\ &+ \sum_{i=1}^n (-1)^i {f(a_1 \otimes
\cdots \otimes a_i a_{i+1} \otimes \cdots \otimes a_{n+1})}\\ &+ (-1)^{n+1}f(a_1 \otimes \cdots \otimes a_n)a_{n+1}
\end{align*}
(for $a_1, \dots, a_{n+1} \in A$) the $k$-modules $C^n(A,M)$, $n \geq 0$, form a complex $\mathbb C (A,M)$. Indeed, we have
$$
(\partial^{n+1} \circ \partial^n)(f) = 0 \quad \forall f \in C^n(A,M).
$$
If $\widetilde{S}^n(A)$ denotes the module $A^{\otimes_k (n+2)}$ we clearly get
$$
\widetilde{S}^n(A) \cong A^\ev \otimes_k S^n(A).
$$
The following definition is fundamental.
\end{nameless}
\begin{defn}\label{def:hcohom}
Let $M$ be an $A^\ev$-module. The \textit{$n$-th Hochschild cohomology of $A$ with coefficients in $M$} is $$\HH^n(A,M) := H^n (\mathbb C(A,M)) = \frac{\Ker(\partial^n)}{\Im(\partial^{n-1})}.$$ 
\end{defn}
\begin{nn}
Let $M$ be an $A^\ev$-module. Denote by $M^A$ the set of all elements $m \in M$ with $am = ma$ for all $a \in A$. It is well known (and easily verfied), that $\HH^0(A,M) \cong M^A$. In particular, $\HH^0(A,A)$ is isomorphic to the center of $A$.
\end{nn}
\begin{nameless}[The cup product]
Let $M$, $N$ be two $A^\ev$-modules, and let $f \in C^m(A,M)$, $g \in C^n(A,N)$ for given integers $m, n \geq 0$. The homomorphism $f \cup g: A^{\otimes_R(m+n)} \rightarrow M \otimes_A N$,
$$
(f \cup g)(a_1 \otimes \cdots \otimes a_{m+n}) = f(a_1 \otimes \cdots \otimes a_m) \otimes g(a_{m+1} \otimes \cdots \otimes a_{m+n}),
$$
is an element of $C^{m+n}(A,M \otimes_A N)$, and thus we get a map
$$
\cup: C^{m}(A,M) \times C^{n}(A,N) \rightarrow C^{m+n}(A,M \otimes_A N).
$$
It turns $\mathbb C(A,A)$ into a differential graded (dg) algebra (after, of course, using the standard identification $A\otimes_A A \cong A$). The product on $\mathbb C(A,A)$ reads as follows
$$
(f \cup g)(a_1 \otimes \cdots \otimes a_{m+n}) = f(a_1 \otimes \cdots \otimes a_m) g(a_{m+1} \otimes \cdots \otimes a_{m+n})
$$
for $f \in C^m(A,A)$, $g \in C^n(A,A)$ and $a_1, \dots, a_{m+n} \in A$. The fact that $\mathbb C(A,A)$ is a dg algebra ensures that the \text{cup product} on $\mathbb C(A,A)$ passes down to a well-defined product map on $\HH^\bullet(A,A):= \bigoplus_{n \geq 0} \HH^n(A,A)$:
$$
\cup: \HH^m(A,A) \times \HH^n(A,A) \rightarrow \HH^{m+n}(A,A).
$$
\end{nameless}
\begin{lem}\label{lem:hhalgebra}
$(\HH^\bullet(A,A),+,\cup)$ is an associative and unital graded $k$-algebra. Its unit is given by the unit map $k \rightarrow A$.
\end{lem}
In view of the considerations above, we make the following definition.
\begin{defn}\label{def:hcohomring}
The graded $k$-algebra
$$
\HH^\bullet(A) := \HH^\bullet(A,A) = \bigoplus_{n \geq 0}\HH^n(A,A)
$$
is the \textit{Hochschild cohomology ring of $A$}.
\end{defn}
\begin{nameless}[Hochschild cohomology and extension groups] Chose a projective resolution $\mathbb P_A \rightarrow A \rightarrow 0$ of $A$ as an $A^\ev$-module. Then, by taking the $n$-th cohomology of $\Hom_{A^\ev}(\PP_A,M)$ for an $A^\ev$-module $M$, we obtain the group of $n$-extensions of $A$ by $M$ over $A^\ev$,
$$
\Ext^n_{A^\ev}(A,M) = H^n(\Hom_{A^\ev}(\PP_A,M)) = \frac{\Ker \Hom_{A^\ev}(p_{n+1},M)}{\Im \Hom_{A^\ev}(p_{n},M)}
$$
($p_\bullet$ denotes the differential on $\mathbb P_A$). This group does not depend on the chosen projective resolution in the sense that if one takes another projective resolution of $A$ over $A^\ev$, then the result will be isomorphic to the group above. The graded $k$-module $\Ext^\bullet_{A^\ev}(A,A)$ is a graded $k$-algebra by the \textit{Yoneda product} (see also Section \ref{sec:yoneda}). For every $A^\ev$-module $M$ there is a map 
$$
\chi_M: \HH^\bullet(A,M) \longrightarrow \Ext^\bullet_{A^\ev}(A,M)
$$
which is a homomorphism of graded $k$-algebras in case $M = A$. To define it, we recall a standard bimodule resolution of $A$. The \textit{bar resolution $\mathbb B_A = (B_\bullet, b_\bullet)$ of $A$} is given by $B_n = A^{\otimes_k (n+2)} = \widetilde{S}^n(A)$ for $n \geq 0$ and $B_{-1} = A$. Furthermore, let $b_n: B_{n} \rightarrow B_{n-1}$ for $n \geq 0$ be given as
$$
b_n (a_0 \otimes \cdots \otimes a_{n+1}) = \sum_{i=0}^n{(-1)^i a_0 \otimes \cdots \otimes a_i a_{i+1} \otimes
\cdots \otimes a_{n+1}},
$$
where $a_0, \dots, a_{n+1} \in A$. Provided by
$$
(a \otimes b)(a_0 \otimes \cdots \otimes a_{n+1}) = aa_0 \otimes \cdots \otimes a_{n+1}b, \quad a,b,a_0,\dots,a_{n+1}
\in A,
$$
each $B_n$, $n \geq -1$, is a left $A^\ev$-module, and the maps $b_n$, $n \geq 0$, are $A^\ev$-linear homomorphisms. It is not hard to see that
$$
\xymatrix@C=20pt{
\BB_A: & \cdots \ar[r]^-{b_2} & B_1 \ar[r]^-{b_1} & B_0 \ar[r]^-{b_0} & A \ar[r] & 0
}
$$
is an acyclic complex. $\BB_A \rightarrow A \rightarrow 0$ will be a projective resolution of $A$ as an $A^\ev$-module, if $A$ is projective over $k$. This follows by putting together the following facts:
\begin{enumerate}
\item $A \otimes_k A$ is a projective $A^\ev$-module.
\item If $P$ and $Q$ are projective $A^\ev$-modules, and if $A$ is projective over $k$, then $P$ and $Q$ are also $k$-projective, since $A^\ev$ is. Hence
$$
\Hom_{A^\ev}(P \otimes_k Q, -) \cong \Hom_k(Q, \Hom_{A^\ev}(P,-))
$$
is an exact functor and consequently, $P \otimes_k Q$ is a projective left $A^\ev$-module.
\item If $n \geq 0$ is even, then $B_n \cong (A^\ev)^{\otimes_k\left(\frac{n}{2} + 1\right)}$ as left $A^\ev$-modules.
\item If $n \geq 0$ is odd, we have $B_n \cong (A^\ev)^{\otimes_k\left(\frac{n+1}{2}\right)} \otimes_k A$ and
$$
\Hom_{A^\ev}(B_n, - ) \cong \Hom_k (A, \Hom_{A^\ev}((A^\ev)^{\otimes_k\left(\frac{n+1}{2}\right)},-).
$$
\end{enumerate}
When applying $\Hom_{A^\ev}(-,M)$ to $\BB_A$ we obtain the cocomplex $\Hom_{A^\ev}(\BB_A,M)$ whose differential is given by $\Hom_{A^\ev}(b_\bullet,M)$. By the adjointess of the functors $A^\ev \otimes_k -$ and $\Hom_{A^\ev}(A^\ev,-)$, we obtain the chain
\begin{align*}
\Hom_k(S^n(A),M) &\cong \Hom_k(S^n(A),\Hom_{A^\ev}(A^\ev,M))\\ &\cong \Hom_{A^\ev}(A^\ev \otimes_k S^n(A), M)\\ & \cong \Hom_{A^\ev}(\widetilde{S}^n(A), M)\\
& = \Hom_{A^\ev}(B_n, M)
\end{align*}
of isomorphisms. Hence we see that the $n$-th degree of $\Hom_{A^\ev}(\BB_A,M)$ coincides with $C^n(A,M)$ for every $n \geq 0$. The following result is well-known.
\begin{lem}\label{lem:HH=HH}
Let $M$ be an $A^\ev$-module. The adjunction isomorphism corresponding to the adjoint pair $(A^\ev \otimes_k -, \Hom_{A^\ev}(A^\ev, -))$ gives rise to an isomrophism $\mathbb C(A,M) \cong \Hom_{A^\ev}(\BB_A,M)$ of complexes.
\end{lem}
Under the above isomorphism the cup product translates to the following multiplication on $\Hom_{A^\ev}(\mathbb B_A,A)$: For integers $m, n \geq 0$ and $f \in \Hom_{A^\ev}(B_m,M)$, $g \in \Hom_{A^\ev}(B_n,M)$ the homomorphism $f \cup g \in \Hom_{A^\ev}(B_{m+n}, A)$ is given by
$$
(f \cup g)(a_0 \otimes \cdots a_{m+n+1}) = f(1_A \otimes a_0 \otimes \cdots \otimes a_{n}) \otimes g(a_{n+1}
\otimes \cdots \otimes a_{m + n +1} \otimes 1_A),
$$
where $a_0, \dots, a_{m+n+1} \in A$.

We are now ready to establish the morphism $\chi_M$. By Lemma \ref{lem:comparison} there is a morphism
$$
\xymatrix{
\mathbb B_A & \equiv \quad & \cdots \ar[r] & B_2 \ar[r] & B_1 \ar[r] & B_0 \ar[r] & A \ar[r] & 0\\
\mathbb P_A \ar[u]^-{\chi_{\mathbb P_A}} & \equiv \quad & \cdots \ar[r] & P_2 \ar[r] \ar[u] & P_1 \ar[r] \ar[u] & P_0 \ar[r] \ar[u] & A \ar[r] \ar@{=}[u] & 0
}
$$
of complexes over $A^\ev$. Hence (by applying the functor $\Hom_{A^\ev}(-,M)$ to the diagram) we get a morphism
$$
\chi_{\mathbb P_A}^\ast = \Hom_{A^\ev}(\chi_{\mathbb P_A}, M): \Hom_{A^\ev}(\mathbb B_A, M) \longrightarrow \Hom_{A^\ev}(\mathbb P_A, M).
$$
The desired map $\chi_M$ is now obtained by taking the cohomology of $\Hom_{A^\ev}(\chi_{\mathbb P_A}, M)$:
$$
\chi_M = H^\bullet(\chi_{\mathbb P_A}^\ast): \HH^\bullet(A,M) \longrightarrow \Ext^\bullet_{A^\ev}(A,M).
$$
Note that by the uniqueness statement in Lemma \ref{lem:comparison}, the map $\chi_M$ does not depend on the chosen chain map $\chi_{\mathbb P_A}$. Moreover, if $\mathbb Q_A \rightarrow A \rightarrow 0$ is another resolution of $A$ by projective $A^\ev$-modules, then there is a (in a certain sense unique) isomorphism $\gamma_n \colon H^n(\Hom_{A^\ev}(\mathbb P_A, M)) \rightarrow H^n(\Hom_{A^\ev}(\mathbb Q_A, M))$ for every integer $n \geq 0$ such that $ \gamma_n \circ H^n(\chi_{\mathbb P_A}^\ast) = H^n(\chi_{\mathbb Q_A}^\ast)$. Since $\mathbb B_A \rightarrow A \rightarrow 0$ is an $A^\ev$-projective resolution of $A$ if $A$ is $k$-projective, $\chi_M$ will be an isomorphism for every $A^\ev$-module $M$ in that case.
\end{nameless}
\section{Gerstenhaber algebras}\label{sec:galgebras}
A satisfying definition of a Gerstenhaber algebra over the commutative ring $k$ is, at least to the knowledge of the author, nowhere written up properly. Following M.\,Gerstenhaber and S.\,D.\,Schack \cite{GeSch86} we make the following attempt (see also \cite{Lei80} and \cite{QFS99} for several remarks on $\mathbb Z_2$-graded, i.e., \textit{super} Lie algebras, which can be easily transfered to the general graded case).

If $M = \bigoplus_{n \in \mathbb Z} M^n$ is a graded $k$-module, and if $i, j \in \Z$, we let $M^{i\bullet+j}$ be the graded $k$-module with
$$
(M^{i\bullet+j})^n = M^{in+j} \quad (\text{for $n \in \Z$}).
$$
\begin{defn}\label{def:galgebra}
Let $A = \bigoplus_{n \in \Z}{A^n}$ be a graded $k$-algebra. Further, let $[-,-]: A \times A \rightarrow A$ be a $k$-bilinear map of degree $-1$ (that is, $[a,b] \in A^{\abs{a}+\abs{b}-1}$ for all homogeneous $a,b \in A$).
\begin{enumerate}[\rm(1)]
\item The pair $(A,[-,-])$ is a \textit{Gerstenhaber algebra} (or \textit{G-algebra}) \textit{over $k$} if
\begin{enumerate}
\item[(G1)] $A$ is graded commutative, i.e., $ab = (-1)^{\abs{a}\abs{b}}ba$ for all homogeneous $a,b \in A$;
\item[(G2)] $[a,b] = -(-1)^{(\abs{a}-1)(\abs{b}-1)}[b,a]$ for all homogeneous $a,b \in A$;
\item[(G3)] $[a,a] = 0$ for all homogeneous $a \in A$ of odd degree;
\item[(G4)] $[[a,a],a] = 0$ for all homogeneous $a \in A$ of even degree;
\item[(G5)] the graded Jacobi identity holds:
$$
[a,[b,c]] = [[a,b],c] + (-1)^{(\abs{a}-1)(\abs{b}-1)}[b,[a,c]]
$$
for all homogeneous $a,b,c \in A$;
\item[(G6)] the graded Poisson identity holds:
$$
[a,bc] = [a,b]c + (-1)^{(\abs{a}-1)\abs{b}} b[a,c]
$$
for all homogeneous $a,b,c \in A$.
\end{enumerate}
\item Assume that $(A,[-,-])$ is a Gerstenhaber algebra over $k$. We call $(A, [-,-])$ a \textit{strict} Gerstenhaber algebra over $k$ if there is a map $sq: A^{2\bullet} \rightarrow A^{4\bullet-1}$ of degree $0$ such that
\begin{enumerate}
\item[(G7)] $sq(ra) = r^2 sq(a)$ for all $r \in k$ and all homogeneous $a \in A^{2\bullet}$;
\item[(G8)] $sq(a+b) = sq(a) + sq(b) + [a,b]$ for all homogeneous $a,b \in A^{2\bullet}$;
\item[(G9)] $[a,sq(b)] = [[a,b],b]$ for all homogeneous $a,b \in A^{2\bullet}$;
\item[(G10)] $sq(ab) = a^2sq(b) + sq(a)b^2 + a[a,b]b$ for all homogeneous $a,b \in A^{2\bullet}$.
\end{enumerate}
\end{enumerate}
\end{defn}
The map $[-,-]$ is called a \textit{Gerstenhaber bracket} for $A$, whereas $sq$ is a \textit{squaring map} for the Gerstenhaber algebra $(A, [-,-])$. Note that any
graded commutative $k$-algebra can be viewed as a (strict) Gerstenhaber algebra over $k$ with trivial bracket (and trivial squaring map).
\begin{rem}\label{rem:strict_galg}
Fix a Gerstenhaber algebra $(A, [-,-])$ over $k$.
\begin{enumerate}[\rm(1)]
\item The graded Jacobi identity measures how far $[-,-]$ is away from being associative, whereas the graded Poisson identity translates to $[a,-]$ being a graded derivation of $A$ of degree $\abs{a} - 1$ (for $a \in A$ homogeneous).

\item Assume that $2$ is a unit in $k$. Then the Gerstenhaber algebra $(A,[-,-])$ becomes a strict one via
$$
sq(a) = 2^{-1}[a,a], \quad (\text{for $a \in A^{2n}, \ n \geq 1$}).
$$
In fact, there is even no other choice due to the next statement.

\item Let $sq: A^{2\bullet} \rightarrow A^{4\bullet-1}$ be a map satisfying (G7) and (G8). For any homogeneous $a \in A$ of degree $\abs{a} \in 2 \mathbb N$ we get:
\begin{align*}
2sq(a) = sq(2a) - 2sq(a) = [a,a].
\end{align*}
Consequently, there is at most one map $sq$ such that $(A, [-,-], sq)$ is a strict Gerstenhaber algeba if $2$ is not a zero devisor in $A$.
\end{enumerate}
\end{rem}
\begin{nn} A familiy $(C^n)_{n \in \Z}$ of $k$-modules is called a \textit{pre-Lie system} if there exist $k$-bilinear maps
$\bullet_i: C^m \times C^n \rightarrow C^{m+n}$, where $m,n \geq 0$ and $0 \leq i \leq m$, such that
$$
(f \bullet_i g) \bullet_j h =
\begin{cases}
(f \bullet_j h) \bullet_{i + p} g, & \textmd{if $0 \leq j \leq i-1$}\\
f \bullet_i (g \bullet_{j-i} h), & \textmd{if $i \leq j \leq n+1$}
\end{cases}
$$
for all $f \in C^m$, $g \in C^n$ and $h \in C^p$. Put
$$
f \bullet g = \sum_{i=0}^m{(-1)^{in}f \bullet_i g} \quad (\text{for $f \in C^m, \ g \in C^n$}).
$$
Several considerations of M.\,Gerstenhaber on this topic (\cite[Sec.\,2, Thm.\,1]{Ge63} and \cite[Sec.\,6, Thm.\,2]{Ge63}) give the following result.
\end{nn}
\begin{thm}\label{thm:pre-Lie-2-Lie}
Let $A = \bigoplus_{n \in \Z}{C^n}$. Then the assignment
$$
[f,g] = f \bullet g - (-1)^{mn}g \bullet f \quad (\text{for $f \in C^m, \ g \in C^n$}),
$$
yields a graded Lie bracket on $A$.
\end{thm}
\begin{nn}
The (shifted) Hochschild complex of $A$ carries the structure of a pre-Lie system. For an $A$-$A$-bimodule $M$ (with central $k$-action), and $f \in C^m(A,M)$, $g \in C^n(A,A)$ define
\begin{align*}
(f {\bullet}_i g)(&a_1 \otimes \cdots a_{m+n-1})\\
&= f(a_1 \otimes \cdots \otimes a_i \otimes g(a_{i+1} \otimes \cdots \otimes a_{i + n}) \otimes a_{i+n+1} \otimes \cdots \otimes a_{m+n-1}),
\end{align*}
where $0 \leq i \leq m-1$. The $k$-bilinearity of the hereby obtained map ${\bullet}_i: C^m(A,M) \times C^n(A,A) \rightarrow C^{m+n-1}(A,M)$ is obvious. It is easy to see, that $C^{\bullet+1}(A,A)$ becomes a pre-Lie system. Hence $C^{\bullet+1}(A,A)$ is a graded Lie algebra with Lie bracket $\{-,-\} = \{-,-\}_A$ obtained from Theorem \ref{thm:pre-Lie-2-Lie}. For this reasons the following equations hold true for any $f \in C^m(A,A)$, $g \in C^n(A,A)$ and $h \in C^p(A,A)$:
\begin{equation}\label{eq:anticom}
\{f,g\} = f \bullet g - (-1)^{(m-1)(n-1)}g \bullet f;
\end{equation}
\begin{equation}\label{eq:gradedjacobi}
\begin{aligned}
(-1)^{(m-1)(p-1)}\{\{f,g\},h\} &+ (-1)^{(m-1)(n-1)}\{\{g,h\},f\}\\ &+ (-1)^{(n-1)(p-1)}\{\{h,f\},g\} = 0
\end{aligned}
\end{equation}
\end{nn}
\begin{lem}[{\cite{Ge63}}]\label{lem:fundformula}
Let $M$ be an $A$-$A$-bimodule and $f \in C^m(A,M), \ g \in C^n(A,A)$. We have
\begin{equation}\label{eq:fundformula}
\partial^{m+n-1}(f \bullet g) = (-1)^{n-1}\partial^m(f) \bullet g + f \bullet \partial^n(g) + (-1)^n [f,g]_\cup,
\end{equation}
where $[f,g]_\cup = f \cup g - (-1)^{mn}g \cup f$.
\end{lem}
If we take $M = A$ and $f \in \Ker(\partial^m)$, $g \in \Ker(\partial^n)$, an immediate consequence of the lemma above is
$$
f \cup g - (-1)^{mn}g \cup f = [f,g]_\cup = (-1)^n\partial^{m+n-1}(f \bullet g) \ \equiv \ 0 \quad (\text{$\mathrm{mod} \Im(\partial^{m+n-1})$}).
$$
So at the level of cohomology, the cup product is graded commutative, and therefore we get the following well known result.
\begin{cor}[{\cite{Ge63}}]\label{cor:gradedcomm}
The Hochschild cohomology ring $\HH^\bullet(A)$ of a $k$-algebra $A$ is graded commutative.
\end{cor}
We observed that the shifted Hochschild complex of an algebra $A$ is a graded Lie algebra, with Lie bracket $\{-,-\}_A$. The fundamental formula stated in Lemma \ref{lem:fundformula} ensures, that $\{-,-\}_A$ is a well-defined map on $\HH^{\bullet+1}(A)$, and thus, also $\HH^{\bullet+1}(A)$ is a graded Lie algebra over $k$. The following even stronger statement holds true.
\begin{thm}[{\cite{Ge63}}]\label{thm:hh_gerstenhaber}
For any $k$-algebra $A$, the triple $$(\HH^\bullet(A), \{-,-\}_A, sq_A)$$ is a strict Gerstenhaber algebra over $k$.
\end{thm}
The squaring map $sq_A: \HH^{2\bullet} \rightarrow \HH^{4\bullet-1}$ arises as follows. The fundamental formula (\ref{eq:fundformula}) shows, that if $f \in C^{2n}(A,A)$ is a cocycle (for some $n \geq 1$), then so is $f \bullet f$. Moreover, $f \bullet f$ will be a coboundary if $f$ was. Hence the assignment $f \mapsto f \bullet f$ establishes the (well-defined) map $sq_A$.
\begin{exa}
Let $k$ be a field of characteristic zero and let $R$ be a commutative ring. Let $\mathfrak{g}$ be a Lie algebra over $R$, with Lie bracket $[-,-]$. The exterior algebra $\Lambda^\bullet_R (\mathfrak{g})$ of $\mathfrak{g}$ over $R$ is a (strict) graded commutative $R$-algebra (see Section \ref{sec:exasHopf}). Thanks to the \textit{Schouten–Nijenhuis bracket}, it also carries the structure of a graded Lie algebra over $R$. More precisely, let $g_1 \wedge \cdots \wedge g_m$ and $h_1 \wedge \cdots \wedge h_n$ be homogeneous elements in $\Lambda^\bullet_R (\mathfrak{g})$. The element
\begin{align*}
[g_1 &\wedge \cdots \wedge g_m, h_1 \wedge \cdots \wedge h_n]\\ &= \sum_{s = 1}^m \sum_{t=1}^n {(-1)^{s+t} g_1 \wedge \cdots \wedge \widehat{g_s} \wedge \cdots \wedge g_m \wedge [g_s, h_t] \wedge h_1 \wedge \cdots \wedge \widehat{h_t} \wedge \cdots \wedge h_n}
\end{align*}
sits inside $\Lambda_R^{m+n-1}(\mathfrak{g})$ (where $\widehat{g_s}$ and $\widehat{h_t}$ stand for the omission of $g_s$ and $h_t$). In fact, it can be shown that $(\Lambda^\bullet_R (\mathfrak{g}), \wedge, [-,-])$ is a Gerstenhaber algebra over $R$.

Now let $A$ be a commutative algebra over $k$. The $k$-linear derivations $\Der_k(A)$ form, via the commutator, a Lie algebra over $k$, but in general not over $A$. However, $\Der_k(A)$ is an $A$-$A$-bimodule and we can still apply all constructions mentioned above. In particular, the Schouten–Nijenhuis bracket can be used to produce a Gerstenhaber algebra
$$
(\Lambda^\bullet_A \Der_k(A), \wedge, [-,-]).
$$
In fact, the bracket $[-,-]$ is the unique extension of $[D,a]=D(a)$ and $[D,E] = D\circ E - E\circ D$ (for $a \in A$ and $D,E \in \Der_k(A)$) such that $(\Lambda^\bullet_A \Der_k(A), \wedge, [-,-])$ is a Gerstenhaber algebra. We have a canonical homomorphism into the Hochschild complex: $$\Lambda^\bullet_A \Der_k(A) \longrightarrow C^\bullet(A,A).$$ The map is given by the assignment
$$
D_1 \wedge \cdots \wedge D_n \ \mapsto \ \frac{1}{n!}\sum_{\sigma \in \mathfrak{S}_n}{(-1)^{\mathrm{sgn}(\sigma)} D_{\sigma(1)} \cup \cdots \cup D_{\sigma(n)}},
$$
where $D_1, \dots, D_n \in \Der_k(A)$. It is well-known that this map can be used to compute the Gerstenhaber algebra $\HH^\bullet(A)$, if $A$ is smooth over $k$:

A $k$-algebra $\Gamma$ is called \textit{smooth} over $k$ if it is a perfect complex in $\mathbf D(\Mod(\Gamma^\ev))$ (i.e., it is quasi-isomorphic to a bounded complex of finitely generated projective $\Gamma^\ev$-modules). One may think of smooth algebras as having finite projective dimension over their enveloping algebra; see \cite[Thm.\,9.1.2]{Gi05} for a more ``classical'' definition of smoothness. In this context, it is also worth recalling that (à la J.\,Cuntz and D.\,Quillen \cite{CuQu95}) \textit{formally smooth} (or \textit{quasi-free}) algebras over $k$ are, by definition, precisely those algebras over $k$ which are of projective dimension at most $1$ over their enveloping algebra. The following theorem is classical. \cite[Thm.\,9.1.3]{Gi05} provides an elementary and very clear proof.
\begin{thm}[Hochschild-Kostant-Rosenberg, {\cite{HKR62}}]
Let $A$ be a commutative $k$-algebra. If $A$ is smooth over $k$ $($e.g., $A = k[x_1, \dots, x_n])$, then the map
$$
(\Lambda^\bullet_A \Der_k(A),0) \longrightarrow (C^\bullet(A,A), \partial^\bullet)
$$
is a quasi-isomorphism. It induces an isomorphism $\Lambda^\bullet_A \Der_k(A) \cong \HH^\bullet(A)$ of $($strict$)$ Gerstenhaber algebras.
\end{thm}
\end{exa}



\chapter{A bracket for monoidal categories}\label{ch:bracket}
Using our explicit description of Retakh's isomorphism for factorizing exact categories, we are going to define a map $[-,-]_\C: \Ext^m_\C(\mathbbm 1_\C,\mathbbm 1_\C) \times \Ext^n_\C(\mathbbm 1_\C,\mathbbm 1_\C) \rightarrow \Ext^{m+n-1}_\C(\mathbbm 1_\C,\mathbbm 1_\C)$ for any strong exact monoidal category $(\C, \otimes_\C, \mathbbm 1_\C)$. After stating some of its properties, we will show that it yields an honest generalization of the map S.\,Schwede introduced in \cite{Sch98}. Thereupon, we will employ this observation to deduce, that the Gerstenhaber bracket is an invariant under Morita equivalence.
\section{The Yoneda product}\label{sec:yoneda}
\begin{nn}
Let $\C$ be an exact $k$-categories. Further, let $X$ be an object in $\C$. The graded $k$-module $\Ext^{\bullet}_\C(X,X)$ may be endowed with the structure of a graded $k$-algebra. To this end, let $A,B,C \in \Ob \C$ be objects in $\C$. Given two admissible extensions $\xi \in \mathcal Ext^m_\C(B,C)$ and $\xi' \in \mathcal Ext^n_\C(A,B)$ we define their \textit{Yoneda product} $\xi \circ \xi' \in \mathcal Ext^{m + n}_\C(A,C)$ to be the following $(m+n)$-extension $\xi \circ \xi'$. If $m, n \geq 1$ we obtain $\xi \circ \xi'$ by simply splicing $\xi$ and $\xi'$ together:
$$
\xymatrix@C=15pt{
\xi & \equiv & 0 \ar[r] & C \ar[r] & E_{m-1} \ar[r] & \cdots \ar[r] & E_0 \ar[d]_-{f_n\circ e_0} \ar[r]^-{e_0} & B \ar[r] \ar@{=}[d]& 0 && \\
&& 0 & A \ar[l] & F_0 \ar[l] & \cdots \ar[l] & \ar[l] F_{n-1} & B \ar[l]_-{f_n} & 0 \ar[l] & \equiv & \xi' \ .
}
$$
It is obvious that the resulting sequence is indeed an admissible exact one. If $m = 0 = n$, $\xi$ and $\zeta$ are morphisms in $\C$ and we let their Yoneda product be their composition (in this sense, the Yoneda product extends the product on the $k$-algebra $\Ext^0_\C(X,X) = \End_\C(X)$ to the whole graded object $\Ext^\bullet_\C(X,X)$ being the justification for denoting it by $\circ$). If $m=0$ and $n \geq 1$, we let $\xi \circ \xi'$ be the lower sequence of the following pushout diagram:
$$
\xymatrix@C=15pt{
\xi' & \equiv & 0 \ar[r] & B \ar[r] \ar[d]_{\xi} & F_{n-1} \ar[r] \ar[d] & F_{n-2} \ar[r] \ar@{=}[d] & \cdots \ar[r] & F_0 \ar[r] \ar@{=}[d] & A \ar@{=}[d] \ar[r] & 0 \ \ \\
&& 0 \ar[r] & C \ar[r]  & P \ar[r] & F_{n-2} \ar[r]  & \cdots \ar[r] & F_0 \ar[r] & A \ar[r] & 0 \ .
}
$$
Similarly, if $m \geq 1$ and $n = 0$, let the lower sequence of the pullback diagram
$$
\xymatrix@C=15pt{
\xi & \equiv & 0 \ar[r] & C \ar[r] \ar@{=}[d] & E_{m-1} \ar[r] \ar@{=}[d] & \cdots \ar[r] & E_1 \ar[r] \ar@{=}[d] & E_0 \ar[r] & B \ar[r] & 0\\
& & 0 \ar[r] & C \ar[r]  & E_{m-1} \ar[r] & \cdots \ar[r] & E_1 \ar[r]  & Q \ar[r] \ar[u] & A  \ar[u]_{\xi'} \ar[r] & 0
}
$$
be $\xi \circ \xi'$. Clearly the Yoneda product respects the equivalence relation defining $\pi_0 \mathcal Ext^n_\C(-,-)$. Therefore it induces a well-defined $k$-bilinear map
$$
\circ : \Ext^m_\C(B,C) \times \Ext^n_\C(A,B) \longrightarrow \Ext^{m+n}_\C(A,C),
$$
and, in particular, maps
\begin{align*}
\nabla_{A,C} &: \Ext^m_\C(C,C) \otimes_k \Ext^n_\C(A,C) \longrightarrow \Ext^{m+n}_\C(A,C) && (\text{for $B = C$}),\\
\tilde{\nabla}_{A,C} &: \Ext^m_\C(A,C) \otimes_k \Ext^n_\C(A,A) \longrightarrow \Ext^{m+n}_\C(A,C) && (\text{for $B = A$}),\\
\intertext{and}
\nabla_{A,A} &: \Ext^m_\C(A,A) \otimes_k \Ext^n_\C(A,A) \longrightarrow \Ext^{m+n}_\C(A,A) && (\text{for $A=B=C$}).
\end{align*}
Notice that the Yoneda product recovers the $k$-action on $\Ext^n_\C(A,B)$ defined in Section \ref{sec:lowhomotopy}:
$$
k \times \Ext^n_\C(X,X) \longrightarrow \Ext^n_\C(X,X), \ (a, x) \mapsto (a \id_X) \circ x = a \vdash \xi.
$$
The following result is classical, and due to N.\,Yoneda (cf. \cite{Yo54}).
\end{nn}
\begin{lem}\label{lem:ext-algebra}
The triple $(\Ext^\bullet_\C(X,X), +, \circ)$ is a \emph{(}positively\emph{)} graded $k$-algebra with unit $\id_A$. The graded $k$-module $\Ext^\bullet_\C(Y,X)$ is a graded $\Ext^\bullet_\C(X,X)$-$\Ext^\bullet_\C(Y,Y)$-bimodule.
\end{lem}
\begin{lem}\label{lem:exactfuncyoneda}
Let $\scrX: \C \rightarrow \D$ be a $k$-linear and exact functor, and let $A \in \Ob\C$. Then, for every object $B \in \Ob \C$, the map
$$
\scrX^\sharp_\bullet: \Ext^\bullet_\C(B,A) \longrightarrow \Ext^\bullet_\D(\scrX B, \scrX A)
$$
is a homomorphism of graded $\Ext^\bullet_\C(A,A)$-modules. In particular, it is a homomorphism of graded $k$-algebras, in case $A = B$.
\end{lem}
\begin{proof}
By definition, it is evident, that $\scrX^\sharp_{\geq 1}$ commutes with the Yoneda product. Since $\scrX$ is exact, it preserves pullbacks along admissible epimorphisms and pushouts along admissible monomorphisms (cf. Lemma \ref{lem:exact_func_push_pull}), and therefore
$$
\scrX^\sharp_{n+1}(f \circ x) = \scrX^\sharp_0(f)\circ\scrX^\sharp_n(x), \quad \scrX^\sharp_{n+1}(x \circ f) = \scrX^\sharp_n(x)\circ\scrX^\sharp_0(f)
$$
(for $n \geq 0$, $f \in \End_\C(X)$, $x \in \Ext^n_\C(X,X)$). In particular, $\scrX^\sharp_\bullet$ is $k$-linear. Since $\scrX(\id_A) = \id_{\scrX A}$, the map $\scrX^\sharp_\bullet$ is also unital.
\end{proof}
\begin{rem}\label{rem:admisexactyoneda}
Note that Definition \ref{def:adexseq} may be reformulated in terms of the Yoneda product: A sequence
$$
\xymatrix@C=15pt{
\xi & \equiv & 0 \ar[r] & Y \ar[r] & E_{n-1} \ar[r] & \cdots \ar[r] & E_0 \ar[r] & X \ar[r] & 0 
}
$$
of morphisms in $\C$ is an admissible $n$-extension if, and only if, there exist objects
$$
X = K_0, K_1, \dots, K_{n-1}, K_n = Y \quad \text{in $\C$}
$$
and admissible short exact sequences
$$
\xi_i \in \Ob \mathcal Ext^1_\C(K_{i-1}, K_i) \quad (\text{for $i = 1, \dots, n$})
$$
such that $\xi_1 \circ \cdots \circ \xi_n$ coincides with $\xi$.
\end{rem}
\section{The bracket and its properties}\label{sec:prop_bracket}
\begin{nn}
Let $(\C, \otimes_\C, \mathbbm 1_\C)$ be a strong exact monoidal $k$-category. For every integer $n \geq 1$, let $\sigma_n$ be $\sigma_n(\mathbbm 1_\C, \mathbbm 1_\C)$ and let $\xi_n$ be an admissible $n$-extension of $\mathbbm 1_\C$ by $\mathbbm 1_\C$ in $\C$. Throughout this section,
$$
u_\C=u_\C^{\id_{\mathbbm 1_\C}, +}: \Ext^\bullet_\C(\mathbbm 1_\C, \mathbbm 1_\C) \longrightarrow \pi_1(\mathcal Ext^{\bullet}_\C(\mathbbm 1_\C, \mathbbm 1_\C),\sigma_\bullet)
$$
and
$$
v_\C(\xi_\bullet) = v_\C(\sigma_\bullet, \xi_\bullet): \pi_1(\mathcal Ext^{\bullet}_\C(\mathbbm 1_\C, \mathbbm 1_\C), \sigma_\bullet) \longrightarrow \pi_1(\mathcal Ext^{\bullet}_\C(\mathbbm 1_\C, \mathbbm 1_\C), \xi_\bullet)
$$
will denote the isomorphisms of graded abelian groups constructed in Section \ref{sec:retakh}. Despite the ambiguity, we are going to suppress the target base point, and write $v_\C$ instead of $v_\C(\xi_n)$ in order to enhance legibility. The reader is hereby forewarned that therefore base points might (and will) be changed in calculations without being explicitly mentioned, and is hence advised to carefully keep track of each modification. We will abbreviate $\otimes_\C$, $\boxtimes_\C$, $\mathbbm 1_\C$ by $\otimes$, $\boxtimes$, $\mathbbm 1$ respectively (see Section \ref{sec:extcats_monoidal} for the definition of $\boxtimes_\C$).
\end{nn}
\begin{nn}
The strong exact monoidal category $(\C, \otimes, \mathbbm 1)$ naturally gives rise to morphisms of admissible extensions linking $\xi \boxtimes \zeta$ with $\xi \circ \zeta$ and $(-1)^{mn}\zeta \circ \xi$ for every pair of admissible extensions $\xi$, $\zeta$ of $\mathbbm 1$ by $\mathbbm 1$ in $\C$. To be more precise, let us fix positive integers $m, n \geq 0$, as well as an $m$-extension $\xi$ and an $n$-extension $\zeta$ of $\mathbbm 1$ by $\mathbbm 1$. The morphisms
$$
(\xi \boxtimes \zeta)_i = \bigoplus_{p + q = i} E_p \otimes F_q
\xymatrix@C=25pt{
\ar[r]^-{\text{can}} & E_0 \otimes F_i \ar[r]^{e_0 \otimes F_i} \ar[r] & \mathbbm 1 \otimes F_i \ar[r]^-{\lambda_{F_i}} & F_i
}
$$
(for $0 \leq i < n$) and
$$
(\xi \boxtimes \zeta)_j = \bigoplus_{p + q = j} E_p \otimes F_q
\xymatrix@C=25pt{
\ar[r]^-{\text{can}} & E_{j-n} \otimes F_n \ar[r]^-{\varrho_{E_{j-n}}} & E_{j-n}
}
$$
(for $n \leq j \leq m+n$) define the components of a morphism 
$$
L = L({\xi,\zeta}): \xi \boxtimes \zeta \longrightarrow \xi \circ \zeta
$$
in $\mathcal Ext^{m+n}_\C(\mathbbm 1, \mathbbm 1)$. Analogously, the morphisms
$$
(\xi \boxtimes \zeta)_i = \bigoplus_{p + q = i} E_p \otimes F_q
\xymatrix@C=40pt{
\ar[r]^-{\text{can}} & E_i \otimes F_0 \ar[r]^{E_i \otimes f_0} & E_i \otimes \mathbbm 1 \ar[r]^-{(-1)^{mn}\varrho_{E_i}} & E_i
}
$$
(for $0 \leq i < m$) and 
$$
(\xi \boxtimes \zeta)_j = \bigoplus_{p + q = j} E_p \otimes F_q
\xymatrix@C=40pt{
\ar[r]^-{\text{can}} & E_m \otimes F_{j-m} \ar[r]^-{(-1)^{\eta_j}\lambda_{F_{j-m}}} & F_{j-m}
}
$$
(for $m \leq j \leq m+n$ and $\eta_j = m(m+n-j)$) give rise to a morphism 
$$
R = R({\xi, \zeta}) : \xi \boxtimes \zeta \longrightarrow (-1)^{mn} \zeta \circ \xi
$$
in $\mathcal Ext^{m+n}_\C(\mathbbm 1, \mathbbm 1)$.
Hence we obtain a roof
\begin{equation}\label{eq:roof}
\begin{aligned}
\xymatrix@!C=25pt{
& \xi \boxtimes \zeta \ar[dl]_{L}  \ar[dr]^{R}& \\
\xi \circ \zeta && (-1)^{mn} \zeta \circ \xi
}
\end{aligned}
\end{equation}
in $\mathcal Ext^{m+n}_\C(\mathbbm 1,\mathbbm 1)$ (i.e., a path of length two), showing that
$$
(-1)^{mn} \zeta \circ \xi \equiv \xi \circ \zeta \quad \text{(mod $\sim$)}.
$$
This should be interpreted as follows.
\end{nn}
\begin{lem}\label{lem:gradcomm}
The graded $k$-algebra $(\Ext^\bullet_\C(\mathbbm 1_\C, \mathbbm 1_\C), +, \circ)$ is graded commutative, that is, $x \circ y = (-1)^{\abs{x} \abs{y}}y \circ x$ for all homogeneous elements $x,y \in \Ext^\bullet_\C(\mathbbm 1_\C, \mathbbm 1_\C)$.
\end{lem}
\begin{nn}
The roof (\ref{eq:roof}) gives rise to a loop $\Omega_\C(\xi,\zeta)$ in $\mathcal Ext^{m+n}_\C(\mathbbm 1, \mathbbm 1)$ (based at $\xi \circ \zeta$):
\begin{equation}\label{eq:loop}
\begin{aligned}
\xymatrix@!C=25pt{
& \xi \boxtimes \zeta \ar[dl]_-{L} \ar[dr]^-{R} & \\
\xi \circ \zeta &	& (-1)^{mn} \zeta \circ \xi &\\
& (-1)^{mn} \zeta \boxtimes \xi \ar[ul]^-{R} \ar[ur]_-{L} & .
}
\end{aligned}
\end{equation}
By taking the maps $u_\C$ and $v_\C$ into account, we get a map
\begin{align*}
[-,-]_\C: \Ext^{m}_\C(\mathbbm 1,\mathbbm 1) & \times \Ext^{n}_\C(\mathbbm 1,\mathbbm 1) \longrightarrow \Ext^{m+n-1}_\C(\mathbbm 1,\mathbbm 1),\\ [\xi, \zeta]_\C &= (u_\C^{-1} \circ v_\C^{-1})(\Omega_\C(\xi, \zeta)).
\end{align*}
If $n$ is even, the northern hemisphere of the loop (\ref{eq:loop}) (that is, the roof (\ref{eq:roof})) is a loop of length two based at $\xi \circ \zeta$; we will denote it by $\square_\C(\xi)$. If $n$ is even, we get $\square_\C(\xi) + \square_\C(\xi) =  \Omega_\C(\xi, \xi)$ in $\pi_1(\mathcal Ext^{2n}_\C(\mathbbm 1,\mathbbm 1), \xi \circ \xi)$, and if, in addition, $2$ is invertible in $k$, $\square_\C(\xi) =  2^{-1}\Omega_\C(\xi, \xi)$. Thus, if $n$ is even, we obtain a map
$$
sq_\C: \Ext^{n}_\C(\mathbbm 1,\mathbbm 1) \longrightarrow \Ext^{2n-1}_\C(\mathbbm 1,\mathbbm 1), \ [\xi] \mapsto (u_\C^{-1} \circ v_\C^{-1})([\square_\C(\xi)]).
$$
satisfying $sq_\C(\xi) + sq_\C(\xi) = [\xi,\xi]_\C$, and $sq_\C(\xi) = 2^{-1} [\xi,\xi]_\C$ in case $2$ is invertible in $k$. The aim of the upcoming considerations is to relate the structure of
$$
(\Ext^\bullet_\C(\mathbbm 1,\mathbbm 1), [-,-]_\C, sq_\C)
$$
to the structure of the underlying monoidal category, and to investigate its behavior with respect to monoidal functors.
\end{nn}
\begin{nn}
If the strong exact monoidal category $(\C, \otimes, \mathbbm 1)$ is lax braided, there are (by definition) functorial morphisms $\gamma_{X,Y}$ in $\C$, relating $X \otimes Y$ and $Y \otimes X$ for every pair of objects in $X, Y$ in $\C$. Following the idea presented in \cite{Ta04}, we will show that a (lax) braiding $\gamma$ will give rise to a morphism $\Gamma(\xi, \zeta): \xi \boxtimes \zeta \rightarrow (-1)^{\abs{\xi}\abs{\zeta}}\zeta \boxtimes \xi$ for every pair $\xi$, $\zeta$ of admissible extensions. The conclusion will be, that $[-,-]_\C$ is constantly zero, if $(\C, \otimes, \mathbbm 1)$ is lax braided.
\end{nn}
\begin{lem}\label{lem:laxbraiding}
Assume that the strong exact monoidal category $(\C, \otimes, \mathbbm 1)$ is lax braided with braiding $\gamma$. Then for every pair $(\xi, \zeta) \in \mathcal Ext^{m}_\C(\mathbbm 1,\mathbbm 1) \times \mathcal Ext^{n}_\C(\mathbbm 1,\mathbbm 1)$ the natural transformation $\gamma$ induces a morphism $\Gamma({\xi, \zeta}) :  \xi \boxtimes \zeta \rightarrow (-1)^{mn} \zeta \boxtimes \xi$ of
admissible $(m+n)$-extensions with
$$
R({\zeta, \xi}) \circ \Gamma({\xi, \zeta}) = L({\xi, \zeta})\quad \text{and} \quad \Gamma({\xi, \zeta}) \circ L({\zeta, \xi}) = R({\xi, \zeta}).
$$
Moreover, in case $(\C, \otimes, \mathbbm 1, \gamma)$ is symmetric, the above morphisms $\Gamma({\xi,\zeta})$ may be chosen to be isomorphisms fulfilling $\Gamma({(-1)^{mn} \zeta, \xi}) \circ \Gamma({\xi, \zeta}) = \id_{\xi \boxtimes \zeta}$.
\end{lem}

\begin{proof}
For every $0 \leq i \leq m + n -1$ let the $i$-th component of $\Gamma({\xi, \zeta})$ be given as
\begin{align*}
\Gamma_i & = \Gamma_i(\xi, \zeta) = (-1)^{mn} \sum_{r=0}^i{(-1)^{r(i-r)} \gamma_{E_r, F_{i-r}}}\\ 
& = (-1)^{mn} \left[
\begin{matrix}
\gamma_{E_0, F_i} & 0 & \cdots & 0 & 0\\
0 & (-1)^{i-1}\gamma_{E_1, F_{i-1}} & \cdots & 0 & 0\\
\vdots & \vdots & \ddots & \vdots & \vdots \\
0 & 0 & \cdots  & (-1)^{i-1}\gamma_{E_{i-1}, F_1} & 0\\
0 & 0 & \cdots & 0 & \gamma_{E_i, F_0}
\end{matrix}\right] .
\end{align*}
We check that all diagrams which (in order to get a morphism $\Gamma({\xi, \zeta}): \xi \otimes \zeta \rightarrow
(-1)^{mn} \zeta \otimes \xi$) should commute, really do so. First of all, we take care of the boundary maps. But the commutativity of the diagrams
$$
\xymatrix@C=60pt{
E_0 \otimes F_0 \ar[r]^-{e_0 \otimes f_0} \ar[d]_{(-1)^{mn}\gamma_{E_0,F_0}} & \mathbbm 1 \otimes \mathbbm 1 \ar[r]^-{\lambda_\mathbbm 1} \ar[d]^{\gamma_{\mathbbm 1,\mathbbm 1}} & \mathbbm 1 \ \ \ar@{=}@<-3pt>[d]\\
F_0 \otimes E_0 \ar[r]^-{(-1)^{mn} f_0 \otimes e_0}	   & \mathbbm 1 \otimes \mathbbm 1 \ar[r]^-{\lambda_\mathbbm 1}	    & \mathbbm 1 \ ,
}
$$
$$
\xymatrix@C=60pt@R=40pt{
\mathbbm 1 \ar[r]^{\lambda_\mathbbm 1} \ar@{=}[d] & \mathbbm 1 \otimes \mathbbm 1 \ar[r]^-{
\left[\begin{smallmatrix}
(-1)^m \mathbbm 1 \otimes f_n\\
e_m \otimes \mathbbm 1
\end{smallmatrix}\right] 
}
\ar[d]_{\gamma_{\mathbbm 1,\mathbbm 1}} & (\mathbbm 1 \otimes F_{n-1}) \oplus (E_{m-1} \otimes \mathbbm 1) \ar[d]^-{
\left[\begin{smallmatrix}
0 & (-1)^{n}\gamma_{E_{m-1},\mathbbm 1}\\
(-1)^{m}\gamma_{\mathbbm 1, F_{n-1}} & 0
\end{smallmatrix}\right] 
}\\
\mathbbm 1 \ar[r]^{\lambda_\mathbbm 1} & \mathbbm 1 \otimes \mathbbm 1 \ar[r]^-{
\left[\begin{smallmatrix}
(-1)^n \mathbbm 1 \otimes e_m\\
f_n \otimes \mathbbm 1
\end{smallmatrix}\right] 
}  & (\mathbbm 1 \otimes E_{m-1}) \oplus (F_{n-1} \otimes \mathbbm 1)
}
$$
is automatic due to the axioms (remember, $\gamma$ is a natural transormation $- \otimes - \rightarrow (- \otimes -) \circ
\tau$ and fulfills $\lambda_{\mathbbm 1} = \varrho_{\mathbbm 1} = \lambda_{\mathbbm 1} \circ \gamma_{\mathbbm 1, \mathbbm 1}$).
Moreover, by the naturality of $\gamma$ and careful sign-reading, the diagrams
$$
\xymatrix@C=60pt@R=40pt{
E_{r} \otimes F_{s} \ar[r]^-{\smatrix{(-1)^{r} E_r \otimes f_s \\ 
e_r \otimes F_s}} \ar[d]_{(-1)^{mn + rs} \gamma_{E_r, E_s}} & (E_{r} \otimes F_{s-1}) \oplus (E_{r-1} \otimes F_{s}) \ar[d]^{
(-1)^{mn + rs}\smatrix{0 & (-1)^{s}\gamma_{E_{r-1}, F_s}\\ (-1)^{r} \gamma_{E_r, F_{s-1}} & 0}
}
\\
F_s \otimes E_r \ar[r]^-{\smatrix{(-1)^{s} E_s \otimes e_r \\ f_s \otimes E_r}} & (F_{s} \otimes E_{r-1}) \oplus (F_{s-1} \otimes E_{r})
}
$$
also commute for all $1 \leq i \leq m+n-1$ and all $(r,s) \in \{0,\dots,m\} \times \{0,\dots,n\}$ with $r+s = i$. We still have to
show the two claimed properties of the family $\Gamma({\xi, \zeta}), \ \xi \in \mathcal Ext^m_\C(\mathbbm 1,\mathbbm 1), \
\zeta \in \mathcal Ext^n_\C(\mathbbm 1,\mathbbm 1)$.

Let $0 \leq i \leq m+n$ and put $\nu := m(m+n-i)$. The commutativity of the diagrams
$$
\xymatrix@!C=65pt@R=30pt{
(\xi \boxtimes \zeta)_i \ar@{->>}[d]_{\text{can}} \ar@{=}[r] & (\xi \boxtimes \zeta)_i \ar@{->>}[d]^{\text{can}} \ar@{=}[r] & (\xi \boxtimes \zeta)_i \ar[d]^-{\Gamma_i} &\\
E_i \otimes F_0 \ar@{=}[r] \ar@{=}[d]_{(-1)^{mn}} & E_i \otimes F_0 \ar[d]^-{(-1)^{mn}\gamma_{E_i, F_0}} & ((-1)^{mn}\zeta \boxtimes \xi)_i \ar@{->>}[d]^-{\text{can}} & \\
E_i \otimes F_0 \ar[r]^{\gamma_{E_i, F_0}} \ar[d]_{E_i \otimes f_0} & F_0 \otimes E_i \ar@{=}[r] \ar[d]^{f_0 \otimes E_i} & F_0 \otimes E_i  \ar[d]^{f_0 \otimes E_i} & (\text{for $0 \leq i < m$})\\
E_i \otimes \mathbbm 1 \ar[r]^{\gamma_{E_i,\mathbbm 1}} \ar[d]_{\varrho_{E_i}} & \mathbbm 1 \otimes E_i \ar@{=}[r] \ar[d]^{\lambda_{E_i}} & \mathbbm 1 \otimes E_i \ar[d]^{\lambda_{E_i}} &\\
E_i \ar@{=}[r] & E_i \ar@{=}[r] & E_i &
}
$$
and (notice that $E_m = \mathbbm 1$)
$$
\xymatrix@!C=65pt@R=30pt{
(\xi \boxtimes \zeta)_i \ar@{->>}[d]_{\text{can}} \ar@{=}[r] & (\xi \boxtimes \zeta)_i \ar@{->>}[d]^{\text{can}} \ar@{=}[r] & (\xi \boxtimes \zeta)_i \ar[d]^-{\Gamma_i} & \\
E_m \otimes F_{i-m} \ar@{=}[r] \ar@{=}[d]_{(-1)^{\nu}} & E_m \otimes F_{i-m} \ar[d]^{(-1)^\nu \gamma_{E_m,F_{i-m}}} & ((-1)^{mn}\zeta \boxtimes \xi)_i \ar@{->>}[d]^-{\text{can}} & \quad (\text{for $m \leq i \leq m + n$}) \\
E_m \otimes F_{i-m} \ar[r]^-{\gamma_{E_m, F_{i-m}}} \ar[d]_{\lambda_{F_{i-m}}} & F_{i-m} \otimes E_m \ar@{=}[r] \ar[d]^{\varrho_{F_{i-m}}} & F_{i-m} \otimes E_m \ar[d]^{\varrho_{F_{i-m}}}  &\\
F_{i-m} \ar@{=}[r] & F_{i-m} \ar@{=}[r] & F_{i-m} &
}
$$
translate to the desired equality: $R_i({\zeta, \xi}) \circ \Gamma_i = L_i({\xi, \zeta})$. The second equation may be deduced similarly. 

If now $(\C, \otimes, \mathbbm 1, \gamma)$ is even symmetric, the axiomatics tell us that, for any two objects $X, Y \in \Ob\C$, the composition $\gamma_{X,Y} \circ \gamma_{X,Y}$ is supposed to be the identity morphism of $X \otimes Y$. Hence the final claim is a direct consequence.
\end{proof}
\begin{thm}\label{thm:trivial_bracket}
Assume that the strong exact monoidal category $(\C, \otimes, \mathbbm 1)$ is lax braided. Then $\Omega_\C(\xi, \zeta)$ is homotopically equivalent to the trivial loop based at $\xi \circ \zeta$ for all
pairs $(\xi, \zeta) \in \mathcal Ext^{m}_\C(\mathbbm 1,\mathbbm 1) \times \mathcal Ext^{n}_\C(\mathbbm 1,\mathbbm 1)$. In particular, the map $[-,-]_\C$ is constantly zero.
\end{thm}
\begin{proof}
The previous Lemma translates into the commutativity of the diagram
$$
\xymatrix@!C=25pt{
& \xi \boxtimes \zeta \ar[dl]_-{L} \ar[dr]^-{R} \ar[dd]_{\Gamma(\zeta, \xi)} & \\
\xi \circ \zeta &	& (-1)^{mn} \zeta \circ \xi\\
& (-1)^{mn} \zeta \boxtimes \xi \ar[ul]^-{R} \ar[ur]_-{L} & .
}
$$
We obtain the following chain of elementary homotopic loops based at $\xi \circ \zeta$:
$$
\xymatrix@!C=25pt{
& \xi \boxtimes \zeta \ar[dl]_-{L} \ar[dr]^-{R} & & & & \xi \boxtimes \zeta \ar[dl]_-{R \circ \Gamma} \ar[dr]^-{L \circ \Gamma} & \\
\xi \circ \zeta &	& (-1)^{mn} \zeta \circ \xi & \quad \sim & \xi \circ \zeta &	& (-1)^{mn} \zeta \circ \xi  \\
& (-1)^{mn} \zeta \boxtimes \xi \ar[ul]^-{R} \ar[ur]_-{L} & & & & (-1)^{mn} \zeta \boxtimes \xi  \ar[ul]^-{R} \ar[ur]_-{L} &
}
$$
\begin{align*}
& \xymatrix@C=25pt{
\sim & \xi \circ \zeta & (-1)^{mn} \zeta \boxtimes \xi \ar@<-1ex>[l]_-{L} \ar@<1ex>[l]^-{L} \ar@<-1ex>[r]_-{R} \ar@<1ex>[r]^-{R} & (-1)^{mn} \zeta \circ \xi
}\\
& \xymatrix@!C=25pt{
\sim & \xi \circ \zeta  \ar[r]^-{\id_{\xi \circ \zeta}} & \xi \circ \zeta
} \quad ,
\end{align*}
that is, $\Omega_\C(\xi, \zeta)$ is homotopically equivalent to the trivial loop. Hence 
$$[-,-]_\C = (u_\C^{-1} \circ v_\C^{-1})(\Omega_\C(-,-))$$
vanishes.
\end{proof}
\begin{rem}
Let $(\C, \otimes, \mathbbm 1)$ be a strong exact monoidal category that possesses a lax braiding $\gamma$. Observe that the proof of Theorem \ref{thm:trivial_bracket} heavily relies on the somehow symmetric shape of $\Omega_\C(\xi,\zeta)$. Hence Lemma \ref{lem:laxbraiding} cannot be used to deduce any triviality statements relating the squaring map $sq_\C$ in general. However, it tells us that, for $\xi$ an admissible self-extension of $\mathbbm 1$ of even length, $sq_\C(\xi)$ coincides with the equivalence class of the loop
$$
\xymatrix@C=30pt{
\xi \circ \xi  \ar[r]^-{\Gamma(\xi, \xi)} & \xi \circ \xi \ .
}
$$
\end{rem}
\begin{nn}
Let us close this section with a comparison result. It will enable us to relate the bracket and the square operations coming from two exact monoidal categories provided they are connected via some exact and almost (co)strong monodial functor. The following lemma is essential for its proof. Let $(\D, \otimes_\D, \mathbbm 1_\D)$ be another strong exact monoidal category.
\end{nn}
\begin{lem}\label{lem:morloopcomp}
Let $\xi$ be a $m$-extension in $\mathcal Ext^{m}_\C(\mathbbm 1_\C, \mathbbm 1_\C)$ and let $\zeta$ be a $n$-extension in $\mathcal Ext^{n}_\C(\mathbbm 1_\C, \mathbbm 1_\C)$. Let $(\scrL, \phi, \phi_0): \C \rightarrow \D$ be an exact and almost strong monoidal functor, and let $\vartheta: \scrL_m \xi \boxtimes_\D \scrL_n \zeta \rightarrow \scrL_{m+n}(\xi \boxtimes_\C \zeta)$ be the morphism in $\mathcal Ext^{m+n}_\D(\mathbbm 1, \mathbbm 1)$ defined within the proof of Proposition \emph{\ref{prop:laxfunc_comm}}. Then the following diagrams commute.
\begin{align*}
\Xi \quad \quad &\equiv \quad \quad \quad \quad \quad \
\begin{aligned}
\xymatrix{
\scrL_m \xi \boxtimes_\D \scrL_n \zeta \ar[d]_-{L} \ar[r]^-\vartheta & \scrL_{m+n}(\xi \boxtimes_\C \zeta) \ar[d]^-{\scrL_{m+n}(L)}\\
(\scrL_m \xi) \circ (\scrL_n \zeta) \ar[r]^-{=} & \scrL_{m+n} (\xi \circ \zeta)
}
\end{aligned}\\&\\
\Xi' \quad \quad &\equiv \quad \quad
\begin{aligned}
\xymatrix{
\scrL_m \xi \boxtimes_\D \scrL_n \zeta \ar[d]_-{R} \ar[r]^-\vartheta & \scrL_{m+n}(\xi \boxtimes_\C \zeta) \ar[d]^-{\scrL_{m+n}(R)}\\
(-1)^{mn}(\scrL_m \zeta) \circ (\scrL_n \xi) \ar[r]^-{=} & (-1)^{mn}\scrL_{m+n} (\zeta \circ \xi)
}
\end{aligned}
\end{align*}
\end{lem}
\begin{proof}
Let us show that the diagram $\Xi$ commutes. Recall that for $0 \leq i \leq m+n-1$, $\vartheta_i$ is given by
$$
\vartheta_i = \left[
\begin{matrix}
\phi_{E_0, F_i} & 0 & \cdots & 0 & 0\\
0 & \phi_{E_1, F_{i-1}} & \cdots & 0 & 0\\
\vdots & \vdots & \ddots & \vdots & \vdots \\
0 & 0 & \cdots  & \phi_{E_{i-1}, F_1} & 0\\
0 & 0 & \cdots & 0 & \phi_{E_i, F_0}
\end{matrix}\right]
$$
Let $0 \leq i < n$. Since the diagram
$$
\xymatrix@C=30pt{
{\bigoplus_{r=0}^i \scrL E_r \otimes_\D \scrL F_{i-r}} \ar[d]_-{\mathrm{can}} \ar[r]^-{\vartheta_i} & {\bigoplus_{r=0}^i \scrL(E_r \otimes_\C F_{i-r})} \ar[d]^-{\mathrm{can}}\\
\scrL E_0 \otimes_\D \scrL F_i \ar[d]_-{(\phi_0^{-1} \otimes_\D \scrL F_i) \circ (\scrL(e_0) \otimes_\D \scrL F_i)} \ar[r]^-{\phi_{E_0, F_i}} & \scrL (E_0 \otimes_\C F_i) \ar[d]^-{\scrL(e_0 \otimes_\C F_i)} \\
\mathbbm 1_\D \otimes_\D \scrL F_i \ar[d]_-{\lambda_{\scrL F_i}} \ar[r]^-{\phi_{E_0, F_i} \circ (\phi_0 \otimes_\D \scrL F_i)} & \scrL (\mathbbm 1_\C \otimes_\C F_i) \ar[d]^-{\scrL(\lambda_{F_i})}\\
\scrL F_i \ar@{=}[r] & \scrL F_i
}
$$
commutes, so does $\Xi_i$. Further, the commutativity of
$$
\xymatrix@C=30pt{
{\bigoplus_{r=0}^i \scrL E_r \otimes_\D \scrL F_{i - r}} \ar[d]_-{\mathrm{can}} \ar[r]^-{\vartheta_i} & {\bigoplus_{r=0}^i \scrL(E_r \otimes_\C F_{i-r})} \ar[d]^-{\mathrm{can}} \\
\scrL E_{i-n} \otimes_\D \scrL \mathbbm 1_\C \ar[r]^-{\phi_{E_{i-n}, \mathbbm 1_\C}} \ar[d]_-{\varrho_{\scrL E_{i-n}} \circ (\scrL E_{i-n} \otimes_\D \phi_0^{-1})} & \scrL (E_{i-n} \otimes_\C \mathbbm 1_\C) \ar[d]^-{\scrL(\varrho_{E_{i-n}})}\\
\scrL E_{i-n} \ar@{=}[r] & \scrL E_{i-n}
}
$$
implies that $\Xi_{i}$ does commute for $n \leq i \leq m+n$ as well. The assertion for the diagram $\Xi'$ follows similarly.
\end{proof}
\begin{prop}\label{prop:functor_homot}
Let $(\scrL,\phi, \phi_0): \C \rightarrow \D$ be an exact and almost strong monoidal functor. Let $(\xi, \zeta)$ be a given pair of extensions in  $\mathcal Ext^m_\C(\mathbbm 1, \mathbbm{1}) \times \mathcal Ext^n_\C(\mathbbm 1_\C, \mathbbm 1_\C)$. Then the loops
$$
\scrL_{m+n}(\Omega_\C(\xi, \zeta)) \quad \text{and} \quad \Omega_\D(\scrL_m \xi, \scrL_n \zeta)
$$
are homotopically equivalent.
\end{prop}
\begin{proof}
Put $\nu := m+n$. According to Lemma \ref{lem:morloopcomp}, the flank rhomboids in the following diagram commute.
$$
\small{
\xymatrix@!C=23pt@R=23pt{
& \scrL_m \xi \boxtimes_\D \scrL_n \zeta \ar@[blue][dl]_-{L} \ar[rr]^-{\vartheta}  & & \scrL_\nu (\xi \boxtimes_\C \zeta) \ar@[red][dl]_-{\scrL_\nu (L)} \ar@[red][dr]^-{\scrL_\nu (R)}  & & \scrL_m \xi \boxtimes_\D \scrL_n \zeta \ar@[blue][dr]^-{R} \ar[ll]_-\vartheta & \\
(\scrL_m \xi) \circ (\scrL_n \zeta) \ar[rr]^-{=} & & \scrL_\nu (\xi \circ \zeta) & & (-1)^{mn}\scrL_\nu (\zeta \circ \xi) & & (-1)^{mn}(\scrL_n \zeta) \circ (\scrL_m\xi) \ar[ll]_-{=} \\
& (-1)^{mn}\scrL_m \zeta \boxtimes_\D \scrL_n \xi \ar@[blue][ul]^-{R} \ar[rr]_-{\vartheta} & & (-1)^{mn}\scrL_\nu (\zeta \boxtimes_\C \xi) \ar@[red][ul]^-{\scrL_\nu (R)} \ar@[red][ur]_{\scrL_\nu(L)} & & (-1)^{mn}\scrL_m \zeta \boxtimes_\D \scrL_n \xi \ar@[blue][ur]_-{L} \ar[ll]^-\vartheta &
}}
$$
It follows that the loop $\scrL_{m+n}(\Omega_\C(\xi, \zeta))$ (represented by the internal diamond) is homotopically equivalent to
$$
\xymatrix@!C=25pt{
& \scrL_m \xi \boxtimes_\D \scrL_n \zeta \ar@[blue][dl]_-{L} \ar@[blue][dr]^-{R}  &  \\
(\scrL_m \xi) \circ (\scrL_n \zeta) &  &(-1)^{mn}(\scrL_n \zeta) \circ (\scrL_m\xi)\\
& (-1)^{mn}\scrL_m \zeta \boxtimes_\D \scrL_n \xi \ar@[blue][ul]^-{R} \ar@[blue][ur]_{L} &
}
$$
which precisely is $\Omega(\scrL_m \xi, \scrL_n \zeta)$.
\end{proof}
\begin{thm}\label{thm:bracketcomm}
Let $(\scrL,\phi, \phi_0): \C \rightarrow \D$ be an exact and almost strong monoidal functor. Then the diagrams
$$
\xymatrix@C=30pt{
\Ext^m_\C(\mathbbm 1_\C, \mathbbm 1_\C) \times \Ext^n_\C(\mathbbm 1_\C, \mathbbm 1_\C) \ar[r]^-{[-,-]_\C} \ar[d]_{\scrL^\sharp_{m} \times \scrL^\sharp_n} & \Ext^{m+n-1}_\C(\mathbbm 1_\C, \mathbbm 1_\C) \ar[d]^-{\scrL^\sharp_{m+n-1}}\\
\Ext^m_{\D}(\mathbbm 1_\D, \mathbbm 1_\D) \times \Ext^n_{\D}(\mathbbm 1_\D, \mathbbm 1_\D) \ar[r]^-{[-,-]_\D} & \Ext^{m+n-1}_{\D}(\mathbbm 1_\D,\mathbbm 1_\D)
}
$$
and
$$
\xymatrix@C=30pt{
\Ext^{2n}_\C(\mathbbm 1_\C, \mathbbm 1_\C) \ar[r]^-{{sq}_\C} \ar[d]_{\scrL^\sharp_{2n}} & \Ext^{4n-1}_\C(\mathbbm 1_\C, \mathbbm 1_\C) \ar[d]^-{\scrL^\sharp_{4n-1}}\\
\Ext^{2n}_{\D}(\mathbbm 1_\D, \mathbbm 1_\D) \ar[r]^-{sq_{\D}} &
\Ext^{4n-1}_{\D}(\mathbbm 1_\D,\mathbbm 1_\D)
}
$$
commute for all integers $m, n \geq 1$.
\end{thm}
\begin{proof}
For an admissible $n$-extension $\xi$ of $\mathbbm 1_\C$ by $\mathbbm 1_\C$, recall that the homomorphisms $\scrL_n^\sharp$ and $\scrL_n^\flat$,
\begin{align*}
\scrL_n^\sharp = \pi_0 \scrL_n &: \Ext^n_\C(\mathbbm 1_\C, \mathbbm 1_\C) \longrightarrow \Ext^{n}_{\D}(\mathbbm 1_\D,\mathbbm 1_\D),\\
\scrL_n^\flat = \pi_1 (\scrL_n,\xi) &: \pi_1 (\mathcal Ext^n_\C(\mathbbm 1_\C, \mathbbm 1_\C), \xi) \longrightarrow \pi_1 (\mathcal Ext^{n}_{\D}(\mathbbm 1_\D,\mathbbm 1_\D), \scrL_n \xi),
\end{align*}
are given by applying $\scrL$ degreewise. Let $x_1=[\xi_1] \in \Ext^m_\C(\mathbbm 1_\C, \mathbbm 1_\C)$ and $x_2 = [\xi_2] \in \Ext^n_\C(\mathbbm 1_\C, \mathbbm 1_\C)$, where $\xi_1$ and $\xi_2$ are admissible extensions in $\C$ representing the classes $x_1$ and $x_2$ respectively. For $\zeta = \xi_1 \circ ((-1)^{mn}\xi_2)$ and $\zeta'=\scrL_{m+n}(\zeta) = \scrL_m\xi_1 \circ ((-1)^{mn}\scrL_n \xi_2)$ we observe that
\begin{equation}\label{thm:bracketcomm:eq1}
\begin{aligned}
\scrL^\sharp_{m+n-1}& \circ u_{\C}^{-1} \circ v_\C^{-1} &&\\
&= u_{\D}^{-1} \circ \scrL^\flat_{m+n} \circ v_\C^{-1} && \text{(by Lemma \ref{lem:grhomcomm})}\\
&= u_{\D}^{-1} \circ v_\D^{-1} \circ \check{F}_{m+n} && \text{(by Lemma \ref{lem:grhomcomm})}\\
&= u_{\D}^{-1} \circ v_\D^{-1} \circ \scrL^\flat_{m+n}  &&
\end{aligned}
\end{equation}
and hence
\begin{equation*}
\begin{aligned}
(\scrL^\sharp_{m+n-1} &\circ [-,-]_\C)(x_1,x_2)\\ &= \scrL^\sharp_{m+n-1}([x_1,x_2]_\C) &&\\
&= \scrL^\sharp_{m+n-1}((u_{\C}^{-1} \circ v_\C^{-1})([\Omega_\C(\xi_1, \xi_2)])) &&\text{(by defintion of $[-,-]_\C$)}\\
&= (u_{\D}^{-1} \circ v_\D^{-1}  \circ \scrL^\flat_{m+n})([\Omega_\C(\xi_1, \xi_2)]) && \text{(by equation (\ref{thm:bracketcomm:eq1}))}\\
&= (u_{\D}^{-1} \circ v_\D^{-1})(\scrL^\flat_{m+n}([\Omega_\C(\xi_1, \xi_2)])) &&\\
&= (u_{\D}^{-1} \circ v_\D^{-1})([{\scrL}_{m+n}\Omega_\C(\xi_1, \xi_2)]) &&\text{(by definition of $\scrL^\flat_{m+n}$)}\\
&= (u_{\D}^{-1} \circ v_\D^{-1})([\Omega_\D(\scrL_m(\xi_1), \scrL_n(\xi_2))]) &&\text{(by Proposition \ref{prop:functor_homot})}\\
&= [[\scrL_m(\xi_1)],[\scrL_n (\xi_2)]]_\D &&\\
&= [\scrL^\sharp_m(x_1), \scrL^\sharp_n(x_2)]_\D &&\text{(by definition of $\scrL^\flat_{m}$ and $\scrL^\flat_{n}$)}\\
&=([-,-]_\D \circ (\scrL^\sharp_m \times \scrL^\sharp_n))(x_1,x_2).&&
\end{aligned}
\end{equation*}
For the commutativity of the second diagram, one argues similarly.
\end{proof}
\begin{rem}\label{rem:thmcomrem}
In light of Remark \ref{rem:laxfunc_comm}, one easily checks that Proposition \ref{prop:functor_homot} also holds true in the setting of exact and \textit{colax} monoidal functors. Hence the statement of Theorem \ref{thm:bracketcomm} remains valid, when the exact and almost strong monoidal functor $(\scrL, \phi, \phi_0)$ is replaced by an exact and almost costrong monoidal functor $(\scrL', \phi', \phi_0') \colon \C \rightarrow \D$.
\end{rem}
\begin{nn}
Since $\Ext^\bullet_\C(\mathbbm 1_\C,\mathbbm 1_\C)$ is a graded commutative $k$-algebra, one may raise the following question.
\begin{quest}\label{con:gerstenhaber}
Under which requirements on $\C$ is $(\Ext^\bullet_\C(\mathbbm 1_\C,\mathbbm 1_\C), [-,-]_\C, sq_\C)$ a $($strict$)$ Gerstenhaber algebra in the sense of Definition $\ref{def:galgebra}$?
\end{quest}
Within the framework of this monograph we are not going to elaborate on this question. However, we remark, that already the $k$-bilinearity of the map $[-,-]_\C$ is not clear at all. Even in the special case of bimodules over a ring, an intrinsic proof could not be deduced. Nevertheless, it will turn out that the abstract construction, and its properties, already yield a couple of interesting applications, as exposed in the following sections. As a first step, we will prove that our bracket in fact recovers S.\,Schwede's original one in case $\C$ is a category of bimodules over a $k$-algebra.
\end{nn}
\section{The module case -- Schwede's original construction}\label{sec:schwede}
\begin{nn} Let $\Gamma$ be any ring and let $n \geq 0$ be an integer. Fix two $\Gamma$-modules $L$ and $M$. In \cite{Sch98} S.\,Schwede (explicitly) constructs an isomorphism
$$
\mu: \Ext^n_\Gamma(M,L) \cong H^n(\Hom_{\Gamma}(\mathbb P_M, L)) \longrightarrow \pi_1(\mathcal Ext^{n+1}_{\Gamma}(M,L), \kappa)
$$
of groups for some specific base point $\kappa$ and any fixed projective resolution $\mathbb P_M \rightarrow M \rightarrow 0$ of $M$ over $\Gamma$. In this section, we will show that, for $\C = \Mod(\Gamma)$, it coincides (up to the sign $(-1)^{n+1}$) with the map
\begin{align*}
v_\C \circ u_\C : \Ext^n_\Gamma(M,L) \longrightarrow\pi_1(\mathcal Ext^{n+1}_{\Gamma}&(M,L), \sigma_n(M,L))\\ &\longrightarrow \pi_1(\mathcal Ext^n_{\Gamma}(M,L), \kappa)
\end{align*}
which we have defined earlier in Section \ref{sec:retakh}. Hence, in particular, the following will hold true.
\begin{thm}[{\cite[Thm.\,3.1]{Sch98}}]\label{thm:schwede_comm}
Let $A$ be a $k$-algebra which is projective over $k$ and let $\P(A) \subseteq \Mod(A^\ev)$ be the full subcategory of all $A^\ev$-modules which are $A$-projective on both sides $($see Example $\ref{exa:bimodules}$$)$. Then, under the identification of Corollary $\ref{cor:isohochschildproj}$, the following diagrams commute $($up to the sign $(-1)^{m+1}$, which may, and will be ignored$)$ for every $m,n \in \mathbb N$.
$$
\xymatrix@C=30pt{
\HH^m(A) \times \HH^n(A) \ar[r]^-{\{-,-\}_A} \ar[d]_-{\chi_A \times \chi_A} & \HH^{m+n-1}(A) \ar[d]^-{\chi_A}\\
\Ext^m_{\P(A)}(A,A) \times \Ext^n_{\P(A)}(A,A) \ar[r]^-{[-,-]_\C} & \Ext^{m+n-1}_{\P(A)}(A,A)
}
$$ 
$$
\xymatrix@C=30pt{
\HH^{2n}(A) \ar[r]^-{sq_A} \ar[d]_-{\chi_A} & \HH^{4n-1}(A) \ar[d]^-{\chi_A}\\
\Ext^{2n}_{\P(A)}(A,A) \ar[r]^-{sq_\C} & \Ext^{4n-1}_{\P(A)}(A,A)
}
$$ 
\end{thm}
The maps $\{-,-\}_A$ and $sq_A$ respectively denote the Gerstenhaber bracket and the squaring map on $\HH^\bullet(A)$. Let us recall S.\,Schwede's map. Assume that the projective resolution $\mathbb P_M \rightarrow M \rightarrow 0$ is given as follows.
$$
\xymatrix@C=20pt{
\cdots \ar[r]^-{\pi_{n+2}} & P_{n+1} \ar[r]^-{\pi_{n+1}} & P_{n} \ar[r]^-{\pi_{n}} & P_{n-1} \ar[r]^-{\pi_{n-1}} & \cdots \ar[r]^-{\pi_1} & P_0 \ar[r]^-{\pi_0} & M \ar[r] & 0
}
$$
Fix a $\Gamma$-linear map $\varphi: P_{n+1} \rightarrow L$ satisfying $\varphi \circ \pi_{n+2} = 0$. The pushout diagram
$$
\xymatrix@C=20pt{
\cdots \ar[r]^-{\pi_{n+2}} & P_{n+1} \ar[d]_-{\varphi} \ar[r]^-{\pi_{n+1}} & P_{n} \ar[r]^-{\pi_{n}} \ar[d] & P_{n-1} \ar[r]^-{\pi_{n-1}} \ar@{=}[d] & \cdots \ar[r]^-{\pi_1} & P_0 \ar[r]^-{\pi_0} \ar@{=}[d] & M \ar[r] \ar@{=}[d] & 0\\
0 \ar[r] & L \ar[r]^-{p_{n+1}} & P \ar[r]^-{p_n} & P_{n-1} \ar[r]^-{p_{n-1}} & \cdots \ar[r]^-{p_1} & P_0 \ar[r]^-{p_0} & M  \ar[r] & 0
}
$$
has exact rows. We may regard $P$ as the quotient
$$
P = \frac{L \oplus P_{n}}{\{(\varphi(x), -\pi_{n+1}(x)) \mid x \in P_{n+1}\}} = \Coker(\varphi \oplus (-\pi_{n+1}))
$$
and hence express the maps $p_{n+1}: L \rightarrow P$, $p_n: P \rightarrow P_{n-1}$ as $p_{n+1}(l) = (l,0)$ and $p_n(l,p) = \pi_{n}(p)$ (for $p \in P_n$, $l \in L$).

If $\psi: P_{n} \rightarrow L$ is a $\Gamma$-module homomorphism, then $\varphi' = \varphi + \psi \circ \pi_{n+1}$ will also satisfy $\varphi' \circ \pi_{n+2} = 0$. Hence we may consider $\kappa(\varphi')$ as well. The assignment $(l,p) \mapsto (l - \psi(p),p)$ gives rise to a well-defined map
$$
\frac{L \oplus P_{n}}{\{(\varphi(x), -\pi_{n+1}(x)) \mid x \in P_{n+1}\}} \longrightarrow \frac{L \oplus P_{n} }{\{(\varphi'(x), -\pi_{n+1}(x)) \mid x \in P_{n+1}\}} \ ,
$$
which itself defines a morphism $\mu(\psi): \kappa(\varphi) \rightarrow \kappa(\varphi + \psi \circ \pi_{n+1})$ of $(n+1)$-extensions. If we chose $\psi$ such that $\psi \circ \pi_{n+1} = 0$, $\mu(\psi)$ will be an endomorphism of $\kappa(\varphi)$ in the category $\mathcal Ext^{n+1}_{\Gamma}(M,L)$, i.e., a loop of length $1$ based at $\kappa(\varphi)$. In \cite[Sec.\,4]{Sch98} it is shown that hereby one obtains a well-defined map
$$
\mu: H^n(\Hom_{\Gamma}(\mathbb P_M, L)) \longrightarrow \pi_1(\mathcal Ext^{n+1}_{\Gamma}(M,L), \kappa(\varphi))
$$
which is an isomorphism (cf. \cite[Thm.\,1.1]{Sch98}). The latter result is going to be recovered by the following lemma.
\end{nn}
\begin{lem}\label{lem:consistancy}
Let $\C$ be the category $\Mod(\Gamma)$. The maps $\mu$ and $v_\C \circ u_\C$ are, up to the sign $(-1)^n$ and conjugation with some path of length $1$, equal $($after identifying $\Ext^n_\Gamma(M,L)$ with $H^n(\Hom_{\Gamma}(\mathbb P_M, L)))$.
\end{lem}
\begin{proof}
First of all, we deal with the case $\varphi = 0$. Let $\kappa$ be $\kappa(0)$. It follows that
$$
P = \frac{L \oplus P_{n}}{\{(\varphi(x), -\pi_{n+1}(x)) \mid x \in P_{n+1}\}} = L \oplus \frac{P_n}{\Im(\pi_{n+1})}
$$
and there is a morphism $\alpha_\kappa: \kappa \rightarrow \sigma_{n+1}(M,L)$ of $(n+1)$-extensions. We claim that the diagram
\begin{equation}\label{eq:uvconjcomm}\tag{$\diamond$}
\begin{aligned}
\xymatrix@C=15pt{
\Ext_\Gamma^n(M,L) \ar[d]_-\cong \ar[r]^-{u_\C} & \pi_1(\mathcal Ext^{n+1}_{\Gamma}(M,L), \sigma_{n+1}(M,L)) \ar[d]^-{c_{\alpha_\kappa}}\\
\OH^n(\Hom_{\Gamma}(\mathbb P_M, L)) \ar[r]^-\mu  & \pi_1(\mathcal Ext^{n+1}_{\Gamma}(M,L), \kappa)
}
\end{aligned}
\end{equation}
commutes. Let an exact sequence $\xi$ in $\mathcal Ext^n_\C(M,L)$ be given whose equivalence class corresponds to $[\psi] \in H^n(\Hom_{\Gamma}(\mathbb P_M, L))$ under the standard isomorphism $\Ext^n_\C(M,L) \cong H^n(\Hom_{\Gamma}(\mathbb P_M, L))$ (cf. \cite[Sec.\,3.4]{Wei94}). Recall that $\psi$ can be chosen to be $\psi_n$ occurring as part of a lifting of the identity morphism of $M$ with respect to $\mathbb P_M$:
$$
\xymatrix@C=20pt{
 \cdots \ar[r]^-{\pi_{n+2}} & P_{n+1} \ar[r]^-{\pi_{n+1}} & P_{n} \ar[r]^-{\pi_{n}} \ar[d]_-{\psi_n} & P_{n-1} \ar[r]^-{\pi_{n-1}} \ar[d]_-{\psi_{n-1}} & \cdots \ar[r]^-{\pi_1} & P_0 \ar[r]^-{\pi_0} \ar[d]^-{\psi_0} & M \ar[r] \ar@{=}[d] & 0 \ \ \\
\xi \quad\quad \equiv & 0 \ar[r] & L \ar[r]^-{e_n} & E_{n-1} \ar[r]^-{e_{n-1}} & \cdots \ar[r]^-{e_1} & E_0 \ar[r]^-{e_0} & M  \ar[r] & 0 \ .
}
$$
Put $P_{n}/ \pi_{n+1} = P_{n} / \Im(\pi_{n+1})$. The conjugate of $(-1)^{n+1} u_\C(\xi)$ by $\alpha_\kappa$ looks as follows (at this point, we use Proposition \ref{lem:mapsagree}).
$$
\xymatrix@C=14pt@R=17pt{
\kappa \ar[d]_-{\alpha_\kappa} & 0 \ar[r] & L \ar@{=}[d] \ar[r] & L \oplus \frac{P_n}{\pi_{n+1}} \ar[d] \ar[r] & P_{n-1} \ar[r] \ar[d] & \cdots \ar[r] & P_1 \ar[r] \ar[d] & P_0 \ar[r] \ar[d] & M \ar[r] \ar@{=}[d] & 0\\
\sigma_{n+1}(M,L) & 0 \ar[r] & L \ar@{=}[r] & L \ar[r] & 0 \ar[r] & \cdots \ar[r] & 0 \ar[r] & M \ar@{=}[r] & M \ar[r] & 0\\
\xi_+ \ar[u]^-{q_1} \ar[d]_-{q_2} & 0 \ar[r] & L \ar[r] \ar@{=}[d] \ar@{=}[u] & L \oplus L \ar[r] \ar[u] \ar[d] & E_{n-1} \ar[u] \ar[d] \ar[r] & \cdots \ar[r] & E_1 \ar[u] \ar[d] \ar[r] & E_0 \ar[u] \ar[d] \ar[r] & M \ar@{=}[u] \ar@{=}[d] \ar[r] & 0\\
\sigma_{n+1}(M,L) & 0 \ar[r] & L \ar@{=}[r] & L \ar[r] & 0 \ar[r] & \cdots \ar[r] & 0 \ar[r] & M \ar@{=}[r] & M \ar[r] & 0\\
\kappa \ar[u]^-{\alpha_\kappa}& 0 \ar[r] & L \ar@{=}[u] \ar[r] & L \oplus \frac{P_n}{\pi_{n+1}} \ar[u] \ar[r] & P_{n-1} \ar[u] \ar[r] & \cdots \ar[r] & P_1 \ar[r] \ar[u] & P_0 \ar[u] \ar[r] & M \ar@{=}[u] \ar[r] & 0
}
$$
We claim that the equivalence class of the loop above equals the equivalence class of $\mu(\kappa)$. In fact, the following assignments define a morphism $\beta: \kappa \rightarrow \xi_+$ of $(n+1)$-extensions
\begin{align*}
\beta_{i}(p_i) &= \psi_i(p_i) && (\text{for $i=0, \dots, n-1$ and $p_i \in P_i$}),\\
\beta_{n}(l, p) &= (l, l - \psi(p)) && (\text{for $l \in L$, $p \in P_{n}/ \Im(\pi_{n+1})$});
\end{align*}
$\beta$ is such that the diagram
$$
\xymatrix@R=12pt@C=25pt{
\kappa \ar[rd]^-{\alpha_\kappa} \ar[dd]_-{\beta} & \\
& \sigma_{n+1}(M,L) \\
\xi_+ \ar[ur]_-{q_1} &
}
$$
commutes. Hence the loop is homotopically equivalent to the following one.
$$
\xymatrix@C=14pt@R=18pt{
\kappa \ar[d]_-{q_2 \circ \beta} & 0 \ar[r] & L \ar@{=}[d] \ar[r] & L \oplus \frac{P_n}{\pi_{n+1}} \ar[d]_-{
\id + (-\psi)
} \ar[r] & P_{n-1} \ar[r] \ar[d] & \cdots \ar[r] & P_1 \ar[r] \ar[d] & P_0 \ar[r] \ar[d]|{=}^-{\pi_0}_-{e_0\circ\psi_0} & M \ar[r] \ar@{=}[d] & 0\\
\sigma_{n+1}(M,L) & 0 \ar[r] & L \ar@{=}[r] & L \ar[r] & 0 \ar[r] & \cdots \ar[r] & 0 \ar[r] & M \ar@{=}[r] & M \ar[r] & 0\\
\kappa \ar[u]^-{\alpha_\kappa}& 0 \ar[r] & L \ar@{=}[u] \ar[r] & L \oplus \frac{P_n}{\pi_{n+1}} \ar[u] \ar[r] & P_{n-1} \ar[u] \ar[r] & \cdots \ar[r] & P_1 \ar[r] \ar[u] & P_0 \ar[u] \ar[r] & M \ar@{=}[u] \ar[r] & 0
}
$$
Since $\mu(\kappa)$ fits inside the commutative diagram
\begin{equation*}
\begin{aligned}
&\xymatrix@R=12pt@C=25pt{
\kappa \ar[rd]^-{q_2 \circ \beta} \ar[dd]_-{\mu(\kappa)} & \\
& \sigma_{n+1}(M,L) \ ,\\
\kappa \ar[ur]_-{\alpha_{\kappa}} &
}
\end{aligned}
\end{equation*}
we have accomplished the desired result (recall the definition of elementary homotopic).
Finally, for arbitrary $\varphi$, we observe that $\mu(\psi) \equiv \mu(\psi)_{\varphi=0} \boxplus \kappa(\varphi)$ (up to conjugation). In fact,
$$
\xymatrix@C=35pt{
P_0 \oplus P_0 \ar[r]^-{
\big[\begin{smallmatrix}
\pi_0 & 0 \\
0 & \pi_0
\end{smallmatrix}\big]}
 & M \oplus M\\
P_0 \oplus \Ker(\pi_0) \ar[r]^-{
\big[\begin{smallmatrix}
\pi_0 & 0
\end{smallmatrix}\big]} \ar[u]^-{
\big[\begin{smallmatrix}
\id & \mathrm{inc} \\
\id & 0
\end{smallmatrix}\big]}
 & M \ar[u]_-{
\big[\begin{smallmatrix}
\id \\ 
\id
\end{smallmatrix}
\big]
}
}
$$
is clearly a pullback diagram. Further,
$$
\xymatrix{
L \oplus L \ar[d]_-{
\big[\begin{smallmatrix}
\id & \id
\end{smallmatrix}
\big]
} \ar[r] & (L \oplus P_n/ \pi_{n+1}) \oplus P \ar[d]^-{} \\
L \ar[r] & P_n/ \pi_{n+1} \oplus P
}
$$
is a pushout diagram, where the maps $L \rightarrow P_n/ \pi_{n+1} \oplus P$ and $(L \oplus P_n/ \pi_{n+1}) \oplus P \rightarrow P_n/ \pi_{n+1} \oplus P$ are induced by the maps
\begin{align*}
L \longrightarrow P_n \oplus (L \oplus P_n), &\ l \mapsto (0,l,0) ,
\intertext{and}
(L \oplus P_n) \oplus (L \oplus P_n) \longrightarrow P_n \oplus (L \oplus P_n), &\ (l,p,l',p') \mapsto (p, l+ l', p')
\end{align*}
respectively. It follows, that
$$
\kappa(0) \boxplus \kappa(\varphi) =  \mathbb{P}^\natural \oplus \kappa(\varphi)
$$
where $\mathbb{P}^\natural$ denotes the $(n-1)$-extension
$$
\xymatrix@C=20pt{
0 \ar[r] & P_{n}/{\pi_{n+1}} \ar[r] & P_{n-1} \ar[r]^-{\pi_{n-1}} & \cdots \ar[r]^-{\pi_2} & P_1 \ar[r]^-{\pi_1} & \Ker(\pi_0) \ar[r] & 0 \ .
}
$$
If $j: \kappa(\varphi) \rightarrow \mathbb{P}^\natural \oplus \kappa(\varphi)$ and $q: \mathbb{P}^\natural \oplus \kappa(\varphi) \rightarrow \kappa(\varphi)$ are the canonical maps, we immediately see that $\mu(\kappa)_{\varphi=0} \oplus \kappa(\varphi)$ conjugated by $j$ is (homotopically equivalent to) the loop
$$
\xymatrix@C=35pt{
\kappa(\varphi) \ar[r]^-j & \mathbb{P}^\natural \oplus \kappa(\varphi) \ar[r]^-{
\big[\begin{smallmatrix}
\id & 0\\
0 & \mu(\kappa)
\end{smallmatrix}\big]
}
& \mathbb{P}^\natural \oplus \kappa(\varphi) \ar[r]^-q & \kappa(\varphi) \ .
}
$$
Ultimately, the diagram below commutes (the triangles containing the arrow $c_{\alpha_\kappa}$ commute by Lemma \ref{lem:fund_conj_commutative} and by the commutativity of (\ref{eq:uvconjcomm}); the commutativity of the lower part, i.e., the equality $c_j \circ v_\C \circ \mu_{\varphi = 0} = \mu$, follows from the observations that were made above).
$$
\xymatrix{
\Ext_\Gamma^n(M,L) \ar[d]_-\cong \ar[r]^-{(-1)^{n+1}u_\C} & \pi_1(\mathcal Ext^{n+1}_{\Gamma}(M,L), \sigma_{n+1}(M,L)) \ar[d]^-{v_\C} \ar@/_1pc/@[red][ldd]_-{c_{\alpha_\kappa}}\\
H^n(\Hom_{\Gamma}(\mathbb P_M, L)) \ar@<-9ex>@/_1.5pc/@[blue][dd]_-\mu \ar@[blue][d]_-\mu & \pi_1(\mathcal Ext^{n+1}_\Gamma(M,L), \sigma_{n+1}(M,L) \boxplus \kappa(\varphi))\ar[d]^-{c_{\alpha_\kappa} \boxplus \kappa(\varphi)}\\
\pi_1(\mathcal Ext^{n+1}_{\Gamma}(M,L), \kappa(0)) \ar@[blue][r]^-{v_\C} & \pi_1 (\mathcal Ext^{n+1}_\Gamma(M,L), \kappa(0) \boxplus \kappa(\varphi)) \ar@[blue][d]^-{c_j}\\
\pi_1(\mathcal Ext^{n+1}_\Gamma(M,L), \kappa(\varphi)) \ar@[blue]@{=}[r] & \pi_1(\mathcal Ext^{n+1}_\Gamma(M,L), \kappa(\varphi))
}
$$
\end{proof}
\section{Morita equivalence}
\begin{nn}\label{nn:morita}
Let $A$ and $B$ be two algebras over the commutative ring $k$. Remember that $A$ and $B$ are called \textit{Morita equivalent} if their associated categories of left modules are equivalent:
$$
\xymatrix@C=18pt{
\mathscr U: \Mod(A) \ar[r]^-\sim & \Mod(B).
}
$$
It is a classical fact (see for example \cite[\S 22]{AnFu92}), that $A$ and $B$ are Morita equivalent if, and only if, there is an $A$-module $P$ fulfilling the following properties:
\begin{enumerate}[\rm(1)]
\item\label{nn:morita:1} $P$ is a \textit{progenerator for $A$} (that is, it is a finitely generated projective $A$-module and a generator for the category $\Mod(A)$ in the sense that for every $A$-module $M$ there is an $A$-module $P_M$ and an epimorphism $P_M \rightarrow M$, where $P_M$ is a possibly infinite direct sum of copies of $P$; e.g., $P=A$ is a progenerator for $A$).
\item\label{nn:morita:2} $B \cong \End_A(P)^\op$ as $k$-algebras.
\end{enumerate}
In this situation, $P$ is naturality a right module over (its opposite endomorphism ring) $B$ and we will call $P$ a \textit{Morita bimodule for $A$}. If $A$ and $B$ are Morita equivalent, the functor $\mathscr U$ is (up to equivalence) given by $\Hom_A(P,-)$, where $P$ is the progenerator $\mathscr U^{-1}(B)$, and the quasi-invers functor $\mathscr U^{-1}$ can be expressed as $P \otimes_B -$. A classical example of a pair of Morita equivalent algebras is given by $A$ and $M_n(A)$ ($=$ the ring of $n \times n$-matrices over $A$) for any $k$-algebra $A$ and any integer $n \geq 1$ (the free module $A^n$ is a progenerator for $A$ with adequate endomorphism ring). It is not surprising that the Hochschild cohomology ring is invariant under Morita equivalence.
\begin{thm}
If $A$ and $B$ are Morita equivalent, then $\HH^\bullet(A)$ and $\HH^\bullet(B)$ are isomorphic as graded $k$-algebras.
\end{thm}
In fact, this even holds true when replacing Morita equivalent by \textit{derived equivalent}\footnote{The $k$-algebras $A$ and $B$ are \textit{derived equivalent}, if their bounded derived categories $\mathbf D^b(\Mod(A))$ and $\mathbf D^b(\Mod(B))$ are equivalent as triangulated categories.} as shown in \cite[Thm.\,4.2]{Ha89} in the special case of finite dimensional algebras and derived equivalences arising from tilting modules, and then in \cite[Prop.\,2.5]{Ri91} in the most general case. (Note that Morita equivalent algebras are always (for trivial reasons) derived equivalent; however the converse does not hold true in general.) The question of whether the Gerstenhaber bracket is preserved under derived equivalences, was considered by B.\,Keller in \cite{Ke04}, wherein he claims that it is indeed a derived invariant. To realize the result, he compares the graded Lie algebra $\HH^{\bullet+1}(A)$ with the Lie algebra associated to the \textit{derived Picard group of $A$} via an isomorphism of graded Lie algebras:
$$
\xymatrix@C=18pt{
\HH^{\bullet+1}(A) \ar[r]^-\sim & \mathrm{Lie}(\mathbf{DPic}_A^\bullet).
}
$$
The right hand side does not depend on $A$, but rather on the underlying triangulated category $\mathbf D^b(\Mod(A))$. However, the proof given by B.\,Keller is, at least for the author, not very enlightening. Using the methods established so far, we will confirm B.\,Keller's observation for Morita equivalent algebras. In fact, we will even show that the whole \textit{strict} Gerstenhaber structure is preserved under Morita equivalence.
\end{nn}
\begin{lem}\label{lem:dereqproj}
Let $A$ and $B$ be Morita equivalent $k$-algebras. Then $A$ is flat $($projective$)$ as a $k$-module if, and only if, $B$ is flat $($projective$)$ as a $k$-module.
\end{lem}
\begin{proof}
In view of \ref{nn:morita}(\ref{nn:morita:1})--(\ref{nn:morita:2}), it suffices to show the following statement: Let $\Gamma$ be a $k$-algebra, and let $P$ be a finitely generated projective $\Gamma$-module. Then $\Gamma$ being a flat (projective) $k$-module implies that $\End_\Gamma(P)$ is a flat (projective) $k$-module. 

Let $Q$ be a $\Gamma$-module with $P \oplus Q \cong \Gamma^m$ for some integer $m \geq 0$. Then, as $k$-modules,
\begin{align*}
\Gamma^{m^2} &\cong \End_\Gamma(\Gamma^m)\\
&\cong \End_\Gamma(P) \oplus \End_\Gamma(Q) \oplus \Hom_\Gamma(P,Q) \oplus \Hom_\Gamma(Q,P),
\end{align*}
which finishes the proof.
\end{proof}
\begin{nn}
In the following, we fix a $k$-algebra $A$ which is projective as a $k$-module. Let $P$ be a progenerator for $A$ and let $B$ be the endomorphism ring $\End_{A}(P)^\op$. Hence $A$ and $B$ are Morita equivalent, and we have the following mutually inverse equivalences of categories. (Remember that $P$ naturally is a right $B$-module.)
\begin{align*}
\mathscr U_A = \Hom_A(P,-) : \Mod(A) \longrightarrow \Mod(B), \quad \mathscr V_A = P \otimes_B - : \Mod(B) \longrightarrow \Mod(A)
\end{align*}
We have already observed that $B$ has to be projective as a $k$-module. Let $(-)^\vee = (-)^\vee_A$ be the contravariant functor $\Hom_A(-,A) : \Mod(A) \rightarrow \Mod(A^\op)$. Since $P$ is a right $B$-module, $P^\vee$ will be a left module over $B$, by $(bf)(p) = (f \circ b)(p)$ (for $b \in B$, $f \in P^\vee$ and $p \in P$). According to the following statement, $(-)^\vee$ defines a duality $\proj(A) \rightarrow \proj(A^\op)$.
\end{nn}
\begin{lem}\label{lem:isodiamonddual}
Let $R$ be a ring, $U$ be a finitely generated projective $R$-module and let $M$ be an $R^\op$-module. Then the map
\begin{align*}
h_U: M \otimes_R U \longrightarrow \Hom&_{R^\op}(\Hom_R(U,R),M)\\
h_U(m \otimes u)(\varphi)&= m \varphi(u)
\end{align*}
is linear, bijective and natural, in both $M$ and $U$. If $M$ is a $S$-$R$-bimodule for some ring $S$, then $h_U$ is also $S$-linear.
\end{lem}
\begin{proof}
The linearity and the naturally are immediate. Further, $h_U$ is bijective, since it is (obviously) bijective in case $U$ is free of finite rank.
\end{proof}
\begin{nn}
The dual $P^\vee$ of $P$ is a progenerator for $A^\op$ with $\End_{A^\op}(P^\vee)^\op \cong \End_A(P) = B^\op$. In particular, we obtain two additional mutually inverse equivalences:
\begin{align*}
\mathscr U_{A^\op} = \Hom_{A^\op}(P^\vee_{A},-) &: \Mod(A^\op) \longrightarrow \Mod(B^\op),\\ \mathscr V_{A^\op} = - \otimes_B P^\vee_{A} &: \Mod(B^\op) \longrightarrow \Mod(A^\op).
\end{align*}
\end{nn}
The isomorphisms stated below are very well-known (see for instance \cite[Ch.\,IX]{CaEi56}); since the classical literature presents them in a slightly too weak fashion, we reprove them in the generality that we need.
\begin{lem}\label{lem:isosmoritaR}
Let $R$ be a $k$-algebra, $U$, $V$ and $W$ be finitely generated projective $R$-modules, and let $S = \End_R(U)^\op$. Further, let $(-)^\vee$ denote the contravariant functor $\Hom_R(-,R): \Mod(R) \rightarrow \Mod(R^\op)$. Then the following assertions hold true.
\begin{enumerate}[\rm(1)]
\item\label{lem:isosmoritaR:1} The map
\begin{align*}
s_M: \Hom_{R \otimes R^\op}(V \otimes_k W^\vee, M) &\longrightarrow \Hom_R(V,\Hom_{R^\op}(W^\vee, M)), \\ s_M(f)(v)(\varphi) &= f(v \otimes \varphi)
\end{align*}
is linear, bijective and natural in $V$, $W$ and in the $R \otimes_k R^\op$-module $M$. If $V = W = U$, the map $s_M$ is in addition a $S$-$S$-bimodule homomorphism.
\item\label{lem:isosmoritaR:2} The map
\begin{align*}
t_M: \Hom_{R \otimes R^\op}(V \otimes_k W^\vee, M) &\longrightarrow \Hom_{R^\op}(W^\vee,\Hom_{R}(V, M)), \\ t_M(f)(\varphi)(v) &= f(v \otimes \varphi)
\end{align*}
is linear, bijective and natural in $V$, $W$ and in the $R \otimes_k R^\op$-module $M$. If $V = W = U$, the map $t_M$ is in addition a $S$-$S$-bimodule homomorphism.
\item\label{lem:isosmoritaR:3} The map
\begin{align*}
a: \Hom_R(U,V) \otimes_k \Hom_{R^\op}(U^\vee, V^\vee) &\longrightarrow \Hom_{R \otimes_k R^\op}(U \otimes_k U^\vee, V \otimes_k V^\vee),\\
a(f \otimes g)(u \otimes \varphi) &= f(u) \otimes g(\varphi)
\end{align*}
is a $S$-$S$-bimodule homomorphism and bijective. Moreover, if $V = U$, it is a homomorphism of $k$-algebras.
\end{enumerate}
\end{lem}
\begin{proof}
One easily checks that the maps in (\ref{lem:isosmoritaR:1}) and (\ref{lem:isosmoritaR:2}) are linear and natural in every variable. Their inverse maps are given by $s_M^{-1}(f)(v \otimes \varphi) = f(v)(\varphi)$ and $t_M^{-1}(f)(v \otimes \varphi) = f(\varphi)(v)$. Assume that $V = W = U$. Take $s \in S = \End_R(U)^\op$ and $f \in \Hom_R(U,V)$, $g \in \Hom_{R^\op}(U^\vee, V^\vee)$ and remember, that the left $S$-module structure on $\Hom_R(U,M)$ is given by 
$$
(sf)(u) = (f \circ s)(u) \quad \text{(for $u \in U$)}
$$
whereas the right $S$-module structure on $\Hom_{R^\op}(U^\vee, M)$ is given by
$$
(gs)(\varphi) = g(\varphi \circ s) \quad \text{(for $\varphi \in W^\vee$)}.
$$
Thus, $s_M$ and $t_M$ are evidently $S$-linear on both sides. The same apparently holds true for the map $a$ in (\ref{lem:isosmoritaR:3}). It is bijective, for it is when $U$ and $V$ are free of finite rank. Finally, it is clear that $a$ is an algebra homomorphism in case $U = V$.
\end{proof}
\begin{nn}
We claim that $P \otimes_k P^\vee$ is a progenerator for $A \otimes_k A^\op$ whose endomorphism ring is isomorphic to $B \otimes_k B^\op$. To begin with, $P \otimes_k P^\vee$ is a projective $A \otimes_k A^\op$-module since the $k$-linear map
\begin{equation*}
\begin{aligned}
s_M: \Hom_{A \otimes A^\op}(P \otimes_k P^\vee, M) &\longrightarrow \Hom_A(P,\Hom_{A^\op}(P^\vee, M)), \\ s_M(f)(p)(\varphi) &= f(p \otimes \varphi)
\end{aligned}
\end{equation*}
defines a natural isomorphism by Lemma \ref{lem:isosmoritaR}(\ref{lem:isosmoritaR:1}). Since $A \otimes_k A^\op$ is a generator, it suffices to name an epimorphism $\bigoplus_i P \otimes_k P^\vee \rightarrow A \otimes_k A^\op$ of $A \otimes_k A^\op$-modules to verify that $P \otimes_k P^\vee$ is a generator for $A \otimes_k A^\op$. But this epimorphism can be realized by tonsoring together an epimorphism $\pi: P^m \rightarrow A$ with the epimorphism $\sigma^\vee: (P^\vee)^m \cong (P^m)^\vee \rightarrow A^\vee \cong A^\op$, where $\sigma: A \rightarrow P^m$ is a (split) monomorphism in $\Mod(A)$. Finally,
\begin{equation*}
\begin{aligned}
a: \Hom_A(P,P) \otimes_k \Hom_{A^\op}(P^\vee, P^\vee) &\longrightarrow \Hom_{A \otimes_k A^\op}(P \otimes_k P^\vee, P \otimes_k P^\vee)\\
a(f \otimes g)(p \otimes \varphi) &= f(p) \otimes g(\varphi)
\end{aligned}
\end{equation*}
is an isomorphism of $k$-algebras (chose $R = A$ and $U = V = P$ in Lemma \ref{lem:isosmoritaR}(\ref{lem:isosmoritaR:3})). Therefore $\End_{A^\ev}(P \otimes_k P^\vee)^\op \cong B \otimes_k B^\op$ as claimed. Hence we have proven the following result.
\begin{lem}
The $k$-algebras $\Gamma:=A \otimes_k A^\op$ and $\Lambda:=B \otimes_k B^\op$ are Morita equivalent.
\end{lem}
Note that $P$, $P^\vee$ and $P \otimes_k P^\vee$ are also finitely generated projective generators when considered as right modules over $B$, $B^\op$ and $B \otimes_k B^\op$ respectively (see \cite[\S 22]{AnFu92}). 

Put $Q := P \otimes_k P^\vee$. The Morita equivalent algebras $\Gamma$ and $\Lambda \cong \End_\Lambda(Q)^\op$ give rise to two mutually inverse equivalences
\begin{align*}
\mathscr U_\Gamma = \Hom_\Gamma(Q,-) : \Mod(\Gamma) \longrightarrow \Mod(\Lambda), \quad \mathscr V_\Gamma = Q \otimes_\Lambda - : \Mod(\Lambda) \longrightarrow \Mod(\Gamma),
\end{align*}
where we consider $Q$ as a right $\Lambda$-module in the usual way. We are interested in how these equivalences work on the full subcategories $\P(A) \subseteq \Mod(\Gamma)$ and $\P(B) \subseteq \Mod(\Lambda)$ (for the definition and properties of $\P(A)$ and $\P(B)$ see Example \ref{exa:bimodules}).
\end{nn}
\begin{prop}\label{prop:laxmodmorita}
The functor $\mathscr U_\Gamma = \Hom_\Gamma(Q,-)$ is a lax monoidal functor $(\Mod(A^\ev), \otimes_A, A) \rightarrow (\Mod(B^\ev), \otimes_B, B)$ with the unit morphism $B \rightarrow \mathscr U_\Gamma A$ being an isomorphism $($that is, $\mathscr U_\Gamma$ is almost strong monoidal$)$. Therefore, $\mathscr V_\Gamma = Q \otimes_\Lambda -$ is an almost costrong monoidal functor $(\Mod(B^\ev), \otimes_B, B) \rightarrow (\Mod(A^\ev), \otimes_A, A)$.
\end{prop}
\begin{proof}
To verify that $\mathscr U_\Gamma$ is a lax monoidal functor, we have to name homomorphisms
$$
\phi_{M,N}: \mathscr U_\Gamma M \otimes_B \mathscr U_\Gamma N \longrightarrow \mathscr U_\Gamma(M \otimes_A N) \quad \text{and} \quad \phi_0: B \longrightarrow \mathscr U_\Gamma A 
$$
of $\Gamma$-modules (the former being natural in $M$ and $N$). As a preliminary observation, note that by Lemma \ref{lem:isodiamonddual}, we obtain the following isomorphism
$$
P^\vee \otimes_A P \cong \Hom_{A^\op}(\Hom_A(P,A), P^\vee) = \End_{A^\op}(P^\vee) \cong \End_A(P)^\op = B,
$$
and thus the $A \otimes_k A^\op$-module homomorphisms stated below:
\begin{align*}
\Delta_P: P \otimes_k P^\vee \longrightarrow P \otimes_k B & \otimes_k P^\vee \cong (P \otimes_k P^\vee) \otimes_A (P \otimes_k P^\vee),\\
\Delta_P(p \otimes \varphi) &= p \otimes 1_B \otimes \varphi
\end{align*}
and
\begin{align*}
\varepsilon_P: P \otimes_k P^\vee \longrightarrow A, \ \varepsilon_P(p \otimes \varphi) = \varphi(p).
\end{align*}
The $B \otimes_k B^\op$-module $\mathscr U_\Gamma(A) = \Hom_{\Gamma}(Q,A) = \Hom_{A \otimes_k A^\op}(P \otimes_k P^\vee, A)$ can be viewed as $k$-algebra through
$$
f \star g := \mu \circ (f \otimes_A g) \circ \Delta_P \quad \text{(for $f,g \in \Hom_\Gamma(Q,A)$)}
$$
with unit $\varepsilon_P$ (here $\mu$ denotes the (unit) isomorphism $A \otimes_A A \rightarrow A$). As a $B \otimes_k B^\op$-module, $\mathscr V_\Gamma(A) = \Hom_{\Gamma}(Q,A)$ is isomorphic to $B$:
\begin{align*}
\mathscr U_\Gamma(A) &= \Hom_{A \otimes_k A^\op}(P \otimes_k P^\vee, A)\\
&\cong \Hom_{A^\op}(P^\vee, \Hom_{A}(P, A)) && \text{(by Lemma \ref{lem:isosmoritaR}(\ref{lem:isosmoritaR:2}))}\\
&= \Hom_{A^\op}(P^\vee,P^\vee)\\
&\cong \Hom_A(P,P)^\op && \text{(since $(-)^\vee$ is a duality)}\\
&= B \ .
\end{align*}
We are now in the position to define the disered homomorphisms $\phi_0$ and $\phi_{M,N}$ for any pair of $A \otimes_k A^\op$-modules $M$ and $N$. Indeed,
\begin{align*}
\phi_{M,N}: \Hom_{\Gamma}(P \otimes_k P^\vee, M) \otimes_B \Hom_{\Gamma}&(P \otimes_k P^\vee, M) \longrightarrow \Hom_{\Gamma}(P \otimes_k P^\vee, M \otimes_A N)\\
\phi_{M,N}(f \otimes g) &= (f \otimes_A g) \circ \Delta_P
\end{align*}
and
$$\xymatrix@C=15pt{
\phi_0: B \ar[r]^-\sim & \Hom_{\Gamma}(P \otimes_k P^\vee,A) = \mathscr U_\Gamma(A)
}$$
fit into the commutative diagrams of Definition \ref{def:monoidalfunc}, as one easily (but tediously) verifies.
\end{proof}
The proposition implies the following well-known fact.
\begin{cor}
If $R$ and $S$ are commutative $k$-algebras, and if $R$ and $S$ are Morita equivalent, then $R \cong S$.
\end{cor}
\begin{proof}
Clearly, $R = Z(R) \cong \Hom_{R^\ev}(R,R)$ and $S = Z(S) \cong \Hom_{S^\ev}(S,S)$. But $\Hom_{R^\ev}(R,R)$ and $\Hom_{S^\ev}(S,S)$ are isomorphic (via $\mathscr U_{R \otimes_k R}$).
\end{proof}
\begin{rem}
The maps $\Delta_P: P \otimes_k P^\vee \rightarrow (P \otimes_k P^\vee) \otimes_A (P \otimes_k P^\vee)$ and $\varepsilon_P: P \otimes_k P^\vee \rightarrow A$ turn $P \otimes_k P^\vee$ into a so called \textit{comonoid} in the monoidal category $(\Mod(A^\ev), \otimes_A, A)$. Other examples of comonoids are coalgebras over the commutative ring $k$ which are precisely the comonoids in the monoidal category $(\Mod(k), \otimes_k, k)$; the coalgebra axioms indicate which requirements are put upon the structure maps $\Delta_P$ and $\varepsilon_P$. See \cite{Sz05} for a characterization of monoidal Morita equivalences in terms of the general language of (strong) comonoids and bialgebroids (the latter are going to be defined in Section \ref{sec:hopfalgebroids}).
\end{rem}
\begin{rem}
The map $\Delta_P$ can be made more precise by utilizing the dual basis lemma. The projectivity of $P$ yields the existence of a set $I$ (which may be chosen as a finite one, since $P$ is finitely generated) and families $(x_i)_{i \in I} \subseteq P$, $(\varphi_i)_{i \in I} \subseteq \Hom_A(P,A)$ such that
$$
x = \sum_{i \in I}\varphi_i(x)x_i \quad \text{(for all $x \in P$)},
$$
that is, $\sum_{i \in I}\varphi_i(-)x_i = \id_P$. Now $\Delta_P$ displays itself as
$$
\Delta_P(p \otimes \varphi) = \sum_{i \in I} (p \otimes \varphi_i) \otimes (x_i \otimes \varphi).
$$
The element $\sum_{i \in I} \varphi_i \otimes x_i$ is (on the nose) the preimage of $\id_P$ under the isomorphism of Lemma \ref{lem:isodiamonddual}.
\end{rem}
\begin{exa}
Let $R$ be a $k$-algebra. Then $R$ is Morita equivalent to $M_n(R)$ with Morita bimodule given by $R^n$. Via the induced Morita bimodule $R^n \otimes_k (R^\op)^n \cong R^n \otimes_k R^n$ for $R^\ev$, the enveloping algebra $R^\ev = R \otimes_k R^\op$ of $R$ is then Morita equivalent to $M_n(R) \otimes_k M_n(R)^\op \cong M_{n^2}(R \otimes_k R^\op)$. The maps $\Delta_{R^n}$ and $\varepsilon_{R^n}$ are given as follows:
$$
\Delta_{R^n}(\underline{r} \otimes \underline{r}') = \sum_{i=1}^n(\underline{r} \otimes e_i) \otimes (e_i \otimes \underline{r}')
$$
and
$$
\varepsilon_{R^n}(\underline{r} \otimes \underline{r}') = \langle \underline{r}, \underline{r}'\rangle = \sum_{i=1}^n r_i r'_i,
$$
where $e_1, \dots, e_n$ is the standard basis of $R^n$, and $\underline{r} = (r_1, \dots, r_n)$, $ \underline{r}' = (r'_1, \dots, r'_n)$ are arbitrary elements.
\end{exa}
\begin{prop}\label{prop:moritamonoequ}
The functors $\mathscr U_\Gamma = \Hom_\Gamma(Q,-)$ and $\mathscr V_\Gamma = Q \otimes_\Lambda -$ restrict to mutually inverse exact and almost strong, respectively almost costrong, monoidal equivalences
\begin{align*}
\Hom_\Gamma(Q,-) |_{\P(A)} &: (\P(A), \otimes_A, A) \longrightarrow (\P(B), \otimes_B, B)\\
Q \otimes_{\Lambda} - \ |_{\P(B)} &: (\P(B), \otimes_B, B) \longrightarrow (\P(A), \otimes_A, A)
\end{align*}
between strong exact monoidal categories.
\end{prop}
\begin{proof}
To see that $\mathscr V_\Gamma$ gives rise to the desired functor $\mathscr B$, we prove the following assertions.
\begin{enumerate}[\rm(i)]
\item\label{enum:projside:1} Let $M$ be a $A \otimes_k A^\op$-module. If $M' = \mathscr U_\Gamma M$ belongs to $\P(B)$, then $M$ belongs to $\P(A)$.
\item\label{enum:projside:2} Let $N$ be a $B \otimes_k B^\op$-module. If $N' = \mathscr V_\Gamma N$ belongs to $\P(A)$, then $N$ belongs to $\P(B)$.
\end{enumerate}
From this it will follow directly that $\mathscr V_\Gamma$ restricts to the claimed equivalence between the categories $\P(B)$ and $\P(A)$ (with quasi-inverse functor given by the restriction of $\mathscr U_\Gamma = \Hom_\Gamma(Q,-)$ to $\P(A))$. Let us first show (\ref{enum:projside:2}). Assume that the $\Lambda$-module $N$ is left and right $B$-projective. Then there is the following isomorphism of $B$-modules.
\begin{align*}
N & \cong \Hom_{B \otimes_k B^\op}(B \otimes_k B^\op, N)\\
&\cong \Hom_{A \otimes_k A^\op}(Q \otimes_\Lambda (B \otimes_k B^\op), Q \otimes_\Lambda N) && \text{($\mathscr V_\Gamma = Q \otimes_\Lambda -$ is an equivalence)}\\
&\cong \Hom_{A \otimes_k A^\op}(P \otimes_k P^\vee, N')\\
&\cong \Hom_A(P, \Hom_{A^\op}(P^\vee, N')) && \text{(by Lemma \ref{lem:isosmoritaR}(\ref{lem:isosmoritaR:1}))}\\
&\cong \Hom_A(P, N' \otimes_A P) && \text{(by Lemma \ref{lem:isodiamonddual})}\\
&= \mathscr U_A (N' \otimes_A P)
\end{align*}
Now note that since $\Hom_A(N' \otimes_A P, -)$ and $\Hom_A(P,\Hom_A(N',-))$ are equivalent functors, the $A$-module $N' \otimes_A P$ is projective. Hence $\mathscr U_A (N' \otimes_A P)$ is $B$-projective, for $\mathscr U_A$ is an equivalence. Similarly,
\begin{align*}
N & \cong \Hom_{B \otimes_k B^\op}(B \otimes_k B^\op, N)\\
&\cong \Hom_{A \otimes_k A^\op}(Q \otimes_\Lambda (B \otimes_k B^\op), Q \otimes_\Lambda N) && \text{($\mathscr V_\Gamma = Q \otimes_\Lambda -$ is an equivalence)}\\
&\cong \Hom_{A \otimes_k A^\op}(P \otimes_k P^\vee, N')\\
&\cong \Hom_{A^\op}(P^\vee, \Hom_{A}(P, N')) && \text{(by Lemma \ref{lem:isosmoritaR}(\ref{lem:isosmoritaR:2}))}\\
&\cong \Hom_{A^\op}(P^\vee, \Hom_{A}(\Hom_{A^\op}(P^\vee,A), N')) && \text{(since $(-)^\vee$ is a duality)}\\
&\cong \Hom_{A^\op}(P^\vee, N' \otimes_{A^\op} P^\vee) && \text{(by Lemma \ref{lem:isodiamonddual})}\\
&\cong \Hom_{A^\op}(P^\vee, P^\vee \otimes_A N')\\
&= \mathscr U_{A^\op} (P^\vee \otimes_A N')
\end{align*}
is a $B^\op$-linear isomorphism. The $A^\op$-module $P^\vee \otimes_A N'$ is projective, and therefore $\mathscr U_{A^\op}(P^\vee \otimes_A N')$ is $B^\op$-projective, as required. The functor $\mathscr U_\Gamma$ admits the following reformulation:
\begin{align*}
\mathscr U_\Gamma &= \Hom_\Gamma(Q,-)\\
&\cong \Hom_\Gamma(\Hom_{\Gamma^\op}(Q^\vee_\Gamma, \Gamma),-)\\
&\cong Q^\vee_\Gamma \otimes_\Gamma - && \text{(by Lemma \ref{lem:isodiamonddual})}
\end{align*}
and, similarly, $\mathscr V_A \cong \Hom_B(P^\vee_{B^\op}, -)$, $\mathscr V_{A^\op} \cong \Hom_{B^\op}((P^\vee_A)^\vee_{B^\op},-)$. Thus the arguments for (\ref{enum:projside:2}) also apply to deduce (\ref{enum:projside:1}). Lastly, the functors $\Hom_\Gamma(Q,-) |_{\P(A)}$ and $Q \otimes_{\Lambda} - \ |_{\P(B)}$ are almost (co)strong monoidal by Proposition \ref{prop:laxmodmorita}.
\end{proof}
\begin{rem}
Supposedly, the statement of the proposition above should remain true, if one replaces $A \otimes_k A^\op$ by any bialgebroid $\mathcal A$ over $A$, $P$ by any comonoidal progenerator $\mathcal P$ in the monoidal category $\Mod(\mathcal A^\natural)$, $B$ by the associated \textit{endomorphism bialgebroid} $\mathcal B :=\End_{\mathcal A}(\mathcal P)^\op$ over the algebra $T := \Hom_{\mathcal A^\natural}(\mathcal P,A)^\op$ and the categories $\P(A)$ and $\P(B)$ by the full subcategories of $\mathcal A^\natural$-modules respectively $\mathcal B^\natural$-modules which are $A$-projective respectively $T$-projective on either side (by $(-)^\natural$ we denote the forgetful functor from the category of bialgebroids to the category of algebras). However, in the special case we have elaborated on in this section, one can now deduce the following satisfying result.
\end{rem}
\begin{thm}\label{thm:moritaHH}
Assume that $A$ and $B$ are Morita equivalent $k$-algebras of which one, and hence both $($due to Lemma $\ref{lem:dereqproj})$, is supposed to be $k$-projective. Then the exact and almost strong monoidal equivalence $\mathscr H := \Hom_\Gamma(Q,-) |_{\P(A)}$ defined in Proposition $\ref{prop:moritamonoequ}$ induces an isomorphism
$$
\mathscr H^\sharp={\mathscr H}_\bullet^\sharp: \HH^\bullet(A) \longrightarrow \HH^\bullet(B)
$$
of graded $k$-algebras such that the diagrams below commute for any choice of integers $m,n \geq 1$.
$$
\xymatrix@C=35pt{
\HH^m(A) \times \HH^n(A) \ar[r]^-{\{-,-\}_A} \ar[d]_-\cong & \HH^{m+n-1}(A) \ar[d]^-\cong \\
\HH^{m}(B) \times \HH^n(B) \ar[r]^-{\{-,-\}_B} & \HH^{m+n-1}(B)
}
\ \quad\quad \
\xymatrix@C=35pt{
\HH^{2n}(A) \ar[r]^{sq_A} \ar[d]_\cong & \HH^{4n-1}(A) \ar[d]^\cong\\
\HH^{2n}(B) \ar[r]^{sq_B} & \HH^{4n-1}(B)
}
$$
Here $\{-,-\}_A$ and $sq_A$ denote the Gerstenhaber bracket and the squaring map on $\HH^\bullet(A)$, whereas $\{-,-\}_B$ and $sq_B$ denote the Gerstehaber bracket and the squaring map on $\HH^\bullet(B)$.
\end{thm}
\begin{proof}
In view of Proposition \ref{prop:moritamonoequ}, the functor $\mathscr H$ induces the morphism $\mathscr H^\sharp : \Ext^\bullet_{\P(A)}(A,A) \rightarrow \Ext^\bullet_{\P(B)}(B,B)$ of graded $k$-algebras. It is an isomorphism, since $\mathscr H$ is an equivalence of categories. Moreover, since both $A$ and $B$ are $k$-projective, the maps $\chi_A$ and $\chi_B$ define isomorphisms $\HH^\bullet(A) \cong \Ext^\bullet_{\P(A)}(A,A)$ and $\HH^\bullet(B) \cong \Ext^\bullet_{\P(B)}(B,B)$ (see Section \ref{sec:basdef} and Corollary \ref{cor:isohochschildproj}). Finally, the commutativity of the diagram is ensured by combining Theorem \ref{thm:bracketcomm} with Theorem \ref{thm:schwede_comm}.
\end{proof}
We will close this chapter with a short survey on the existence of braidings on the monoidal category $(\Mod(A^\ev), \otimes_A, A)$.
\section{The monoidal category of bimodules}\label{sec:catbimodules}
\begin{nn}
Let $A$ be an algebra over the commutative ring $k$. Regarding the considerations of Section \ref{sec:prop_bracket}, where we introduced and studied the bracket for monoidal categories, the question of as to which requirements have to be put on $A$ such that the monoidal category $(\Mod(A^\ev), \otimes_A, A)$ is braided naturally arises. In fact, this subject was treated in the quite recent article \cite{AgCaMi12} where it is actually shown that braidings on $(\Mod(A^\ev), \otimes_A, A)$ are in $1$-$1$ correspondence with certain elements of the $3$-fold tensor product $A \otimes_k A \otimes_k A$ (so called ``canonical R-matrices"). We will shortly explain the results, and relate them to the theory that we have developed so far.
\end{nn}
\begin{nn}
For an element $\mathbf r = \mathbf r_1 \otimes \mathbf r_2 \otimes \mathbf r_3 \in A \otimes_k A \otimes_k A$ consider the following five equations (where implicit summation is understood).
\begin{align}
a\mathbf r_1 \otimes \mathbf r_2 \otimes \mathbf r_3 &= \mathbf r_1 \otimes \mathbf r_2 a \otimes \mathbf r_3,\label{ali:matrix:1}\\
\mathbf r_1 \otimes a\mathbf r_2 \otimes \mathbf r_3 &= \mathbf r_1 \otimes \mathbf r_2 \otimes \mathbf r_3a,\\
\mathbf r_1 \otimes \mathbf r_2 \otimes a\mathbf r_3 &= \mathbf r_1a \otimes \mathbf r_2 \otimes \mathbf r_3,
\intertext{for all $a \in A$, and}
\mathbf r_1 \otimes \mathbf r_2 \otimes 1_A \otimes \mathbf r_3 &= \mathbf r_1\mathbf r_1 \otimes \mathbf r_2 \otimes \mathbf r_3\mathbf r_2 \otimes \mathbf r_3\\
\mathbf r_1 \otimes 1_A \otimes \mathbf r_2 \otimes \mathbf r_3 &= \mathbf r_1 \otimes \mathbf r_2 \mathbf r_1 \otimes \mathbf r_2 \otimes \mathbf r_3 \mathbf r_3.\label{ali:matrix:5}
\end{align}
In the terminology of \cite{AgCaMi12}, an invertible element $\mathbf r \in A \otimes_k A \otimes_k A$ satisfying the above equations is called a \textit{canonical R-matrix for $A$}. By a couple of simple calculations, it may be shown that an invertible element $\mathbf r \in A \otimes_k A \otimes_k A$ is a canonical R-matrix for $A$ if, and only if, it satisfies the following three equations (cf. \cite[Thm.\,3.2]{AgCaMi12}).
\begin{align*}
\mathbf r_1 \otimes a\mathbf r_2 \otimes \mathbf r_3 &= \mathbf r_1 \otimes \mathbf r_2 \otimes \mathbf r_3a \quad \text{(for all $a \in A$)},\\
\mathbf r_1\mathbf r_2 \otimes \mathbf r_3 &= 1_A \otimes 1_A, \\
\mathbf r_2 \otimes \mathbf r_3\mathbf r_1 &= 1_A \otimes 1_A .
\end{align*}
\end{nn}
\begin{defn}
An element $\mathbf r \in A \otimes_k A \otimes_k A$ is called \textit{semi-canonical R-matrix for $A$} if it satisfies the equations (\ref{ali:matrix:1})\,--\,(\ref{ali:matrix:5}), and $\mathbf r_1 \mathbf r_2 \mathbf r_3 = 1_A$.
\end{defn}
From the above remarks it is clear that every canonical R-matrix for $A$ is a semi-canonical R-matrix for $A$. Braidings on the category of bimodules over $A$ are completely classified by the upcoming theorem. Although the given reference claims it in a weaker form (without the parenthesised assertions), its proof shows, that it is valid as stated below.
\begin{thm}[{\cite[Thm.\,3.1]{AgCaMi12}}]\label{thm:corrbraiding}
There is a bijective correspondence between the class of all $($lax$)$ braidings $\gamma$ on $(\Mod(A^\ev), \otimes_A, A)$ and the set of all $($semi-$)$canonical R-matrices for $A$. Every $($semi-$)$canonical R-matrix $\mathbf r \in A \otimes_k A \otimes_k A$ gives rise to a $($lax$)$ braiding $\gamma^{\mathbf r}$ as follows:
$$
\gamma^{\mathbf r}_{M,N}: M \otimes_A N \longrightarrow N \otimes_A M, \ \gamma^{\mathbf r}_{M,N}(m \otimes n) = \mathbf r_1 n \mathbf r_2 \otimes m \mathbf r_3.
$$
\end{thm}
\begin{nn}\label{nn:PlPrP}
Recall from Example \ref{exa:bimodules} that the full subcategory
$$
\P(A) := \{ M \in \Mod(A^\ev) \mid \text{$M$ is a projective left and right $A$-module}\}
$$
of $\Mod(A^\ev)$ is a monoidal category (with inherited monoidal structure from the monoidal category $(\Mod(A^\ev), \otimes_A, A)$). Likewise, the full subcategory
$$
\mathsf F(A) := \{M \in \Mod(A^\ev) \mid \text{$M$ is a flat left and right $A$-module}\}
$$
of $\Mod(A^\ev)$ is monoidal (since, for instance, $\big((M \otimes_A N) \otimes_A -\big) \cong (M \otimes_A - ) \circ (N \otimes_A -)$). Clearly we have $\P(A) \subseteq \mathsf F(A) \subseteq \Mod(A^\ev)$. In addition to $\P(A)$ and $\mathsf F(A)$ let us also consider the full subcategories
\begin{align*}
\P_\lambda(A) &:= \{ M \in \Mod(A^\ev) \mid \text{$M$ is a projective left $A$-module}\}\\
\P_\varrho(A) &:= \{ M \in \Mod(A^\ev) \mid \text{$M$ is a projective right $A$-module}\}
\intertext{and}
\mathsf F_\lambda(A) &:= \{M \in \Mod(A^\ev) \mid \text{$M$ is a flat left $A$-module}\}\\
\mathsf F_\varrho(A) &:= \{M \in \Mod(A^\ev) \mid \text{$M$ is a flat right $A$-module}\}
\end{align*}
of $\Mod(A^\ev)$. Yet again, these are monoidal subcategories and $\P(A) = \P_\lambda(A) \cap \P_\varrho(A)$, $\mathsf F(A) = \mathsf F_\lambda(A) \cap \mathsf F_\varrho(A)$. By basically copying the arguments given in \cite{AgCaMi12} we are able to prove the following surprising (and very elementary) result.
\end{nn}
\begin{cor}\label{cor:braid_flatprojbraid}
Consider the following statements on the $k$-algebra $A$.
\begin{enumerate}[\rm(1)]
\item\label{cor:braid_flatprojbraid:1} The monoidal category $(\Mod(A^\ev), \otimes_A, A)$ is $($lax$)$ braided.
\item\label{cor:braid_flatprojbraid:2l} The monoidal category $(\mathsf F_\lambda(A), \otimes_A, A)$ is $($lax$)$ braided.
\item\label{cor:braid_flatprojbraid:2r} The monoidal category $(\mathsf F_\varrho(A), \otimes_A, A)$ is $($lax$)$ braided.
\item\label{cor:braid_flatprojbraid:2} The monoidal category $(\mathsf F(A), \otimes_A, A)$ is $($lax$)$ braided.
\item\label{cor:braid_flatprojbraid:3l} The monoidal category $(\P_\lambda(A), \otimes_A, A)$ is $($lax$)$ braided.
\item\label{cor:braid_flatprojbraid:3r} The monoidal category $(\P_\varrho(A), \otimes_A, A)$ is $($lax$)$ braided.
\item\label{cor:braid_flatprojbraid:3} The monoidal category $(\P(A), \otimes_A, A)$ is $($lax$)$ braided.
\end{enumerate}
Then the implications 
\begin{gather*}
(\ref{cor:braid_flatprojbraid:1}) \Longrightarrow (\ref{cor:braid_flatprojbraid:2l}) \Longrightarrow (\ref{cor:braid_flatprojbraid:3l}), \quad (\ref{cor:braid_flatprojbraid:1}) \Longrightarrow (\ref{cor:braid_flatprojbraid:2r}) \Longrightarrow (\ref{cor:braid_flatprojbraid:3r}),\quad (\ref{cor:braid_flatprojbraid:1}) \Longrightarrow (\ref{cor:braid_flatprojbraid:2}) \Longrightarrow (\ref{cor:braid_flatprojbraid:3}),\\ (\ref{cor:braid_flatprojbraid:2l}) \Longrightarrow (\ref{cor:braid_flatprojbraid:2}),\quad (\ref{cor:braid_flatprojbraid:2r}) \Longrightarrow (\ref{cor:braid_flatprojbraid:2}), \quad (\ref{cor:braid_flatprojbraid:3l}) \Longrightarrow (\ref{cor:braid_flatprojbraid:3}) \quad \text{and} \quad (\ref{cor:braid_flatprojbraid:3r}) \Longrightarrow (\ref{cor:braid_flatprojbraid:3})
\end{gather*}
hold true. If $A$ is $k$-flat, then $(\ref{cor:braid_flatprojbraid:2}) \Longrightarrow (\ref{cor:braid_flatprojbraid:1})$. If $A$ is $k$-projective, then $(\ref{cor:braid_flatprojbraid:3}) \Longrightarrow (\ref{cor:braid_flatprojbraid:1})$, and hence all statements are equivalent.
\end{cor}
\begin{proof}
The implications $(\ref{cor:braid_flatprojbraid:1}) \Longrightarrow (\ref{cor:braid_flatprojbraid:2l}) \Longrightarrow (\ref{cor:braid_flatprojbraid:3l})$, $(\ref{cor:braid_flatprojbraid:1}) \Longrightarrow (\ref{cor:braid_flatprojbraid:2r}) \Longrightarrow (\ref{cor:braid_flatprojbraid:3r})$, $(\ref{cor:braid_flatprojbraid:1}) \Longrightarrow (\ref{cor:braid_flatprojbraid:2}) \Longrightarrow (\ref{cor:braid_flatprojbraid:3})$, $(\ref{cor:braid_flatprojbraid:2l}) \Longrightarrow (\ref{cor:braid_flatprojbraid:2})$, $(\ref{cor:braid_flatprojbraid:2r}) \Longrightarrow (\ref{cor:braid_flatprojbraid:2})$, $(\ref{cor:braid_flatprojbraid:3l}) \Longrightarrow (\ref{cor:braid_flatprojbraid:3})$ and $(\ref{cor:braid_flatprojbraid:3r}) \Longrightarrow (\ref{cor:braid_flatprojbraid:3})$ are valid for trivial reasons. For the converse implications, we show the following slightly more general statement: Let $\C \subseteq \Mod(A^\ev)$ be a $k$-linear and full subcategory such that $A$, $A \otimes_k A \in \C$, and such that  $M \otimes_A N \in \C$ for all $M, N \in \C$. Then, if $(\C, \otimes_A, A)$ is (lax) braided, so is $(\Mod(A^\ev), \otimes_A, A)$. This then will finish the proof, since $A \otimes_k A \in \mathsf F(A)$ if $A$ is $k$-flat, and $A \otimes_k A \in \P(A)$ if $A$ is $k$-projective.

Let $\gamma$ be a lax braiding on $(\C, \otimes_A, A)$. In view of theorem \ref{thm:corrbraiding}, it is enough to name a (semi-)canonical R-matrix for $A$. Let $\mu: A \otimes_A A \rightarrow A$ be the isomorphism $a \otimes b \mapsto ab$. Consider the homomorphism
$$
\gamma^\circ := (A \otimes_k \mu^{-1} \otimes_k A) \circ \gamma_{A \otimes_k A, A \otimes_k A} \circ (A \otimes_k \mu \otimes_k A).
$$
We claim that $\mathbf r:= \gamma^\circ(1_A \otimes 1_A \otimes 1_A) \in A \otimes_k A \otimes_k A$ is a semi-canonical R-matrix for $A$. Let $M$, $N$ and $P$ be modules in $\C$. For fixed elemts $m \in M$ and $n \in N$, consider the $A^\ev$-linear maps $f_m: A \otimes_k A \rightarrow M$ and $g_n: A \otimes_k A \rightarrow N$ given by $f_m(a \otimes b) = amb$ and $g_n(a \otimes b) = anb$. Since $\gamma$ is a natural tranformation, we get
\begin{equation}\label{eq:nat_braided}
(g_n \otimes_A f_m) \circ \gamma_{A \otimes_k A, A \otimes_k A} = \gamma_{M,N} \circ (f_m \otimes_A g_n).
\end{equation}
Furthermore, the following equations hold true.
\begin{align*}
\gamma_{M,N}(am \otimes n) &= a\gamma_{M,N}(m \otimes n), \\
\gamma_{M,N}(m \otimes na) &= \gamma_{M,N}(m \otimes n)a, \\
\gamma_{M,N}(ma \otimes n) &= \gamma_{M,N}(m \otimes an),
\intertext{for all $a \in A$, $m \in M$, $n \in N$, and}
\gamma_{M \otimes_A N, P} &= (\gamma_{M,P} \otimes_A N) \circ (M \otimes_A \gamma_{N,P}),\\
\gamma_{M, N \otimes_A P} &= (N \otimes_A \gamma_{M,P}) \circ (\gamma_{M,N} \otimes_A P),\\
\gamma_{A,A} &= \mu^{-1} \circ \mu = \id_{A \otimes_A A}.
\end{align*}
By specializing to $M = N = P = A \otimes_k A$ one now can easily deduce the desired equations (also use equation (\ref{eq:nat_braided}) for $m = n = 1_A \otimes 1_A$). In case $\gamma^\circ$ is invertible (for instance, take $\gamma$ to be a braiding), the semi-canonical R-matrix $\mathbf r = \gamma^\circ(1_A \otimes 1_A \otimes 1_A)$ will be invertible, that is, it will be a canonical R-matrix, with $\mathbf r^{-1} = (\gamma^\circ)^{-1}(1_A \otimes 1_A \otimes 1_A)$.
\end{proof}
\begin{nn}
In the commutative world as well as in the world of finite dimensional algebras over a field, the question whether a canonical R-matrix does exist, has been answered by the following classification result (see \cite{AgCaMi12}).
\begin{enumerate}[\rm(1)]
\item Let $A$ be a commutative $k$-algebra. Then $\mathbf r \in A \otimes_k A \otimes_k A$ is a canonical R-matrix for $A$ if, and only if, $\mathbf r = 1_A \otimes 1_A \otimes 1_A$ and the unit map $k \rightarrow A$ is an epimorphism in the category of rings.
\item Let $k$ be a field and $A$ a finite dimensional algebra over $k$. Then there is a (necessarily unique) canonical R-matix for $A$ if, and only if, $A$ is a simple $k$-algebra and $Z(A) = k$.
\end{enumerate}
From the Theorems \ref{thm:bracketcomm} and \ref{thm:schwede_comm} we deduce the following.
\begin{cor}\label{cor:vanishing_gersten_canbraid1}
Assume that there is a semi-canonical R-matrix for $A$ and that $A$ is $k$-projective. Then if $A'$ is Morita equivalent to $A$, the Gerstenhaber bracket and the squaring map on $\HH^\bullet(A')$ are trivial in degrees $m, n \geq 1$. \qed
\end{cor}
\begin{cor}\label{cor:vanishing_gersten_canbraid}
The Gerstenhaber bracket and the squaring map on $\HH^\bullet(A)$ are trivial in the following situations.
\begin{enumerate}[\rm(1)]
\item The algebra $A$ is $($Morita equivalent to an algebra which is$)$ commutative, $k$-projective and the unit map $k \rightarrow A$ is an epimorphism in the category of rings.
\item The base ring $k$ is a field and $A$ is a finite dimensional simple algebra over $k$ whose center agrees with $k$.\qed
\end{enumerate} 
\end{cor}
\end{nn}
\begin{nn}
Let $k$ be an algebraically closed field, and $A$ a finite dimensional algebra over $k$ which admits a canonical R-matrix. In this situation the statement of Corollary \ref{cor:vanishing_gersten_canbraid1} can be deduced easily from \cite[Lem.\,1.5]{Ha89}. In fact, it follows that $\HH^\bullet(A)$ is concentrated in degree zero (for $\gldim(A) = 0$). \cite[Lem.\,1.5]{Ha89} has been generalized to perfect fields by R.\,Rouquier in \cite{Rou08}. For non-perfect fields, the situation is different, as the first out of the following three examples shows.
\begin{exa}
Assume that $k$ is a field with $\text{char}(k) = 2$. Consider the following purely inseparable field extension:
$$
K = k(t) \subseteq \frac{k(t)[u]}{(u^2-t)} = L.
$$
Then the $K$-algebra $L$ has a 2-periodic minimal projective resolution over its (local) enveloping algebra
$$
L^\ev = L \otimes_K L \cong \frac{k(t)[u',u'']}{(u'^2-t, u''^2 - t)} \cong \frac{L[v]}{(v-u)^2},
$$
from which one can deduce, that $\HH^\bullet(L)$ is infinite dimensional over $K$ (in fact, $\HH^n(L) \cong L$ for all $n \geq 0$).
\end{exa}
\begin{exa}
Let $f = t^2 + 3t +1 \in \mathbb Z[t]$ and consider the commutative $\mathbb Z$-algebra $\Lambda = \Lambda_f := \mathbb Z[t]/(f)$. As a $\mathbb Z$-module, it is free of rank $2$. The equivalence class of the formal derivative $f' = 2t + 3$ of $f$ in $\Lambda$ is a non-zero devisor in $\Lambda$, but not a unit. It follows (from \cite[Prop.\,2.2]{Holm00}) that the Hochschild cohomology ring $\HH^\bullet(\Lambda)$ of $\Lambda$ is concentrated in even degrees, wherein it is given by
$$
\HH^{2n}(\Lambda) \cong \Lambda / f' \Lambda \neq 0 \quad \text{(for all $n \geq 0$)}.
$$
Therefore the Gerstenhaber bracket and the squaring map on $\HH^\bullet(\Lambda)$ are trivial. However, the unit map $\mathbb Z \rightarrow \Lambda$ is not an epimorphism in the category of rings (use the criterion \cite[Prop.\,17.2.6]{Gr67}, or notice that the clearly distict ring homomorphisms $f, g: \Lambda \rightarrow \Lambda$ given by $f(t) = t$, $g(t) = -t-3$ are equalized by $\mathbb Z \rightarrow \Lambda$) and hence there cannot be a canonical R-matrix for $\Lambda$.
\end{exa}
\begin{exa}
For an example of a finite dimensional algebra over a field, that does not admit a (semi-)canonical R-matrix, but whose strict Gerstenhaber structure is trivial, let $\overrightarrow{\Delta}$ be a finite quiver without an oriented cycle. Assume further that $\overrightarrow{\Delta}$ is a tree and that $k$ is an algebraically closed field (of any characteristic). Then $\HH^n(k\overrightarrow{\Delta}) = 0$ for all $n \geq 1$, as D.\,Happel showed in \cite{Ha89}. But, of course, $k\overrightarrow{\Delta}$ will almost never be simple.
\end{exa}
Having these examples at hand, it appears that the vanishing of the (strict) Gerstenhaber structure on $\HH^\bullet(A)$ is not correlated to $(\Mod(A^\ev), \otimes_A, A)$ (and hence to $(\P(A), \otimes_A, A)$; see Corollary \ref{cor:braid_flatprojbraid}) being lax braided in general. The following question therefore does not seem to have an answer in terms of lax braidings on a monoidal category of bimodules over $A$.
\end{nn}
\begin{quest}\label{quest:vanishingmean}
Let $A$ be an algebra over the commutative ring $k$. What does it mean for the algebra $A$ $($and its category of modules$)$ that the Gerstenhaber bracket $($and the squaring map$)$ on $\HH^\bullet(A)$ vanish?
\end{quest}



\chapter{Application I: The kernel of the Gerstenhaber bracket}\label{cha:app1}
\section{Introduction and motivation}\label{int:kernel}
\begin{nn}
Let $k$ be a commutative ring and $A$ be an associative and unital $k$-algebra. Let $\{-,-\}_A$ be the Gerstenhaber bracket on the corresponding Hochschild cohomology ring $\HH^\bullet(A)$. The initial spark leading to the results of this chapter was the following fundamental question raised by R.\,-O.\,Buchweitz.
\begin{quote}
\textit{Is it, by any means, possible to deduce those} (\textit{homogeneos}) \textit{elemenets $f \in \HH^\bullet(A)$ such that $\{f,-\}_A = 0$}? \textit{In other words, can we describe} (\textit{parts}) \textit{of the kernel of the adjoint representation
$$
\mathrm{ad}: \HH^\bullet(A) \longrightarrow \Der_k^\mathbb{Z}(\HH^\bullet(A)), \ f \mapsto \{f,-\}_A?
$$}
\end{quote}
By combining results acquired in the previous chapter with S.\,Schwede's interpretation of the bracket presented in \cite{Sch98} (see also Theorem \ref{thm:schwede_comm}), we will be able to determine a considerable part of $\Ker(\{-,-\}_A) \subseteq \HH^\bullet(A) \otimes_k \HH^\bullet(A)$ for a class of interesting $k$-algebras (containing, for instance, cocommutative Hopf algebras which are $k$-projective), pointing into the direction of the above question. Let us give further motivation.
\end{nn}
\begin{nn}
Assume that $k$ is a field of prime characteristic $p$ and let $G$ be a group with identity element $e_G$. The group algebra $kG$ of $G$ carries the structure of a cocommutative Hopf algebra over $k$ with $1_{kG} = e_G$, comultiplication $\Delta(g) = g \otimes g$ and counit $\varepsilon: kG \rightarrow k$, $\varepsilon(g) = 1_k$ (for $g \in G$). There are two graded commutative $k$-algebras associated to $kG$, namely, the Hochschild cohomology ring
$$
\HH^\bullet(kG) \cong \Ext^\bullet_{kG^\ev}(kG,kG)
$$
of $kG$ (note that $kG$ is $k$-projective) and, as for any $k$-Hopf algebra, the $\Ext$-algebra
$$
\OH^\bullet(G,k) = \Ext^\bullet_{kG}(k,k)
$$
of the trivial $kG$-module $k$. If $G$ is finite and abelian, it was shown by T.\,Holm in \cite{Ho96} that these graded $k$-algebras are related as follows:
\begin{equation}\tag{$\dagger$}\label{eq:isocoh}
\HH^\bullet(kG) \cong kG \otimes_k \OH^\bullet(G,k).
\end{equation}
In \cite{CiSo97}, C.\,Cibils and A.\,Solotar generalized the result to arbitrary base rings. The isomorphism is a very explicit one and may be formulated in terms of the defining (standard) resolutions of $kG$ and $k$ respectively. However, it seemed to be unknown for quite a while how to translate the Gerstenhaber structure on $\HH^\bullet(kG)$ to $kG \otimes_k \OH^\bullet(G,k)$ along this isomorphism in an intrinsic manner. In \cite{Sa12} S.\,Sánchez-Flores established a graded Lie structure on $kG \otimes_k \OH^\bullet(G,k)$ such that the isomorphism (\ref{eq:isocoh}) is an isomorphism of Gerstenhaber algebras. When specializing to the case $G \cong \mathbb Z / r\mathbb Z$ for some integer $r \geq 1$ divisible by $p$, one can describe the bracket in terms of a certain $k$-basis of $\OH^\bullet(G,k)$. More concretely, the bracket on the right hand side in (\ref{eq:isocoh}) shows up as follows (cf. \cite[Theorem 5.5]{Sa12}). 

Assume that, as above, $G$ is cyclic and finite whose order \abs{G} is divided by $p$. Let $g \in G$ be a generator. Further, let $\beta: kG \rightarrow k$ be the map $\beta(g^0) = 0$,
$$
\beta(g^i) = \varepsilon(g^0 + g + \cdots + g^{i-1}) \quad (\text{for $1 \leq i \leq n-1$}).
$$
Since $G$ is cyclic, it admits a total ordering of its elements, namely, define
$$
e_G < g < g^2 < \cdots < g^{\abs{G}-1}.
$$
The hereby defined sets $Q(x)$,
$$
Q(x) = \{ y \in G \mid xy < y\} \quad (\text{for $x \in G$}),
$$
have the following properties: $Q(e_G) = \emptyset$, $Q(g) = \{g^{p-1}\}$ and $Q(g^{p-1}) = G \setminus \{e_G\}$. Let $n \geq 2$ be an integer and $\underline{x} = (x_1, \dots, x_n) \in G^{\times n}$. By definition, $\underline{x}$ satisfies condition $C(n)$ if
\begin{itemize}
\item $x_i \in Q(x_{i+1})$ for all odd $i$, $1 \leq i \leq n-1$, in case $n$ is even, or
\item $x_i \in Q(x_{i+1})$ for all even $i$, $1 \leq i \leq n-1$, in case $n$ is odd.
\end{itemize}
Having made this definition, one can show that for fixed $n \geq 1$, $\dim_k \OH^n(G,k) = 1$ and that the map $\beta^n \in \Map(G^{\times n}, k) \cong \Hom_{kG}(kG^{\otimes_k n}, k)$,
$$
\beta^n(\underline{x}) =
\begin{cases}
1 & \text{if $n$ is even and $\underline{x}$ satisfies $C(n)$},\\
\beta(x_1) & \text{if $n$ is odd and $\underline{x}$ satisfies $C(n)$,}\\
0 & \text{otherwise},
\end{cases}
$$
(for $\underline{x} = (x_1, \dots, x_n) \in G^{\times n}$) defines a non-zero element in $\OH^n(G,k)$; hence $(\beta^n)_{n \geq 0}$ is a $k$-basis of $\OH^\bullet(G,k)$. With this basis at hand, the Gerstenhaber bracket on $kG \otimes_k \OH^\bullet(G,k)$ displays itself as
\begin{equation}\tag{$\dagger\dagger$}\label{eq:bracketdes}
\{x \otimes \beta^m, y \otimes \beta^n\}_{kG} = xy \otimes (\varphi(x,\beta^m) - \varphi(y, \beta^n)) \beta^{m+n-1}
\end{equation}
(\text{for $m,n \geq 1$, $x,y \in G$}), where $\varphi: G \times \OH^\bullet(G,k) \rightarrow k$ is the map
$$
\varphi(x, \beta^n) =
\begin{cases}
\beta(x) & \text{if $n$ is odd},\\
0 & \text{otherwise}
\end{cases}
\quad (\text{for $x \in G$, $n \geq 1$}).
$$
\textit{Observation}: For $n \geq 1$ and $x \in G$, let $\underline{x}(n)$ be the element $(x,e_G, \dots, e_G) \in G^{\times n}$. If $n$ is even, $\underline{x}(n)$ will not satisfy condition $C(n)$ meaning that $\beta^n(\underline{x}(n)) = 0$. Conversely, if $n$ is odd and $n \geq 3$, then $\underline{x}(n)$ satisfies condition $C(n)$. Hence the map $\varphi$ is actually given by evaluation at $\underline{x}(n)$:
$$
\varphi(x, \beta^n) = \beta^n(\underline{x}(n)) \quad (\text{for $x \in G$, $n \geq 1$}).
$$
Now specialize to $x = e_G$; by definition, $\beta(x) = 0$. Therefore $\varphi(x, \beta^n) = \beta^n(\underline{x}(n)) = 0$ for all $n \geq 1$.

\vspace*{11pt}

By mapping $\xi \in \OH^\bullet(G,k)$ to $1_{kG} \otimes \xi$, $\OH^\bullet(G,k)$ can be identified with a graded subalgebra of $kG \otimes_k \OH^\bullet(G,k) \cong \HH^\bullet(kG)$. From the above observation and the description (\ref{eq:bracketdes}) of the Gerstenhaber bracket on $kG \otimes_k \OH^\bullet(G,k)$, it is evident that
$$
\{1_{kG} \otimes \xi, 1_{kG} \otimes \xi'\}_{kG} = 0 \quad (\text{for $G$ cyclic and $\xi$, $\xi' \in \OH^\bullet(G,k)$}),
$$
that is, 
\begin{equation}\tag{$\#$}\label{eq:hhker}
\begin{aligned}
\OH^\bullet(G,k) \otimes_k & \OH^\bullet(G,k) \subseteq \Ker(\{-,-\}_{kG}),
\end{aligned}
\end{equation}
where we regard $\{-,-\}_{kG}$ as a $k$-linear map $\HH^\bullet(kG) \otimes_k \HH^\bullet(kG) \rightarrow \HH^\bullet(kG)$. The statement (\ref{eq:hhker}) can be generalized massively, in the following sense.
\begin{thm}[$\subseteq$ Corollary \ref{cor:gerstenhabervanishhopfalgebra}]
Let $G$ be a group and let $kG$ be the corresponding group algebra over the commutative ring $k$. Then
$$
\OH^\bullet(G,k) \otimes_k \OH^\bullet(G,k) \subseteq \Ker(\{-,-\}_{kG}),
$$
where we view $\{-,-\}_{kG}$ as a $k$-linear map $\HH^\bullet(kG) \otimes_k \HH^\bullet(kG) \rightarrow \HH^\bullet(kG)$.
\end{thm}
The aim of this chapter is to prove the above result not only for group algebras, but rather for any quasi-triangular Hopf algebroid satisfying some rather mild projectivity requirements.
\end{nn}
\section{Bialgebroids}\label{sec:hopfalgebroids}
Let $k$ be a commutative ring and let $R$ be a $k$-algebra. In this section, we introduce the natural generalization of bialgebras over non-commutative base rings. Most of the upcoming definitions and results are extracted from \cite{Bo08}, \cite{DoMu06} and \cite{Ko09}. As usual, we let $R^\ev = R \otimes_k R^\op$ be the enveloping algebra of $R$. Remember that left $R^\ev$-modules bijectively correspond to $R$-$R$-bimodules on which $k$ acts centrally.
\begin{defn}
Let $A$ be an $R^\ev$-module and let $\nabla: A \otimes_R A \rightarrow A$, $\eta: R \rightarrow A$ be $R^\ev$-module homomorphisms. The triple $\mathcal A = (A, \nabla, \eta)$ is called an \textit{$R$-ring} if the diagrams given below commute.
\begin{equation}\tag{R1}
\begin{aligned}
\xymatrix@C=30pt{
A \otimes _R (A \otimes_R A) \ar[r]^-{A \otimes_R \nabla} \ar[d]_-\cong & A \otimes_R A \ar[r]^-{\nabla} & A \ar@{=}[d]\\
(A \otimes_R A) \otimes_R A \ar[r]^-{\nabla \otimes_R A} & A \otimes_R A \ar[r]^-{\nabla} & A
}
\end{aligned}
\end{equation}
\begin{equation}\tag{R2}
\begin{aligned}
\xymatrix{
R \otimes_R A \ar[r]^-\cong \ar[d]_-{\eta \otimes_R A} & A \ar@{=}[d] & A \otimes_R R \ar[l]_-\cong \ar[d]^-{A \otimes_R \eta} \\
A \otimes_R A \ar[r]^-\nabla & A & A \otimes_R A \ar[l]_-\nabla
}
\end{aligned}
\end{equation}
The map $\nabla$ is the \textit{multiplication map} whereas $\eta$ is the \textit{unit map}. Let $(A,\nabla_A, \eta_A)$ and $(B, \nabla_B, \eta_B)$ be $R$-rings. An $R^\ev$-module homomorphism $f: A \rightarrow B$ is called \textit{homomorphism of $R$-rings} if it is compatible with the additional structure that $A$ and $B$ carry, namely, $\nabla_B\circ f = (f \otimes_R f) \circ \nabla_A$ and $f \circ \eta_A = \eta_B$.
\end{defn}
\begin{rem}\label{rem:Rrings}
Let $\mathcal A = (A, \nabla, \eta)$ be an $R$-ring. Then $A$ may be viewed as an associative and unital $k$-algebra with multiplication map obtained by composing the natural epimorphism $A \otimes_k A \rightarrow A \otimes_R A$ with $\nabla: A \otimes_R A \rightarrow A$ and unit map obtained by composing the unit map $k \rightarrow R$ with $\eta: R \rightarrow A$. We will denote the corresponding $k$-algebra to the $R$-ring $(A, \nabla, \eta)$ by $\mathcal A^\natural$ or (by slight abuse of notation) simply by $A$. Having made these observations, the map $\eta: R \rightarrow A$ will be a unital $k$-algebra homomorphism $R \rightarrow \mathcal A^\natural$. In fact, we even have the following result.
\begin{lem}[{\cite[Lem.\,2.2]{Bo08}}]
The assignment which maps an $R$-ring $\mathcal A = (A, \nabla, \eta)$ to the $k$-algebra homomorphism $R \rightarrow \mathcal A^\natural$ defines a bijective correspondence between the classes
\begin{enumerate}[\rm(1)]
\item of all $R$-rings, and
\item $\bigcup_{\Lambda \in \text{$k$-}\mathsf{Alg}}\Hom_{\text{$k$-}\mathsf{Alg}}(R,\Lambda)$,
\end{enumerate}
where $\text{$k$-}\mathsf{Alg}$ denotes the category of $k$-algebras.
\end{lem}
From now on, we will switch between those two characterizations of $R$-rings at will. By slightly abusing notation, we are going to write $A$ for $\mathcal A^\natural$ when there cannot be any misunderstanding.
\end{rem}
\begin{defn}
Let $C$ be an $R^\ev$-module and let $\Delta: C \rightarrow C \otimes_R C$, $\varepsilon: C \rightarrow R$ be $R^\ev$-module homomorphisms. The triple $\mathcal C = (C, \Delta, \varepsilon)$ is called an \textit{$R$-coring} if the diagrams given below commute.
\begin{equation}\tag{CR1}
\begin{aligned}
\xymatrix@C=30pt{
C \otimes _R (C \otimes_R C) & C \otimes_R C \ar[l]_-{C \otimes_R \nabla} & C \ar[l]_-{\nabla}\\
(C \otimes_R C) \otimes_R C \ar[u]^-\cong & C \otimes_R C \ar[l]_-{\nabla \otimes_R C} & C  \ar@{=}[u] \ar[l]_-{\nabla}
}
\end{aligned}
\end{equation}
\begin{equation}\tag{CR2}
\begin{aligned}
\xymatrix{
R \otimes_R C & C \ar[l]_-\cong \ar[r]^-\cong & C \otimes_R R \\
C \otimes_R C \ar[u]^-{\varepsilon \otimes_R C} & C \ar[l]_-\Delta \ar[r]^-\Delta \ar@{=}[u] & C \otimes_R C \ar[u]_-{C \otimes_R \varepsilon} 
}
\end{aligned}
\end{equation}
The map $\Delta$ is the \textit{comultiplication map} whereas $\varepsilon$ is the \textit{counit map}. Let $(C,\Delta_C, \varepsilon_C)$ and $(D, \Delta_D, \varepsilon_D)$ be $R$-corings. an $R^\ev$-module homomorphism $f: C \rightarrow D$ is called \textit{homomorphism of $R$-corings} if it is compatible with the additional structure that $C$ and $D$ carry, namely, $(f \otimes_R f) \circ \Delta_C = \Delta_D \circ f$ and $\varepsilon_D \circ f = \varepsilon_C$.
\end{defn}
\begin{rem}
In contrast to Remark \ref{rem:Rrings}, an $R$-coring does not have to be a $k$-coalgebra in general (which already is evident from the definition). Later examples will illustrate this. 

Let us introduce some notation. If $(A, \nabla, \eta)$ is an $R$-ring, we write $\nabla(a \otimes a') = aa'$ for $a,a' \in A$; in analogy to Sweedler's notation for coalgebras, we write
$$
\Delta(c) = \sum_{(c)}c_{(1)} \otimes c_{(2)} = c_{(1)} \otimes c_{(2)}
$$
if $(C,\Delta,\varepsilon)$ is an $R$-coring and $c$ is in $C$.
\end{rem}
\begin{nn}
Let $\mathcal A = (A, \nabla, \eta)$ be an $R^\ev$-ring; in particular $\eta$ is a map $R^\ev \rightarrow A$. The $k$-algebra homomorphisms
$$
s_A = \eta(- \otimes 1_R): R \longrightarrow A \quad \text{and} \quad t_A = \eta(1_R \otimes -): R^\op\longrightarrow A
$$
are called the \textit{source} and \textit{target maps} of the $R^\ev$-ring $\mathcal A$, and they respectively turn $A$ into a left and a right $R$-module. By
$$
(r \otimes r')(a \otimes a') = s_A(r)a \otimes t_A(r')a' \quad (\text{for $r,r' \in R$ and $a,a' \in A$})
$$
the tensor product $A \otimes_R A$ becomes an $R^\ev$-module. Consider the following $k$-submodule $A \times_R A$ of $A \otimes_R A$:
\begin{align*}
A \times_R A := \left\{ \sum_i{a_i \otimes a_i'} \mathrel{\Big|}  \sum_i a_it_A(r)\otimes a_i' = \sum_i a_i \otimes a_i' s_A(r) \ \forall r \in R\right\} .
\end{align*}
By factorwise multiplication and $\eta_{A \times_R A}(r \otimes r') = s_A(r) \otimes t_A(r')$ (for $r,r' \in R$ and $a,a' \in A$) it is an $R^\ev$-ring. In particular, $A \times_R A$ is a $k$-algebra with unit $1_A \otimes 1_A$.
\end{nn}
\begin{defn}
Let $B$ be a $k$-module. The $5$-tuple $\mathcal B = (B,\nabla, \eta, \Delta, \varepsilon)$ is called a (\textit{left}) \textit{$R$-bialgebroid} if $(B, \nabla, \eta)$ is an $R^\ev$-ring and $(B, \Delta, \varepsilon)$ is an $R$-coring subject to the following compatibility axioms.
\begin{enumerate}[\rm(1)]
\item The $R^\ev$-module structure on $B$ as part of the $R$-coring structure is related to the $R^\ev$-module structure of $B$ as part of the $R^\ev$-ring structure via
$$
(r \otimes r') b = \eta(r \otimes 1_R)\eta(1_R \otimes r') b \quad \text{(for $b \in B$, $r,r' \in R$).}
$$
\item The map $\Delta$ factors through the $k$-algebra $B \times_R B$ such that $\Delta: B \rightarrow B \times_R B$ is a $k$-algebra homomorphism.
\item The counit $\varepsilon: B \rightarrow R$ satisfies
\begin{enumerate}
\item $\varepsilon(\eta(r \otimes 1_R)\eta(1_R \otimes r')b) = r\varepsilon(b)r'$ for $b \in B$ and $r,r' \in R$,
\item $\varepsilon(bb') = \varepsilon(b\eta(\varepsilon(b') \otimes 1_R)) = \varepsilon(b\eta(1_R \otimes \varepsilon(b')))$ for $b,b' \in B$, and
\item $\varepsilon(1_B) = 1_R$.
\end{enumerate}
\end{enumerate}
\end{defn}
Note that there is also the notion of right bialgebroids. We prefer to stick to left bialgebroids (since we work with left modules), and hence will suppress the indication of the side.
\begin{exas}\label{exa:bialg}\begin{enumerate}[\rm(1)]
\item Any $k$-algebra is a $k$-ring, any $k$-coalgebra is a $k$-coring and any $k$-bialgebra is a bialgebroid over $k$.
\item\label{exa:bialg:2} Let $A$ be a $k$-algebra. The enveloping algebra $A^\ev = A \otimes_k A^\op$ is a bialgebroid over $A$. Namely, it is an $A^\ev$-ring via the $k$-algebra homomorphism $\eta = \id_{A^\ev}$ and an $A$-coring via
\begin{align*}
\Delta: A \otimes_k A^\op \longrightarrow (A \otimes_k A^\op) \otimes_A (A \otimes_k A^\op),& \ \Delta(a \otimes a') = (a \otimes 1_A) \otimes (1_A \otimes a')\\
\varepsilon: A \otimes_k A^\op \longrightarrow A,& \ \varepsilon(a \otimes a') = aa'.
\end{align*}
\end{enumerate}
\end{exas}
\begin{nn}
Let $\mathcal B = (B, \nabla, \eta, \Delta, \varepsilon)$ be an $R$-bialgebroid. Then $B$ is equipped with four $R$-module structures in the following ways:
\begin{align*}
r \rhd b \lhd r' &:= \eta(r \otimes r')b,\\
r \RHD b \LHD r' &:= b\eta(r' \otimes r) \quad \text{(for $r,r' \in R$, $b \in B$)}.
\end{align*}
In order to distinguish these structures we will use the following notation:
\begin{align*}
{_\rhd B} & \quad \text{the left $R$-module with structure map $r \rhd b = \eta(r \otimes 1_R)b$;}\\
{B_\lhd} & \quad \text{the right $R$-module with structure map $b \lhd r = \eta(1_R \otimes r)b$;}\\
{_\RHD B} & \quad \text{the left $R$-module with structure map $r \RHD b = b\eta(1_R \otimes r)$;}\\
{B_\LHD} & \quad \text{the right $R$-module with structure map $b \LHD r = b\eta(r \otimes 1_R)$}.
\end{align*}
Let $M$ be a left $B$-module and let $M'$ be a right $B$-module. We set
\begin{align*}
{_\rhd M} &:=  {_\rhd B} \otimes_{B}M, & {_\RHD M'} &:= M' \otimes_B {_\RHD B},
\\
{M_\lhd} &:= B_\lhd \otimes_B M, & {M'_\LHD} &:= M' \otimes_B {B_\LHD} \ .
\end{align*}
When not stated otherwise, we view $B$ as a left $R$-module through ${_\rhd B}$ and as a right $R$-module through ${B_\lhd}$.
\end{nn}
\begin{nn}
Let $\mathcal B = (B,\nabla, \eta, \Delta, \varepsilon)$ be an $R$-bialge\-broid, and let $M$ and $N$ be $B$-modules. The ring $R$ may be turned into a $B$-module by putting
$$
br := \varepsilon(b \LHD r) = \varepsilon(r \RHD b) \quad (\text{for $b \in B$, $r \in R$}).
$$
Note that the requirements on the counit $\varepsilon$ (precisely) guarantee that this indeed gives rise to a well-defined $B$-module structure on $R$. The $R$-module ${_\rhd M_\lhd} \otimes_R {_\rhd N}$ may be turned into a $B$-module by
$$
b(b' \otimes x) := bb' \otimes m \quad (\text{for $b,b' \in B$, $m \in M$}).
$$
By using $\Delta$ one can define an additional $B$-module structure on $M \otimes_R N$. For $b \in B$ and $m \in M$, $n \in N$ set
\begin{align*}
b(m \otimes n) &\mathrel{\mathop:}= \Delta(b)(m \otimes n) = b_{(1)}m \otimes b_{(2)}n.
\end{align*}
Denote this $B$-module by $M \boxtimes_R N$. If $N$ is a $B$-$B$-bimodule (e.g., $N = B$), both modules $M \otimes_R N$ and $M \boxtimes_R N$ are $B$-$B$-bimodules whose right $B$-module structure comes from right multiplication on the factor $N$.
\begin{lem}
Let $\mathcal B = (B,\nabla, \eta, \Delta, \varepsilon)$ be a left $R$-bialgebroid. The category of all $B$-modules is a tensor category with respect to the tensor product $\boxtimes_R$ and the tensor unit $R$. Moreover, the forgetful functor $(\Mod(B), \boxtimes_R,-) \rightarrow (\Mod(R^\ev),\otimes_R,R)$ is strict monoidal. \qed
\end{lem}
In fact, bialgebroid structures on an $R^\ev$-ring $\mathcal B = (B, \nabla, \eta)$ bijectively correspond to monoidal structures on the category $\Mod(B)$ such that the forgetful functor $(\Mod(B), \boxtimes_R,-) \rightarrow (\Mod(R^\ev),\otimes_R,R)$ is strict monoidal (cf. \cite[Thm.\,5.1]{Scha98}). This recovers the classical statement for bialgebras (see Proposition \ref{prop:bialg_monoidal}). We have a natural map
$$
U_B: {_\RHD B} \otimes_{R^\op} B_\lhd \longrightarrow B_\lhd \boxtimes_R {_\rhd B}, \ U_B(b \otimes b') = \Delta(b) (1_B \otimes b')
$$
which is $B^\ev$-linear since $\Delta$ is an algebra homomorphism $B \rightarrow B \times_R B$.
\end{nn}
\begin{defn}\label{def:hopflalgebroid}
A bialgebroid $\mathcal H = (H,\nabla, \eta, \Delta, \varepsilon)$ over $R$ is a \textit{$R$-Hopf algebroid} if $U_H$ is a bijection.
\end{defn}
\begin{exa}
The enveloping algebra of a $k$-algebra $A$ is not only a bialgebroid, but also a Hopf algebroid over $A$.
\end{exa}
We extent our pool of examples by showing that Hopf algebras over $k$ are $k$-Hopf algebroids.
\begin{lem}[{\cite[Prop.\,3.1.5]{Be98}}]\label{lem:boxtimes-otimes}
Let $\mathcal H = (H, \nabla, \eta, \Delta, \varepsilon,S)$ be a Hopf algebra over $k$, and let $M$ be an $H$-module. The $k$-linear map
$$
U_M: H \otimes_k M \longrightarrow H \boxtimes_k M, \ U_M(h \otimes m) = \Delta(h)(1_H \otimes m)
$$
is bijective. Further, $U_M$ is a $H$-$H$-bimodule homomorphism if $M$ is a $H$-$H$-bimodule.
\end{lem}
\begin{proof}
The $B$-linearity follows from the fact that $\Delta$ is a $k$-algebra homomorphism. We show that the map
$$
V_M: H \boxtimes_k M \longrightarrow H \otimes_k M, \ V_M(h \otimes m) =  \sum_{(h)}{h_{(1)} \otimes S(h_{(2)})m}
$$
is the inverse map of $U_M$. Fix $h \in H$ and $m \in M$. First of all, notice that
\begin{align*}
\Delta_S &= (H \otimes_k \nabla) \circ (H \otimes_k S \otimes_k H) \circ (\Delta \otimes_k H) \circ \Delta\\
& = (H \otimes_k \nabla)\circ(H \otimes_k S \otimes_k H)\circ(H \otimes_k \Delta)\circ \Delta && \text{(by the coassociativity)}\\
& = (H \otimes_k \eta) \circ (H \otimes_k \varepsilon) \circ \Delta && \text{(by the antipode axiom)},\\
\Delta'_S &= (H \otimes_k \nabla) \circ (H \otimes_k H \otimes_k S) \circ (\Delta \otimes_k H) \circ \Delta\\
& = (H \otimes_k \nabla)\circ(H \otimes_k H \otimes_k S)\circ(H \otimes_k \Delta)\circ \Delta && \text{(by the coassociativity)}\\
& = (H \otimes_k \eta) \circ (H \otimes_k \varepsilon) \circ \Delta && \text{(by the antipode axiom)},
\end{align*}
and hence $\Delta_S(h) = h \otimes 1_H = \Delta_S'(h)$ by the counitary axiom. It follows that
\begin{align*}
(V_M \circ U_M)(h \otimes m) & = V_M(h_{(1)} \otimes h_{(2)}m)\\
				& = h_{(1,1)} \otimes S(h_{(1,2)})h_{(2)}m\\
				& = \Delta_S(h)(1_H \otimes m)\\
				& = (h \otimes 1_H)(1_H \otimes m)\\
				& = h \otimes m ,
\end{align*}
and conversely,
\begin{align*}
(U_M \circ V_M)(h \otimes m) & = {U_M(h_{(1)} \otimes S(h_{(2)})m)}\\
				& = {h_{(1,1)} \otimes h_{(1,2)}S(h_{(2)})m}\\
				& = \Delta'_S(h)(1_H \otimes m)\\
				& = (h \otimes 1_H)(1_H \otimes m)\\
				&= h \otimes m.
\end{align*}
Thus we have finished the proof.
\end{proof}
\begin{lem}\label{lem:isohopfalgebroid}
Let $\mathcal H = (H,\nabla, \eta, \Delta, \varepsilon)$ be a Hopf algebroid over $R$ and let $M$ be a $H$-module. Then the left $H$-modules ${_\RHD H} \otimes_{R^\op} M_\lhd$ and $H_\lhd \boxtimes_R {_\rhd M}$ are isomorphic.
\end{lem}
\begin{proof}
This is an easy but important observation. From Definition \ref{def:hopflalgebroid} it follows that ${_\RHD H} \otimes_{R^\op} H_\lhd \cong H_\lhd \boxtimes_R {_\rhd H}$; thus,
\begin{align*}
{_\RHD H} \otimes_{R^\op} M_\lhd &= {_\RHD H} \otimes_{R^\op} (H_\lhd  \otimes_ H M)\\
& = ({_\RHD H} \otimes_{R^\op} H_\lhd)  \otimes_ H M\\
& \cong ({H_\lhd} \boxtimes_{R} {_\rhd H}) \otimes_ H M = H_\lhd \boxtimes_R {_\rhd M},
\end{align*}
as required.
\end{proof}
\begin{lem}\label{lem:hopfRporjHproj}
Let $\mathcal H = (H,\nabla, \eta, \Delta, \varepsilon)$ be a Hopf algebroid over $R$ and let $M$ be a $H$-module. ${M_\lhd}$ is a projective right $R$-module if, and only if, $H_\lhd \boxtimes_R {_\rhd M}$ is a projective left $H$-module.
\end{lem}
\begin{proof}
Since $\Hom_H({_\RHD H} \otimes_{R^\op} M_\lhd,-) \cong \Hom_{R^\op}(M_\lhd, \Hom_H({_\RHD H},-))$, this is an immediate consequence of Lemma \ref{lem:isohopfalgebroid}.
\end{proof}
Note that Lemma \ref{lem:hopfRporjHproj} is a generalization of Lemma \ref{lem:kprojproj}.
\section{A monoidal functor}\label{sec:kernel}
\begin{nn}
Let $R$ be an algebra over the commutative ring $k$, and $\mathcal B = (\nabla, \eta, \Delta, \varepsilon)$ be a bialgebroid over $R$. Consider the following full subcategories of $\Mod(B)$:
\begin{align*}
\mathsf C_\lambda(\mathcal B) &= \left\{ M \in \Mod(B) \mid \text{${_\rhd M}$ is a projective left $R$-modules} \right\}\\
\mathsf C_\varrho(\mathcal B) &= \left\{ M \in \Mod(B) \mid \text{$M_\lhd$ is a projective right $R$-modules} \right\}\\
\mathsf C(\mathcal B) &= \left\{ M \in \Mod(B) \mid \text{${_\rhd M}$ and $M_\lhd$ are projective $R$-modules} \right\}\\
&= \mathsf C_\lambda(\mathcal B) \cap \mathsf C_\varrho(\mathcal B)
\end{align*}
Clearly, $R$ belongs to all of the categories $\mathsf C_\lambda(\mathcal B)$, $\mathsf C_\varrho(\mathcal B)$ and $\mathsf C(\mathcal B)$. Note that if $A$ is a $k$-algebra, the categories $\mathsf P_\lambda(\mathcal B)$, $\mathsf P_\varrho(\mathcal B)$ and $\mathsf P(A)$ (see \ref{nn:PlPrP}) coincide with $\mathsf C_\lambda(\H_A)$, $\mathsf C_\varrho(\H_A)$ and $\C(\H_A)$ respectively, where $\H_A$ denotes the $A$-Hopf algebroid described in Example \ref{exa:bialg}(\ref{exa:bialg:2}). In particular, any statement claimed for $\mathsf C_\lambda(\mathcal B)$, $\mathsf C_\varrho(\mathcal B)$ and $\mathsf C(\mathcal B)$ will, if true, automatically also hold for $\mathsf P_\lambda(A)$, $\mathsf P_\varrho(A)$ and $\mathsf P(A)$. Despite the potential redundancy, we will repeat certain assertions for $\mathsf P_\lambda(A)$, $\mathsf P_\varrho(A)$ and $\mathsf P(A)$ individually, on the one hand in order to emphasize, but on the other hand also because of slight changes in the required assumptions. Notice that if $\mathcal B$ is a bialgebra over $k$, $\mathsf C_\lambda(\mathcal B) = \mathsf C_\varrho(\mathcal B) = \mathsf C(\mathcal B)$ recovers the category $\mathsf C_{\Proj(k)}(\mathcal B)$ defined in Example \ref{exa:bialgebras}.

Keep in mind, that if $\scrX : \C \rightarrow \D$ is an exact functor between exact categories, we denote by $\scrX^\sharp_n$ and $\scrX^\flat_n$ the induced maps $\pi_0 \scrX_n$ and $\pi_1 \scrX_n$,
$$
\Ext^n_\C(C,D) \longrightarrow \Ext^n_\D(\scrX C, \scrX D),
$$
$$
\pi_1 \mathcal Ext^n_\C(C,D) \longrightarrow \pi_1 \mathcal Ext^n_\D(\scrX C, \scrX D)
$$
(for $C,D \in \Ob \C$). The corresponding graded maps $\scrX^\sharp_\bullet$ and $\scrX^\flat_\bullet$ will simply be denoted by $\scrX^\sharp$ and $\scrX^\flat$.
\end{nn}
\begin{lem}\label{lem:candaexactmono}
\begin{enumerate}[\rm(1)]
\item\label{lem:candaexactmono:1} The categories $\mathsf C_\lambda(\mathcal B)$, $\mathsf C_\varrho(\mathcal B)$ and $\mathsf C(\mathcal B)$ are closed under arbitrary direct sums, direct summands and taking extensions in $\Mod(B)$. In particular, they are $k$-linear exact categories, and as such, they are closed under kernels of epimorphisms.
\item\label{lem:candaexactmono:2} The triples $(\C_\lambda(\cB), \boxtimes_R, R)$, $(\C_\varrho(\cB), \boxtimes_R, R)$ and $(\C(\cB), \boxtimes_R, R)$ are $($strict$)$ mono\-idal subcategories of $(\Mod(B), \boxtimes_R, R)$. Moreover, the induced monoidal structures turn $\C_\lambda(\cB)$ and $\C_\varrho(\cB)$ into strong exact monoidal categories, and $\C(\cB)$ into a very strong exact monoidal category.
\end{enumerate}
\end{lem}
\begin{proof}
Arbitrary sums and summands of projectives are projective. If
$$
\xi \quad \equiv \quad 0 \longrightarrow N \longrightarrow X \longrightarrow M \longrightarrow 0
$$
is an exact sequence in $\Mod(B)$ whose outer terms belong to $\C_\lambda(\cB)$, then ${_\rhd \xi}$ defines a spilt exact sequences of left $R$-modules. Hence $X$ belongs to $\C_\lambda(\cB)$, and $\C_\lambda(\cB)$ is extension closed. Finally, if $f:M \rightarrow N$ is an epimorphism in $\Mod(B)$ with $M$ and $N$ belonging to $\C_\lambda(\cB)$, the $R$-module homomorphism ${_\rhd f}$ splits, and therefore ${_\rhd \Ker(f)}$ is a projective $R$-modules, i.e., $\C_\lambda(\cB)$ is closed under kernels of epimorphisms. Completely analogous arguments apply to $\mathsf C_\varrho(\mathcal B)$. The claimed properties of $\mathsf C(\mathcal B)$ now follow from the fact that it coincides with $\mathsf C_\lambda(\mathcal B) \cap \mathsf C_\varrho(\mathcal B)$; therefore (\ref{lem:candaexactmono:1}) is established.

Since proven similarly for $\mathsf C_\lambda(\mathcal B)$ and $\mathsf C_\varrho(\mathcal B)$, we only prove (\ref{lem:candaexactmono:2}) for the category $\mathsf C(\mathcal B)$. Let $M$ and $N$ be $B$-modules belonging to $\C(\cB)$. We have to convince ourselves that both ${_\rhd X}$ and ${X_\lhd}$ are $R$-projective, where $X := M \boxtimes_R N$. But this immediately follows from the following isomorphisms of functors:
\begin{align*}
\Hom_R({_\rhd X}, -) &\cong \Hom_R({_\rhd N}, \Hom_R({_\rhd M},-)),\\
\Hom_{R^\op}({X_\lhd}, -) &\cong \Hom_{R^\op}({M_\lhd}, \Hom_{R^\op}({N_\lhd},-)).
\end{align*}
Since $R$ automatically is contained in $\C(\cB)$, the triple $(\C(\cB), \boxtimes_R, R)$ is indeed a monoidal subcategory of $(\Mod(B), \boxtimes_R, R)$. To finish the proof, use Corollary \ref{cor:tensorexact} and recall that projective modules are flat.
\end{proof}
\begin{nn}\label{nn:requirements}
From now on, we will assume that the base algebra $R$ and the $R$-bialgebroid $\mathcal B$ satisfy the following conditions.
\begin{enumerate}[\rm(i)]
\item\label{item:requirements:1} The $k$-algebra $R$ is a projective $k$-module.
\item\label{item:requirements:2} The $R$-modules ${_\rhd B}$ and $B_\lhd$ are projective, i.e., $B$ belongs to $\mathsf C_\lambda(\mathcal B)$, $\mathsf C_\varrho(\mathcal B)$ and $\mathsf C(\cB)$.
\end{enumerate}
Requirements (\ref{item:requirements:1})--(\ref{item:requirements:2}) imply that $B$ is a projective $k$-module. Recall that $B$ being $k$-projective means that $\Proj(B^\ev)$ belongs to $\mathsf P(B)$.
\end{nn}
\begin{exa}
Let $A$ be a $k$-algebra which is projective over $k$. Then the $A$-bialgebroid $\H_A$ associated to the enveloping algebra $A^\ev$ of $A$ satisfies \ref{nn:requirements}(\ref{item:requirements:1})--(\ref{item:requirements:2}). We will encounter further examples later on.
\end{exa}
\begin{defn}
Let $M$ be a $B$-module. The graded $k$-module $\OH^\bullet(\mathcal B,M) = \Ext^\bullet_{B}(R,M)$ is the \textit{graded cohomology module of the bialgebroid $\mathcal B$ with coefficients in $M$}. The graded algebra $\OH^\bullet(\mathcal B,R)$ is the \textit{cohomology ring of $\mathcal B$}.
\end{defn}
\begin{nn}
Let $i_\lambda := i_{\C_\lambda(\cB)}$, $i_\varrho := i_{\C_\varrho(\cB)}$ and $i := i_{\C(\cB)}$ be the inclusion functors $\C_\lambda(\cB) \rightarrow \Mod(B)$, $\C_\varrho(\cB) \rightarrow \Mod(B)$ and $\C(\cB) \rightarrow \Mod(B)$ respectively. We get induced graded $k$-algebra homomorphisms:
\begin{align*}
i^\sharp_\lambda:& \Ext^\bullet_{\C_\lambda(\cB)}(R,R) \longrightarrow \Ext^\bullet_{\cB}(R,R) = \OH^\bullet(\cB,R)\, ,\\
i^\sharp_\varrho:& \Ext^\bullet_{\C_\varrho(\cB)}(R,R) \longrightarrow \Ext^\bullet_{\cB}(R,R) = \OH^\bullet(\cB,R)\, ,\\
i^\sharp:& \Ext^\bullet_{\C(\cB)}(R,R) \longrightarrow \Ext^\bullet_{\cB}(R,R) = \OH^\bullet(\cB,R) \, .
\end{align*}
They are isomorphisms due to the following lemma.
\end{nn}
\begin{lem}\label{lem:iso_ext_algeb}
The exact categories $\C_\lambda(\cB)$, $\C_\varrho(\cB)$ and $\C(\cB)$ are entirely extension closed subcategories of $\Mod(B)$.
\end{lem}
\begin{proof}
We use Corollary \ref{cor:exhproj}. Thanks to Lemma \ref{lem:candaexactmono}(\ref{lem:candaexactmono:1}) we know that $\C_\lambda(\cB)$, $\C_\varrho(\cB)$ and $\C(\cB)$ are 1-extension closed and closed under kernels of epimorphisms. It remains to show that $\Proj(B)$ is contained in every of those categories. But this is clear, because $B$ does by assumption \ref{nn:requirements}(\ref{item:requirements:2}), and because $\Proj(B) = \Add(B)$ and $\C_\lambda(\cB)$, $\C_\varrho(\cB)$, $\C(\cB)$ are closed under taking direct sums and direct summands.
\end{proof}
\begin{cor}
Let $A$ be a $k$-algebra which is projective over $k$. The exact categories $\P_\lambda(A)$, $\P_\varrho(A)$ and $\P(A)$ are entirely extension closed subcategories of $\Mod(A^\ev)$. \qed
\end{cor}
\begin{nn}
By fixing the second argument, the bifunctors $- \boxtimes_R -$ and $- \otimes_B -$ give rise to $k$-linear functors between the categories $\Mod(B)$ and $\Mod(B^\ev)$:
\begin{align*}
\scrL_\cB : \Mod(B) \longrightarrow \Mod(B^\ev),& \ \scrL_\cB M = M \boxtimes_R B,\\
\scrL^\cB : \Mod(B^\ev) \longrightarrow \Mod(B),& \ \scrL^\cB N = N \otimes_B R .
\end{align*}
We are interested in to what extend these functors respect the rich structures that the categories they are defined on and into possess.
\end{nn}
\begin{lem}\label{lem:func_prop}
The $k$-linear functors $\L_\cB$ and $\mathscr L^\cB$ have the following properties.
\begin{enumerate}[\rm(1)]
\item\label{lem:func_prop:1} If $0 \rightarrow N \rightarrow E \rightarrow M \rightarrow 0$ is an exact sequence in $\Mod(B)$ which, as a sequence of right $R$-modules, splits, then the induced sequence $0 \rightarrow \scrL_\cB N \rightarrow \scrL_\cB E \rightarrow \scrL_\cB M \rightarrow 0$ is exact in $\Mod(B^\ev)$.
\item\label{lem:func_prop:2} If $M$ is a $B$-module sucht that $M_\lhd$ is $R$-projective $($that is, if $M$ belongs to $\C_\varrho(\cB))$, then $\scrL_\cB M$ is a $B^\ev$-module which is projective when considered as a right $B$-module $($that is, it belongs to $\mathsf P_\varrho(B))$.
\item\label{lem:func_prop:3} If $0 \rightarrow N \rightarrow E \rightarrow M \rightarrow 0$ is an exact sequence in $\Mod(B^\ev)$ which, as a sequence of right $B$-modules, splits, then the induced sequence $0 \rightarrow \scrL^\cB N \rightarrow \scrL^\cB E \rightarrow \scrL^\cB M \rightarrow 0$ is exact in $\Mod(B)$.
\item\label{lem:func_prop:5} The functors $\scrL^\cB \scrL_\cB$ and $\Id_{\Mod(B)}$ are equivalent.
\end{enumerate}
\end{lem}
\begin{proof}
The additivity of $\mathscr L_\cB$ is responsible for (\ref{lem:func_prop:1}) holding true. The module $\scrL_\cB M$ is $B$-projective on the right since
$$
\Hom_{B^\op}(\scrL_\cB M,-) \cong \Hom_{B^\op}(M_\lhd \otimes_R {_\rhd B},-) \cong \Hom_{R^\op}(M_\lhd, \Hom_{B^\op}(B,-));
$$
hence (\ref{lem:func_prop:2}) follows. Assertion (\ref{lem:func_prop:3}) is again by the additivity of $\mathscr L^\cB$. The $k$-linear isomorphism
$$
\theta_M: (M \boxtimes_R B) \otimes_B R \longrightarrow M \boxtimes_R (B \otimes_B R) \longrightarrow M, \ \theta_M((m \otimes b) \otimes r) = \eta(1_R \otimes \varepsilon(b)r)m
$$
is natural in the $B$-module $M$ (cf. \cite[Prop.\,IX.2.1]{CaEi56}) and easily checked to be $B$-linear. Thus, also (\ref{lem:func_prop:5}) is established.
\end{proof}
\begin{prop}\label{prop:functor_restriction}
The $k$-linear functor $\scrL_\cB$ is a strong monoidal functor between the tensor categories $(\Mod(B), \boxtimes_R, R)$ and $(\Mod(B^\ev), \otimes_B, B)$. Furthermore, the functors $\scrL_\cB$ and $\scrL^\cB$ restrict to exact functors
$$
\xymatrix@C=32pt@R=10pt{
(\C_\varrho(\cB), \boxtimes_R, R) \ar[r]^-{\scrL_\cB} &  (\P_\varrho(B), \otimes_B, B) \, , \\ (\Mod(B), \boxtimes_R, R)  & (\P_\varrho(B), \otimes_B, B) \ar[l]_-{\scrL^\cB} \, .}
$$
In particular, there is the diagram
$$
\xymatrix@C=32pt@R=30pt{
\C_\varrho(\cB) \ar[r]^-{\scrL_\cB} \ar[d]|=_-{\subseteq}^-{i_\varrho} &  \P_\varrho(B) \ar@<-1ex>@{-->}[r]_-{B \otimes_k(-)} \ar[d]^-\subseteq & \P(B) \ar@<1ex>[r]^-{\subseteq} \ar@<-1ex>[l]_-{\subseteq} \ar[d]|{\subseteq} & \P_\varrho(B) \ar@<1ex>@{-->}[l]^-{(-) \otimes_k B} \ar[r]^-{\scrL^\cB} \ar[d]_-{\subseteq} & \Mod(B) \ar@{=}[d] \, \, \\
\Mod(B) \ar[r]^-{\scrL_\cB} & \Mod(B^\ev) \ar@{=}[r] & \Mod(B^\ev) \ar@{=}[r] & \Mod(B^\ev) \ar[r]^-{\scrL^\cB} & \Mod(B) \, ,
}
$$
wherein the squares given by the solid arrows commute.
\end{prop}
\begin{proof}
First of all, $\mathscr L_\cB R \cong B$ holds true for trivial reasons. Let $M$ and $N$ be $B$-modules; there is a well-defined and natural (in both arguments) $k$-linear isomorphism
\begin{equation*}\label{eq:careil_iso}
\begin{aligned}
(M \boxtimes_R N) \boxtimes_R B &\cong \big(M \boxtimes_R (B \otimes_B N)\big) \boxtimes_R B\\
&\cong (M \boxtimes_R B) \otimes_B (N \boxtimes_R B)
\end{aligned}
\end{equation*}
which is given by $(m \otimes n) \otimes b \mapsto (m \otimes 1_B) \otimes (n \otimes b)$ (cf. \cite[Ch.\,IX]{CaEi56}). It is also a homomorphism of $B^\ev$-modules. In fact, if we fix $b,b',b'' \in B$, $m \in M$ and $n \in N$, the element
\begin{align*}
(b' \otimes b'')((m \otimes n) \otimes b) &= \Delta(b')\cdot\big((m \otimes n) \otimes bb''\big)\\
&= \big( \Delta(b'_{(1)})\cdot(m \otimes n)\big) \otimes b'_{(2)}bb''\\
&= \big(b'_{(1,1)}m \otimes b'_{(1,2)}n\big) \otimes b'_{(2)}hb\\
&= \big((\Delta \otimes_R B) \circ \Delta(b')\big) \big((m \otimes n) \otimes bb''\big)\\
&= \big((B \otimes_R \Delta) \circ \Delta(b')\big) \big((m \otimes n) \otimes bb''\big)\\
&= (b'_{(1)}m \otimes b'_{(2,1)}n) \otimes b'_{(2,2)}bb''
\end{align*}
in $(M \boxtimes_R N) \boxtimes_R B$ is being mapped to
\begin{align*}
(b'_{(1)}m \otimes 1_B) \otimes (b'_{(2,1)}n \otimes b'_{(2,2)}bb'') &= (b'_{(1)}m \otimes 1_B) \otimes \big(\Delta(b'_{(2)}) \cdot (n \otimes bb'')\big)\\
&= (b'_{(1)}m \otimes 1_B) \otimes b'_{(2)}(n \otimes bb'')\\
&= (b'_{(1)}m \otimes 1_B)b'_{(2)} \otimes (n \otimes bb'')\\
&= (b'_{(1)}m \otimes b'_{(2)}) \otimes (n \otimes bb'')\\
&= \big(\Delta(b') \cdot (m \otimes 1_B)\big) \otimes (n \otimes bb'')\\
&= (b' \otimes b'')\big((m \otimes 1_B)\big) \otimes (n \otimes b)\big)
\end{align*}
in $(M \boxtimes_R B) \otimes_B (N \boxtimes_R B)$. Thus we have obtained natural isomorphisms
$$
\phi_{M,N}: \scrL_\cB M \otimes_B \scrL_\cB N \longrightarrow \scrL_\cB(M \boxtimes_R N), \quad \phi_0: B \longrightarrow \scrL_\cB R
$$
in $\Mod(B^\ev)$. They fit into (all of) the diagrams of Definition \ref{def:monoidalfunc}

Lemma \ref{lem:func_prop} tells us, that the functors ${\mathscr L}_{\mathcal B}$ and ${\mathscr L}^{\mathcal B}$ restrict as claimed, and, moreover, that these restrictions are exact functors on the exact categories $\C_\varrho(\cB)$ and $\P_\varrho(B)$. The commutativity of the diagram is clear.
\end{proof}
\begin{thm}\label{thm:commutativeHopfalgebroid}
Let $\C$ and $\P$ denote the strong exact monoidal categories $\C_\varrho(\cB)$ and $\P_\varrho(B)$ respectively. The map ${\mathscr L}_{\mathcal B}^\sharp$ $($obtained from the functor $\mathscr L_{\mathcal B}|_{\C})$, defines a split monomorphism $\OH^\bullet(\mathcal B,R) \rightarrow \HH^\bullet(B)$ of graded $k$-algebras such that the induced diagrams
$$
\xymatrix{
\OH^m(\mathcal B,R) \times \OH^n(\cB,R) \ar[r]^-{\beta_\cB} \ar[d] & \OH^{m+n-1}(\mathcal B,R) \ar[d]\\
\HH^m(B) \times \HH^n(B) \ar[r]^-{\{-,-\}_B} & \HH^{m+n-1}(B)
}
\ \quad \quad \
\xymatrix{
\OH^{2n}(\mathcal B,R)\ar[r]^-{\sigma_\cB} \ar[d] & \OH^{4n-1}(\mathcal B,R) \ar[d]\\
\HH^{2n}(B) \ar[r]^-{sq_B} & \HH^{4n-1}(B)
}
$$
commute for every pair of integers $m,n \geq 1$. Here $\{-,-\}_B$ and $sq_B$ denote the Gerstenhaber bracket and the squaring map on $\HH^\bullet(B)$, whereas $\beta_{\mathcal B}$ and $\sigma_{\mathcal B}$ are given by $\beta_\cB = {i}_\varrho^\sharp \circ [-,-]_{\C} \circ ({i}^\sharp_\varrho \times {i}^\sharp_\varrho)^{-1}$ and $\sigma_\cB = {i}^\sharp_\varrho \circ sq_{\C} \circ ({i}^\sharp_\varrho)^{-1}$. In particular, $(\OH^\bullet(\mathcal B,R), \beta_\cB, \sigma_\cB)$ is a strict Gerstenhaber subalgebra of $(\HH^\bullet(B), \{-,-\}_B, sq_B)$.
\end{thm}
\begin{proof}
Recall that by Lemma \ref{lem:iso_ext_algeb} and Remark \ref{nn:requirements}, the categories $\C = \C_\varrho(\mathcal B)$ and $\P = \P_\varrho(B)$ are entirely extension closed in $\Mod(B)$ and $\Mod(B^\ev)$ respectively. The monoidal functors $\mathscr L_\cB|_{\C_\varrho(\mathcal B)}$ and $\mathscr L^\cB|_{\P_\varrho(B)}$ (see Proposition \ref{prop:functor_restriction}) induce the $k$-algebra homomorphisms in the sequence
$$
\xymatrix@C=20pt{
\Ext^\bullet_{\C_\varrho(\mathcal B)}(R,R) \ar[r]^-{{\mathscr L}_\cB^\sharp} & \Ext^\bullet_{\P_\varrho(B)}(B,B) \ar[r]^-{{\mathscr L}^\cB_\sharp} & \Ext^\bullet_{B}(R,R) = \OH^\bullet(\cB,R)
}.
$$
Their composition is an isomorphism: According to Lemma \ref{lem:func_prop}(\ref{lem:func_prop:5}), there is a natural isomorphism $\theta: \mathscr L^\cB\mathscr L_\cB \rightarrow \Id_{\Mod(B)}$; if now $\xi$ is an $n$-extension $0 \rightarrow R \rightarrow \mathbb E \rightarrow R \rightarrow 0$ in $\mathsf C$, then we obtain the commutative diagram
$$
\xymatrix@C=20pt{
0 \ar[r] & R \ar@{=}[d] \ar[r]^-{e'_n} & \mathscr L^{\mathcal B} \mathscr L_{\mathcal B}(E_{n-1}) \ar[d]_-{\theta_{E_{n-1}}} \ar[r] & \cdots \ar[r] & \mathscr L^{\mathcal B}\mathscr L_{\mathcal B}(E_0) \ar[r]^-{e_0'} \ar[d]^-{\theta_{E_0}} & R \ar[r] \ar@{=}[d] & 0\\
0 \ar[r] & R \ar[r] & E_{n-1} \ar[r] & \cdots \ar[r] & E_0 \ar[r] & R \ar[r] & 0
}
$$
in $\Mod(B)$, where $e_0' = \theta_R \circ \mathscr L^\cB\mathscr L_\cB(e_0)$ and $e'_n = \mathscr L^\cB\mathscr L_\cB(e_n)\circ \theta_R^{-1}$. Note that $\theta_R$ is the isomorphism $\mathscr L^\cB\mathscr L_\cB(R) \cong \mathscr L^\cB(B) \cong R$ obtained from the unit isomorphisms. It follows that ${\mathscr L}^\cB_\sharp \circ {\mathscr L}_\cB^\sharp$ coincides with the map $i_\varrho^\sharp : \Ext^\bullet_\C(R,R) \rightarrow \OH^\bullet(\cB,R)$, and thus is bijective. Let us now consider the diagram
$$
\xymatrix@C=40pt{
\OH^m(\cB,R) \times \OH^n(\cB,R) \ar[r]^-{\beta_{\cB}} \ar[d]_-{({i}^\sharp_\varrho \times {i}^\sharp_\varrho)^{-1}} & \OH^{m+n-1}(\cB,R) \ar@<-2pt>[d]^-{({i}^\sharp_\varrho)^{-1}} \ \ \\
\Ext^{m}_{\C}(R,R) \times \Ext^{n}_{\C}(R,R) \ar[r]^-{[-,-]_{\C}} \ar[d]_-{{\mathscr L}_\cB^\sharp \times {\mathscr L}_\cB^\sharp} & \Ext^{m+n-1}_{\C}(R,R) \ar@<-2pt>[d]^-{{\mathscr L}_\cB^\sharp}  \ \ \\
\Ext^{m}_{\P}(B,B) \times \Ext^{n}_{\P}(B,B) \ar[r]^-{[-,-]_{\P}} \ar[d]_-\cong & \Ext^{m+n-1}_{\P}(B,B) \ar@<-2pt>[d]^-\cong \ \ \\
\HH^m(B) \times \HH^n(B) \ar[r]^-{\{-,-\}_{B}} & \HH^{m+n-1}(B) \ .
}
$$
The top square commutes by definition, the middle square by Theorem \ref{thm:bracketcomm}, and the bottom square by Theorem \ref{thm:schwede_comm}. Similar arguments show that also the diagram for the squaring map $sq_B$ commutes.
\end{proof}
In what follows, we will identify the map $\scrL_{\cB}^\sharp$ with the map $$\xymatrix{\OH^\bullet(\cB, R) \ar[r]^-\cong & \Ext^{\bullet}_{\C_\varrho(\cB)}(R,R) \ar[r]^-{\scrL_{\cB}^\sharp} & \Ext^{\bullet}_{\P_\varrho(B)}(B,B) \ar[r]^-\cong & \HH^\bullet(B)}.$$
\begin{defn}
The bialgebroid $\mathcal B$ is 
\begin{enumerate}[\rm(1)]
\item \textit{pre-triangular}, if the tensor category $(\Mod(B), \boxtimes_R,R)$ is lax braided,
\item \textit{quasi-triangular}, if the tensor category $(\Mod(B), \boxtimes_R,R)$ is braided, and
\item \textit{cocommutative}, if the tensor category $(\Mod(B), \boxtimes_R,R)$ is symmetric.
\end{enumerate}
A Hopf algebroid over $R$ is \textit{pre-triangular} (\textit{quasi-triangular}, \textit{cocommutative}), if it is pre-triangular (quasi-triangular, cocommutative) as a bialgebroid over $R$.
\end{defn}
Note that (of course) cocommutative implies quasi-triangular, and quasi-tri\-an\-gular implies pre-triangular.
\begin{exa}
The $A$-Hopf algebroid $\mathcal H_A$ associated to the enveloping algebra $A^\ev$ is pre-triangular (quasi-triangular) if, and only if, $A$ admits a canonical R-matrix (a semi-canonical R-matrix) in the sense of Section \ref{sec:catbimodules}.
\end{exa}
\begin{cor}\label{cor:vanishHopfalgebroid}
Suppose that the bialgebroid $\mathcal B$ is pre-triangular. Then ${\mathscr L}^\sharp_{\mathcal B}$ defines a split monomorphism $\OH^\bullet(\mathcal B,R) \rightarrow \HH^\bullet(B)$ such that the diagram
$$
\xymatrix{
\OH^m(\mathcal B,R) \times \OH^n(\mathcal B,R) \ar[r]^-{0} \ar[d] & \OH^{m+n-1}(\mathcal B,R) \ar[d]\\
\HH^m(B) \times \HH^n(B) \ar[r]^-{\{-,-\}_B} & \HH^{m+n-1}(B)
}
$$
commutes for all integers $m, n \geq 1$, that is,
$$
\{\alpha, \beta\}_B = 0 \quad \text{$($for all $\alpha, \beta \in \OH^{\geq 1}(\cB,R))$.}
$$
\end{cor}
\begin{proof}
Take the commutative diagram occuring in Theorem \ref{thm:commutativeHopfalgebroid} and note that $[-,-]_{\mathsf C_\varrho(\cB)} = 0$ since $\mathcal B$ is pre-triangular (i.e., $(\Mod(B), \boxtimes_R, R)$ is lax braided; see Theorem \ref{thm:trivial_bracket} for the vanishing).
\end{proof}
\begin{nn}
By Lemma \ref{lem:triangular_braided} we know, that those bialgebras that admit a lax braiding on their category of left modules, are precisely the pre-triangular ones (pre-triangular in the sense that there is a semi-canonical R-matrix for the bialgebra in the sense of Definition \ref{defn:qtrhopf}).

Let $\mathcal B = (B, \nabla, \eta, \Delta, \varepsilon)$ be a bialgebra such that $B$ is a projective $k$-module (for instance, take the Hopf algebra $kG$ for \textit{any} group $G$). We will denote its cohomology ring $\OH^\bullet(\mathcal B,k) = \Ext^\bullet_B(k,k)$ by $\OH^\bullet(B,k)$. Due to Theorem \ref{thm:commutativeHopfalgebroid} the cohomology ring $\OH^\bullet(B,k)$ embeds into the Hochschild cohomology ring $\HH^\bullet(B)$ of $B$. Note that for Hopf algebras $\mathcal H = (H, \nabla, \eta, \Delta, \varepsilon, S)$ over fields, the existence of a (split) embedding $\OH^\bullet(H,k) \rightarrow \HH^\bullet(H)$ was already observed in \cite{GiKu98}.

Corollary \ref{cor:vanishHopfalgebroid} implies the following result, which, over fields, was conjectured by L.~Menichi in \cite{Me11}.
\end{nn}
\begin{cor}[{\cite[Conj.~23]{Me11}}]\label{cor:gerstenhabervanishhopfalgebra}
Assume that $\mathcal B$ is pre-triangular. The map ${\mathscr L}_{\mathcal B}^\sharp$ defines a split monomorphism $\OH^\bullet(B,k) \rightarrow \HH^\bullet(B)$ of graded $k$-algebras. Moreover, the Gerstenhaber bracket $\{-,-\}_B$ on $\HH^\bullet(B)$ vanishes on the subalgebra $\OH^\bullet(B,k) \subseteq \HH^\bullet(B)$, i.e.,
$$
\{\alpha, \beta\}_{B} = 0 \quad \text{$($for all $\alpha, \beta \in \OH^\bullet(B,k))$.}
$$
\end{cor}
\begin{proof}
By Corollary \ref{cor:vanishHopfalgebroid}, it is evident that
$$
\OH^{\geq 1}(B,k) \otimes_k \OH^{\geq 1}(B,k) \subseteq \Ker(\{-,-\}_B).
$$
But since $\OH^{0}(B,k) = \Hom_{B}(k,k) = k$, and $\{x,-\}_B$ is the zero map $\HH^\bullet(B) \rightarrow \HH^{\bullet - 1}(B)$ for all $x \in k$ (remember that $\{-,\xi\}_B$ is $k$-linear and a derivation for all $\xi \in \HH^\bullet(B)$), the result is established.
\end{proof}
\begin{rem}
In \cite{Me11}, Menichi proves Corollary \ref{cor:gerstenhabervanishhopfalgebra} for cocommutative Hopf algebras $\mathcal H$ over fields $K$ in the context of operatic actions and Batalin-Vilkovisky algebra structures. He conjectures, that his result should extend to quasi-triangular bialgebras over $K$, which he was not able to show by adapting the proof for Hopf algebras, since it highly depends on the existence of an antipode. The main obstacle is that Menichi deduces the vanishing of the BV-operator he considers from the fact that, in the cocommutative case, the antipode $S$ on $\mathcal H$ is involutive, i.e., $S \circ S = \id_H$ (which, in general, is false for quasi-triangular Hopf algebras; see \ref{exas:qtcocom}).

Note that the construction presented in this monograph bypasses this obstacle completely. Our approach not only admits a more general result (covering Menichi's conjecture), but is also more elementary, and thus, as we believe, more accessible.
\end{rem}
\begin{rem}
Due to M.\,A.\,Farinati and A.\,L.\,Solotar (cf. \cite{FaSo04}), the cohomology ring of a Hopf algebra $\mathcal H$ over $k$ is a Gerstenhaber algebra. Its Gerstenhaber bracket coincides with $[-,-]_{\C_\varrho(\mathcal H)}$ in case the $k$-module underlying $\mathcal H$ is projective. All of this remains true, if the Hopf algebra is replaced by a bialgebra (see \cite{Me11}). Hence Corollary \ref{cor:gerstenhabervanishhopfalgebra} may be read as follows: If $\mathcal B$ is a pre-triangular bialgebra being projective over $k$, then the Gerstenhaber bracket on the cohomology ring of $\mathcal B$ is trivial.
\end{rem}
\begin{rem}
There is a rich pool of quasi-triangular (even cocommutative) Hopf algebras over $k$ which are projective as $k$-modules. Various examples of such Hopf algebras will be presented in the appendix of this monograph.
\end{rem}
We conclude this section with an observation on the functors $\scrL_\H$ and $\scrL^\H$ for quasi-triangular Hopf algebroids $\mathcal H$.
\begin{lem}\label{lem:hopfconditionsat}
Let $\mathcal H = (H, \nabla, \eta, \Delta, \varepsilon)$ be a Hopf algebroid over $R$. Assume that $\mathcal H$ is quasi-triangular. Then the left $H$-module $M_\lhd \boxtimes_R {_\rhd H}$ is projective for all modules $M \in \mathsf C(\mathcal H)$.
\end{lem}
\begin{proof}
Let $\gamma$ be a braiding on $(\Mod(H), \boxtimes_R,R)$. In particular, we have the isomorphism
$$
\gamma_{M,H}: M_\lhd \boxtimes_R {_\rhd H} \longrightarrow H_\lhd \boxtimes_R {_\rhd M}
$$
of $H$-modules for every $H$-module $M$. If $M$ is in $\mathsf C(\mathcal H)$, then $H_\lhd \boxtimes_R {_\rhd M}$ is a projective $H$-module by Lemma \ref{lem:hopfRporjHproj}, and therefore so is $M_\lhd \boxtimes_R {_\rhd H}$.
\end{proof}
\begin{cor}\label{cor:quasiHopfrestr}
Assume that the base algebra $R$ is a projective $k$-module. Further, suppose that $\mathcal H = (H, \nabla, \eta, \Delta, \varepsilon)$ is a quasi-triangular Hopf algebroid, such that ${_\rhd H}$ and $H_\lhd$ are projective $R$-modules. Then the $k$-linear functors $\scrL_\H$ and $\scrL^\H$ between the tensor categories $(\Mod(H), \boxtimes_R, R)$ and $(\Mod(H^\ev), \otimes_H, H)$ restrict to the exact functors below, the first of which being also strong monoidal.
$$
\xymatrix@C=32pt@R=10pt{
(\C(\H), \boxtimes_R, R) \ar[r]^-{\scrL_\H} &  (\P(H), \otimes_H, H) \\ (\Mod(H), \boxtimes_R, R)  & (\P(H), \otimes_H, H) \ar[l]_-{\scrL^\H}}
$$
\end{cor}
\begin{proof}
Combine Lemma \ref{lem:hopfconditionsat} with Lemma \ref{lem:func_prop}.
\end{proof}
If $\mathcal H$ is a Hopf algebra over $k$, the functor $\scrL^{\mathcal H} : \Mod(H^\ev) \rightarrow \Mod(H)$ is strong monoidal as well. Let us refine some of our observations in this particular situation.
\begin{cor}
Suppose that $\mathcal H = (H, \nabla, \eta, \Delta, \varepsilon, S)$ is a quasi-triangular Hopf algebra over $k$, such that $H$ is projective as a $k$-module. Then the $k$-linear monoidal functors $\scrL_\H$ and $\scrL^\H$ between the tensor categories $(\Mod(H), \boxtimes_k, k)$ and $(\Mod(H^\ev), \otimes_H, H)$ restrict to exact and strong monoidal functors
$$
\xymatrix@C=35pt{
(\C(\H), \boxtimes_k, k) \ar@<1ex>[r]^-{\scrL_\H} & \ar@<1ex>[l]^-{\scrL^\H} (\P(H), \otimes_H, H).
}
$$
\end{cor}
\begin{proof}
Fix a module $M \in \P(\mathcal H)$. As $r(m \otimes s) = m \otimes sr$ for all $m \in M$ and $r,s \in k$, we get isomorphisms
$$
\Hom_k(M \otimes_H k, -) \cong \Hom_{H^\op}(M, \Hom_k(k,-)) \cong \Hom_{H^\op}(M, -)
$$
and therefore $\Hom_k(M \otimes_H k, -)$ is exact. Thus, the functor $\scrL^\H$ takes $\P(\mathcal H)$ to $\C(\mathcal H)$ and the assertion follows from Corollary \ref{cor:quasiHopfrestr}.
\end{proof}
Under the assumptions of the above corollary, the right inverse of the graded algebra map $\mathscr L_{\mathcal H}^\sharp: \OH^\bullet(H,k) \rightarrow \HH^\bullet(H)$ is given by
$$
\xymatrix{\HH^\bullet(H) \ar[r]^-\cong & \Ext^{\bullet}_{\P(H)}(H,H) \ar[r]^-{\scrL^{\mathcal H}_\sharp} & \Ext^{\bullet}_{\C(\mathcal H)}(k,k) \ar[r]^-\cong & \OH^\bullet(H,k)}.
$$
\section{Comparison to Linckelmann's result}\label{sec:linckelmann}
\begin{nn}
Throughout this section, the symbol $k$ will denote a commutative ring. We close the chapter by presenting a (more or less) immediate consequence of Theorem \ref{thm:commutativeHopfalgebroid} (and its Corollaries \ref{cor:vanishHopfalgebroid} and \ref{cor:gerstenhabervanishhopfalgebra}). Remember that in \cite{CiSo97}, the authors describe the very close connection between the graded commutative algebras $\Ext^\bullet_{kG^\ev}(kG,kG)$ and $\Ext^\bullet_{kG}(k,k)$ in case $G$ is a finite abelian group. M.\,Linckelmann offers the following generalization of this result.
\end{nn}
\begin{thm}[{\cite{Li00}}]\label{thm:generalcibsol}
Let $\mathcal H$ be a commutative Hopf algebra over the commutative ring $k$. Assume that its underlying algebra $A = \mathcal H^\natural$ is finitely generated projective as a $k$-module. Then there is an isomorphism
$$
\theta: \OH^\bullet(A,k) \otimes_k A \longrightarrow \HH^\bullet(A)
$$
of graded $k$-algebras.
\end{thm}
\begin{nn}\label{nn:extkdiso}
Linckelsmann's isomorphism is easy to construct, and we will reprove the result in a moment. Beforehand, we like to recall some basic facts.  Let $R$ be a ring. The derived category $\mathbf D(\Mod(R))$ of $R$ is the localization of the homotopy category $\mathbf K(\Mod(R))$ with respect to the class of all quasi-isomorphisms in $\mathbf K(\Mod(R))$, and hence admits a \textit{calculus of fractions} (see \cite{Wei94}). In what follows, we will write $\mathbf K(R) := \mathbf K(\Mod(R))$ and $\mathbf D(R) := \mathbf D(\Mod(R))$; we let $[1]$ be the suspension functor on $\mathbf D(R)$. Remember that $\Mod(R)$ fully faithfully embeds into $\mathbf D(R)$ by sending an $R$-module $M$ to the complex $\mathbb M$ concentrated in degree $0$, wherein it is $M$. We will usually write $M$ instead of $\mathbb M$. 

Let $\mathbbm P_M \rightarrow M \rightarrow 0$ be a projective resolution of the finitely generated $R$-module $M$. It is well-known that there are isomorphisms
$$
\xymatrix@C=15pt{
\Hom_{\mathbf K(R)}(\mathbb P_M, \mathbb P_M[n]) \ar[r] & \Hom_{\mathbf D(R)}(M, M[n]) \ar[r] & \Ext^n_R(M,M)
}\quad (\text{for $n \geq 0$}),
$$
where, as mentioned, we regard $M$ as a complex concentrated in degree $0$ (cf. \cite{Kr04} and \cite{Wei94}). Observe that, in $\mathbf D(R)$, $M$ and $\mathbb P_M$ are isomorphic. The first of the above maps is then simply given by the canonical assignment $f \mapsto (\id, f)$. By \cite[Prop.\,3.2.2]{Ve67} (see also \cite[III.\S5]{GeMa03}) the (inverse of the) second map arises from the following assigment (note that this map does not rely on the existence of injective/projective resolutions). Namely, map an $n$-extension
$$
\xymatrix@C=20pt{
\xi & \equiv & 0 \ar[r] & M \ar[r]^-{e_n} & E_{n-1} \ar[r]^-{e_{n-1}} & \cdots \ar[r]^{e_1} & E_0 \ar[r]^-{e_0} & M \ar[r] & 0 
}
$$
to the \textit{fraction} defined by the roof
$$
\xymatrix@!C=18pt@R=22pt{
M & \cdots \ar[r] & 0 \ar[r] & 0 \ar[r]^-{} & 0 \ar[r]^-{} & \cdots \ar[r]^{} & M \ar[r] & 0 \ar[r] & \cdots \ \ \\
\xi^\natural \ar[u] \ar[d]  & \cdots \ar[r] & 0 \ar[r] \ar[u] \ar[d] & M \ar[r]^-{e_n} \ar@{=}[d] \ar[u] & E_{n-1} \ar[d] \ar[u] \ar[r]^-{e_{n-1}} & \cdots \ar[r]^{e_1} & E_0 \ar[u]_-{e_0} \ar[r] \ar[d] & 0 \ar[r] \ar[d] \ar[u] & \cdots \ \ \\
M[n] & \cdots \ar[r] & 0 \ar[r] & M \ar[r]^-{} & 0 \ar[r]^-{} & \cdots \ar[r]^{} & 0 \ar[r] & 0 \ar[r] & \cdots \ .
}
$$
The chain map $\xi^\natural \rightarrow M$ is a quasi-isomorphism. To proceed further, we need two additional lemmas. 
\end{nn}
\begin{lem}[{\cite[Prop.\,2]{Li00}}]\label{lem:algebraiso}
Let $\mathcal H$ be a commutative Hopf algebra over $k$ with underlying $k$-algebra $A = \mathcal H^\natural$. The map $\alpha = (A \otimes_k \nabla) \circ (\Delta \otimes_k A)$ is a $k$-algebra isomorphism $A \otimes_k A \rightarrow A \otimes_k A$ with inverse map $\beta =(A \otimes_k \nabla) \circ (A \otimes_k S \otimes_k A) \circ (\Delta \otimes_k A)$. Both maps send $1_A \otimes a$ to $1_A \otimes a$.
\end{lem}
\begin{proof}
It is clear, that both maps agree on $1_A \otimes A = \{ 1_A \otimes a \mid a \in A \}$. Hence it suffices to show that $(\beta \circ \alpha)(a \otimes 1_A) = a \otimes 1_A = (\alpha \circ \beta)(a \otimes 1_A)$ holds for all $a \in A$ to conclude the claimed statement. $\alpha$ maps $a \otimes 1_A$ to $\Delta(a)$ and hence
\begin{align*}
(\beta \circ \alpha)(a \otimes 1_A) &= ((A \otimes_k \nabla) \circ (A \otimes_k S \otimes_k A) \circ (\Delta \otimes_k A) \circ \Delta)(a)\\
&= ((A \otimes_k \nabla) \circ (A \otimes_k S \otimes_k A) \circ (A \otimes_k \Delta) \circ \Delta)(a)\\
&= ((A \otimes_k (\eta\circ\varepsilon)) \circ \Delta(a)\\
& = a \otimes 1_A,
\end{align*}
where we successively used the coassociativity and the counitarity. On the other hand, $\beta(a \otimes 1_A) = (A \otimes_k S) \circ \Delta(a)$. Since $(\Delta \otimes_k A) \circ (A \otimes_k S) = (A \otimes_k A \otimes_k S) \circ (\Delta \otimes_k A)$ one may deduce that $(\alpha \circ \beta)(a \otimes 1_A) = a \otimes 1_A$.
\end{proof}
\begin{lem}[{\cite[Cor.\,3]{Li00}}]
Under the assumptions of Lemma $\ref{lem:algebraiso}$, and for an $A^\ev$-module $M'$, we let $\mathrm{Res}_\alpha(M')$ be the $A^\ev$-module with underlying $k$-module $M'$ and twisted $A^\ev$-action
$$
(a \otimes b)m := \alpha(a \otimes b)m = \alpha(a \otimes 1_A)(1_A \otimes b)m \quad (\text{for $a,b \in A$, $m \in M'$}).
$$
Then for any $A$-module $M$, the modules $\mathrm{Res}_\alpha(M \otimes_k A)$ and $M \boxtimes_k A$ coincide. Moreover, $\mathrm{Res}_\alpha(P \otimes_k A)$ is a projective $A^\ev$-module if $P$ is projective over $A$.
\end{lem}
\begin{proof}
The first claim follows from the fact that $\alpha(1_A \otimes b) = 1_A \otimes b$ for all $b \in A$ and the observation
$$
(a \otimes 1_A)(b \otimes m) = \alpha(a \otimes 1_A)(b \otimes m) = \Delta(a)(b \otimes m) \quad (\text{for $a, b \in A$, $m \in M$}).
$$
In light of the Lemma \ref{lem:boxtimes-otimes}, we know that (specialize to $P = A$) $\mathrm{Res}_\alpha(A \otimes_k A) = A \boxtimes_k A \cong A \otimes_k A \cong A^\ev$ as $A^\ev$-modules. Since $\mathrm{Res}_\alpha(- \otimes_k A)$ is an additive functor $\Mod(A) \rightarrow \Mod(A^\ev)$, and every projective module is a summand of a free one, we are done.
\end{proof}
\begin{proof}[Proof of theorem $\ref{thm:generalcibsol}$]
In degree zero, the homomorphism $\theta$ is delivered by the natural isomorphisms
\begin{equation}\label{eq:thetaiso}
\begin{aligned}
\theta_{M,N}: \Hom_A(M,N) \otimes_k A &\longrightarrow \Hom_{A^\ev}(\mathrm{Res}_\alpha(M \otimes_k A), \mathrm{Res}_\alpha(N \otimes_k A)), \\ \theta_{M,N}(f \otimes a) &= (m \otimes b \mapsto f(m) \otimes ba),
\end{aligned}
\end{equation}
where $M$ and $N$ are $A$-modules which are finitely generated projective. To see how the isomorphism operates in higher degrees, let $\mathbb P_k \rightarrow k \rightarrow 0$ be a projective resolution of $k$ as an $A$-module, which can be chosen to have finitely generated components (over $k$, hence also over $A$). Since $A$ is $k$-projective (and hence also $k$-flat), the complex $\mathrm{Res}_\alpha(\mathbb P_A \otimes_k A) \rightarrow \mathrm{Res}_\alpha(k \otimes_k A) \rightarrow 0$ is a projective resolution of $\mathrm{Res}_\alpha(A \otimes_k k)$ as an $A^\ev$-module. But since $A \cong k \boxtimes_k A = \mathrm{Res}_\alpha(k \otimes_k A)$ as $A^\ev$-modules, we may regard it as a bimodule resolution $\mathbb P_A := \mathrm{Res}_\alpha(P_k \otimes_k A) = \mathbb P_k \boxtimes_k A$ of $A$. Now, the map $\theta$ induces an isomorphism
$$
\Hom_{\mathbf K(A)}(\mathbb P_k, \mathbb P_k[n]) \otimes_k A \longrightarrow \Hom_{\mathbf K(A^\ev)}(\mathbb P_A, \mathbb P_A[n])
$$
for, if $f: M \rightarrow M'$, $g: N \rightarrow N'$ and $\varphi: M \rightarrow N$, $\varphi': M' \rightarrow N' $ are $A$-linear homomorphisms between $A$-modules which are finitely generated projective over $k$ such that $\varphi' \circ f = g \circ \varphi$, the diagram
$$
\xymatrix@C=30pt{
M \otimes_k A \ar[r]^-{f \otimes_k A} \ar[d]_-{\theta(\varphi \otimes a)} & M' \otimes_k A \ar[d]^-{\theta(\varphi' \otimes a)}\\
N \otimes_k A \ar[r]^-{g \otimes_k A} & N' \otimes_k A
}
$$
commutes for every $a \in A$. From \ref{nn:extkdiso}, we immediately deduce that
$$
\Ext^n_A(k,k) \otimes_k A \cong \Ext^n_{A^\ev}(A,A) \quad (\text{for all $n \geq 0$}).
$$
The isomorphism is such of graded $k$-algebras.
\end{proof}
\begin{prop}\label{prop:morphismagree}
Let $\mathcal H$ be a commutative and quasi-triangular Hopf algebra over $k$ with underlying $k$-algebra $A = \mathcal H^\natural$. Assume that $A$ is finitely generated projective as a $k$-module. Then the injective homomorphism
$$
\vartheta: \OH^\bullet(A,k) \longrightarrow \OH^\bullet(A,k) \otimes_k A \cong \HH^\bullet(A), \ \xi \mapsto \xi \otimes 1_A
$$
of graded $k$-algebras agrees with the homomorphism $\scrL_\H^\sharp: \OH^\bullet(A,k) \rightarrow \HH^\bullet(A)$ constructed earlier.
\end{prop}
\begin{proof}[Proof]
Let $\mathbb P_k \rightarrow k \rightarrow 0$ be a $A$-projective resolution of $k$, and let $\mathbb P_A \rightarrow A \rightarrow 0$ be the $A^\ev$-projective resolution $\mathbb P_A = \mathbb P_k \boxtimes_k A$. Let $\theta$ be the map (\ref{eq:thetaiso}). It suffices to show that
$$
\xymatrix@!C=68pt{
\Hom_{\mathbf K(A)}(\mathbb P_k,\mathbb P_k[n]) \ar[dr] \ar[ddd]|=_-{\vartheta}^{\theta(- \otimes 1_A)} \ar[rrr]^-\cong &&&\Ext^n_{\C(\H)}(k,k)\ar[ddd]^-{\scrL_\H^\sharp}\\
& \Hom_{\mathbf D(A)}(k,k[n]) \ar[r]^-{\delta} \ar[d] &  \ar[d] \Ext_A^n(k,k) \ar[ur] &\\
& \Hom_{\mathbf D(A^\ev)}(A,A[n]) \ar[r]^-{\partial} & \Ext^n_{A^\ev}(A,A) \ar[dr] &\\
 \Hom_{\mathbf K(A^\ev)}(\mathbb P_A, \mathbb P_A[n]) \ar[ur] \ar[rrr]^-\cong &&& \Ext^n_{\P(A)}(A,A)
}
$$
commutes for every $n \geq 1$ to establish the proposition. We have to show that the internal square commutes. Since $A$ is $k$-flat, the functor $\scrL_\H = (- \boxtimes_k A): \Mod(A) \rightarrow \Mod(A^\ev)$ is exact. Hence it gives rise to a functor $\mathbf D(\scrL_\H): \mathbf D(A) \rightarrow \mathbf D(A^\ev)$, and the lefthand arrow in the internal square is given by applying it. According to \ref{nn:extkdiso}, an $n$-extionsion
$$
\xymatrix@C=20pt{
\xi & \equiv & 0 \ar[r] & k \ar[r]^-{e_n} & E_{n-1} \ar[r]^-{e_{n-1}} & \cdots \ar[r]^{e_1} & E_0 \ar[r]^-{e_0} & k \ar[r] & 0 
}
$$
of $A$-modules is mapped to the roof
\begin{equation}\label{eq:roof1}
\begin{aligned}
\xymatrix@!C=25pt{
A & \scrL_\H(\xi^\natural) \ar[r] \ar[l] & A[n]
}
\end{aligned}
\end{equation}
by $\mathbf D(\scrL_\H) \circ \delta^{-1}$, where $\xi^\natural$ is the complex $0 \rightarrow k \rightarrow E_{n-1} \rightarrow \cdots \rightarrow E_1 \rightarrow E_0$ and the only non-trivial components of the occurring chain maps are given by $e_0 \boxtimes_k A$ and $\id_A$ respectively. On the other hand, let $\psi: \mathbb P_k \rightarrow \xi^\natural$ be a chain map lifting the identity of $k$. By pushing out $(e_n, \psi_n)$ we obtain an $n$-extension $\xi'$ of $k$ by $k$ and a morphism $\alpha_\xi: \xi' \rightarrow \xi$ of $n$-extensions. Now the righthand arrow in the internal square maps $\xi$ to $\scrL_\H(\xi')$, which is sent to
\begin{equation}\label{eq:roof2}
\begin{aligned}
\xymatrix@!C=25pt{
A & \scrL_\H(\xi'^\natural) \ar[r] \ar[l] & A[n]
}
\end{aligned}
\end{equation}
by $\partial^{-1}$. The commutative diagram
$$
\xymatrix@!C=25pt@R=25pt{
& \scrL_\H(\xi^\natural) \ar[dl] \ar[dr]  & \\
A & \scrL_{\H}(\xi'^\natural) \ar@{=}[d] \ar[r] \ar[l] \ar[u]|{\scrL_\H (\alpha_\xi)}& A[n]\\
& \scrL_{\H}(\xi'^\natural) \ar[ur] \ar[ul] &
}
$$
tells us that the roofs (\ref{eq:roof1}) and (\ref{eq:roof2}) define the same equivalence class in $\Hom_{\mathbf D(A^\ev)}(A,A[n])$.
\end{proof}
\begin{cor}\label{cor:centerdetermins}
Under the assumptions of Proposition $\ref{prop:morphismagree}$ and after letting $\OH^\bullet$ denote $\OH^\bullet(A,k)$, the Gerstenhaber bracket $\{-,-\}_A$ on $\HH^\bullet(A) \cong \OH^\bullet(A,k) \otimes_k A$ is given as follows:
$$
\{\xi \otimes x, \zeta \otimes y\}_A = \{\xi \otimes 1_A, 1_{\OH^\bullet} \otimes y\}_A (\zeta \otimes x) + (-1)^{\abs{\zeta}} (\xi \otimes y) \{\zeta \otimes 1_A, 1_{\OH^\bullet} \otimes x\}_A,
$$
where $x,y \in A$ and $\xi, \zeta \in \OH^\bullet(A,k)$ are homogeneous. Hence the graded Lie algebra structure on $\HH^{\bullet + 1}(A)$ is completely determined by the $k$-linear maps $$\{- \otimes 1_A, 1_{\OH^\bullet} \otimes x\}_A : \OH^\bullet(A,k) \longrightarrow \HH^{\bullet - 1}(A), \quad x \in A,$$ or, respectively, by the $k$-linear maps $$\{\xi \otimes 1_A, 1_{\OH^\bullet} \otimes -\}_A : A \longrightarrow \HH^{\abs{\xi} - 1}(A), \quad \xi \in \OH^\bullet(A,k).$$
$($In other words, the Gerstenhaber bracket is completely determined by the induced action of the center $Z(A) = A$ on the cohomology ring $\OH^\bullet(A,k).)$
\end{cor}
\begin{proof}
For simplicity, we will write $1$ for both $1_A$ and $1_{\OH^\bullet}$. Fix $x,y \in A$ and homogeneous elements $\xi, \zeta \in \OH^\bullet(A,k)$. After observing that $(\xi \otimes x) = (1 \otimes x)(\xi \otimes 1)$ and $(\zeta \otimes y) = (1 \otimes y)(\zeta \otimes 1)$, the axioms for $\{-,-\}_A$ yield:
\begin{align*}
\{\xi \otimes x, 1 \otimes y\}_A & = -(-1)^{(\abs{\xi}-1)(\abs{1 \otimes y}-1)} \{1 \otimes y, \xi \otimes x\}_A\\
&= (-1)^{\abs{\xi}}\Big( \{1 \otimes y, 1 \otimes x\}_A  (\xi \otimes 1)\\ 
& \quad \quad + (-1)^{(\abs{1 \otimes y}-1)\abs{1 \otimes x}}(1 \otimes x) \{1 \otimes y, \xi \otimes 1\}_A\Big)\\
&= (-1)^{\abs{\xi}} (1 \otimes x) \{1 \otimes y, \xi \otimes 1\}_A\\
&= \{\xi \otimes 1, 1 \otimes y\}_A (1 \otimes x),
\\
\{\xi \otimes x, \zeta \otimes 1\}_A & = -(-1)^{(\abs{\xi}-1)(\abs{\zeta}-1)} \{\zeta \otimes 1, \xi \otimes x\}_A\\
& = -(-1)^{(\abs{\xi}-1)(\abs{\zeta}-1)}\Big(\{\zeta \otimes 1, 1 \otimes x\}_A  (\xi \otimes 1)\\
&\quad \quad + (-1)^{(\abs{\zeta}-1)\abs{1 \otimes x}}(1 \otimes x)\{\xi \otimes 1, \zeta \otimes 1\}_A\Big)\\
&= (-1)^{(\abs{\xi}-1)(\abs{\zeta}-1) + 1}\{\zeta \otimes 1, 1 \otimes x\}_A  (\xi \otimes 1)\\
&= (-1)^{\abs{\zeta}} (\xi \otimes 1) \{\zeta \otimes 1, 1 \otimes x\}_A,
\end{align*}
and hence
\begin{align*}
\{\xi \otimes x, \zeta & \otimes y\}_A\\ &= \{\xi \otimes x, 1 \otimes y\}_A  (\zeta \otimes 1) + (-1)^{(\abs{\xi}-1)\abs{1 \otimes y}} (1 \otimes y)\{\xi \otimes x, \zeta \otimes 1\}_A\\
&= \{\xi \otimes x, 1 \otimes y\}_A  (\zeta \otimes 1) + (1 \otimes y)\{\xi \otimes x, \zeta \otimes 1\}_A\\
&= \{\xi \otimes 1, 1 \otimes y\}_A (1 \otimes x) (\zeta \otimes 1)\\
&\quad \quad + (-1)^{\abs{\zeta}} (1 \otimes y) (\xi \otimes 1) \{\zeta \otimes 1, 1 \otimes x\}_A\\
&=  \{\xi \otimes 1, 1 \otimes y\}_A (\zeta \otimes x) + (-1)^{\abs{\zeta}} (\xi \otimes y) \{\zeta \otimes 1, 1 \otimes x\}_A.
\end{align*}
Note that we did use that $\{1 \otimes x, 1 \otimes y\}_A = 0 = \{\xi \otimes 1, \zeta \otimes 1\}_A$.
\end{proof}
\begin{rem}
In a very recent article (see \cite{LeZh13}), J.\,Le and G.\,Zhou prove that if $A$ and $B$ are $k$-algebras over a field $k$, such that one amongst the two is finite dimensional over $k$, then
$$
\HH^\bullet(A \otimes_k B) \cong \HH^\bullet(A) \otimes_k \HH^\bullet(B)
$$
as Gerstenhaber algebras. They use the isomorphism to deduce the multiplicative structure, as well as the Lie structure of $\HH^\bullet(kG)$ where $G$ is an elementary abelian group of finite rank. More specifically, they demonstrate the following.
\begin{prop}[{\cite[Thm.\,4.3]{LeZh13}}]
Let $k$ be a field of characteristic $p > 0$, and let $G = (\mathbb Z/p\mathbb Z)^n$ be the elementary abelian $p$-group of rank $n \geq 1$.
\begin{enumerate}[\rm(1)]
\item If $p \neq 2$, then
\begin{align*}
\HH^\bullet(kG) &\cong \left(\frac{k[x_1, \dots,x_n]}{(x_1^p, \dots, x_n^p)}\right) \otimes_k \left(\Lambda(y_1, \dots, y_n) \otimes_k k[z_1, \dots, z_n]\right)\\
&\cong \left(\frac{k[x_1, \dots,x_n,z_1, \dots, z_n]}{(x_1^p, \dots, x_n^p)}\right) \otimes_k \Lambda(y_1, \dots, y_n)
\end{align*}
where $\abs{x_i} = 0$, $\abs{y_i} = 1$ and $\abs{z_i}=2$ for $i=1,\dots,n$. For the Gerstenhaber bracket $\{-,-\}_G$ on $\HH^\bullet(kG)$ we get
\begin{alignat*}{3}
\{x_i,y_j\}_G &= \delta_{ij}, \quad &\{x_i,x_j\}_G &= 0, \quad &\{y_i,y_j\}_G &= 0,\\
 \{x_i,z_j\}_G &= 0, \quad &\{y_i,z_j\}_G &= 0, \quad &\{z_i,z_j\}_G &= 0
\end{alignat*}
for all $i,j = 1, \dots, n$.
\item If $p = 2$, then
$$
\HH^\bullet(kG) \cong \left(\frac{k[x_1, \dots,x_n]}{(x_1^2, \dots, x_n^2)}\right)
\otimes_k k[y_1, \dots, y_n] \cong \frac{k[x_1, \dots,x_n, y_1, \dots, y_n]}{(x_1^2, \dots, x_n^2)},
$$
where $\abs{x_i} = 0$ and $\abs{y_i} = 1$ for $i=1,\dots,n$. For the Gerstenhaber bracket $\{-,-\}_G$ on $\HH^\bullet(kG)$ we get
\begin{align*}
\{x_i,y_j\}_G = \delta_{ij}, \quad \{x_i,x_j\}_G = 0, \quad \{y_i,y_j\}_G = 0
\end{align*}
for all $i,j = 1, \dots, n$.
\end{enumerate}
\end{prop}
Observe that these results for the bracket not only match the observations made in the Introduction \ref{int:kernel} for finite cyclic groups, but also confirm our more general result stated in Corollary \ref{cor:centerdetermins} (see Corollary \ref{cor:gerstenhabervanishhopfalgebra} as well).
\end{rem}



\chapter{Application II: The $\mathbf{\Ext}$-algebra of the identity functor}\label{chap:identityfunc}
Over fields, Hochschild cohomology of associative algebras has at least three incarnations, two of which we have already encountered previously. The third one manifests as the $\Ext$-algebra of the identity functor on the category of left modules. To be more precise, let $k$ be a commutative ring, and let $A$ be a $k$-algebra. There is a canonical commutative triangle of graded $k$-algebra homomorphisms
$$
\xymatrix@!C=50pt@R=10pt{
 && \Ext^\bullet_{\mathsf{End}_k(\Mod(A))}(\Id_{\Mod(A)}, \Id_{\Mod(A)}) \ar[dd]^-{\beta} \ \ \\
\HH^\bullet(A) \ar@/^1pc/[urr]^-{\alpha} \ar@/_0.8pc/[drr]^-{\chi_A} &&\\
&& \Ext^\bullet_{A^\ev}(A,A) \ ,
}
$$
where $\sfEnd_k(\Mod(A))$ is the abelian category of endofunctors on $\Mod(A)$. It is known (as sort of a folklore, since a reference is nowhere to be found), that $\beta$ is an isomorphism if $A$ is a flat $k$-module. The map $\chi_A$ is known to be bijective if $A$ is $k$-projective. Our construction introduced in Chapter \ref{ch:bracket} yields a bracket $[-,-]_\A$ on $\Ext^\bullet_{\mathsf{End}_k(\Mod(A))}(\Id_{\Mod(A)}, \Id_{\Mod(A)})$. If $A$ is a projective $k$-module, we will prove that $\alpha = \beta^{-1} \circ \chi_A$ preserves the brackets in the sense that it maps the Gerstenhaber bracket $\{-,-\}_A$ on $\HH^\bullet(A)$ to the bracket $[-,-]_\A$. In what follows, let $k$ be a fixed commutative ring. Set theoretical issues of any kind will generously be ignored (as usual).
\section{The evaluation functor}\label{sec:eval}
\begin{nn}
Let $\A$ be an abelian $k$-category and let $X$ be an object in $\A$. Let $\sfEnd_k(\A)$ be the abelian $k$-category $\Fun_k(\A,\A)$, which we will refer to as the \textit{category of endofunctors on $\A$}. Moreover, let $E_\A(X)$ be the opposite endomorphism ring $\End_\A(X)^\op$ of $X$. The \textit{evaluation functor at $X$}, denoted by $\ev_X$, is given by
$$
\ev_X: \sfEnd_k(\A) \longrightarrow \A, \ \ev_X(\scrX) = \scrX(X).
$$
When composed with $\Hom_\A(X,-)$, we obtain a functor
$$
H_X: \sfEnd_k(\A) \longrightarrow \Mod(E_\A(X)), \ H_X(\scrX) = \Hom_\A(X,\scrX(X)).
$$
Note that the left $E_\A(X)$-module action on $H_X(\scrX)$ for $\scrX \in \Ob \sfEnd_k(\A)$ is given by
$$
ah := h \circ a \quad (\text{for $a \in E_\A(X), \ h \in \Hom_\A(X,\mathscr \scrX(X))$}).
$$
In fact, $H_X$ takes values in $\Mod(E_\A(X)^\ev)$ for
$$
ha := \scrX(a) \circ h \quad (\text{for $a \in E_\A(X), \ h \in \Hom_\A(X,\scrX(X))$})
$$
defines a right $E_\A(X)$-module structure on $\Hom_\A(X,\scrX(X))$, turning it into an $E_\A(X)$-$E_\A(X)$-bimodule with central $k$-action.
\end{nn}
\begin{lem}
The functor
$$
H_X: \sfEnd_k(\A) \longrightarrow \Mod(E_\A(X)^\ev), \ H_X(\scrX) = \Hom_\A(X,\scrX(X))
$$
is $k$-linear and commutes with limits $($if existent$)$. Hence $H_X$ has a left adjoint if $\A$ is complete.
\end{lem}
\begin{proof}
The first two assertions are clear, because $H_X$ is the composite of two $k$-linear and limit preserving functors. The latter follows form the fact below.
\end{proof}
\begin{thm}[{\cite[Thm.\,X.1.2]{MaL98}}]
Let $\scrX: \A \rightarrow \mathsf B$ be a $k$-linear functor between complete abelian $k$-linear categories. Then $\scrX$ has a left adjoint functor if it commutes with limits $($that is, if $\scrX$ is \textit{continuous}$)$. In this case, a left adjoint $\scrX^\lambda$ is given by
$$
\scrX^\lambda(A) = \li(\mathscr Q: (A \downarrow \scrX) \longrightarrow \A) \quad \text{$($for $A$ in $\A$$)$},
$$
where $\mathscr Q$ denotes the projection functor from the slice category $A \downarrow \scrX$ to $\A$.
\end{thm}
\begin{nn}
Let $A$ be a $k$-algebra. For the remainder of this section, let $\A$ be the complete category $\Mod(A)$ of all left $A$-modules. If $\scrX$ is a $k$-linear endofunctor of $\A$, then $\scrX(A)$ is a right $A$-module through $xa = \scrX(\cdot a)(x)$ ($a \in A, \ x \in \scrX(A)$). Note that $E_\A(A) = \End_A(A)^\op \cong A$ and that
$$
\Hom_A(A,\scrX(A)) \cong \scrX(A) = \ev_A(\scrX) \quad (\text{for $\scrX \in \Ob \sfEnd_k(\Mod(A))$})
$$
naturally in $\scrX$ and as $A^\ev$-modules. Hence $H_A: \sfEnd_k(\A) \rightarrow \Mod(A^\ev)$ coincides with the evaluation functor $\ev_A$ at $A$.
\end{nn}
\begin{lem}\label{lem:adjeval}
The left adjoint functor $\ev_A^\lambda: \Mod(A^\ev) \rightarrow \End_k(\A)$ of $\ev_A$ is given by
$$
\ev_A^\lambda(M) = (M \otimes_A -) \quad \text{$($for $M \in \Ob \Mod(A^\ev))$}.
$$
\end{lem}
\begin{proof}
We will show that the functor $\ev_A^\lambda$, $\ev_A^\lambda(M) = (M \otimes_A -)$, is a left adjoint functor of $\ev_A$. To this end, we will define the unit and the counit of the claimed adjunction:
$$
\eta: \Id_{\Mod(A^\ev)} \longrightarrow \text{ev}_A \circ \text{ev}_A^\lambda, \quad \varepsilon: \text{ev}_A^\lambda \circ \text{ev}_A \longrightarrow \Id_{\End_k(\A)}.
$$
Clearly, we have the following functorial isomorphisms:
$$
\text{ev}_A(M \otimes_A -) = M \otimes_A A \cong M \quad \text{(for $M \in \Mod(A^\ev)$)},
$$
giving rise to the unit. For $\scrX \in \Ob \End_k(\A)$, define
$$
\varepsilon_{\scrX}: (\scrX(A) \otimes_A -) \longrightarrow \scrX
$$
as follows. Let $M$ be an $A^\ev$-module and $m \in M$. Denote by $\cdot m$ the $A$-linear homomorphism $A \rightarrow M, \ a \mapsto am$ and define
$$
\varepsilon_{\scrX, M}: {\scrX}(A) \otimes_A M \longrightarrow \scrX(M), \ x \otimes m \mapsto {\scrX}(\cdot m)(x). 
$$
Indeed, we obtain a commutative diagram
$$
\xymatrix@C=44pt{
\scrX(A) \otimes_A M \ar[r]^{\id_{\scrX(A)} \otimes_A f} \ar[d]_{\varepsilon_{\scrX, M}} & \scrX(A) \otimes_A N \ar[d]^{\varepsilon_{{\scrX}, N}}\\
{\scrX}(M) \ar[r]^{{\scrX}(f)} & \scrX(N)
}
$$
for every $A$-linear homomorphism $f: M \rightarrow N$ since $f \circ (\cdot m) = \cdot f(m)$, and therefore
\begin{align*}
(\varepsilon_{\scrX,N} \circ \id_{\scrX(A)} \otimes_A f) (a \otimes m) = \varepsilon_{\scrX,N}(a \otimes f(m)) = \scrX(\cdot f(m))(a)
\end{align*}
and
\begin{align*}
(\scrX(f) \circ \varepsilon_{\scrX,M}) (a \otimes m) = \scrX(f)(\scrX(\cdot m)(a)) = \scrX(f \circ (\cdot m))(a) = \scrX(\cdot f(m))(a).
\end{align*}
We now check, that the morphisms $\varepsilon_{\scrX}, \ {\scrX \in \Ob \sfEnd_k(\Mod(A^\ev))},$ are functorial in $\scrX$. Let $\scrX, \scrY \in \Ob \sfEnd_k(\A)$ and let $M$ be an $A$-module. Moreover, let $\theta:
\scrX \rightarrow \scrY$ be a natural transformation. The diagram
$$
\xymatrix@C=40pt{
\scrX(A) \otimes_A M \ar[r]^-{\theta_A \otimes_A M} \ar[d]_-{\varepsilon_{\scrX,M}} & \scrY(A) \otimes_A M
\ar[d]^-{\varepsilon_{\scrY,M}}\\
\scrX(M) \ar[r]^-{\theta_M} & \scrY(M)
}
$$
is commutative since
\begin{align*}
(\varepsilon_{\scrY,M} \circ (\theta_A \otimes_A M))(a \otimes m) &= \varepsilon_{\scrY,M}(\theta_A(a) \otimes m) = \scrY(\cdot m)(\theta_A(a)),\\
(\theta_M \circ \varepsilon_{\scrX,M})(a \otimes m) &= \theta_M(\scrX(\cdot m)(a))
\end{align*}
for all $a \in \scrX(A), \ m \in M$, and since
$$
\xymatrix@C=30pt{
\scrX(A) \ar[r]^-{\scrX(\cdot m)} \ar[d]_-{\theta_A} & \scrX(M) \ar[d]^-{\theta_M}\\
\scrY(A) \ar[r]^-{\scrY(\cdot m)} & \scrY(M)
}
$$
commutes due to the naturality of $\theta$. To completely establish the adjunction, we finally have to verify that the compositions
$$
\xymatrix@C=30pt{
\mathrm{ev}_A(\scrX) \ar[r]^-{\eta_{\mathrm{ev}_A(\scrX)}} & \mathrm{ev}_A \mathrm{ev}^{\lambda}_A
\mathrm{ev}_A(\scrX) \ar[r]^-{\mathrm{ev}_A(\varepsilon_{\scrX})} & \mathrm{ev}_A(\scrX)
}
$$
and
$$
\xymatrix@C=30pt{
\mathrm{ev}^{\lambda}_A(M) \ar[r]^-{\mathrm{ev}^{\lambda}_A(\eta_{M})} & \mathrm{ev}^{\lambda}_A
\mathrm{ev}_A
\mathrm{ev}^{\lambda}_A(M) \ar[r]^-{\varepsilon_{\mathrm{ev}^{\lambda}_A(M)}} &
\mathrm{ev}^{\lambda}_A(M)
}
$$
are the identity morphisms on $\scrX(A)$ and $M \otimes_A -$ respectivley (for $\scrX \in \sfEnd_k(\A), \ M \in \Mod(A^{\ev}))$. In fact, for $a \in \scrX(A)$, $N \in \Mod(A)$ and $m \otimes n \in M \otimes_A N$, we get that
\begin{align*}
(\mathrm{ev}_A(\varepsilon_{\scrX}) \circ \eta_{\mathrm{ev}_A(\scrX)})(a) &= \varepsilon_{\scrX,A}(a \otimes 1_A) = \scrX(\cdot 1_A)(a) = a
\end{align*}
as well as
\begin{align*}
(\varepsilon_{\mathrm{ev}^{\lambda}_A(M)} \circ \mathrm{ev}^{\lambda}_A(\eta_{M}))_N(m \otimes n) &=
\varepsilon_{\mathrm{ev}^{\lambda}_A(M), N}((m \otimes 1_A) \otimes n)\\
&= \mathrm{ev}^{\lambda}_A(M)(\cdot n)(m \otimes 1_A)\\ &= m \otimes n. 
\end{align*}
Hence the lemma is established.
\end{proof}
\section{Exact endofunctors}
The following lemma generalizes a known result for the tensor functor $\otimes_R$ over a ring $R$.
\begin{lem}[{\cite[Ch.\,1, Prop.\,2.6]{Liu02}}]\label{lem:flatexact}
Let $\A_1$, $\A_2$ and $\mathsf B$ be abelian categories and let $\mathscr R: \A_1 \times \A_2 \rightarrow \mathsf B$ be an additive functor such that the functors $\mathscr R^N = \mathscr R(-,N)$ are right exact for every object $N \in \A_2$. Further, assume that $\A_2$ and $\mathscr R$ are such that $\A_2$ admits a generating class $\U \subseteq \A_2$ with $\mathscr R^U=\mathscr R(-,U)$ being an exact functor $\A_1 \rightarrow \mathsf B$ for every $U \in \U$. Let $0 \rightarrow X \rightarrow Y \rightarrow Z \rightarrow 0$ be a short exact sequence in $\A_1$ such that $\mathscr R_X = \mathscr R(X,-)$ and $\mathscr R_Y = \mathscr R(Y,-)$ are right exact, and $\mathscr R_Z = \mathscr R(Z,-)$ is exact. Then the sequence
$$
0 \longrightarrow \mathscr R^N(X) \longrightarrow \mathscr R^N(Y) \longrightarrow \mathscr R^N(Z) \longrightarrow 0
$$
is exact in $\mathsf B$ for every object $N \in \A$, that is, the induced sequence
$$
0 \longrightarrow \mathscr R_X \longrightarrow \mathscr R_Y \longrightarrow \mathscr R_Z \longrightarrow 0
$$
of functors is exact.
\end{lem}
\begin{proof}
By assumption, the sequence $\mathscr R^N(X) \rightarrow \mathscr R^N(Y) \rightarrow \mathscr R^N(Z) \rightarrow 0$ is exact. It remains to show that $\mathscr R^N(X) \rightarrow \mathscr R^N(Y)$ is a monomorphism. Let $\pi: U \rightarrow N$ be an epimorphism in $\A_2$ with $U \in \U$ and set $K:=\Ker(\pi)$. The diagram
$$
\xymatrix{
& \Ker(f) \ar[d] \ar[r] & \Ker(g) \ar[r] \ar[d] & 0 \ar[d] \ar[r] & 0\\
& \mathscr R(X,K) \ar[r] \ar[d]_-f & \mathscr R(Y,K) \ar[r] \ar[d]_-g & \mathscr R(Z,K) \ar[d] \ar[r] & 0\\
0 \ar[r] & \mathscr R(X,U) \ar[r] \ar[d] & \mathscr R(Y,U) \ar[r] \ar[d] & \mathscr R(Z,U) \ar[r] \ar[d] & 0\\
& \mathscr R(X, N) \ar[r] \ar[d] & \mathscr R(Y,N) \ar[d] \ar[r] & \mathscr R(Z,N) \ar[r] \ar[d] & 0 \\
& 0 & 0 & 0
}
$$
has exact rows and columns, for $\mathscr R_Z(K) \rightarrow \mathscr R_Z(U)$ is a monomorphism ($\mathscr R_Z$ is exact). So we may apply the Snake Lemma to find that $\mathscr R^N(X)\rightarrow \mathscr R^N(Y)$ has indeed vanishing kernel.
\end{proof}
\begin{exa}
Let $A$ be a $k$-algebra. As we already know, the category $\Mod(A^\ev)$ is monoidal thanks to the tensor functor $\otimes_A$ and to the tensor unit $A$. The class of projective $A^\ev$-modules is a class $\U$ that the previous lemma can be applied to.
\end{exa}
\begin{nn}
If $\A$ is any $k$-linear abelian category, the category $\sfEnd_k(\A)$ carries additional structure. Namely, it is a monoidal (even tensor) category with tensor functor given by the composition $\circ: \sfEnd_k(\A) \times \sfEnd_k(\A) \rightarrow \sfEnd_k(\A)$ of functors. Hence the tensor unit is the identity functor $\Id_\A$. The category $\underline{\sfEnd}_k(\A)$ of \textit{exact} $k$-linear endofunctors of $\A$ is an extension closed and monoidal subcategory of $\sfEnd_k(\A)$ by the following lemma.
\end{nn}
\begin{lem}\label{lem:exfunctors}\begin{enumerate}[\rm(1)]
Let $\C$ be an exact $k$-category. 
\item\label{lem:exfunctors:1} The full subcategory $\underline{\Fun}_k(\C, \A)$ of $\Fun_k(\C,\A)$ consisting of all exact $k$-linear functors $\C \rightarrow \A$ has the two out of three property, that is, if $0 \rightarrow \scrX \rightarrow \scrY \rightarrow \scrZ \rightarrow 0$ is a short exact sequence in $\Fun_k(\C,\A)$ with two out of the objects $\scrX, \scrY, \scrZ$ belonging to $\underline{\Fun}_k(\C,\A)$, then so does the third. In particular, $\underline{\Fun}_k(\C,\A)$ is an exact category that is closed under kernels of epimorphisms and under cokernels of monomorphisms.
\item\label{lem:exfunctors:2} For every $\scrX \in \Ob \sfEnd_k(\A)$, the functor $- \circ \scrX: \sfEnd_k(\A) \rightarrow \sfEnd_k(\A)$ is an exact functor. If $\scrX$ is exact, then $\scrX \circ -: \sfEnd_k(\A) \rightarrow \sfEnd_k(\A)$ is also exact. Hence, the composition of endofunctors is exact in both variables if one views it as a bifunctor $\circ : \underline{\sfEnd}_k(\A) \times \underline{\sfEnd}_k(\A) \longrightarrow \underline{\sfEnd}_k(\A)$.
\end{enumerate}
\end{lem}
\begin{proof}
Let us deduce assertion (\ref{lem:exfunctors:1}). It will turn out to be an elementary consequence of the Snake Lemma. Let $0 \rightarrow \scrX \rightarrow \scrY \rightarrow \scrZ \rightarrow 0$ be an exact sequence in $\Fun_R(\C,\A)$. Take an admissible short exact sequence $0 \rightarrow X \rightarrow Y \rightarrow Z \rightarrow 0$ in $\C$ and consider the commutative diagram
$$
\xymatrix@C=25pt{
0 \ar[r] & \Ker(f) \ar[r] \ar[d] & \Ker(g) \ar[r] \ar[d] & \Ker(h) \ar[d] & &\\
0 \ar[r] & \scrX(X) \ar[r] \ar[d]_-f & \scrY(X) \ar[r] \ar[d]_-g & \scrZ(X) \ar[r] \ar[d]^-h & 0 &\\
0 \ar[r] & \scrX(Y) \ar[r] \ar@/_1pc/[dr] \ar[dd] & \scrY(Y) \ar[r] \ar@/_1pc/[dr] \ar'[d][dd] & \scrZ(Y) \ar[r] \ar@/_1pc/[dr] \ar'[d][dd] & 0 &\\
& & \Coker(f) \ar[r] \ar@/_1pc/@{-->}[dl]^-{a}  & \Coker(g) \ar[r] \ar@/_1pc/@{-->}[dl]^{b} & \Coker(h) \ar[r] \ar@/_1pc/@{-->}[dl]^-{c}  & 0\\
0 \ar[r] & \scrX(Z) \ar[r] & \scrY(Z) \ar[r] & \scrZ(Z) \ar[r] & 0 &
}
$$
having exact rows (here the dotted morphisms are the unique ones induced by the universal property cokernels possess). It is clear that if $g$ and $h$ were monomorphisms, then so was $f$. Similarly, if $\Ker(f) = 0$ and $\Ker(h) = 0$, then $\Ker(g)$ injects into $\Ker(h) = 0$ and hence has to be $0$ itself. Now let $\scrX$ and $\scrY$ be exact functors. Thus, $\Ker(f) = 0 = \Ker(g)$ and the morphisms $a, b$ in the diagram above are isomorphisms. In particular, $\Coker(f) \rightarrow \Coker(g)$ is a monomorphism, and the Snake Lemma therefore yields the exact sequence
$$
\xymatrix{
0 \ar[r] & \Ker(h) \ar[r]^-0 & \Coker(f) \ar[r] & \Coker(g) \ar[r] & \Coker(h) \ar[r] & 0 \ .
}
$$
This implies that $h$ is a monomorphism. We have therefore shown, that if two out of the functors $\scrX, \scrY, \scrZ$ are exact, the third will preserve monomorphisms. To complete the proof, consider the commutative diagram
$$
\xymatrix{
\Ker(h) \ar[r] & \Coker(f) \ar[r] \ar[d]_-a  & \Coker(g) \ar[r] \ar[d]_-{b} & \Coker(h) \ar[r] \ar[d]^-c & 0 \ \ \\
0 \ar[r] & \scrX(Z) \ar[r] & \scrY(Z) \ar[r] & \scrZ(Z) \ar[r] & 0 \ ;
}
$$
as we have seen, $\Ker(h) = 0$ as soon as two functors out of $\scrX, \scrY, \scrZ$ are exact. Assume that $\scrY$ and $\scrZ$ are exact. Then $b$ and $c$ are isomorphisms and by the 5-Lemma, so is $a$. Hence $\scrX(Y) \rightarrow \scrX(Z)$ is a cokernel of $h$. Analogously, $\scrX$ and $\scrY$ being exact implies that $\scrZ(Y) \rightarrow \scrZ(Z)$ is a cokernel of $h$, whereas $\scrX$ and $\scrZ$ being exact yields that $\scrY(Y) \rightarrow \scrY(Z)$ is a cokernel of $g$. Hereby we have established (\ref{lem:exfunctors:1}).

If $0 \rightarrow \scrX' \rightarrow \scrY' \rightarrow \scrZ' \rightarrow 0$ is exact in $\sfEnd_k(\A)$, then so is
$$
0 \longrightarrow (\scrX'\circ\scrX)(X) \longrightarrow (\scrY'\circ\scrX)(X) \longrightarrow (\scrZ'\circ\scrX)(X) \longrightarrow 0
$$
for every object $X$ in $\A$. The second statement in (\ref{lem:exfunctors:2}) is clear.
\end{proof}
\begin{lem}\label{lem:endofuncprojres}
Let $A$ be a $k$-algebra, $\A = \Mod(A)$ and let $\scrX$ be in $\sfEnd_k(\A)$. Assume that $A$ is a projective $k$-module. If $\scrX$ is equivalent to $M \otimes_A -$ for some $A^\ev$-module $M$ which is flat as a right $A$-module, then $\scrX(A) \cong M$ and $\scrX$ admits a projective resolution in $\sfEnd_k(\A)$ by exact endofunctors.
\end{lem}
\begin{proof}
The isomorphism $\scrX(A) \cong M$ follows immediately. It remains to exhibit a projective resolution of $\scrX$ by exact functors. Chose a projective resolution
$$
\cdots \longrightarrow P_2 \longrightarrow P_1 \longrightarrow P_0 \longrightarrow \scrX(A) \longrightarrow 0,
$$
(denoted by $\mathbb P_{\scrX(A)} \rightarrow \scrX(A) \rightarrow 0$ for short) of $\scrX(A)$ over $A^\ev$. After applying $\ev_A^\lambda$, we obtain a sequence
$$
\cdots \longrightarrow (P_2 \otimes_A -) \longrightarrow (P_1 \otimes_A -) \longrightarrow (P_0 \otimes_A -) \longrightarrow \scrX \longrightarrow 0
$$
of exact functors. This sequence is exact, since $\Ker(P_{i} \rightarrow P_{i-1})$ is flat as a right $A$-module for every $i \geq 1$ (use Lemma \ref{lem:flatexact} and Lemma \ref{lem:exfunctors}.(\ref{lem:exfunctors:1}) to perform an induction on $i$; alternatively, observe that $H_n(\mathbb P_{\scrX(A)} \otimes_A N) = \Tor_n^A(\scrX(A),N)$ for every left $A$-module $N$). Moreover, the endofunctors $P_i \otimes_A -$ are projective objects in $\sfEnd_k(\A)$ since $\ev_A^\lambda$ is a left adjoint functor of an exact functor (and therefore, it preserves projectivity; see \ref{nn:leftadjointprop}).
\end{proof}
\begin{rem}\label{rem:eilenbergwatts}
By the Eilenberg-Watts-Theorem (see \cite{NySm08} for a generalized version), we precisely know which endofunctors have the shape of those in the previous lemma: Every exact functor $\scrX: \Mod(A) \rightarrow \Mod(A)$ which commutes with arbitrary coproducts is isomorphic to $\scrX(A) \otimes_A -$.
\end{rem}
\begin{prop}
Let $A$ be projective over $k$ and let $\A = \Mod(A)$. Then the functors
$$
j_n: \mathcal Ext^n_{\underline{\sfEnd}_k(\A)}(\Id_\A, \Id_\A) \longrightarrow \mathcal Ext^n_{{\sfEnd_k(\A)}}(\Id_\A, \Id_\A)\quad \text{$($for $n \geq 0)$},
$$
induced by the inclusion $j: \underline{\sfEnd}_k(\A) \rightarrow {\End}_k(\A)$, define an isomorphism
$$
\Ext^\bullet_{\underline{\sfEnd}_k(\A)}(\Id_\A, \Id_\A) \cong \Ext^\bullet_{{\sfEnd_k(\A)}}(\Id_\A, \Id_\A)
$$
of graded $k$-algebras.
\end{prop}
\begin{proof}
We check the conditions (\ref{lem:exhproj:0})--(\ref{lem:exhproj:2}) of Proposition \ref{lem:exhproj}. By Lemma \ref{lem:exfunctors}, the category $\underline{\sfEnd}_k(\A)$ already is ($1$-)extension closed; also it is closed under kernels of epimorphisms. Lemma \ref{lem:endofuncprojres} shows that there is a projective resolution of $\Id_\A$ which belongs to $\underline{\sfEnd}_k(\A)$.
\end{proof}
\section{$\mathbf{\Ext}$-algebras and adjoint functors}
\begin{nn}
Troughout this section, let $k$ be a commutative ring and let $\A$ and $\sfB$ be abelian $k$-categories. Assume that we are given a pair
$$
\xymatrix{
\A \ar@<-.5ex>[r]_-\scrR & \sfB \ar@<-.5ex>[l]_-\scrL
}
$$
of functors. Remember that $\scrL$ is \textit{left adjoint to $\scrR$} if there are $k$-linear isomorphisms
$$
\varphi_{A,B}: \Hom_{\A}(\scrL(B), A) \longrightarrow \Hom_{\B}(B, \scrR(A)) \quad (A \in \Ob(\A), \ B \in \Ob(\B))
$$
which are natural in $A$ and $B$. Any other functor $\scrL': \B \rightarrow \A$ admitting such isomorphisms will be isomorphic to $\scrL$. It is well known (a proof is given in \cite{NySm08}) that if $\scrR$ is $k$-linear, then so is $\scrL$. Recall that $\scrL$ is left adjoint to $\scrR$ if, and only if, there are natural transformations $\eta: \Id_\B \rightarrow \scrR \circ \scrL$ (the \textit{unit} of the adjunction) and $\varepsilon: \scrL \circ \scrR \rightarrow \Id_\A$ (the \textit{counit} of the adjunction) such that the compositions
$$
\xymatrix@C=20pt{
\scrR(A) \ar[r]^-{\eta_{\scrR(A)}} & (\scrR \circ \scrL \circ\scrR)(A) \ar[r]^-{\scrR(\varepsilon_A)} & \scrR(A) \ , & \scrL(B) \ar[r]^-{\scrL(\eta_B)} & (\scrL \circ \scrR \circ \scrL)(B) \ar[r]^-{\varepsilon_{\scrL(B)}} & \scrL(B)
}
$$
are the identity morphisms for every $A \in \Ob(\A)$ and $B \in \Ob(\B)$. When $\eta$ and $\varepsilon$ are given, the adjunction isomorphisms
$$
\varphi_{A,B}: \Hom_{\A}(\scrL(B), A) \longrightarrow \Hom_{\B}(B, \scrR(A)) \quad (A \in \Ob(\A), \ B \in \Ob(\B))
$$
may be defined by $\varphi_{A,B}(f) = \scrR(f) \circ \eta_B$ with inverse map given as $\varphi_{A,B}^{-1}(g) = \varepsilon_A \circ \scrL(g)$ (for $f \in \Hom_\A(\scrL(B),A)$ and $g \in \Hom_{\B}(B, \scrR(A))$). In the remainder, we will assume that $\scrR$ is $k$-linear and $\scrL$ is left adjoint to $\scrR$. Further, we abbreviate $\scrL \circ \scrR$ and $\scrR \circ \scrL$ by $\scrL\scrR$ and $\scrR\scrL$ respectivley.
\end{nn}
\begin{nn}\label{nn:leftadjointprop}
Let $\scrL$ be left adjoint to $\scrR$. Then $\scrR$ is full and faithful if, and only if, the corresponding counit $\varepsilon: \scrL\scrR \rightarrow \Id_\A$ is an isomorphism. If $\scrR$ is an exact functor, then $\scrL$ will preserve pushouts, cokernels and coproducts (dual statement of \cite[Thm.\,II.7.7]{HiSt97}). Further, if $P$ is a projective object in $\B$, $\scrL(P)$ will be projective in $\A$ for $\Hom_{\A}(\scrL(P),-) \cong \Hom_{\B}(P,\scrR(-))$ is the composition of exact functors.
\end{nn}
\begin{nn}\label{nn:projlifting}
Recall the following construction (which we already have used intensively in previous chapters). Let $X$, $X'$ and $Y$ be objects in $\A$, and let $f: X \rightarrow X'$ be a morphism. Let $\xi \ \equiv \ 0 \rightarrow Y \rightarrow E_{n-1} \rightarrow \cdots E_0 \rightarrow X' \rightarrow 0$ be an $n$-extension in $\A$ for some fixed integer $n \geq 1$. There is a morphism of complexes,
$$
\xymatrix@C=20pt@R=20pt{
\cdots \ar[r] & P_{n} \ar[r]^-{p_n} \ar[d]_-{\varphi_n} & P_{n-1} \ar[r]^-{p_{n-1}} \ar[d]_-{\varphi_{n-1}} & \cdots \ar[r] & P_{1} \ar[r]^-{p_1} \ar[d]^-{\varphi_1} & P_{0} \ar[r]^-{p_0} \ar[d]^-{\varphi_0} & X \ar[r] \ar[d]^-f & 0 \ \ \\
0 \ar[r] & Y \ar[r] & E_{n-1} \ar[r] & \cdots \ar[r] & E_{1} \ar[r] & E_{0} \ar[r] & X' \ar[r] & 0 \ ,
}
$$
lifting $f$ (cf. Lemma \ref{lem:comparison}). By pushing out $(\varphi_n, p_n)$ we obtain an $n$-extension $\xi'=\xi(\varphi_n)$, and a commutative diagram
$$
\xymatrix@C=20pt@R=20pt{
\xi' & \equiv & 0 \ar[r] & Y \ar[r] \ar@/^1.801pc/@[black][dd] & P \ar@/^1.8pc/@[black]@{-->}[dd] \ar[r] & \cdots \ar[r] & P_{1} \ar[r] & P_{0} \ar[r] & X \ar[r] & 0 \ \ \\
&& \cdots \ar[r] & P_{n} \ar[r]|{ \ } \ar[d]_-{\varphi_n} \ar[u]^-{\varphi_n} & P_{n-1} \ar[r]|{ \ } \ar[u] \ar[d]_-{\varphi_{n-1}} & \cdots \ar[r] & P_{1} \ar[r] \ar[d]^-{\varphi_1} \ar@{=}[u] & P_{0} \ar[r] \ar[d]^-{\varphi_0} \ar@{=}[u] & X \ar[r] \ar[d]^-f \ar@{=}[u] & 0 \ \ \\
\xi & \equiv & 0 \ar[r] & Y \ar[r] & E_{n-1} \ar[r] & \cdots \ar[r] & E_{1} \ar[r] & E_{0} \ar[r] & X' \ar[r] & 0 \ .
}
$$
Suppose $X = X'$ and $f = \id_X$. Thanks to the dashed arrow in the above diagram, a morphism $\alpha_\xi: \xi' \rightarrow \xi$ in $\mathcal Ext^n_\A(X,Y)$ is given; therefore $\xi$ and $\xi'$ define the same element in $\Ext^n_\A(X,Y)$.
\end{nn}
\begin{prop}\label{prop:isoadjointfull}
Let $A$ be an object in $\A$ and $B$ be an object in $\B$ such that $B$ admits a projective resolution $\cdots \rightarrow P_1 \rightarrow P_0 \rightarrow B \rightarrow 0$ in $\B$. Assume that
\begin{enumerate}[\rm(1)]
\item\label{prop:isoadjointfull:1} $\scrL$ is left adjoint to $\scrR$,
\item the unit $\eta: \Id_\B \rightarrow \scrR \scrL$ is an isomorphism,
\item\label{prop:isoadjointfull:2} $\scrR$ is an exact functor, and that
\item\label{prop:isoadjointfull:3} the sequence $\cdots \rightarrow \scrL(P_1) \rightarrow \scrL(P_0) \rightarrow \scrL(B) \rightarrow 0$ is exact in $\A$.
\end{enumerate}
Then the $k$-linear maps
$$
\scrR^\sharp_n: \Ext^n_{\A}(\scrL(B), A) \longrightarrow \Ext^n_{\B}(\scrR\scrL(B), \scrR(A))  \quad \text{$($for $n \geq 0)$}
$$
induced by $\scrR$ are surjective. They are also injective, in case $A = \scrL(B')$ for some object $B'$ in $\B$. In particular,
$$
\scrR_\bullet^\sharp: \Ext^\bullet_{\A}(\scrL(B), \scrL(B)) \longrightarrow \Ext^\bullet_{\B}(\scrR\scrL(B), \scrR\scrL(B)) \cong \Ext^\bullet_\B(B,B)
$$
is an isomorphism of graded $k$-algebras for every $B \in \Ob\B$.
\end{prop}
\begin{proof}
The proof is spread over the upcoming Lemmas \ref{lem:surjectivityprop} and \ref{lem:injectivityprop}.
\end{proof}
\begin{lem}\label{lem:surjectivityprop}
Under the assumptions of Proposition $\ref{prop:isoadjointfull}$, the $k$-linear maps
$$
\scrR_n^\sharp : \Ext^n_{\A}(\scrL(B), A) \longrightarrow \Ext^n_{\B}(\scrR\scrL(B), \scrR(A))\quad \text{$($for $n \geq 0)$}
$$
induced by $\scrR$ are surjective.
\end{lem}
\begin{proof}
Let $\eta: \Id_\B \rightarrow \scrR \scrL$ be the unit and $\varepsilon: \scrL\scrR \rightarrow \Id_\A$ the counit of the adjunction. The commutative diagram
$$
\xymatrix{
\Hom_\A(\scrL(B),A) \ar[r]^-{\scrR_{\scrL(B),A}}  & \Hom_\B(\scrR\scrL(B), \scrR(A)) \ar@{=}[d]\\
\Hom_\A(\scrL\scrR\scrL(B),A) \ar[r]^-\cong \ar@{->>}[u]^-{\Hom_\A(\eta_{\scrL(B)}, A)}   & \Hom_\B(\scrR\scrL(B),\scrR(A))
}
$$
verifies the claim in degree zero: it shows that $\scrR^\sharp_0$ is surjective. Let $n \geq 1$ be an integer. Take an $n$-extension $\xi \ \equiv \ 0 \rightarrow \scrR(A) \rightarrow E_{n-1} \rightarrow \cdots \rightarrow E_0 \rightarrow \scrR\scrL(B) \rightarrow 0$ in $\B$. There is a morphism
$$
\xymatrix@C=23pt@R=21pt{
\cdots \ar[r] & P_{n} \ar[r]^-{d_n} \ar[d]_-{\varphi_n} & P_{n-1} \ar[r]^-{p_{n-1}} \ar[d]_-{\varphi_{n-1}} & \cdots \ar[r] & P_{1} \ar[r]^-{p_1} \ar[d]^-{\varphi_1} & P_{0} \ar[r]^-{\eta_{B}^{-1} \circ p_0} \ar[d]^-{\varphi_0} & \scrR\scrL(B) \ar[r] \ar@{=}[d] & 0\\
0 \ar[r] & \scrR(A) \ar[r] & E_{n-1} \ar[r] & \cdots \ar[r] & E_{1} \ar[r] & E_{0} \ar[r] & \scrR\scrL(B) \ar[r] & 0
}
$$
of complexes and hence the follwing pushout diagram $\Xi$ is formable.
$$
\xymatrix@C=20pt@R=20pt{
&& \cdots \ar[r] & P_{n} \ar[r] \ar[d]_-{\varphi_n} & P_{n-1} \ar[r] \ar[d] & \cdots \ar[r] & P_{0} \ar[r] \ar@{=}[d] & \scrR\scrL(B) \ar[r] \ar@{=}[d] & 0\\
\xi' & \equiv & 0 \ar[r] & \scrR(A) \ar[r] & P \ar[r] & \cdots \ar[r] & P_{0} \ar[r] & \scrR\scrL(B) \ar[r] & 0
}
$$
The lower $n$-extension is equivalent to $\xi$ by observations made in \ref{nn:projlifting}. Since $\scrL$ preserves pushouts, the rows in $\scrL(\Xi)$ are exact. Consider the diagram
$$
\small
\xymatrix@C=13pt@R=18pt{
& 0 \ar[r] & \scrL\scrR(A) \ar[r] \ar[d]_-{\varepsilon_A} & \scrL(P) \ar[r] \ar[d]_-{r} & \scrL(P_{n-2}) \ar[r] \ar@{=}[d] & \cdots \ar[r] & \scrL(P_0) \ar[r] \ar@{=}[d] & \scrL\scrR\scrL(B) \ar[r] \ar@{=}[d] & 0\\
& 0 \ar[r] & A \ar[r] & Q \ar[r] & \scrL(P_{n-2}) \ar[r] & \cdots \ar[r] & \scrL(P_0) \ar[r] & \scrL\scrR\scrL(B) \ar[r] & 0\\
\zeta \quad \equiv & 0 \ar[r] & A \ar[r] \ar@{=}[u] & Q \ar[r] \ar@{=}[u] & \scrL(P_{n-2}) \ar[r] \ar@{=}[u] & \cdots \ar[r] & R \ar[r] \ar[u] & \scrL(B) \ar[r] \ar[u]_-{\scrL(\eta_B)} & 0\\
& \cdots \ar[r] & \scrL(P_n) \ar[r] \ar[u]^-{\psi_{n}} \ar[d]_-{\psi_{n}} & \scrL(P_{n-1}) \ar[r] \ar[u]^-{\psi_{n-1}} \ar[d] & \scrL(P_{n-2}) \ar[r] \ar[u]^-{\psi_{n-2}} \ar@{=}[d] & \cdots \ar[r] & \scrL(P_0) \ar[r] \ar[u]_-{\psi_0} \ar@{=}[d] & \scrL\scrR\scrL(B) \ar[r] \ar[u]_-{\varepsilon_{\scrL(B)}} \ar@{=}[d] & 0\\
\zeta' \quad \equiv & 0 \ar[r] & A \ar[r] & S \ar[r] & \scrL(P_{n-2}) \ar[r] & \cdots \ar[r] & \scrL(P_0) \ar[r] & \scrL\scrR\scrL(B) \ar[r] & 0
}
$$
with exact rows, where $Q$, $R$ and $S$ denote the pushouts and pullbacks of the obvious morphisms. Moreover, the morphism $\psi$ of chain complexes comes from the fact, that (since $\scrL$ preserves projectivity) $\cdots \rightarrow \scrL(P_1) \rightarrow \scrL(P_0) \rightarrow \scrL(B) \rightarrow 0$ is a projective resolution of $\scrL(B)$. We claim that $\scrR^\sharp_n$ maps the equivalence class of $\zeta$ to the equivalence class of $\xi$. In fact, there is a morphism $\alpha_\zeta: \zeta' \rightarrow \zeta$ of complexes (see \ref{nn:projlifting} for its definition) and, furthermore, a commutative diagram
$$
\small
\xymatrix@C=18pt@R=20pt{
0 \ar[r] & \scrR(A) \ar[r] \ar[d]_-{\eta_{\scrR(A)}} & P \ar[r] \ar[d]_-{\eta_{P}} & P_{n-2} \ar[r] \ar[d]_-{\eta_{P_{n-2}}} & \cdots \ar[r] & P_{0} \ar[r] \ar[d]^-{\eta_{P_0}} & \scrR(B) \ar[r] \ar[d]^-{\eta_{\scrR(B)}} & 0\\
0 \ar[r] & \scrR\scrL\scrR(A) \ar[r] \ar[d]_-{\scrR(\varepsilon_A)} & \scrR\scrL(P) \ar[r] \ar@/^2pc/[dd]^(0.3){\scrR(r)} & \scrR\scrL(P_{n-2}) \ar[r] \ar@{=}[d] & \cdots \ar[r] & \scrR\scrL(P_0) \ar[r] \ar@{=}[d] & \scrR\scrL\scrR(B) \ar[r] \ar@{=}[d] & 0\\
0 \ar[r] & \scrR(A) \ar[r] \ar@{=}[d] & \scrR(S) \ar[r]|(0.42){ \ } \ar[d] & \scrR\scrL(P_{n-2}) \ar[r] \ar@{=}[d] & \cdots \ar[r] & \scrR\scrL(P_0) \ar[r] \ar[d] & \scrR\scrL\scrR(B) \ar[r] \ar[d]^-{\scrR(\varepsilon_{(B)})} & 0\\
0 \ar[r] & \scrR(A) \ar[r] & \scrR(Q) \ar[r] & \scrR\scrL(P_{n-2}) \ar[r] & \cdots \ar[r] & \scrR(R) \ar[r] & \scrR(B) \ar[r] & 0
}
$$
which, when read from top to bottom, defines a morphism $\xi' \rightarrow \scrR(\zeta)$ of $n$-extensions. Hence, $\xi$ is linked to $\scrR(\zeta)$ in the following manner:
$$
\xymatrix{
\xi & \xi' \ar[l]_-{\alpha_\xi} \ar[r] & \scrR(\zeta).
}
$$
This completes the proof.
\end{proof}
\begin{lem}\label{lem:injectivityprop}
Under the assumptions of Proposition $\ref{prop:isoadjointfull}$, the $k$-linear maps
$$
\scrR_n^\sharp : \Ext^n_{\A}(\scrL(B), \scrL(B')) \longrightarrow \Ext^n_{\B}(\scrR\scrL(B), \scrR\scrL(B'))\quad \text{$($for $n \geq 0)$}
$$
induced by $\scrR$ are injective.
\end{lem}
\begin{proof}
Let $\xi$ be a $n$-extension $0 \rightarrow \scrL(B') \rightarrow E_{n-1} \rightarrow \cdots \rightarrow E_0 \rightarrow \scrL(B) \rightarrow 0$ in $\A$ such that $\zeta:=\scrR(\xi)$ is equivalent to the trivial $n$-extension $\sigma_n(\scrR\scrL(B), \scrR\scrL(B'))$ in $\B$. Chose a chain map
$$
\xymatrix@C=18pt{
\cdots \ar[r] & P_{n} \ar[r]^-{p_n} \ar[d]_-{\varphi_n} & P_{n-1} \ar[r]^-{p_{n-1}} \ar[d]_-{\varphi_{n-1}} & \cdots \ar[r]^-{p_2} & P_{1} \ar[r]^-{p_1} \ar[d]^-{\varphi_1} & P_{0} \ar[r]^-{\eta_{B}^{-1} \circ p_0} \ar[d]^-{\varphi_0} & \scrR\scrL(B) \ar[r] \ar@{=}[d] & 0 \ \ \\
0 \ar[r] & \scrR\scrL(B') \ar[r] & \scrR (E_{n-1}) \ar[r] & \cdots \ar[r] & \scrR (E_{1}) \ar[r] & \scrR (E_{0}) \ar[r] & \scrR\scrL(B) \ar[r] & 0 \ ,
}
$$
lifting the indentity of $B$. Then $\zeta$ is equivalent to $\zeta' = \zeta(\varphi_n)$ (see \ref{nn:projlifting}). Note that $\scrL$, as a left adjoint, preserves pushouts and hence the sequence $\scrL(\zeta')$ remains exact. We obtain the following commutative diagram, within the top and the bottom rows are exact, and $Q$ denotes $\scrR\scrL(B') \oplus_{P_n} P_{n-1}$.
$$
\small
\xymatrix@C=13pt@R=20pt{
\scrL(\zeta') \ar[d]_-{\scrL(\alpha_\zeta)} & \equiv & 0 \ar[r] & \scrL(B') \ar[d]_-{{\scrL(\eta_{B'})}} \ar[r] & \scrL(Q) \ar[d] \ar[r] & \cdots \ar[r] & \scrL(P_{0}) \ar[d] \ar[r] & \scrL(B) \ar[d]^-{\scrL(\eta_{B})} \ar[r] & 0 \ \ \\
\scrL(\zeta) \ar[d] & \equiv & 0 \ar[r] & \scrL\scrR\scrL(B') \ar[d]_-{\varepsilon_{\scrL(B')}} \ar[r] & \scrL\scrR (E_{n-1}) \ar[d]_-{\varepsilon_{E_{n-1}}} \ar[r] & \cdots \ar[r] & \scrL\scrR (E_{0}) \ar[d]^-{\varepsilon_{E_0}} \ar[r] & \scrL\scrR\scrL(B) \ar[d]^-{\varepsilon_{\scrL(B)}} \ar[r] & 0 \ \ \\
\xi & \equiv & 0 \ar[r] & \scrL(B') \ar[r] & E_{n-1} \ar[r] & \cdots \ar[r] & E_0 \ar[r] & \scrL(B) \ar[r] & 0 \ .
}
$$
When read from top to bottom, this is a morphism of $n$-extensions and thus, $\scrL(\zeta')$ is equivalent to $\xi$. It remains to show, that $\scrL(\zeta')$ is equivalent to the trivial $n$-extension in $\A$. Since $\zeta$ is equivalent to the trivial $n$-extension, there is a sequence $\zeta = \zeta_0, \zeta_1, \dots, \zeta_r = \sigma_n(\scrR \scrL(B),\scrR \scrL(B'))$ of objects in $\mathcal Ext^n_\B(\scrR \scrL(B),\scrR \scrL(B'))$ and a sequence of morphisms $\alpha_1, \dots, \alpha_{r}$, where
$$
\alpha_i: \zeta_{i-1} \longrightarrow \zeta_{i} \quad \text{or} \quad \alpha_i:  \zeta_{i} \longrightarrow  \zeta_{i-1} \quad (\text{for $1 \leq i \leq r$}).
$$
By the weak functoriality discussed in the proof of Proposition \ref{lem:exhproj}, morphisms $\alpha_i'$ between $\zeta_{i-1}'$ and $\zeta_{i}'$ will be given (for $i=1,\dots, r$). Since $\mathscr L$ preserves pushouts, $\scrL(\alpha_1'), \dots, \scrL(\alpha_{r}')$ is a sequence of morphisms between $n$-extensions that connects $\scrL(\zeta')$ and $\sigma_n(\scrL(B), \scrL(B'))$.
\end{proof}
\section{Hochschild cohomology for abelian categories}
\begin{defn}
Let $\A$ be an abelian $k$-category, $n \geq 0$ be an integer and $\scrX$ a $k$-linear endofunctor of $\mathsf A$. The \textit{$n$-th Hochschild cohomology group of $\mathsf A$ with coefficients in $\scrX$} is
$$
\HH^n(\mathsf A, \scrX) = \Ext^n_{\sfEnd_k(\A)}(\Id_{\mathsf A}, \scrX).
$$
The graded algebra $\HH^\bullet(\A) := \HH^\bullet(\A, \Id_\A)$ is the \textit{Hochschild cohomology ring of $\A$}.
\end{defn}
We prove that $\HH^n(\mathsf A, \Id_\A)$ coincides with $\Ext^n_{A^\ev}(A,A)$ for all $n \geq 0$ if $\A$ is the category of left modules over the $k$-projective algebra $A$.
\begin{lem}\label{lem:funcexactiso}
Let $A$ be a $k$-algebra, $\A = \Mod(A)$ and let $\scrX$ be an object in $\sfEnd_k(\Mod(A))$. If $A$ is projective over $k$, then the functor $\ev_A$ induces $k$-linear epimorphisms
$$
\beta_n = {\ev}_{A}^\sharp: \Ext^n_{\sfEnd_k(\A))}(\Id_\mathsf{A}, \scrX) \longrightarrow \Ext^n_{A^\ev}(A, \scrX(A))\quad \text{$($for $n \geq 0)$}
$$
being also monomorphisms in case $\scrX = \Id_{\A}$. In particular, we obtain an isomorphism
$$
\beta = {\ev}_{A}^\sharp: \Ext^\bullet_{\sfEnd_k(\A)}(\Id_{\A}, \Id_{\A}) \longrightarrow \Ext^\bullet_{A^\ev}(A, A)
$$
of graded $k$-algebras.
\end{lem}
\begin{proof}
In fact, this is a direct consequence of Proposition \ref{prop:isoadjointfull}, since $\ev_\A$ is exact, the unit corresponding to the adjoint pair $(\ev_A^\lambda, \ev_A)$ is an isomorphism, and $\ev_A^\lambda$ sends any projective resolution of $A$ to a projective resolution of $\Id_\A$ (compare with Lemma \ref{lem:endofuncprojres}).
\end{proof}
\begin{prop}\label{abel:monoidal}
Let $A$ be a $k$-algebra and $\A = \Mod(A)$. If $A$ is projective over $k$, then the $k$-linear functor $\emph{ev}_A^\lambda$ restricts to an exact and strong monoidal functor
$$
\ev_A^\lambda: (\P(A), \otimes_A, A) \longrightarrow (\underline{\sfEnd}_k(\A), \circ, \Id_{\mathsf A}).
$$
between very strong exact monoidal categories. The induced $k$-linear maps
$$
\beta_n^- = ({\ev}^{\lambda}_{A})_n^\sharp: \Ext^n_{\P(A)}(A,A) \longrightarrow  \Ext^n_{\underline{\sfEnd}_k(\A)}(\Id_{\mathsf A}, \Id_{\mathsf A})\quad \text{$($for $n \geq 0)$}
$$
are isomorphisms which make the diagrams
\begin{equation}\label{prop:monoidal:eq1}
\begin{aligned}
\xymatrix@C=30pt{
\Ext^m_{\P(A)}(A,A) \times \Ext^n_{\P(A)}(A,A) \ar[r]^-{[-,-]_{\P(A)}} \ar[d] & \Ext^{m + n}_{\P(A)}(A,A) \ar[d]\\
\Ext^m_{\underline{\sfEnd}_k(\A)}(\Id_{\mathsf A},\Id_{\mathsf A}) \times \Ext^n_{\underline{\sfEnd}_k(\A)}(\Id_{\mathsf A},\Id_{\mathsf A}) \ar[r]^-{[-,-]_\A} & \Ext^{m+n-1}_{\underline{\sfEnd}_k(\A)}(\Id_{\mathsf A},\Id_{\mathsf A})
}
\end{aligned}
\end{equation}
\begin{equation}\label{prop:monoidal:eq2}
\begin{aligned}
\xymatrix@C=30pt{
\Ext^{2n}_{\P(A)}(A,A) \ar[r]^-{sq_{\P(A)}} \ar[d] & \Ext^{4n-1}_{\P(A)}(A,A) \ar[d]\\
\Ext^{2n}_{\underline{\sfEnd}_k(\A)}(\Id_{\mathsf A},\Id_{\mathsf A}) \ar[r]^-{sq_\A} & \Ext^{4n-1}_{\underline{\sfEnd}_k(\A)}(\Id_{\mathsf A},\Id_{\mathsf A})
}
\end{aligned}
\end{equation}
commutative for every pair of integers $m, n \geq 1$. Here $[-,-]_\A$ and $sq_\A$ denote the maps $[-,-]_{\underline{\sfEnd}_k(\A)}$ and $sq_{{\underline{\sfEnd}_k(\A)}}$.
\end{prop}
\begin{proof}
Because every object in $\mathsf P(A)$ is flat over $A^\op$, $\ev_A^\lambda$ clearly restricts as claimed. It is apparent that $(A \otimes_A -) \cong \Id_{\Mod(A)}$. Further, 
$$
\big((M \otimes_A N) \otimes_A - \big) \cong M \otimes_A (N \otimes_A -) = (M \otimes_A -) \circ (N \otimes_A -)
$$
for all $A^\ev$-modules $M$ and $N$. Thus $\ev_A^\lambda$ is a bistrong monoidal functor. It is also exact, when restricted to $\P(A)$: Let $0 \rightarrow P' \rightarrow P \rightarrow P'' \rightarrow 0$ be an admissible short exact sequence in $\P(A)$. The induced sequence
$$
0 \longrightarrow (P' \otimes_A -) \longrightarrow (P \otimes_A -) \longrightarrow (P'' \otimes_A -) \longrightarrow 0
$$
of endofunctors on $\Mod(A)$ is exact in $\sfEnd_k(\Mod(A))$, since $P''$ is projective over $A^\op$ (cf. Lemma \ref{lem:flatexact}). Finally, let $i: \P(A) \rightarrow \Mod(A^\ev)$ and $j: \underline{\sfEnd}_k(\A) \rightarrow \sfEnd_k(\A)$ be the inclusion functors. The commutative diagram
$$
\xymatrix{
\Ext^n_{\P(A)}(A,A) \ar[r]^-{{i}^\sharp_{n}}_-\cong \ar[d]_-{\beta_n^-} & \Ext^n_{A^\ev}(A,A)\\
\Ext^n_{\underline{\sfEnd}_k(\A)}(\Id_\A,\Id_\A) \ar[r]^-{{j}^\sharp_{n}}_-\cong & \Ext^n_{\sfEnd_k(\A)}(\Id_\A, \Id_\A) \ar[u]_-{\beta_n}^-\cong
}
$$
tells us, that $\beta^- = ({\ev}^\lambda_{A})^\sharp$ has to be an isomorphism. The claimed commutativity of the diagrams (\ref{prop:monoidal:eq1}) and (\ref{prop:monoidal:eq2}) follows from Theorem \ref{thm:bracketcomm}.
\end{proof}
\begin{cor}\label{cor:abel:monoidal}
Let $A$ be a $k$-algebra. If $A$ is projective over $k$, then the Hochschild cohomology ring $\HH^\bullet(A)$ of $A$ agrees with the Hochschild cohomology ring $\HH^\bullet(\A)$ of $\A = \Mod(A)$ $($in every possible meaning, i.e., the graded ring structures as well as the brackets and the squaring maps coincide$)$.
\end{cor}
\begin{proof}
Combine Proposition \ref{abel:monoidal} with Corollary \ref{cor:isohochschildproj} and Lemma \ref{lem:funcexactiso}.
\end{proof}
\vfill
\section*{Acknowledgements}
This monograph is a modified version of my doctoral thesis, submitted at Bielefeld University. I am indebted to my supervisor Henning Krause for his continuing support and his encouragement. Many thanks are given to Ragnar-Olaf Buchweitz and Greg Stevenson for very valuable and stimulating conversations on the present work. I also like to thank Ragnar-Olaf Buchweitz for his warm hospitality during my stays at the University of Toronto. Finally, I wish to thank Rachel Taillefer for pointing out the article \cite{Me11} to me, and the anonymous referees for their comments.


\appendix

\renewcommand{\thesection}{\Alph{chapter}.\arabic{section}}
\chapter{Basics}
\section{Homological lemmas}
\begin{nn}
Let $k$ be a commutative ring. In this section, we will recall the basic notions of (pre)additive, abelian and $k$-linear categories, as well as structure preserving functors between them. Further, we are going to state some fundamental homological results for abelian categories. Any statement within this section is beyond classical and should be well-known to everyone who attended an introductory course on homological algebra.
\end{nn}
\begin{defn}\label{def:linearcat}
Let $\C$ be a category. The category $\C$ is called a \textit{$k$-linear} category (or simply a \textit{$k$-category}) if
\begin{itemize}
\item each set $\Hom_\C(X,Y)$ is a $k$-module ($X, Y \in \Ob \C$);
\item the composition maps
$$
\circ: \Hom_\C(Y, Z) \times \Hom_\C(X,Y) \rightarrow \Hom_\C(X,Z)
$$
are $k$-bilinear (for all $X, Y, Z \in \Ob \C$).
\end{itemize}
The category $\C$ is called \textit{preadditive} if it is $\mathbb Z$-linear. It is called \textit{additive} if it is preadditive and has finite products.
\end{defn}
\begin{defn}\label{def:addfunc}
A functor $\mathscr X: \C \rightarrow \D$ between preadditive categories $\C$ and $\D$ is called \textit{additive} if the induced maps
$$
\mathscr X_{X,Y}: \Hom_\C(X,Y)  \rightarrow \Hom_\D(\mathscr X X, \mathscr XY), \quad X,Y \in \Ob \C,
$$
are homomorphisms of abelian groups. If the categories $\C$ and $\D$ are $k$-linear, the functor $\mathscr X$ is called \textit{$k$-linear} if $\mathscr X_{X,Y}$ is a $k$-linear map for every pair $X, Y \in \Ob \C$.
\end{defn}
\begin{rem}\label{rem:charaddfunc}
Clearly every $k$-linear category is preadditive; the product category of two preadditive ($k$-linear/additive) categories is again preadditive ($k$-linear/additive). Notice that an additive category $\C$ automatically has finite coproducts and finite direct sums (i.e., biproducts), and thus a zero object. Moreover, a functor $\mathscr X: \C \rightarrow \D$ between additive categories is additive if, and only if, the canonical maps (that are induced by the inclusions of summands into their direct sum)
$$
a_{X,Y}: \mathscr X X \oplus \mathscr X Y \longrightarrow \mathscr X (X \oplus Y) \quad \text{(for $X, Y \in \Ob\C$)}
$$
are isomorphisms.
\end{rem}
\begin{defn}\label{def:abelian}
An additive category $\A$ is called \textit{abelian} if every morphism $f: X \rightarrow Y$ in $\A$ has a kernel $(\Ker(f), \ker(f))$ and a cokernel $(\Coker(f), \coker(f))$, and the bottom morphism in the induced diagram
$$
\xymatrix{
\Ker(f) \ar[r]^-{\ker(f)} & X \ar[d]_-{\coker(\ker(f))} \ar[r]^-f & Y \ar[r]^-{\coker(f)} & \Coker(f)\\
& \Coker(\ker(f)) \ar@{-->}[r]^-{\cong} & \Ker(\coker(f)) \ar[u]_-{\ker(\coker(f))}
}
$$
is an isomorphism. The object $\Ker(\coker(f)) \cong \Coker(\ker(f))$ is the \textit{image of $f$} and will be denoted by $\Im(f)$.

Let $X \rightarrow Y \rightarrow Z$ be a sequence in an abelian category $\A$. The sequence is called \textit{exact at $Y$} if $\Ker(Y \rightarrow Z) = \Im(X \rightarrow Y)$. A sequence of morphisms
$$
\xymatrix{
\cdots \ar[r] & X_{i-1} \ar[r] & X_i \ar[r] & X_{i+1} \ar[r] & \cdots
}
$$
in the abelian category $\mathsf A$ is called \textit{exact}, if it is exact at $X_i$ for all $i \in \Z$.
\end{defn}

\begin{lem}[5-Lemma, {\cite[Prop.\,2.72]{Ro09}}]\label{lem:fivelem}
Let $\A$ be an abelian category. Assume that a diagram
$$
\xymatrix{
A \ar[r] \ar[d]_-f & B \ar[r] \ar[d]_-g & C \ar[r] \ar[d]_-h & D \ar[r] \ar[d]^-i & E \ar[d]^-j\\
A' \ar[r] & B' \ar[r] & C' \ar[r] & D' \ar[r] & E'
}
$$
with exact rows is given in $\A$. 
\begin{enumerate}[\rm(1)]
\item If $g$ and $i$ are monomorphisms and $f$ is an epimorphism, then $h$ is a monomorphism.
\item If $g$ and $i$ are epimorphisms and $j$ is a monomorphism, then $h$ is an epimorphism.
\item If $f, g$ and $i,j$ are isomorphisms, then $h$ is an isomorphism.
\end{enumerate}
\end{lem}
\begin{lem}[Snake Lemma, {\cite[Ex.\,6.5]{Ro09}}]\label{lem:snake}
Let $\A$ be an abelian category. Assume that
$$
\xymatrix{
& A \ar[r]^-f \ar[d]_-a & B \ar[r]^-g \ar[d]_-b & C \ar[r] \ar[d]^-c & 0\\
0 \ar[r] & A' \ar[r]^-{f'} & B' \ar[r]^-{g'} & C' &
}
$$
is a commutative diagram in $\A$ with exact rows. Then, after fixing kernels and cokernels for $a, b$ and $c$, there is a unique commutative diagram
$$
\xymatrix{
& \Ker(a) \ar[r] \ar[d] & \Ker(b) \ar[r] \ar[d] & \Ker(c) \ar[d] &\\
& A \ar[r]^-f \ar[d]_-a & B \ar[r]^-g \ar[d]_-b & C \ar[r] \ar[d]^-c & 0\\
0 \ar[r] & A' \ar[r]^-{f'} \ar[d] & B' \ar[r]^-{g'} \ar[d] & C' \ar[d] &\\
& \Coker(a) \ar[r] & \Coker(b) \ar[r] & \Coker(c) &
}
$$
and a morphism $d: \Ker(c) \rightarrow \Coker(a)$ such that
$$
\xymatrix{
\Ker(a) \ar[r]^-p & \Ker(b) \ar[r] & \Ker(c) \ar[r]^-d & \Coker(a) \ar[r] & \Coker(b) \ar[r]^-q & \Coker(c)
}
$$
is exact. Further, $p$ is a monomorphism if $f$ is, and $q$ is an epimorphism if $g'$ is.
\end{lem}

\begin{lem}[Comparison Lemma, {\cite[Thm.\,6.16]{Ro09}}]\label{lem:comparison}
Let $\A$ be an abelian category and let $f: A \rightarrow B$ be a morphism in $\A$. Suppose that
$$
\xymatrix@R=12pt{
\mathbb P \quad : \quad & \cdots \ar[r]^-{d_3} & P_2 \ar[r]^-{d_2} & P_1 \ar[r]^-{d_1} & P_0 \ar[r]^-{d_0} & A \ar[r] & 0\\
\mathbb Q \quad : \quad & \cdots \ar[r]^-{} & Q_2 \ar[r]^-{} & Q_1 \ar[r]^-{} & Q_0 \ar[r]^-{} & B \ar[r] & 0
}
$$
are sequences in $\A$ such that $P_i$ is a projective object in $\A$ for all $i \geq 0$, $\xi$ is a complex \emph{(}i.e., $d_{i} \circ d_{i+1} = 0$ for all $i \geq 0$\emph{)} and such that $\xi'$ is exact. Then there are morphisms $f_i: P_i \rightarrow Q_i$, $i \geq 0$, such that
\begin{equation}\label{eq:comparison_lemma}
\begin{aligned}
\xymatrix{
\cdots \ar[r]^-{} & P_2 \ar[r]^-{} \ar[d]^-{f_2} & P_1 \ar[r]^-{} \ar[d]^-{f_1} & P_0 \ar[r]^-{} \ar[d]^-{f_0} & A \ar[r] \ar[d]^-f & 0\\
\cdots \ar[r]^-{} & Q_2 \ar[r]^-{} & Q_1 \ar[r]^-{} & Q_0 \ar[r]^-{} & B \ar[r] & 0
}
\end{aligned}
\end{equation}
commutes. The family $(f_i)_i$ is unique up to chain homotopy, that is, if $(f_i': P_i \rightarrow Q_i)_i$ is another family of morphisms such that the diagram \emph{(\ref{eq:comparison_lemma})} commutes, then $(f_i - f_i ')_i$ is null-homotopic. 
\end{lem}
\section{Algebras, coalgebras, bialgebras and Hopf algebras}\label{sec:algcoalhopf}
\begin{nn}
The main purpose of this and the subsequent section is to recollect the classical definitions of (quasi-triangular, triangular, cocommutative) bialgebras and Hopf algebras, along with elementary results related to these algebras and their categories of modules. We are going to illustrate the various definitions by promintent examples of Hopf algebras, to which we will apply our theory in order to obtain further insights into the Gerstenhaber algebra structure of their Hochschild cohomology rings.
\end{nn}
\begin{nn}
For the entire section, we let $k$ be a commutative ring. If $\Gamma$ is a ring, we let $Z(\Gamma)$ be the \textit{center of $\Gamma$}:
$$
Z(\Gamma) := \{\gamma \in \Gamma \mid \gamma \gamma' = \gamma' \gamma \text{ for all $\gamma' \in \Gamma$}\}.
$$
Remember that an (associative, unital) $k$-algebra may be defined as a (associative, unital) ring $A$ which additionally carries the structure of a $k$-module such that
$$
r(ab) = (ra)b = a(rb) \quad \text{(for all $r \in k$, $a,b \in A$)},
$$
that is, $k.1_A \subseteq Z(A)$. Alternatively (and more conceptually), one may regard a $k$-algebra as a $k$-module $A$ together with $k$-linear maps $\nabla: A \otimes_k A \rightarrow A$, $\eta: k \rightarrow A$ which are subject to certain axioms. We will recall this approach in full detail. For a $k$-module $V$, let $\tau: V \otimes_k V \rightarrow V \otimes_k V$ denote the map $\tau(v \otimes w) = w \otimes v$.
\end{nn}
\begin{defn}\label{def:algebra}
Let $A$ be a $k$-module. Let $\nabla : A \otimes_k A \rightarrow A$ and $\eta: k \rightarrow A$ be $k$-linear maps. The tuple $(A, \nabla, \eta)$ is called \textit{$($associative, unital$)$ $k$-algebra} if the following diagrams commute.
\begin{equation}\begin{aligned}
\xymatrix@C=40pt{
A \otimes_k (A \otimes_k A) \ar[d]_-{\cong} \ar[r]^-{A \otimes_k \nabla} & A \otimes_k A \ar[r]^-{\nabla} & A \ar@{=}[d]\\
(A \otimes_k A) \otimes_k A \ar[r]^-{\nabla \otimes_k A} & A \otimes_k A \ar[r]^-{\nabla} & A
}
\end{aligned}\tag{A1}\label{alg1}\end{equation}
\begin{equation}
\begin{aligned}
\xymatrix{
R \otimes_k A \ar[r]^-\cong \ar[d]_-{\eta \otimes_k A} & A \ar@{=}[d] & A \otimes_k k \ar[l]_-\cong \ar[d]^-{A \otimes_k \eta} \\
A \otimes_k A \ar[r]^-\nabla & A & A \otimes_k A \ar[l]_-\nabla
}
\end{aligned}\tag{A2}\label{alg2}
\end{equation}
(\ref{alg1}) is the \textit{associativity axiom} and (\ref{alg2}) is the \textit{unitary axiom}. Henceforth, $\nabla$ will be refered to as the \textit{multiplication map} and $\eta$ as the \textit{unit map}. A $k$-algebra $A$ is called \textit{commutative} if its underlying ring is commutative, i.e., if the diagram
\begin{equation}\begin{aligned}
\xymatrix@C=25pt{
A \otimes_k A \ar[d]_{\tau} \ar[r]^-{\nabla} & A \ar@{=}[d] \\
A \otimes_k A \ar[r]^-{\nabla} &  A
}
\end{aligned}\tag{A3}\label{alg3}\end{equation}
commutes.
\end{defn}
\begin{defn}\label{def:algebrahom}
Let $(A, \nabla_A, \eta_A)$ and $(B, \nabla_B, \eta_B)$ be $k$-algebras, and let $f: A \rightarrow B$ be a $k$-linear map. The map $f$ is a \textit{$($unital$)$ $k$-algebra homomorphism} if it commutes with the respective structure maps, that is,
$$
f \circ \nabla_A = \nabla_B \circ (f \otimes_k f) \quad \text{and} \quad f \circ \eta_A = \eta_B.
$$
In other words, $f$ is a homomorphism of the underlying rings.
\end{defn}
\begin{defn}\label{def:module}
Let $A$ be a $k$-algebra. Let $M$ be a $k$-module and $\nabla_M \colon A \times M \rightarrow M$ be a $k$-bilinear map. The pair $(M, \nabla_M)$ is called a \textit{$($left$)$ $A$-module} if the following induced diagrams of $k$-linear maps commute.
\begin{equation}\begin{aligned}
\xymatrix@C=40pt{
A \otimes_k (A \otimes_k M) \ar[d]_{\cong} \ar[r]^-{A \otimes_k \nabla_M} & A \otimes_k M \ar[r]^-{\nabla_M} & M \ar@{=}[d]\\
(A \otimes_k A) \otimes_k M \ar[r]_-{\nabla \otimes_k M} & A \otimes_k M \ar[r]_-{\nabla_M} & M
}
\end{aligned}\tag{M1}\label{mod1}\end{equation}
\begin{equation}\begin{aligned}
\xymatrix@C=14pt{
k \otimes_k M \ar[dr]_\cong \ar[rr]^{\eta \otimes_k M} && A \otimes_k M \ar[dl]^{\nabla_M} \\
& M &
}
\end{aligned}\tag{M2}\label{mod2}\end{equation}
\end{defn}
Coalgebras are precisely the dual analogue of algebras. They arise by inverting the arrows within the diagrammatically stated axioms of an algebra.
\begin{defn}\label{def:coalgebra}
Let $C$ be a $k$-module and $\Delta: C \rightarrow C \otimes_k C$, $\varepsilon: C \rightarrow k$ be  $k$-linear maps. The tuple $(C, \Delta, \varepsilon)$ is called \textit{$($coassociative, counital$)$ $k$-coalgebra} provided the diagrams below commute.
\begin{equation}
\begin{aligned}
\xymatrix@C=40pt{
C \ar@{=}[d] \ar[r]^-{\Delta} & C \otimes_k C \ar[r]^-{\Delta \otimes_k C} & (C \otimes_k C) \otimes_k C \ar[d]^-\cong\\
C \ar[r]_-{\Delta} & C \otimes_k C \ar[r]_-{C \otimes_k \Delta} & C \otimes_k (C \otimes_k C)
}
\end{aligned}\tag{C1}\label{coalg1}
\end{equation}
\begin{equation}
\begin{aligned}
\xymatrix{
k \otimes_k C & C \ar[l]_-\cong \ar[r]^-\cong & C \otimes_k k \\
C \otimes_k C \ar[u]^-{\varepsilon \otimes_k C} & C \ar[l]_-\Delta \ar[r]^-\Delta \ar@{=}[u] & C \otimes_k C \ar[u]_-{C \otimes_k \varepsilon} 
}
\end{aligned}\tag{C2}\label{coalg2}
\end{equation}
The axiom (\ref{coalg1}) is called \textit{coassociativity axiom}, while (\ref{coalg2}) is the \textit{counitary axiom}. Accordingly, we shall call $\Delta$ the \textit{comultiplication map} and $\varepsilon$ the \textit{counit map}. A $k$-coalgebra $C$ is said to be \textit{cocommutative} if the diagram
\begin{equation}
\begin{aligned}
\xymatrix@C=25pt{
C \ar[r]^-{\Delta} \ar@{=}[d] & C \otimes_k C \ar[d]^{\tau}\\
C \ar[r]^-{\Delta} & C \otimes_k C
}
\end{aligned}\tag{C3}\label{coalg3}
\end{equation}
commutes.
\end{defn}
\begin{defn}\label{def:algebrahom}
Let $(C, \Delta_C, \varepsilon_C)$ and $(D, \Delta_D, \varepsilon_D)$ be $k$-coalgebras, and let $f: C \rightarrow D$ be a $k$-linear map. The map $f$ is a \textit{$($unital$)$ $k$-coalgebra homomorphism} if it commutes with the respective structure maps, that is,
$$
\Delta_D \circ f = (f \otimes_k f) \circ \Delta_C \quad \text{and} \quad \varepsilon_D \circ f = \varepsilon_C.
$$
\end{defn}
\begin{defn}\label{def:comodule}
Let $(C, \Delta, \varepsilon)$ be a $k$-coalgebra and $M$ be a $k$-module. Let $\Delta_M: M \rightarrow C \otimes_k M$ be a $k$-linear map. The pair $(M, \Delta_M)$ is a (\textit{left}) \textit{$C$-comodule} if the following diagrams commute.
\begin{equation}\begin{aligned}
\xymatrix@C=40pt{
M \ar[r]^-{\Delta_M} \ar@{=}[d] & C \otimes_k M \ar[r]^-{\Delta \otimes_k C} & (C \otimes_k C) \otimes_k M \ar[d]^\cong \\
M \ar[r]_-{\Delta_M}		      & C \otimes_k M \ar[r]_-{M \otimes_k \Delta_M} & C \otimes_k (C \otimes_k M)}
\end{aligned}\tag{CM1}\label{comod1}\end{equation}
\begin{equation}\begin{aligned}
\xymatrix@C=14pt{
& M \ar[rd]^{\cong} \ar[dl]_{\Delta_M} & \\
C \otimes_k M \ar[rr]_{\varepsilon \otimes_k M} & & k \otimes_k M
}
\end{aligned}\tag{CM2}\label{comod2}\end{equation}
\end{defn}
\begin{nota}\label{not:sweedler}
Let $C$ be a $k$-coalgebra and let $M$ be a $C$-comodule. For $m \in M$ we may write
\begin{equation}
\Delta_M(m) = \sum_{i=1}^n {m_{(1,i)} \otimes m_{(2,i)}}
\end{equation}
for an integer $n \geq 0$ and elemtens $m_{(1,i)} \in C$, $m_{(2,i)} \in M$ for $i = 1, \dots, n$. We will use \textit{Sweedler's notation} (cf. \cite{Sw69}) to abbreviate this expression to
\begin{equation}
\Delta_M(m) = \sum_{(m)} {m_{(1)} \otimes m_{(2)}}
\end{equation}
Thus, for an element $c \in C$, the coalgebra axioms may be stated as follows.
\begin{itemize}
\item[(\ref{coalg1})] $\begin{aligned}[t] \sum_{(c)}{c_{(1,1)} \otimes c_{(1,2)} \otimes c_{(2)}} = \sum_{(c)}{c_{(1)}
\otimes c_{(2,1)} \otimes c_{(2,2)}}\end{aligned}$;
\item[(\ref{coalg2})] $\begin{aligned}[t]\sum_{(c)}{\varepsilon(c_{(1)}) \otimes c_{(2)}} = 1_k \otimes c, \quad c
\otimes 1_k = \sum_{(c)}{c_{(1)} \otimes \varepsilon(c_{(2)})}\end{aligned}$;
\item[(\ref{coalg3})] $\begin{aligned}[t]\sum_{(c)}{c_{(1)} \otimes c_{(2)}} = \sum_{(c)}{c_{(2)} \otimes
c_{(1)}}\end{aligned}$.
\end{itemize}
In order to provided enhanced readability, we will often even drop the summation symbol: $\Delta_M(m) = m_{(1)} \otimes m_{(2)}$.
\end{nota}
\begin{nn}
Modules that simultaneously carry an algebra and a coalgebra structure such that relative structure maps are compatible with the complementary structure appear quite frequently in nature and will be the topic of the upcoming considerations.
\end{nn}
\begin{defn}\label{def:bialgebra}
Let $\mathcal B = (B, \nabla, \eta, \Delta, \varepsilon)$ be such that $(B, \nabla, \eta)$ is an algebra and $(B, \Delta,\varepsilon)$ is a coalgebra over $k$. The 5-tuple $\mathcal B = (B, \nabla, \eta, \Delta, \varepsilon)$ is a \textit{$k$-bialgebra}, if $\Delta$ and $\varepsilon$ are $k$-algebra homomorphisms. We denote the underlying $k$-algebra of $\mathcal B$ by $\mathcal B^\natural$, or (when no confusion is possible) simply by $B$.
\end{defn}
\begin{defn}\label{def:bialgebrahom}
Let $\mathcal B$ and $\mathcal C$ be $k$-bialgebras. A $k$-linear map $\mathcal B^\natural \rightarrow \mathcal C^\natural$ is a \textit{$k$-bialgebra homomorphism} if it respects both structures that $\mathcal B$ and $\mathcal C$ possess, i.e., if it simultaneously is a homomorphism of $k$-algebras and of $k$-coalgebras.
\end{defn}
\begin{rem} Let $\mathcal B = (B, \nabla, \eta, \Delta, \varepsilon)$ be a $k$-bialgebra.
\begin{enumerate}[\rm(1)]
\item Since the counit $\varepsilon: B \rightarrow k$ is a $k$-algebra homomorphism, $k$ may be viewed as a (left) $B$-module:
$$
br:= \varepsilon(b)r \quad (\text{for $b \in B$, $r \in k$}).
$$
It is then called the \textit{trivial $B$-module}.
\item As stated in \cite[Thm.\,2.1.1]{Ab80}, $\Delta$ and $\varepsilon$ being $k$-algebra homomorphisms is equivalent to $\nabla$ and $\eta$ being $k$-coalgebra homomorphisms.
\end{enumerate}
\end{rem}
\begin{defn}\label{defn:properties}
Let $\mathcal B = (B, \nabla, \eta, \Delta, \varepsilon)$ be a $k$-bialgebra. Let P be any attribugte an algebra or a coalgebra can possess (e.g., simple, semi-simple, commutative, cocommutative etc.). If not stated differently, we say that $\mathcal B$ is P if its underlying algebra (respectively coalgebra) is P.
\end{defn}
\begin{nn}
Fix a $k$-bialgebra $(B, \nabla, \eta, \Delta, \varepsilon)$ and two $B$-modules $X$ and $Y$. The $k$-module $X \otimes_k Y$ becomes a $B$-$B$-bimodule in two (possibly different) ways: 
\begin{equation}\label{eq:tensmod1}
\begin{aligned}
b(x \otimes y) &:= bx \otimes y, \\
(x \otimes y)b &:= x \otimes yb
\end{aligned}
\end{equation}
as well as
\begin{equation}\label{eq:tensmod2}
\begin{aligned}
b (x \otimes y) &:= \Delta(b)\cdot(x \otimes y) = \sum_{(b)}{{b}_{(1)}x \otimes b_{(2)}y}, \\
(x \otimes y)b &:= x \otimes yb,
\end{aligned}
\end{equation}
where, in both cases, $b \in B, \ x \in X$ and $y \in Y$. One effortlessly checks that these indeed give rise to sensible $B$-actions. We are going to use the notation
\begin{itemize}
\item $X \otimes_k Y$ for the module obtained by (\ref{eq:tensmod1}), and
\item $X \boxtimes_R Y$ for the module obtained by (\ref{eq:tensmod2}).
\end{itemize}
The hereby acquired $k$-linear bifunctor
$$
-\boxtimes_k - : \Mod(B) \times \Mod(B) \rightarrow \Mod(B)
$$
gives rise to a tensor structure on $\Mod(B)$. The trivial module $k$ acts as the tensor unit and the forgetful functor $\Mod(B) \rightarrow \Mod(k)$ is \textit{strict} bistrong monoidal\footnote{A bistrong monoidal functor is called \textit{strict} if its structure morphisms (i.e., $\phi$, $\phi_0$, $\psi$ and $\psi_0$) are the identity morphisms.}. Bialgebras are, in some sense, determined by this property.
\end{nn}
\begin{prop}\label{prop:bialg_monoidal}
Let $(B, \nabla, \eta)$ be a $k$-algebra. The following data are equivalent:
\begin{enumerate}[\rm(1)]
\item A $k$-bialgebra structure on $(B,\nabla,\eta)$.
\item A monodical structure on the category $\Mod(B)$ such that the forgetful functor $\Mod(B) \rightarrow \Mod(k)$ is strict monoidal.
\end{enumerate}
\end{prop}
\begin{proof}
A proof may be deduced from \cite[Thm.\,15]{Pa80}.
\end{proof}
\begin{rem}\label{rem:cocomsym}
Let $(B,\nabla, \eta, \Delta, \varepsilon)$ be a cocommutative bialgebra over $k$. Let $X$ and $Y$ be $B$-modules. There is a natural isomorphism of $B$-modules,
$$
\gamma_{X,Y}: X \boxtimes_k Y \rightarrow Y \boxtimes_k X,
$$
induced by the $k$-linear map $X \otimes_k Y \rightarrow Y \otimes_k X, \ x \otimes y \mapsto y \otimes x$. It is such that $\gamma_{X,Y}^{-1} = \gamma_{Y,X}$, and hence the resulting natural transformation $\gamma$ turns $(\Mod(B), \boxtimes_k, k)$ into a symmetric monoidal category. 
\end{rem}
\begin{nn}
Bialgebras $(B, \nabla, \eta, \Delta, \varepsilon)$ over $k$ that yield a braiding on the monoidal category $(\Mod(B), \boxtimes_k, k)$ can be described by an invertible element $\mathbf{r} \in B \otimes_k B$ satisfying certain axioms. Such bialgebras are called \textit{quasi-triangular}. For integers $i \geq 2$ and $1 \leq s < t \leq i$ let $\phi^i_{st}$ be the map
$$
\xymatrix@C=18pt{
B \otimes_k B \ar[r]^-\cong & k \otimes_k \cdots \otimes_k B \otimes_k \cdots \otimes_k B \otimes_k \cdots \otimes_k k \ar[r]^-{\mathrm{can}} & B^{\otimes_k i},
}
$$
where the two copies of $B$ in the middle term occur as the $s$-th and the $t$-th factor. For an element $\mathbf{r} = \mathbf r_1 \otimes \mathbf r_2 \in B \otimes_k B$ (implicit summation is understood) consider the following equations.
\begin{equation}\tag{QT1}\label{defn:qtrhopf:qt1}
\begin{aligned}
\mathbf{r} \Delta(b) = (\tau \circ \Delta)(b) \mathbf{r} \quad \text{for all $b \in B$;}
\end{aligned}
\end{equation}
\begin{equation}\tag{QT2}\label{defn:qtrhopf:qt2}
\begin{aligned}
(\Delta \otimes_k B)(\mathbf{r}) = \mathbf{r}_{13}\mathbf{r}_{23};
\end{aligned}
\end{equation}
\begin{equation}\tag{QT3}\label{defn:qtrhopf:qt3}
\begin{aligned}
(B \otimes_k \Delta)(\mathbf{r}) = \mathbf{r}_{13}\mathbf{r}_{12},
\end{aligned}
\end{equation}
where $\mathbf r_{st} = \phi_{st}^3(\mathbf{r})$ for $1 \leq s < t \leq 3$. The elements $\mathbf{r}_{13}, \mathbf{r}_{23}$ and $\mathbf{r}_{12}$ can be described as follows. If $\mathbf r = \sum_i {a_i \otimes b_i}$ then
$$
\mathbf{r}_{13} = \sum_i{a_i \otimes 1_B \otimes b_i}, \quad \mathbf{r}_{23} = \sum_i{1_B \otimes a_i \otimes b_i}, \quad \mathbf{r}_{12} = \sum_i{a_i \otimes b_i \otimes 1_B}.
$$
The following definition goes back to  V.\,G.\,Drinfel'd in \cite{Dr86}.
\end{nn}
\begin{defn}\label{defn:qtrhopf}
Let $\mathcal B = (B, \nabla, \eta, \Delta, \varepsilon)$ be a bialgebra over $k$.

\begin{enumerate}[\rm(1)]
\item The bialgebra $\mathcal B$ is called \textit{pre-triangular}, if there exists an element $\mathbf{r} \in B \otimes_k B$ which satisfies (\ref{defn:qtrhopf:qt1})--(\ref{defn:qtrhopf:qt3}) and $\varepsilon(\mathbf r_1 \varepsilon(\mathbf r_2)) = 1_R$. In this case, $\mathbf r$ is called a \textit{semi-canonical R-matrix for $\mathcal B$}.
\item The bialgebra $\mathcal B$ is called \textit{quasi-triangular} if there exists an invertible element $\mathbf r \in B \otimes_k B$ satisfying (\ref{defn:qtrhopf:qt1})--(\ref{defn:qtrhopf:qt3}). In this case, $\mathbf r$ is called a \textit{canonical R-matrix for $B$}. 
\item The bialgebra $\mathcal B$ is \textit{triangular} if it is quasi-triangular with canonical R-matrix $\mathbf r \in B \otimes_k B$ such that (in $B \otimes_k B \otimes_k B$)
$$
\mathbf r_{12}^{-1} = (\tau \otimes_k B)(\mathbf r_{12}),
$$
where, as always, $\tau: B \otimes_k B \rightarrow B \otimes_k B$ is the map given by commuting the factors.
\end{enumerate}
\end{defn}
\begin{lem}\label{lem:triangular_braided}
Let $(B, \nabla, \eta, \Delta, \varepsilon)$ be a $k$-bialgebra and let $\mathbf r \in B \otimes_k B$ be an element such that \emph{(\ref{defn:qtrhopf:qt1})--(\ref{defn:qtrhopf:qt3})} and $\varepsilon(\mathbf r_1 \mathbf r_2) = 1_k$ hold. Then the maps
$$
\gamma_{X,Y}: X \boxtimes_k Y \longrightarrow Y \boxtimes_k X, \ \gamma_{X,Y}(x \otimes y) = \tau(\mathbf{r}(x \otimes y)) \quad \text{\emph{(}for $B$-modules $X,Y$\emph{)},}
$$
are homomorphisms of $B$-modules and give rise to a lax braiding on $(\Mod(B), \boxtimes_k, k)$. Further, the tuple $(\Mod(B), \boxtimes_k, k, \gamma)$ is a braided monoidal category if, in addition, $\mathbf r$ is invertible.
\end{lem}
\begin{proof}
In order to see that $\gamma_{X,Y}$ is $B$-linear, let $b\in B$, $x \in X$ and $y \in Y$. We compute
\begin{align*}
\gamma_{X,Y}(b(x \otimes y)) &= \gamma_{X,Y}(\Delta(b)(x \otimes y)) &\\
&= \tau(\mathbf r\Delta(b)(x \otimes y)) &\\
&= \tau((\tau \circ \Delta)(b) \mathbf r (x \otimes y)) & (\text{by (\ref{defn:qtrhopf:qt1})})\\
&= \tau^2(\Delta(b))\tau(\mathbf r(x \otimes y)) &\\
&= \Delta(b)\gamma_{X,Y}(x \otimes y) &\\
&= b\gamma_{X,Y}(x \otimes y).
\end{align*}
If $\mathbf r$ is invertible, $\gamma_{X,Y}$ is bijective for every pair of $A$-modules $X,Y$, since the assignment
$$
\gamma^{-1}_{X,Y}(y \otimes x) = \mathbf{r}^{-1}\tau(y \otimes x) \quad \text{(for $x \in X$, $y \in Y$),}
$$
gives rise to the inverse map of $\gamma_{X,Y}$. The triangle equations follow from {(\ref{defn:qtrhopf:qt2})} and {(\ref{defn:qtrhopf:qt3})}.
\end{proof}
\begin{nn}
Note that by \cite[Prop.\,10.1.8]{Mo93} every quasi-triangular bialgebra is pre-triangular. Further, observe that if $\mathcal B = (B, \nabla, \eta, \Delta, \varepsilon)$ is a cocommutative $k$-bialgebra, the element $\mathbf{r} = 1_B \otimes 1_B \in B \otimes_k B$ will satisfy the equations (\ref{defn:qtrhopf:qt1}), (\ref{defn:qtrhopf:qt2}) as well as (\ref{defn:qtrhopf:qt3}), and hence turns $\mathcal B$ into a quasi-triangular bialgebra. In particular the lemma above is applicable (compare this with Remark \ref{rem:cocomsym}).

The existence of a canonical R-matrix for $B$ precisely means that $(\Mod(B), \boxtimes_k, k)$ is a (lax) braided monoidal category.
\end{nn}
\begin{thm}[{\cite[Thm.\,10.4.2]{Mo93}}]\label{prop:quasitri_bialg}
Let $\mathcal B$ be a $k$-bialgebra. Then there is a bijective correspondence between $($lax$)$ braidings $\gamma$ on $(\Mod(B), \boxtimes_k, k)$ and \text{$($semi-$)$}canonical R-matrices $\mathbf r \in B \otimes_k B$.
\end{thm}
\begin{defn}\label{def:hopf}
Let $\mathcal B = (H, \nabla, \eta, \Delta, \varepsilon)$ be a $k$-bialgebra and $S: H \rightarrow H$ be a $k$-linear map. The pair $\mathcal H = (\mathcal B, S)$ is called \textit{$k$-Hopf algebra}, if the diagram
\begin{equation}\label{def:hopf:1}
\begin{aligned}
\xymatrix{
& H \otimes_k H \ar[rr]^{S \otimes_k \id_H} & & H \otimes_k H \ar[rd]^{\nabla} & \\
H \ar[rr]^{\varepsilon} \ar[ru]^{\Delta} \ar[rd]_{\Delta} & & R \ar[rr]^{\eta} & & H \\
& H \otimes_k H \ar[rr]_{\id_H \otimes_k S} & & H \otimes_k H \ar[ru]_{\nabla} &
}
\end{aligned}
\end{equation}
commutes. In this case, the map $S$ is called an \textit{antipode for} $\mathcal H$. A Hopf algebra is \textit{commutative} (\textit{cocommutative}, \textit{quasi-triangular}, \textit{triangular} etc.) if its underlying bialgebra is commutative (cocommutative, quasi-triangular, triangular etc.).
\end{defn}
\begin{defn}
Let $\mathcal G$ and $\mathcal H$ be $k$-Hopf algebras with antipodes $S_{\mathcal G}$ and $S_{\mathcal H}$, and let $f: \mathcal G^\natural \rightarrow \mathcal H^\natural$ be a $k$-linear map. The map $f$ is a \textit{$k$-Hopf algebra homomorphism} if it is a homomorphism of the $k$-bialgebras lying under $\mathcal G$ and $\mathcal H$, and
$$
S_{\mathcal H} \circ f = f \circ S_{\mathcal G}.
$$
\end{defn}
\begin{nn}If $\mathcal H$ is a $k$-Hopf algebra with underlying algebra $A$, its antipode $S$ is a $k$-algebra homomorphism $A^\op \rightarrow A$ satisfying the equation 
$$
\tau \circ (S \otimes_k S) \circ \Delta = \Delta \circ S
$$
(cf. \cite[Thm.\,2.1.4]{Ab80}). Again, using Sweedler's notation for the comultiplication, the commutativity of the diagram (\ref{def:hopf:1}) translates to
$$
\sum_{(a)}{S(a_{(1)})a_{(2)}} = \eta(\varepsilon(a)) = \sum_{(a)}{a_{(1)}S(a_{(2)})} \quad (\text{for $a \in A$}).
$$
If $\mathcal H$ is commutative or cocommutative, then $S \circ S = \id_A$ (cf. \cite[Theorem 2.1.4]{Ab80}).
\end{nn}
\begin{exa}\label{exas:qtcocom}
Let $N \geq 2$ be an integer and let $1 \neq \zeta \in k$ be a $N$-th root of unity. Consider the $k$-algebra
$$
H_{2N} := \frac{k\langle g,x \rangle}{(g^N - 1, x^N, xg - \zeta gx)}.
$$
It is a Hopf algebra, denoted by $\mathcal H_{2N}$, via $\Delta(g) = g \otimes g$, $\Delta(x) = x \otimes 1 + g \otimes x$ and $\varepsilon(g) = 1$, $\varepsilon(x) = 0$ and $S(g) = g^{-1}$, $S(x) = - g^{-1}x$. Note that $H_{2N}$ is free of rank $2N$ over $k$, $S$ has order $2N$ and that $1,g,g^2,\dots,g^{N-1}$ are the only elements $r$ in $H_{2N}$ with $\Delta(r) = r \otimes r$ and $\varepsilon(r) = 1$. Moreover, $\mathcal H_{2N}$ is neither commutative, nor cocommutative. However, if $2$ is invertible in $k$ (and $N=2$), there is a family of canonical R-matrices $\mathbf r_\alpha$, $\alpha \in k$, for $\mathcal H_4$ (cf. \cite[Ex.\,10.1.17]{Mo93}):
$$
\mathbf r_\alpha := \frac{1}{2}(1 \otimes 1 + 1 \otimes g + g \otimes 1 - g \otimes g) + \frac{\alpha}{2}(x \otimes x - x \otimes gx +gx \otimes x + gx \otimes gx).
$$
Note that there cannot be any additional canonical R-matrices for $H_4$ (see \cite[Sec.\,4.2.F]{ChPr94}). Moreover, $\H_{2N}$ does not admit a quasi-triangular structure for $N > 2$. The Hopf algebras $\mathcal H_{2N}$ are the so called \textit{Taft algebras} (cf. \cite{Ta71}).
\end{exa}
\begin{nn}
There exists a more general (and less familiar) notion of weak bialgebras and weak Hopf algebras (see \cite{BNS99} and \cite{Ni98}). Bialgebras (Hopf algebras) are examples of weak bialgebras (weak Hopf algebras). We are not going to dive into this weakened theory any deeper, but will give the definition of a weak bialgebra for the record.
\end{nn}
\begin{defn}\label{def:weakbialgebra}
Let $(B, \nabla, \eta, \Delta, \varepsilon)$ be such that $(B, \nabla, \eta)$ is a $k$-algebra and $(B, \Delta, \varepsilon)$ is a $k$-coalgebra. The $5$-tuple $(B, \nabla, \eta, \Delta, \varepsilon)$ is a \textit{weak $k$-bialgebra} if the following equations hold true.
\begin{equation}\tag{WB1}
\Delta(bb') = \Delta(b)\Delta(b') \quad (\text{for all $b,b' \in B$});
\end{equation}
\begin{equation}\tag{WB2}
\begin{aligned}
(\Delta(1_B) \otimes 1_B)(1_B \otimes \Delta(1_B)) &= (\Delta \otimes_k B \circ \Delta)(1_B)\\ &= (1_B \otimes \Delta(1_B))(\Delta(1_B) \otimes 1_B);
\end{aligned}
\end{equation}
\begin{equation}\tag{WB3}
\begin{aligned}
\sum_{(c)}\varepsilon(bc_{(1)})\varepsilon(c_{(2)}d) &= \varepsilon(bcd)\\ &= \sum_{(c)} \varepsilon(bc_{(2)})\varepsilon(c_{(1)}d) && \text{(for all $b,c,d \in B$)}.
\end{aligned}
\end{equation}
\end{defn} 
\begin{rem}
What we have actually defined above is, in terms of \cite{Ni98}, a monoidal and comonoidal weak bialgebra, which is simply called a weak bialgebra in \cite{BNS99}. A weak bialgebra over $k$, as defined above, is called \textit{weak Hopf algebra} if there is a $k$-linear map $S: B \rightarrow B$ satisfying a set of axioms (which we will not state here; consult the references for details). Note that weak bialgebras also give rise to monoidal categories.
\end{rem}
\section{Examples: Hopf algebras}\label{sec:exasHopf}
We consider further important examples of quasi-triangular Hopf algebras, most of them being even cocommutative. As exposed in the previous section, they give rise to braided monoidal categories. Let $k$ be a commutative ring.
\begin{exa}\label{exa:groupalgebra}
Let $G$ be a group (not necessarily finite) with identity element $e_G$. The group algebra $kG$ of $G$ is the free $k$-module with basis $G$, and multiplication and unit are given by, for $g, h \in G$:
$$
\nabla_G: kG \otimes_k kG \longrightarrow G, \ \nabla_G(g \otimes h) = gh, \quad \eta_G: k \longrightarrow kG, \ \eta_G(1_k) = e_G.
$$
Clearly, $kG$ is a unital $k$-algebra with $e_G$ being the unit element. The algebra $kG$ is commutative if, and only if, $G$ is abelian. $kG$ also carries the structure of a $k$-coalgebra with comultiplication
$$
\Delta_G: G \longrightarrow kG \otimes_k kG, \ \Delta_G(g) = g \otimes g,
$$
and counit $\varepsilon_G: kG \rightarrow k, \ \varepsilon_G(g) = 1_k$. The $k$-linear map $S_G: kG \rightarrow kG, \ S_G(g) = g^{-1}$ is such that $(kG, \nabla_G, \eta_G, \Delta_G, \varepsilon_G, S_G)$ is a Hopf algebra over $k$. It is apparent from the definition that it is a cocommutative Hopf algebra (however, there are, in general, a lot more canonical R-matrices for $kG$ than $e_G \otimes e_G$; see for instance \cite{Wa10}). Note that $kG \otimes_k kG \cong k(G \times G)$ as $k$-algebras, and thus, $kG \otimes_k kG$ naturally is a Hopf algebra again. In case $G$ is $\mathbb Z_n := \mathbb Z / n\mathbb Z$ for some $n \geq 0$, the group algebra $kG$ is isomorphic to $k[x,x^{-1}]$ if $n = 0$, and $k[x]/(x^n - 1)$ otherwise. Comultiplication, counit and antipode are given by
$$
\Delta_{\mathbb Z_n}(x) = x \otimes x, \quad \varepsilon_{\mathbb Z_n}(x) = 1_k, \quad S_{\mathbb Z_n}(x) = -x.
$$
However, if $G$ is any finite and non-abelian group, then $kG^\ast = \Hom_k(kG,k)$ is a Hopf algebra (the structure maps for $kG^\ast$ simply arise from the structure maps for $kG$ by applying the contravariant functor $(-)^\ast$; see Example \ref{exa:groupalgebra} for details) with comultiplication $\nabla^\ast$, which cannot be quasi-triangular. To see this, let $\{\langle g, -\rangle \mid g \in G \}$ the dual basis of $G$, and let $g, h \in G$ with $gh \neq hg$. For the $kG^\ast$-modules $M_g = k\langle g, -\rangle$ and $M_h = k\langle h, - \rangle$ we get:
$$
\langle gh, - \rangle (M_g \boxtimes_k M_h) \neq 0 = \langle gh, - \rangle(M_h \boxtimes_k M_g),
$$
where the module structure on $M_g$ is given by $\langle x, - \rangle \langle g, -\rangle = \delta_{x,g}\langle g, -\rangle$. Note that in any case, the antipode of $kG^\ast$ has order 2.
\end{exa}
\begin{exa}
Let $\mathcal H$ be a Hopf algebra over $k$. An element $x \in H$ is called \textit{group like} if $\Delta(x) = x \otimes x$ and $\varepsilon(x) = 1_k$. If $x, y \in H, \ x \neq y$, are group like elements, then $x$ and $y$ are $k$-independent (cf. \cite[Remark after Def.\,1.3.4]{Mo93}). Multiplication of elements in $H$ turns the set
$$
\Gamma(\mathcal H) := \{ x \in H \mid \text{$x$ is group like} \} \subseteq H \setminus \{0\}
$$
into a group with identity element $1_H$. In fact, if $x,y$ are group like, then so is $xy$ since $\Delta$ and $\varepsilon$ are algebra homomorphisms; moreover the formula $\tau \circ (S \otimes_k S) \circ \Delta = \Delta \circ S$ (cf. \cite[Thm.\,2.1.4]{Ab80}), tells us that $S(x)$ is group like. Finally, the axiom for the antipode $S$ leads to $xS(x) = 1_H = S(x)x$. The corresponding group algebra $k\Gamma(\mathcal H)$ is a sub-Hopf algebra of $\mathcal H$ (i.e., the inclusion $k\Gamma(\mathcal H) \subseteq H$ is an injective homomorphism of Hopf algebras). Hence every Hopf algebra admits a (possibly non-trivial) cocommutative sub-Hopf algebra. The construction of $\Gamma(-)$ is involutive in the sense that if $G$ is a group, then $k\Gamma(kG) = kG$.

If $\mathrm{char}(k) \neq 2$ and if we consider the $k$-Hopf algebra $\mathcal H_{2N}$ discussed as part of the Examples \ref{exas:qtcocom}, the group-like elements are precisely $1, g, g^2, \dots, g^{N-1}$; thus, $k\Gamma(\mathcal H_{2N}) \cong k \mathbb Z_N \cong k[t]/(t^N - 1)$.
\end{exa}
\begin{prop}\label{prop:hopfalgvanish}
Let $\mathcal H$ be a Hopf algebra over $k$ whose underlying $k$-algebra $A$ is projective as a $k$-module. Let $f: \mathcal H \rightarrow \mathcal Q$ be a Hopf algebra homomorphism into a quasi-triangular Hopf algebra $\mathcal Q$. Let $Q$ be the underlying $k$-algebra of $\mathcal Q$. The graded algebra $\OH^\bullet(\mathcal H,k) = \Ext^\bullet_{A}(k,k)$ is a subalgebra of $\HH^\bullet(A)$, and the Gerstenhaber bracket $\{-,-\}_A$ on $\HH^\bullet(A)$ fulfils
$$
\{\Res^\sharp_{f}(\alpha), \Res^\sharp_{f}(\beta)\}_A = 0 \quad \text{$($for all $\alpha, \beta \in \OH^\bullet(\mathcal Q, k))$.}
$$
Here $\Res_{f}: \Mod(Q) \rightarrow \Mod(H)$ denotes the restriction functor along $f$, and $\Res_{f}^\sharp$ is the induced map between the corresponding cohomology rings.
\end{prop}
\begin{proof}
The restriction functor $\mathrm{Res}_f: \Mod(Q) \rightarrow \Mod(A)$ is (strict) monoidal, exact and sends $k$-projective $Q$-modules to $k$-projective $A$-modules. Further, since $\mathcal Q$ is quasi-triangular, Corollary \ref{cor:gerstenhabervanishhopfalgebra} applies, and we obtain the commutative diagram
$$
\xymatrix@C=40pt{
\OH^m(\mathcal Q,k) \times \OH^n(\mathcal Q, k) \ar[r]^-{0} \ar[d]_-{\Res_{f}^\sharp \times \Res_{f}^\sharp} & \OH^{m+n-1}(\mathcal Q,k) \ar@<-2pt>[d]^-{\Res_{f}^\sharp} \, \, \\
\OH^m(\mathcal H,k) \times \OH^n(\mathcal H, k) \ar[r]^-{[-,-]_{\mathsf C(\mathcal H)}} \ar[d]_-{\subseteq} & \OH^{m+n-1}(\mathcal H,k) \ar@<-2pt>[d]^-{\subseteq} \, \, \\
\HH^m(A) \times \HH^n(A) \ar[r]^-{\{-,-\}_{A}} & \HH^{m+n-1}(A) \, .
}
$$
\end{proof}
\begin{prop}\label{prop:grouplike}
Let $\mathcal H$ be a Hopf algebra over $k$ whose underlying $k$-algebra $A$ is projective as a $k$-module. Let $G$ be the group $\Gamma(\mathcal H)$ of group like elements. Let $U \subseteq G$ be a subgroup and let $f: \mathcal H \rightarrow kU$ be a Hopf algebra homomorphism. The graded algebra $\OH^\bullet(\mathcal H,k) = \Ext^\bullet_{A}(k,k)$ is a subalgebra of $\HH^\bullet(A)$, and the Gerstenhaber bracket $\{-,-\}_A$ on $\HH^\bullet(A)$ fulfils
$$
\{\Res^\sharp_{f}(\alpha), \Res^\sharp_{f}(\beta)\}_A = 0 \quad \text{$($for all $\alpha, \beta \in \OH^\bullet(kU, k))$.}
$$
Here $\Res_{f}: \Mod(kU) \rightarrow \Mod(H)$ denotes the restriction functor along $f$, and $\Res_{f}^\sharp$ is the induced map between the corresponding cohomology rings.\qed
\end{prop}
\begin{cor}
Let $N \geq 2$ be an integer, $1 \neq \zeta \in k$ a $N$-th root of unity, and let $\mathcal H_{2N} = (H_{2N}, \nabla, \eta, \Delta, \varepsilon, S)$ be the corresponding Taft $($Hopf$)$ algebra described in Example \emph{\ref{exas:qtcocom}}. For any factor $d$ of $N$, there is a $($unital, graded$)$ algebra homomorphism $\psi_d: \OH^\bullet(k\mathbb Z_d,k) \rightarrow \HH^\bullet(H_{2N})$ $($induced by the map $f_d: \mathcal H_{2N} \rightarrow k\mathbb Z_d$, $f_d(g) = g^{N/d}$, $f_d(x) = 0)$ such that
$$
\{\psi_d(\alpha),\psi_d(\beta)\}_{H_{2N}} = 0 \quad \text{$($for all $\alpha, \beta \in \OH^\bullet(k\mathbb Z_d,k))$.}
$$
\end{cor}
\begin{proof}
The group $G = \Gamma(\mathcal H_{2N})$ is isomorphic to $\mathbb Z_N$. Since $H_{2N}$ is projective over $k$ (even free of rank $2N$), the claim follows from Proposition \ref{prop:grouplike} by considering the map $f_d$.
\end{proof}
\begin{exa}
Let $\mathcal H = (H, \nabla, \eta, \Delta, \varepsilon, S)$ be a Hopf algebra whose underlying $k$-module $H$ is finitely generated and projective over $k$. If $V$ is a finitely generated projective $k$-module, we let $V^\ast$ be its $k$-dual $\Hom_k(V,k)$. Observe that for finitely generated projective $k$-modules $U,V,V',W,W'$ there are the following isomorphisms.
\begin{align*}
V^\ast \otimes_k W &\cong \Hom_k(V,W),\\
\Hom_k(U,V) \otimes_k W &\cong \Hom_k(U, V \otimes_k W),\\
\Hom_k(V,W) \otimes_k \Hom_k(V',W') &\cong \Hom_k(V \otimes_k V', W \otimes_k W')
\end{align*}
The latter specializes to the $k$-linear isomorphism $\alpha_V$ below.
$$
\alpha_V: V^\ast \otimes_k V^\ast \longrightarrow (V \otimes_k V)^\ast, \ \varphi \otimes \psi \mapsto (\langle \varphi \otimes \psi, v \otimes w\rangle = \varphi(v) \otimes \psi(w))
$$
The \textit{dual Hopf algebra of $\mathcal H$} is given by
$$
\mathcal H^\ast = (H^\ast, \Delta^\ast, \varepsilon^\ast, \nabla^\ast, \eta^\ast, S^\ast),
$$
where the appearing maps are defined as follows.
\begin{align*}
\Delta^\ast : H^\ast \otimes_k H^\ast &\longrightarrow H^\ast, \ \Delta^\ast(\varphi \otimes \psi) = \alpha_H(\varphi \otimes \psi) \circ \Delta,\\
\varepsilon^\ast : k^\ast &\longrightarrow H^\ast, \ \varepsilon^\ast(\varphi) = \varphi \circ \varepsilon,\\
\nabla^\ast : H^\ast &\longrightarrow H^\ast \otimes_k H^\ast, \ \nabla^\ast(\varphi) = \alpha_H^{-1}(\varphi \circ \nabla)\\
\eta^\ast : H^\ast &\longrightarrow k^\ast, \ \eta^\ast(\varphi) = \varphi \circ \eta\\
\intertext{and}
S^\ast: H^\ast &\longrightarrow H^\ast, \ S^\ast(\varphi) = \varphi \circ S.
\end{align*}
The \textit{Drinfel'd double of $\mathcal H$} is the $k$-Hopf algebra $\mathcal D(\mathcal H)$ whose underlying $k$-module is $H^\ast \otimes_k H$; see \cite{Dr86} and \cite{Mo93} for details on the precise $k$-algebra structure. It can be shown that $\mathcal D(\mathcal H)$ is quasi-triangular, with canonical R-matrix $\mathbf r$ defined by
$$
\xymatrix@C=-3pt@R=11pt{
\Hom_k(H,H) & \cong & H^\ast \otimes_k H \\
\id_{H} \ar@{}[u]|{\mathbin{\rotatebox[origin=c]{90}{$\in$}}} \ar@{|->}[rr] && \mathbf r \ar@{}[u]|{\mathbin{\rotatebox[origin=c]{90}{$\in$}}}
}
$$
The homomorphism $\beta_\H: H \rightarrow H^\ast \otimes_k H$ given by $\beta_\H(h) = 1_{H^\ast} \otimes h$ defines an injective Hopf algebra homomorphism $\mathcal H \rightarrow \mathcal D(\mathcal H)$ (see \cite[Cor.\,10.3.7]{Mo93}). If $\mathcal H$ is cocommutative, then the underlying algebra of $\mathcal D(\H)$ is isomorphic to $H^\ast \# \mathcal H$ (cf. \cite[Cor.\,10.3.10]{Mo93}), where $H^\ast \# \mathcal H$ is the following algebra: Let $A$ be a \textit{$\mathcal H$-module algebra}, that is, a $k$-algebra which simultaneously is a $H$-module, such that both structures are compatible with the structure maps of $\mathcal H$ (see \cite[Def.\,4.1.1]{Mo93}). Then the \textit{smash product} of $A$ with $\mathcal H$, denoted by $A \# \mathcal H$, has $A \otimes_k H$ as underlying $k$-module. The multiplication is given by
$$
(a \otimes h)(a' \otimes h') = a (h_{(1)}b) \otimes h_{(2)}h' \quad (\text{for $a,a' \in A$, $h,h' \in H$}),
$$
i.e., $(a \otimes h)(a' \otimes h') = (a \otimes 1_H) \cdot \Delta(h) \cdot (b \otimes h')$.
\end{exa}
\begin{prop}
Let $\mathcal H$ be a Hopf algebra over $k$ whose underlying $k$-algebra $A$ is, as a $k$-module, finitely generated projective. Let $D$ be the underlying $k$-algebra of $\mathcal D(\mathcal H)$. The graded algebra $\OH^\bullet(\mathcal H,k) = \Ext^\bullet_A(k,k)$ is a subalgebra of $\HH^\bullet(A)$ and the Gerstenhaber bracket $\{-,-\}_{A}$ on $\HH^\bullet(A)$ fulfils
$$
\{\Res^\sharp_{\beta_\H}(\alpha), \Res^\sharp_{\beta_\H}(\beta)\}_A = 0 \quad \text{$($for all $\alpha, \beta \in \OH^\bullet(\mathcal D(\mathcal H), k))$.}
$$
Here $\Res_{\beta_\H}: \Mod(D) \rightarrow \Mod(H)$ denotes the restriction functor along $\beta_\H$, and $\Res_{\beta_\H}^\sharp$ is the induced map between the corresponding cohomology rings.
\end{prop}
\begin{proof}
This is a direct consequence of Proposition \ref{prop:hopfalgvanish}.
\end{proof}
\begin{exa}
Let $V$ be a $k$-module. The tensor algebra $T(V)$,
$$
T(V) = \bigoplus_{i \geq 0}V^{\otimes_k i},
$$
comes with two (in general) non-isomorphic (graded\footnote{A \textit{graded Hopf algebra over $k$} is a $k$-Hopf algebra whose underlying $k$-module is a $\mathbb Z$-graded one, and whose structure maps are homogeneous $k$-linear maps of degree $0$.}) Hopf algebra structures, both with respect to the unit $\eta(\alpha) = \alpha$ for $\alpha \in k$ and counit $\varepsilon(v) = 0$ for $v \in V$. On the one hand, the \textit{concatenation},
$$
\nabla_{\rm c}((v_1 \otimes \cdots \otimes v_r) \otimes (v_{r+1} \otimes \cdots \otimes v_{r+s})) = v_1 \otimes \cdots \otimes v_r \otimes v_{r+1} \otimes \cdots \otimes v_{r+s},
$$
along with the \textit{shuffle coproduct},
$$
\Delta_{\rm sh}(v_1 \otimes \cdots \otimes v_n) = \sum_{p=0}^n \sum_{\sigma \in \mathfrak{S}_{p,n-p}} (v_{\sigma(1)} \otimes \cdots \otimes v_{\sigma(p)}) \otimes (v_{\sigma(p+1)} \otimes \cdots \otimes v_{\sigma(n)})
$$
put a bialgebra structure on $T(V)$, where, for $n \geq 1$ and $0 \leq p \leq n$, $\mathfrak{S}_{p,n-p}$ is the subset of $\mathfrak{S}_n$ containing all permutations $\sigma$ with
$$
\sigma(1) < \sigma(2) < \cdots < \sigma(p) \quad \text{and} \quad \sigma(p+1) < \sigma(p+2) < \cdots < \sigma(n).
$$
An antipode is given by
$$
S(v_1 \otimes \cdots \otimes v_n) = (-1)^n v_n \otimes \cdots \otimes v_1.
$$
Note that $\Delta_{\rm sh}$ and $S$ are induced by the universal property of $T(V)$ applied to the maps $V \rightarrow V \oplus V,\ v \mapsto v \otimes 1 + 1 \otimes v$ (under the identification $V \otimes_k k \cong V \cong k \otimes_k V$) and $V \rightarrow V, \ v \mapsto -v$. Thus they automatically are algebra homomorphisms. With respect to these structure maps, the tensor algebra is a cocommutative Hopf algebra. Now, on the other hand, the \textit{shuffle product},
$$
\nabla_{\rm sh}((v_1 \otimes \cdots \otimes v_r) \otimes (v_{r+1} \otimes \cdots \otimes v_{r+s})) = \sum_{\sigma \in \mathfrak{S}_{r,s}} v_{\sigma(1)} \otimes \cdots \otimes v_{\sigma(r)} \otimes v_{\sigma(r+1)} \otimes \cdots \otimes v_{\sigma(r+s)}
$$
along with the \textit{deconcatenation},
$$
\Delta_{\rm dec}(v_1 \otimes \cdots \otimes v_n) = \sum_{i=0}^n(v_1 \otimes \cdots \otimes v_i) \otimes (v_{i+1} \otimes \cdots \otimes v_n),
$$
turn $T(V)$ into a cocommutative bialgebra as well. The tensor algebra gives rise to several popular graded Hopf algebras by factoring out a certain graded Hopf ideal $\mathcal I$.
\begin{enumerate}[\rm(1)]
\item The \textit{symmetric algebra $\mathrm{Sym}(V)$ of $V$} is obtained from $(T(V), \nabla_{\rm c}, \eta, \Delta_{\rm sh}, \varepsilon, S)$ by factoring out the ideal $\mathcal I$ generated by the elements $v \otimes w - w \otimes v$ for $v,w \in V$. It is a commutative and cocommutative graded Hopf algebra. If $V$ is free of finite $k$-rank $n$, then the underlying $k$-algebra of $\mathrm{Sym}(V)$ is isomorphic to $k[x_1, \dots, x_n]$.
\item The \textit{exterior algebra $\Lambda(V)$ of $V$} is obtained from $(T(V), \nabla_{\rm c}, \eta, \Delta_{\rm sh}, \varepsilon, S)$ by factoring out the ideal $\mathcal I$ generated by the elements $v \otimes v$ for $v \in V$. We write $v_1 \wedge \cdots \wedge v_n$ for the equivalence class of $v_1 \otimes \cdots \otimes v_n$. Note that $\Lambda(V)$ is \textit{strict graded commutative} (i.e., $v \wedge w = -(w \wedge v)$ and $v \wedge v = 0$ for all $v, w \in V$) and cocommutative; further, if $V$ is finitely generated projective, $\Lambda(V^\ast) \cong \Lambda(V)^\ast$ via
$$
f_1 \wedge \cdots \wedge f_n \mapsto \langle f_1 \wedge \cdots \wedge f_n , v_1 \wedge \cdots \wedge v_n\rangle = \mathrm{det}(f_i(v_j)).
$$ 
The isomorphism is such of Hopf algebras (cf. \cite{ABW82}).
\item Let $\mathfrak g$ be a Lie algebra over $k$ with underlying $k$-module $V$. The \textit{universal enveloping algebra $\mathcal U(\mathfrak g)$ of $\mathfrak g$} is obtained from $(T(V), \nabla_{\rm c}, \eta, \Delta_{\rm sh}, \varepsilon, S)$ by factoring out the ideal $\mathcal I$ generated by the elements  $[g,h]_{\mathfrak g} - g \otimes h + h \otimes g$ for $g,h \in \mathfrak g$. It is a cocommutative Hopf algebra.
\end{enumerate}
Note that in all cases, the algebras are projective $k$-modules if, and only if, the $k$-module $V$ (respectively $\mathfrak g$) is projective. All of them have, over $k$, finitely generated components if, and only if, the $k$-module $V$ (respectively $\mathfrak g$) is finitely generated. See \cite[Chap.\,III]{Bou89} for a very detailed treatment of tensor, symmetric and exterior algebras, and \cite[Chap.\,4]{ChPr94} for additional examples of (quasi-triangular) Hopf algebras. For the ones stated above we obtain the following.
\end{exa}
\begin{prop}\label{prop:apptensor}
Let $V$ be a $k$-module and let $\mathfrak g$ be a Lie algebra over $k$. Assume that both $V$ and $\mathfrak g$ are projective over $k$.
\begin{enumerate}[\rm(1)]
\item The cohomology ring $\OH^\bullet(T(V), k) = \Ext^\bullet_{T(V)}(k,k)$ of $T(V)$ identifies with a subalgebra of $\HH^\bullet(T(V))$ and, when restricted to it, the Gerstenhaber bracket $\{-,-\}_{T(V)}$ on $\HH^\bullet(T(V))$ vanishes, that is,
$$
\{\alpha, \beta\}_{T(V)} = 0 \quad \text{$($for all $\alpha, \beta \in \OH^\bullet(T(V),k))$.}
$$
\item The cohomology ring $\OH^\bullet(\mathrm{Sym}(V), k) = \Ext^\bullet_{\mathrm{Sym}(V)}(k,k)$ of $\mathrm{Sym}(V)$ identifies with a subalgebra of $\HH^\bullet(\mathrm{Sym}(V))$ and, when restricted to it, the Gerstenhaber bracket $\{-,-\}_{\mathrm{Sym}(V)}$ on $\HH^\bullet(\mathrm{Sym}(V))$ vanishes, that is,
$$
\{\alpha, \beta\}_{\mathrm{Sym}(V)} = 0 \quad \text{$($for all $\alpha, \beta \in \OH^\bullet(\mathrm{Sym}(V),k))$.}
$$
\item The cohomology ring $\OH^\bullet(\Lambda(V), k) = \Ext^\bullet_{\Lambda(V)}(k,k)$ of $\Lambda(V)$ identifies with a subalgebra of $\HH^\bullet(\Lambda(V))$ and, when restricted to it, the Gerstenhaber bracket $\{-,-\}_{\Lambda(V)}$ on $\HH^\bullet(\Lambda(V))$ vanishes, that is,
$$
\{\alpha, \beta\}_{\Lambda(V)} = 0 \quad \text{$($for all $\alpha, \beta \in \OH^\bullet(\Lambda(V),k))$.}
$$
\item The cohomology ring $\OH^\bullet(\mathcal U(\mathfrak g), k) = \Ext^\bullet_{\mathcal U(\mathfrak g)}(k,k)$ of $\mathcal U(\mathfrak g)$ identifies with a subalgebra of $\HH^\bullet(\mathcal U(\mathfrak g))$ and, when restricted to it, the Gerstenhaber bracket $\{-,-\}_{\mathcal U(\mathfrak g)}$ on $\HH^\bullet(\mathcal U(\mathfrak g))$ vanishes, that is,
$$
\{\alpha, \beta\}_{\mathcal U(\mathfrak g)} = 0 \quad \text{$($for all $\alpha, \beta \in \OH^\bullet(\mathcal U(\mathfrak g),k))$.}
$$
\end{enumerate}
\end{prop}
\begin{proof}
These are special cases of Corollary \ref{cor:gerstenhabervanishhopfalgebra}.
\end{proof}
\begin{cor}
Let $V$ be a finitelty generated projective $k$-module.
\begin{enumerate}[\rm(1)]
\item The exterior algebra $\Lambda(V^\ast)$ of $V^\ast = \Hom_k(V,k)$ identifies with a subalgebra of $\HH^\bullet(\mathrm{Sym}(V))$ and, when restricted to it, the Gerstenhaber bracket $\{-,-\}_{\mathrm{Sym}(V)}$ on $\HH^\bullet(\mathrm{Sym}(V))$ vanishes, that is,
$$
\{v, w\}_{\mathrm{Sym}(V)} = 0 \quad \text{$($for all $v,w \in \Lambda(V^\ast))$.}
$$
\item The symmetric algebra $\mathrm{Sym}(V^\ast)$ of $V^\ast$ identifies with a subalgebra of $\HH^\bullet(\Lambda(V))$ and, when restricted to it, the Gerstenhaber bracket $\{-,-\}_{\Lambda(V)}$ on $\HH^\bullet(\Lambda(V))$ vanishes, that is,
$$
\{\alpha, \beta\}_{\Lambda(V)} = 0 \quad \text{$($for all $\alpha, \beta \in \mathrm{Sym}(V^\ast))$.}
$$
\end{enumerate}
\end{cor}
\begin{proof}
Via classical Koszul duality, the cohomology rings of $\mathrm{Sym}(V)$ and $\Lambda(V)$ compute as $\Lambda(V^\ast)$ and $\mathrm{Sym}(V^\ast)$ respectively (see for instance \cite{Kr13}). The assertion now follows from Proposition \ref{prop:apptensor}.
\end{proof}
\begin{rem}
The above statements yet again illustrate that, in contrast to \cite{Me11}, we neither have to require that $k$ is a field nor that our algebras are finitely generated over $k$ (as $k$-modules) to achieve the vanishing of the restriction of the Gerstenhaber bracket to the cohomology ring.
\end{rem}


\backmatter



\printindex

\end{document}